\documentclass[11pt]{amsart} \textwidth=14.5cm \oddsidemargin=1cm
\usepackage{amsmath,wasysym,array,longtable}
\usepackage{amsxtra}
\usepackage{amscd}
\usepackage{amsthm}
\usepackage{enumitem}
\usepackage{amsfonts}
\usepackage{amssymb}
\usepackage{eucal}
\usepackage{graphics, color}
\usepackage{hyperref,mathrsfs}
\usepackage[dvipsnames]{xcolor}
\usepackage{tikz}
\usepackage{tikz-cd}
\usepackage{stmaryrd}
\usepackage[all]{xy}
\usepackage{hyperref}
\usepackage{color}
\usepackage{bbm} 
\usepackage{upgreek}
\usepackage{rotating}
\usepackage[T1]{fontenc}
\allowdisplaybreaks

\hypersetup{colorlinks,linkcolor=blue, urlcolor=blue,
	citecolor=blue}
\newcommand{\Ql}{\overline{\mathbb Q}_{\ell}}
\newcommand{\IC}{\operatorname{IC}}

\newcommand{\GL}{\operatorname{GL}}

\newcommand{\DD}{\mathbb{D}}
\newcommand{\nDD}{\mathfrak{D}}
\newcommand{\boxt}{\boxtimes}

\newtheorem{thm}{Theorem} [section]
\newtheorem{lem}[thm]{Lemma}
\newtheorem{cor}[thm]{Corollary}
\newtheorem{prop}[thm]{Proposition}

\theoremstyle{definition}

\newtheorem{rem}[thm]{Remark}

\newcommand{\e}[0]{\mathbf{e}}

\newcommand{\f}[0]{\mathbf{f}}

\newcommand{\bd}[0]{\mathbf{d}}

\newcommand{\bt}[0]{\mathbf{t}}

\numberwithin{equation}{section}

\newcommand{\A}{\mathcal A}

\newcommand{\End}{{\mathrm{End}}}

\newcommand{\ff}{\mathtt{f}}
\newcommand{\sF}{{\mathscr F}}

\newcommand{\Hom}{{\mathrm{Hom}}}
\newcommand{\id}{\mathbbm{1}}

\newcommand{\q}{\mathbf{q}}

\newcommand{\T}{\mathbb T}

\newcommand{\Mat}{\mathrm{Mat}}
\newcommand{\nc}{\newcommand}
\nc{\redtext}[1]{\textcolor{red}{#1}}
\nc{\red}[1]{\redtext{ #1}}
\nc{\browntext}[1]{\textcolor{brown}{#1}}
\nc{\greentext}[1]{\textcolor{green}{#1}}
\nc{\bluetext}[1]{\textcolor{blue}{#1}}
\nc{\brown}[1]{\browntext{ #1}}
\nc{\green}[1]{\greentext{ #1}}
\nc{\blue}[1]{\bluetext{ #1}}
\providecommand{\Ql}{\overline{\mathbb Q}_{\ell}}
\providecommand{\IC}{\operatorname{IC}}
\providecommand{\Fq}{\mathbb F_q}
\providecommand{\Fqq}{\mathbb F_{q^2}}
\providecommand{\Gr}{\operatorname{Gr}}

\providecommand{\sh}[1]{\langle #1\rangle}
\providecommand{\const}[1]{\underline{\Ql}_{#1}}
\providecommand{\Flag}{\operatorname{Fl}}

\providecommand{\PP}{\mathbb P}
\providecommand{\tIC}{\widetilde{\IC}}
\providecommand{\cP}{\mathcal P}
\providecommand{\et}{\mathrm{\acute{e}t}}
\providecommand{\sm}[1]{\left(\begin{smallmatrix}#1\end{smallmatrix}\right)}

\title[Non-split groups and $\mathrm{i}$Schur algebras]
{Non-split groups of Lie type and $\mathrm{i}$Schur algebras with unequal parameters}

\author[Li Luo]{Li Luo}
\author[Rui Song]{Rui Song}
\author[Yang Yang]{Yang Yang}
\address{School of Mathematical Sciences, Key Laboratory of MEA (Ministry of Education) \& Shanghai Key Laboratory of PMMP, East China Normal University, Shanghai 200241, China}
\email{lluo@math.ecnu.edu.cn (Luo)}
\email{51265500035@stu.ecnu.edu.cn (Song)} 
\email{52275500011@stu.ecnu.edu.cn (Yang)}

\begin{document}

\begin{abstract}
We present a geometric framework for realizing quantum Schur algebras with unequal parameters, via constructible sheaves on flag varieties arising from non-split groups of Lie type. In the cases of general unitary groups and minus-type orthogonal groups, this framework yields the iSchur algebras with unequal parameters, which are closely related to the quasi-split iquantum groups of type AIII. Such a geometric realization enables the construction of canonical bases, integral forms, and monoidal categorification for these unequal-parameter iSchur algebras and the corresponding iquantum groups. This material is not a formal generalization from the equal-parameter case to the unequal-parameter case. A nontrivial arithmetic structure is imposed on the geometric structure of split groups, giving rise to many new phenomena. For example, the support varieties of generators of an integral form may not be rationally smooth, the categorification is provided by the category of descent data of semisimple complexes fixed by the twisted Frobenius, and the positivity of the canonical basis is established using a refined covering.
\end{abstract}

\maketitle
\tableofcontents

\section{Introduction}
\subsection{Background}
%
In geometric representation theory, one of the initial steps is to realize algebraic structures geometrically. A foundational example is Iwahori's pioneering construction of the Hecke algebra associated with a finite Chevalley group $\mathrm{G}$, as the convolution algebra of $\mathrm{G}$-invariant functions on the product of two complete flag varieties (see \cite{Iw64}), where the parameter $q$ arises as the size of the finite field $\mathbb{F}_q$. This construction was generalized by Beilinson, Lusztig and MacPherson (BLM) to realize quantum Schur algebras of type A in their influential work \cite{BLM90}, using $n$-step flags in place of complete flags. Furthermore, they introduced a stabilization procedure to obtain quantum $\mathfrak{gl}$. See also \cite{Lu99,DF15} for the BLM-type realization of affine Schur algebras and quantum affine $\mathfrak{gl}$.

However, for more than two decades, the construction along the line of Beilinson-Lusztig-MacPherson could not be extended to other classical quantum groups. The reason lies in a fundamental fact: unlike the type A case, quantum orthogonal and symplectic groups do not directly form a Schur-type duality with the corresponding Hecke algebras. After Bao and Wang \cite{BW18} established a Schur-type duality between iquantum groups of type AIII and Hecke algebras of type B/C, this puzzle was eventually solved in \cite{BKLW18}, which shows that the stabilization procedure for quantum Schur algebras of type B/C produces quasi-split iquantum groups of type AIII, rather than quantum orthogonal or symplectic groups. For this reason, we refer to the quantum Schur algebras associated with Hecke algebras of type B/C as iSchur algebras. This geometric construction was later extended to affine iSchur algebras and affine iquantum groups of type AIII in \cite{FLLLW20,FLLLW23}. Related constructions for quantum Schur algebras of type D were studied in \cite{FL15, DLZ25, CF24}.

In contrast to the type A case, the Hecke algebras of type B/C admit two independent parameters, corresponding to the two conjugacy classes of simple reflections. The associated iSchur algebras and iquantum groups also admit two-parameter formulations (cf. \cite{BWW18,LL21}). In \cite{LL21}, Lai and the first author employed these unequal-parameter Hecke algebras to give a BLM-type realization for iSchur algebras and type AIII iquantum groups with two arbitrary parameters. This generalizes the equal-parameter construction in \cite{BKLW18}; see also \cite{LY27} for its affinization with three parameters. These constructions are algebraic and  do not provide a constructible-sheaf model on flag varieties for the unequal-parameter algebras and their canonical bases. 

Although geometric architectures in the spirit of Iwahori's realization have not yet been fully investigated for arbitrary choices of the two parameters, some special unequal-parameter cases are already available for Hecke algebras via non-split groups of Lie type. These unequal parameters encode the Galois action on the root datum of a non-split group. For split groups, this construction recovers the equal-parameter case. Thus, in passing from split to non-split groups, Frobenius affects both the quantum parameter and the algebraic structure.

It is remarkable that Kato \cite{Ka09} provided an equivariant K-theoretic construction of affine Hecke algebras of type C with three arbitrary parameters. This was generalized by Antor \cite{An25} to affine Hecke algebras of types B/C/F/G with two arbitrary parameters. In \cite{LXY26}, the first and third authors, together with Xu, gave an equivariant K-theoretic construction of affine iSchur algebras and iquantum groups with three or two arbitrary parameters based on the achievements of Kato and Antor. Here unequal parameters arise by a different mechanism. They come from the separate equivariant scalings of the summands of exotic representation, rather than from the Galois action induced by field extensions. The relation between coherent and constructible sheaf constructions suggests a form of Langlands reciprocity (cf. Bezrukavnikov's derived equivalence \cite{Be16}). Relating these two mechanisms for unequal parameters is a further motivation for studying the non-split setting.

\subsection{New contributions}

The present paper establishes a geometric framework for iSchur algebras with unequal parameters using non-split groups. 
A general construction is provided in Section~\ref{sec:generalframework}. We then focus on the general unitary groups $\mathbf{U}_D$ and the minus-type orthogonal groups $\mathbf{O}^-_{2d+2}$. The cases $\mathbf{U}_{2d}$, $\mathbf{U}_{2d+1}$, and $\mathbf{O}^-_{2d+2}$ yield three specializations of the two-parameter iSchur algebras. Each case produces an $\imath$-family and a $\jmath$-family. We denote the unequal-parameter iSchur algebras arising from $\mathbf{U}_D$ and $\mathbf{O}^-_{2d+2}$ by $\mathbf{S}^{\mathfrak{c},D}_{n,d}$ and $\mathbf{S}^{\mathfrak{c},O}_{n,d}$, respectively, where $\mathfrak{c}=\imath,\jmath$. We give geometric constructions of their canonical bases and integral forms, and obtain their monoidal categorifications through Frobenius descent.

To the best of our knowledge, this is the first work to study iSchur algebras and
iquantum groups via non-split groups, for which a nontrivial arithmetic structure is imposed on the geometry of split groups. After base change to $\mathbb{F}_{q^2}$, the flag varieties and their orbit closures belong to the split setting. To recover the unequal-parameter algebras, one must retain the action of the twisted Frobenius on their intersection cohomology. This action affects both the canonical-basis coefficients and the convolution structure.
Thus, this paper is neither a straightforward generalization of the work in \cite{BLM90,BKLW18,FL15} from split groups to non-split groups, nor merely a geometric formulation of some results in \cite{LL21}. 

First, the geometry clarifies the extra generator $\mathbf{t}$ in the
iquantum group of type $\mathrm{AIII}_{2n-1}$ and its corresponding
iSchur algebra. In the equal-parameter case, its support variety is
rationally smooth, which gives the sheaf-theoretic description in
\cite[\S5.3]{BKLW18}. In the present setting, this support is no longer
rationally smooth in general. Moreover, in the unitary cases, the
algebraic generator $\mathbf{t}$ from \cite{BWW18,LL21} differs from
the corresponding IC trace functions. We compute the relevant
Kazhdan-Lusztig polynomials and obtain an explicit correction
$\mathbf{t}'$, which is adapted to canonical bases and integral forms.
In the orthogonal case, $\mathbf{t}'=\mathbf{t}$. These computations
require the Frobenius action on the IC stalks, rather than their
dimensions alone. In particular, the resulting polynomials may have
negative coefficients.

Secondly, we construct the monoidal category
$\mathrm{Semis}^{\mathfrak{c},W}_\square$ of Frobenius descent data on
semisimple complexes over $\mathbb{F}_{q^2}$. The characteristic
homomorphism identifies a quotient of its split Grothendieck algebra
with $\mathbf{S}^{\mathfrak{c},\square}_{n,d}$
(Theorem~\ref{cat-main}). The category contains the positive and
negative descents $\widetilde{\mathrm{IC}}_A^\pm$ of $F$-stable IC
complexes, as well as the exchange descents
$P_{\{B,B^\vartheta\}}$ attached to pairs of orbits interchanged by
$F$. The positive descents alone are not closed under convolution.
The kernel of the characteristic homomorphism is generated by the
sums of the positive and negative classes and by the exchange classes.
Through the finite-dimensional iSchur quotients, this also gives a
weak categorification of the corresponding iquantum groups with unequal parameters.

The orthogonal case illustrates why the descent category must retain
more orbit data than the function algebra. On the rational points of
the double flag varieties considered here, the
$\mathbf{O}^-_{2d+2}(\mathbb{F}_q)$-orbits coincide with the
$\mathbf{SO}^-_{2d+2}(\mathbb{F}_q)$-orbits. Over
$\mathbb{F}_{q^2}$, however, some split orthogonal orbits break into
two special orthogonal orbits, which are interchanged by $F$.
We therefore use the refined special orthogonal orbit decomposition
when constructing the descent category.

Keeping the descent data also gives a refined positivity statement.
We introduce a sign parameter $\varpi$ with $\varpi^2=1$ to distinguish
$\widetilde{\mathrm{IC}}_A^+$ from
$\widetilde{\mathrm{IC}}_A^-$. The refined covering records the
positive, negative and exchange summands of a convolution with
nonnegative multiplicities (Theorem~\ref{thm:positivity}). Under the
characteristic homomorphism, the negative descents contribute a minus
sign and the exchange classes vanish. This recovers the
canonical-basis multiplication in the iSchur algebra, whose structure
constants may have negative coefficients. Thus the refined covering
retains the positivity of the decomposition multiplicities before
taking Frobenius traces.

Finally, we construct the bar involution by Verdier duality with a
dimension normalization on each component of the double flag variety.
This normalized duality is compatible with convolution and preserves
the kernel of the characteristic homomorphism. It therefore induces
a $\mathbb{Z}$-algebra involution on
$\mathbf{S}^{\mathfrak{c},\square}_{n,d}$, sending $v$ to $v^{-1}$
and fixing the canonical basis. Under the three specializations, it
agrees with the algebraic bar involution in \cite{LL21} and is
compatible with the iquantum-group bar involutions in \cite{BWW18}.
The mixed-sheaf convolution arguments used in this construction also
apply to the equal-parameter setting. We include the required details
in Appendix~\ref{app:mixed-convolution}.

\subsection{Follow-up work}
Let us briefly discuss some follow-up work that we have begun to prepare.

A natural continuation is an affine generalization of the geometric framework established in this paper. More precisely, we aim to construct affine iSchur algebras and affine iquantum groups with unequal parameters geometrically, using non-split groups over the local field $\mathbb{F}_q((\epsilon))$.

In light of Langlands reciprocity, we also plan to study equivariant K-theoretic realizations of these unequal-parameter iSchur algebras, and then obtain their finite-dimensional modules by Springer theory. 
Several works have generalized Bezrukavnikov's derived equivalence to the setting of non-split reductive groups. For a Frobenius-equivariant formulation in the unramified case,
see \cite[Theorem~5.1 \& Remark~5.2]{Zh25}; for the tamely ramified equal-characteristic case, see \cite[Theorem~0.1.1]{Ci25}. We hope to extend these equivalences to affine iSchur algebras using constructible sheaves on partial flag varieties and coherent sheaves on generalized Steinberg varieties.

\subsection{Organization}
 The paper is organized as follows. 
 
 In Section~\ref{sec:generalframework}, we recall basic facts about Frobenius morphisms and non-split groups of Lie type. We then give a general convolution construction of quantum Schur algebras with unequal parameters and their Schur dualities with Hecke algebras. 

 In Section~\ref{sec:UandO}, we describe the general unitary group $\mathbf{U}_D$ and the minus type orthogonal group $\mathbf{O}^-_{2d+2}$, their twisted Frobenius morphisms, and the associated $N$-step isotropic flag varieties. We also describe the double flag orbits and their behavior under base change and Frobenius.

 In Section~\ref{sec:structure_Sj_U(D)}, we introduce the iSchur algebras $\mathbf{S}_{n,d}^{\mathfrak{c},\square}$ with unequal parameters through $\mathbf{U}_D$ and $\mathbf{O}^-_{2d+2}$. We obtain multiplication formulas for their Chevalley generators. These formulas help us to establish explicit maps from the iSchur algebras with two arbitrary parameters to  $\mathbf{S}_{n,d}^{\mathfrak{c},\square}$.

Section~\ref{sec:canonicalintegral} is devoted to canonical bases and integral forms of $\mathbf{S}_{n,d}^{\mathfrak{c},\square}$. We compute the Frobenius traces on the IC stalks of the support of the extra generator $\mathbf{t}$ and introduce the modified generator $\mathbf{t}'$.

In Section~\ref{sec:categorification}, we classify the $F$-descent data on semisimple complexes, 
and prove that the category $\mathrm{Semis}^{\mathfrak{c},W}_\square$ of all descent data of the $F$-fixed semisimple complexes is closed under convolution. Then we show that the Grothendieck group $K_0^\oplus(\mathrm{Semis}^{\mathfrak{c},W}_\square)$ is a monoidal categorification of $\mathbf{S}^{\mathfrak{c},\square}_{n,d}$. Precisely, we give the characteristic homomorphism, describe its kernel, and establish the positivity property for the refined covering. Finally, we construct the bar involution on $\mathbf{S}^{\mathfrak{c},\square}_{n,d}$ via Verdier duality.

In Section~\ref{sec:iquantumgroups}, we add some material on iquantum groups. Using the projections from the iquantum groups onto the iSchur algebras, a weak categorification for the iquantum groups is obtained, with explicit formulations for generators. 
 
Appendix~\ref{app:mixed-convolution} treats convolution of mixed complexes on flag varieties, and proves the closure statement needed in Section~\ref{sec:iquantumgroups}. As a byproduct, it supplies the mixed-sheaf arguments for the usual monoidal categorifications in the equal-parameter constructions of \cite{BLM90,BKLW18,FL15}. 

Appendix~\ref{AppendixB} gives explicit convolution examples in $K_0^\oplus(\mathrm{Semis}^{\mathfrak{c},W}_\square)$. They exhibit the occurrence of negative IC descents and the exchange descents in convolutions of positive IC descents, and show that the positive descent objects do not form a monoidal subcategory.

\subsection*{Notation and conventions}

Let $\ell\ne p>3$ be two primes, and $q$ a power of $p$. All sheaves and complexes are $\ell$-adic with
$\overline{\mathbb Q}_{\ell}$-coefficients.
All derived categories are taken with these coefficients.
Fix an embedding $\overline{\mathbb Q}_{\ell}\hookrightarrow\mathbb C$
and choose $q^{1/2}$ with a positive image.
The half Tate twist $\overline{\mathbb Q}_{\ell}(1/2)$ is the rank-one
Weil sheaf on $\operatorname{Spec}\mathbb F_q$ on which the geometric
Frobenius acts by $q^{-1/2}$.
For a complex $\mathcal F$, we write
\[
  \mathcal F\langle m\rangle=\mathcal F[m](m/2)
\]
for its normalized shift.

For an $\mathbb F_q$-variety $\mathbf{X}$, the geometric Frobenius morphism on
its base change is $F=F^{\mathrm{abs}}_{q,\mathbf{X}}\times\mathrm{id}$,
where $F^{\mathrm{abs}}_{q,\mathbf{X}}$ is the absolute $q$-Frobenius of $\mathbf{X}$.
We also write $F$ for the induced maps on rational or geometric points.
Both the geometric Frobenius morphism and these induced maps
will be referred to as Frobenius.
The relation with the Frobenius on geometric-point groups is explained in \S\ref{Frobenius} \& \S\ref{subsec:frobenius-revisited}.
Connectedness of varieties means geometric connectedness.
Stalks at rational points are geometric stalks with the Frobenius
actions supplied by the chosen Weil structures.
The symbols $\Gamma$ and $R\Gamma$ denote geometric global sections
and geometric derived global sections, respectively.
They carry the Frobenius actions induced by the chosen Weil structures.

The main notation is collected below.
Detailed definitions are given in the indicated sections.
In Sections~\ref{sec:UandO}--\ref{sec:categorification}, the bold flag and orbit symbols denote varieties,
while the corresponding non-bold symbols denote finite sets of rational points.
Other uses of the same symbol for a variety and its rational points are clear from the context.

\begingroup
\small
\setlength{\tabcolsep}{5pt}
\setlength{\LTpre}{5pt}
\setlength{\LTpost}{5pt}
\renewcommand{\arraystretch}{1.13}
\edef\NotationSavedTable{\number\value{table}}
\begin{longtable}{@{}>{\raggedright\arraybackslash}p{0.30\textwidth}
                    >{\raggedright\arraybackslash}p{\dimexpr0.70\textwidth-10pt\relax}@{}}
\hline
\textbf{Notation} & \textbf{Meaning} \\
\hline
\endfirsthead
\multicolumn{2}{@{}l}{\textit{Notation and conventions (continued)}}\\[3pt]
\hline
\textbf{Notation} & \textbf{Meaning} \\
\hline
\endhead
\hline
\endfoot
\hline
\endlastfoot

\multicolumn{2}{@{}l}{\textit{Parameters and algebraic groups}}\\*[3pt]
$N=2n+1$, $D$
& The number of flag steps and the unitary dimension
  $D=2d$ or $2d+1$ (\S\ref{ex:U_D}--\S\ref{HermitF_D}). \\[3pt]
$\mathfrak c$, $\square$
& $\mathfrak c\in\{\imath,\jmath\}$ labels the two families.
  $\square\in\{D, O\}$ labels the unitary and orthogonal cases;
  the $\square=\mathrm{SO}$ convention is specified below. \\[3pt]
$v$, $\mathbf q$, $\mathcal A$
& Indeterminates $\mathbf q$ and $v$ with $\mathbf q=v^2$, and the coefficient ring
  $\mathcal A=\mathbb Z[v,v^{-1}]$ (\S\ref{subsec:General}).
  The formal parameter $\mathbf q$ is distinct from the field size $q$. \\[3pt]
$\mathbf u$, $\mathbf v$
& The two independent parameters of the generic Hecke, Schur,
  and $\imath$quantum algebras (\S\ref{sec:Hecke} \& \S\ref{subsec:iquantum-groups}). \\[3pt]
$\mathbf G$, $\mathrm G$
& In Section~\ref{sec:generalframework}, the split algebraic group and its geometric-point
  group $\mathrm G=\mathbf G(\overline{\mathbb F}_q)$.
  
  \\[3pt]
$\mathbf G_{q^2}$
& The base change $\mathbf G\times_{\mathbb F_q}\mathbb F_{q^2}$
  (\S\ref{Frobenius}). \\[3pt]
$F_q$, $\tau$, $F$
& In \S\ref{Frobenius}, the standard split Frobenius, the diagram
  automorphism, and the twisted Frobenius $F=F_q\circ\tau$.
  The order of $\tau$ is two in the cases treated after Section~\ref{sec:generalframework}. \\[3pt]
$\mathrm G^F$, $\mathrm G(q)$
& The twisted finite group and the split finite group
  $\mathbf G(\mathbb F_q)$, respectively (\S\ref{Frobenius}). \\[3pt]
$W$, $W^F=W^\tau$, $\ell(\cdot)$
& The Weyl group, the folded Weyl group, and the length function
  on $W$, not on $W^F$ (\S\ref{subsec:nonsplit-groups}). \\[5pt]

\multicolumn{2}{@{}l}{\textit{Flags, orbits, and indices}}\\*[3pt]
$\Lambda_{n+1,d}$
& Weak compositions of $d$ into $n+1$ parts (\S\ref{HermitF_D}). \\[3pt]
$\pmb{\mathscr F}^{\mathfrak c}_{\square}$,
$\mathscr F^{\mathfrak c}_{\square}$
& The flag variety over $\mathbb F_q$ and its set of
  $\mathbb F_q$-rational points (\S\ref{HermitF_D} \& \S\ref{O-flag}). \\[3pt]
$\widetilde{\pmb{\mathscr F}}^{\mathfrak c}_{\square}$,
$\widetilde{\mathscr F}^{\mathfrak c}_{\square}$
& Its base change to $\mathbb F_{q^2}$ and the set of
  $\mathbb F_{q^2}$-rational points (\S\ref{sec:basechange}). \\[3pt]
$\Theta^{\mathfrak c}_{\square}$,
$\widetilde\Theta^{\mathfrak c}_{\square}$
& Double-flag orbit indices over $\mathbb F_q$ and in the split
  setting over $\mathbb F_{q^2}$ (\S\ref{HermitF_D}--\S\ref{sec:basechange}). \\[3pt]
$\pmb{\mathcal O}_A$, $\mathcal O_A$, $\widetilde{\mathcal O}_A$
& The orbit variety over $\mathbb F_q$ and its finite orbit of
  rational points. Tilded symbols denote the corresponding
  objects over $\mathbb F_{q^2}$ (\S\ref{HermitF_D}--\S\ref{sec:basechange}). \\[3pt]
$A^\vartheta$
& The involution on split orbit indices induced by $F$;
  it is central reversal of matrices in the unitary case
  (\S\ref{subsec:frobenius-orbits} \& \S\ref{subsec:so-orbits}). \\[3pt]
$\widetilde\Theta^{\mathfrak c}_{\mathrm{SO}}$
& The refined index set distinguishing the split
  $\mathrm{SO}$-orbits (\S\ref{subsec:so-orbits}).
  In Section~\ref{sec:categorification} and \S\ref{subsec:iquantum-categorification}, $\mathrm{O}$ is replaced by $\mathrm{SO}$;
  the flag varieties and Schur algebras are unchanged. \\[3pt]
$\operatorname{ro}(A)$,
$\operatorname{co}(A)$
& Row-sum and column-sum vectors of $A$ (\S\ref{sec:basechange}). \\[3pt]
$E_{ij}^{\theta}$
& $E_{ij}+E_{N+1-i,N+1-j}$, where $E_{ij}$ is a matrix unit
  (\S\ref{subsec:bruhatorder}). The superscript $\theta$ is distinct from
  the orbit involution $\vartheta$. \\[5pt]

\multicolumn{2}{@{}l}{\textit{Algebras, bases, and generators}}\\*[3pt]
$\mathbb H$, $\mathbb T_\ff$, $\mathbb S_{\mathtt f}$
& The generic two-parameter Hecke algebra, fock space and Schur algebra; 
  (\S\ref{sec:Hecke}). \\[3pt]
$\mathbb S_{n,d}^{\mathfrak c}$,
$\mathbf S_{n,d}^{\mathfrak c,\square}$
& The two-parameter $\imath$Schur algebra and its geometric
  specialization over $\mathcal A$ (\S\ref{flag F_D} \& \S\ref{subsec:specialization}). \\[3pt]
  $\ast$
  & Convolution of geometric Hecke algebra and Schur algebra (\ref{subsec:General}).  \\ [3pt]
$d(A)$, $r(A)$
& The dimension of the orbit indexed by $A$ and of its image
  under the first projection (\S\ref{subsec:normalization-D} \& \S\ref{subsec:normalization-O}). \\[3pt]
$e_A$, $[A]$
& The orbit characteristic function and the normalized standard
  basis element $[A]=v^{-d(A)+r(A)}e_A$
  (\S\ref{flag F_D}, \S\ref{subsec:normalization-D} \& \S\ref{subsec:normalization-O}). \\[3pt]
$P_{B,A}$, $\{A\}$
& The Kazhdan--Lusztig polynomial and the canonical
  basis element $\{A\}=\sum_{B\leq A}P_{B,A}[B]$ (\S\ref{subsec:canonical-basis}). \\[3pt]
$1_\lambda$
& The idempotent corresponding to the diagonal matrix $\lambda$
  (\S\ref{generator}). \\[3pt]
$\mathbf e_i$, $\mathbf f_i$,
$\mathbf d_a^{\pm1}$
& The generators of the geometric $\mathrm{i}$Schur algebras
  (\S\ref{generator}). \\[3pt]
$\mathbf t$, $\mathbf t'$
& The extra generator in the $\imath$-family and its
  modification (\S\ref{generator}--\S\ref{subsec:integral-form}). \\[3pt]
$\mathbb U_n^\jmath$,
$\mathbb U_n^\imath$
& The quasi-split $\imath$quantum groups of types
  $\mathrm{AIII}_{2n}$ and $\mathrm{AIII}_{2n-1}$.
  Their generators $e_i,f_i,d_a^{\pm1},t$ are unbold
  (\S\ref{subsec:iquantum-groups}). \\[5pt]

\multicolumn{2}{@{}l}{\textit{Mixed complexes and categorification}}\\*[3pt]
$D^b_{\mathrm m,\mathbf G}(X)$
& The bounded mixed $\mathbf G$-equivariant derived category
  of $\ell$-adic complexes on $X$ with
  $\overline{\mathbb Q}_{\ell}$-coefficients
  (\S\ref{subsec:ICC}). \\[3pt]
$\widetilde{\mathrm{IC}}_A$,
$\mathrm{IC}_A$, $\mathrm{IC}'_A$
& For $A\in\Theta^{\mathfrak c}_{\square}$, the intersection
  complexes on the split $\mathbb F_{q^2}$-orbit
  closure and its twisted and split $\mathbb F_q$-forms,
  respectively. They restrict to the constant sheaf in degree zero,
  rather than being perverse-normalized (\S\ref{subsec:ICC}). \\[3pt]
$\mathrm{Semis}^{\mathfrak c}_{\square}$
& The additive category generated by
  $\widetilde{\mathrm{IC}}_A\langle m\rangle$ for
  $A\in\widetilde\Theta^{\mathfrak c}_{\square}$ and $m\in\mathbb Z$
  (\S\ref{subsec:semisimple-complexes}). \\[3pt]
$\mathrm{Semis}^{\mathfrak c,F}_{\square}$,
$\mathrm{Semis}^{\mathfrak c,W}_{\square}$
& The full subcategory of $F$-stable objects and the category
  of objects with $F$-descent data, respectively (\S\ref{subsec:semisimple-complexes}). \\[3pt]
$\varphi$, $\Psi_{\mathcal F}$
& The descent datum $\varphi:F^*\mathcal F\xrightarrow{\sim}\mathcal F$
  and the prescribed $\mathbb F_{q^2}$-Weil structure,
  with $\varphi\circ F^*(\varphi)=\Psi_{\mathcal F}$ (\S\ref{subsec:semisimple-complexes}). \\[3pt]
$\widetilde{\mathrm{IC}}_A^{\pm}$,
$P_{\{B,B^\vartheta\}}$
& Positive and negative descents of an $F$-stable IC complex,
  and the exchange descent for a two-element orbit
  (Proposition~\ref{prop:main}). \\[3pt]
$\star$, $\diamond$, $\odot$, $\ast_\mathbf{P}$
& Convolution of complexes on $\mathbb F_{q^2}$-varieties, convolution of descent
  data, convolution on the descended $\mathbb F_q$-variety and twisted-product convolution with the middle variety $\mathbf{G}/\mathbf{P}$,
  respectively (\S\ref{subsec:convolution}, \S\ref{subsec:bar-involution}~\& \S \ref{app:convolution-models}). \\[3pt]
  $\otimes^{\mathbb{L}}$,
  $\boxtimes$, $\widetilde{\boxtimes}$
  & Derived tensor product, external product and twisted external product, respectively (\S \ref{subsec:convolution},~\S\ref{subsec:bar-involution}~\& \S\ref{app:convolution-models}).
  \\[3pt]
$K_0^\oplus(\mathcal C)$
& The split Grothendieck group of an additive category;
  normalized shifts give $v[M]=[M\langle-1\rangle]$
  (\S\ref{subsec:monoidal-categorification}). \\[3pt]
$\chi$, $\mathrm{ch}$
& The Frobenius trace homomorphism at $v=-q^{1/2}$ and
  the generic characteristic homomorphism (\S\ref{subsec:monoidal-categorification}). \\[3pt]
$\mathbb D$,
$\mathfrak D^{\mathfrak c}_{\square}$,
$\overline{(\,\cdot\,)}$
& Verdier duality, its dimension-normalized version,
  and the induced bar involution (\S\ref{subsec:bar-involution}). \\[3pt]
$\mathcal L_A$
& $\widetilde{\mathrm{IC}}_A^+\langle d(A)-r(A)\rangle$,
  with $\mathrm{ch}([\mathcal L_A])=\{A\}$ (\S\ref{subsec:iquantum-categorification}). \\[3pt]
$\mathsf R_{H,J}$
& The unshifted equivariant-descent equivalence between
  the double-flag and single-flag models (\S~\ref{app:convolution-models}). \\
\end{longtable}
\setcounter{table}{\NotationSavedTable}
\endgroup

\subsection*{Acknowledgments}
We thank Weiqiang Wang for his helpful comments on the positivity property using a refined covering. The work is partially supported by the National Key R\&D Program of China (No. 2024YFA1013802) and the NSF of China (Nos. 12371028, 125B2001 and 12631002).

\section{General construction} \label{sec:generalframework}
In this section, we provide a general framework for quantum Schur algebras with unequal parameters. In particular, we recall the Frobenius morphisms that define the $\mathbb{F}_q$-structure of non-split groups, by which we present a geometric realization of several quantum Schur algebras with unequal parameters via convolution algebras of invariant functions on products of flag varieties.

\subsection{Frobenius morphism}\label{Frobenius}
Let $q=p^r$ be a power of a prime number $p>3$, where the condition $p>3$ is imposed to avoid bad characteristic.
The Frobenius map (relative to the finite field $\mathbb{F}_q$ of $q$ elements)
$$
    F_q : \overline{\mathbb{F}_q} \to \overline{\mathbb{F}_q}, \quad x \mapsto x^q
$$
fixes the finite subfield $\mathbb{F}_q \subset \overline{\mathbb{F}_q}$ pointwise.
Let $\mathbf{G}$ be a split connected reductive algebraic group defined over $\mathbb{F}_q$.  
For any $\mathbb{F}_q$-algebra $R$, the set of $R$-points of $\mathbf{G}$ is
$$
    \mathbf{G}(R) = \operatorname{Hom}_{\mathbb{F}_q\text{-alg}}(\mathbb{F}_q[\mathbf{G}], R),
$$
where $\mathbb{F}_q[\mathbf{G}]$ denotes the coordinate ring of $\mathbf{G}$.  
Taking $R = \overline{\mathbb{F}_q}$, we obtain the group of $\overline{\mathbb{F}_q}$-points, which we denote by 
$\mathrm{G}:=\mathbf{G}(\overline{\mathbb{F}_q})$, e.g. the general linear group and the special orthogonal group over $\overline{\mathbb{F}_q}$: $$\mathrm{GL}_D:=\mathbf{GL}_D(\overline{\mathbb{F}_q}) \quad\text{and}\quad \mathrm{SO}_{2d+2}:=\mathbf{SO}_{2d+2}(\overline{\mathbb{F}_q}).$$  
The map $F_q$ induces a group automorphism, called the standard Frobenius morphism relative to $\mathbb{F}_q$ and still denoted by $F_q$, on $\mathrm{G}$ by
$$
F_q\!(g)(x)=F_q(g(x))= g(x)^q, \qquad
\forall g \in \mathrm{G}, x\in\mathbb{F}_q[\mathbf{G}].
$$
equivalently, if $\mathbf{G}$ is embedded as a closed subgroup of $\mathbf{GL}_n$ via equations with coefficients in $\mathbb{F}_q$, then $F_q$ acts entrywise by raising matrix entries to their $q$-th powers.

An endomorphism $F:\mathrm{G}\to \mathrm{G}$ is called a Frobenius morphism if some power $F^a$ is the standard Frobenius morphism relative to $\mathbb{F}_{p^b}$ for some integers $a,b\geq 1$. The fixed-point subgroup 
$$\mathrm{G}^F=\{g\in \mathrm{G}~|~F(g)=g\}$$ is then a finite group of Lie type. In particular, the standard Frobenius morphism $F_q$ itself is a Frobenius morphism, and $\mathrm{G}^{F_q}$ is the split group 
$$\mathrm{G}(q):=\mathbf{G}(\mathbb{F}_q)$$ consisting of the $\mathbb{F}_q$-points of $\mathrm{G}$. 

Other Frobenius morphisms can be obtained as follows.
Fix a maximal torus a Borel subgroup $T\subset B$ of $\mathrm{G}$ that are both $F_q$-invariant. Let $\tau$ be an automorphism of $\mathrm{G}$ defined over $\mathbb{F}_q$, of finite order $a$, preserving $T\subset B$ and inducing a nontrivial automorphism of the Dynkin diagram. Then $\tau$ commutes with $F_q$, so $F=F_q\circ\tau$ satisfies $F^a=F_{q^a}$ and is a Frobenius morphism preserving $T$ and $B$.

In the following, we always consider Frobenius morphisms (called twisted Frobenius morphisms) of the form 
$$F = F_q \circ \tau \quad\text{with $\tau$ inducing a nontrivial graph automorphism}.$$ 
In fact, nontrivial graph automorphisms occur only in the cases of types $A$, $D$ and $E_6$ illustrated in Figures 1-5 below. 

\begin{center}
\label{tab:diagrams}

\begin{tikzpicture}[thick, baseline=(current bounding box.center), scale=0.8]
\matrix [column sep={0.5cm}, row sep={0.35cm}, nodes={draw=none, inner sep=2pt}, ampersand replacement=\&] {
    \node(U1) [draw, circle, fill=white, scale=0.6, minimum size=5mm, inner sep=0pt, label={above:$1$}] {}; 
    \& \node(U2) [draw, circle, fill=white, scale=0.6, minimum size=5mm, inner sep=0pt, label={above:$2$}] {}; 
    \& \node(U3) {$\cdots$}; 
    \& \node(U4) [draw, circle, fill=white, scale=0.6, minimum size=5mm, inner sep=0pt, label={above:$d$}] {}; \\
    \& \& \& \\
    \node(L1) [draw, circle, fill=white, scale=0.6, minimum size=5mm, inner sep=0pt, label={below:$2d$}] {}; 
    \& \node(L2) [draw, circle, fill=white, scale=0.6, minimum size=5mm, inner sep=0pt, label={below:$2d-1$}] {}; 
    \& \node(L3) {$\cdots$}; 
    \& \node(L4) [draw, circle, fill=white, scale=0.6, minimum size=5mm, inner sep=0pt, label={below:$d+1$}] {}; \\
};
\draw (U1) -- (U2);
\draw (U2) -- (U3);
\draw (U3) -- (U4);
\draw (U4) -- (L4);
\draw (L1) -- (L2);
\draw (L2) -- (L3);
\draw (L3) -- (L4);
\draw[color=blue, <->, bend right, shorten >=3pt, shorten <=3pt] (L1) to (U1);
\draw[color=blue, <->, bend right, shorten >=3pt, shorten <=3pt] (L2) to (U2);
\draw[color=blue, <->, bend left, shorten >=3pt, shorten <=3pt] (L4) to (U4);
\node at (current bounding box.south) [yshift=-8pt] {Figure 1: Type $A_{2d}$};
\end{tikzpicture}
\hspace{0.5cm}
\begin{tikzpicture}[thick, baseline=(current bounding box.center), scale=0.8]
\matrix [column sep={0.5cm}, row sep={0.15cm}, nodes={draw=none, inner sep=2pt}, ampersand replacement=\&] {
    \node(U1) [draw, circle, fill=white, scale=0.6, minimum size=5mm, inner sep=0pt, label={above:$1$}] {}; 
    \& \node(U2) [draw, circle, fill=white, scale=0.6, minimum size=5mm, inner sep=0pt, label={above:$2$}] {}; 
    \& \node(U3) {$\cdots$}; 
    \& \node(U4) [draw, circle, fill=white, scale=0.6, minimum size=5mm, inner sep=0pt, label={above:$d-1$}] {}; \\
    \& \& \& \& \node(C) [draw, circle, fill=white, scale=0.6, minimum size=5mm, inner sep=0pt, label={right:$d$}] {}; \\
    \node(L1) [draw, circle, fill=white, scale=0.6, minimum size=5mm, inner sep=0pt, label={below:$2d-1$}] {}; 
    \& \node(L2) [draw, circle, fill=white, scale=0.6, minimum size=5mm, inner sep=0pt, label={below:$2d-2$}] {}; 
    \& \node(L3) {$\cdots$}; 
    \& \node(L4) [draw, circle, fill=white, scale=0.6, minimum size=5mm, inner sep=0pt, label={below:$d+1$}] {}; \\
};
\draw (U1) -- (U2);
\draw (U2) -- (U3);
\draw (U3) -- (U4);
\draw (L1) -- (L2);
\draw (L2) -- (L3);
\draw (L3) -- (L4);
\draw (U4) -- (C);
\draw (C) -- (L4);
\draw[color=blue, <->, bend right, shorten >=3pt, shorten <=3pt] (L1) to (U1);
\draw[color=blue, <->, bend right, shorten >=3pt, shorten <=3pt] (L2) to (U2);
\draw[color=blue, <->, bend left, shorten >=3pt, shorten <=3pt] (L4) to (U4);
\node at (current bounding box.south) [yshift=-8pt] {Figure 2: Type $A_{2d-1}$};
\end{tikzpicture}

\vspace{0.4cm}

\begin{tikzpicture}[thick, baseline=(current bounding box.center), scale=0.8]
\node (U1) at (0,0) [draw, circle, fill=white, scale=0.6, minimum size=5mm, inner sep=0pt, label={above:$1$}] {};
\node (U2) at (1.0,0) [draw, circle, fill=white, scale=0.6, minimum size=5mm, inner sep=0pt, label={above:$2$}] {};
\node (U3) at (2.0,0) [draw=none] {$\cdots$};
\node (U4) at (3.0,0) [draw, circle, fill=white, scale=0.6, minimum size=5mm, inner sep=0pt, label={above:$d-1$}] {};
\node (D1) at (4.0,0.6) [draw, circle, fill=white, scale=0.6, minimum size=5mm, inner sep=0pt, label={above:$d$}] {};
\node (D2) at (4.0,-0.6) [draw, circle, fill=white, scale=0.6, minimum size=5mm, inner sep=0pt, label={below:$d+1$}] {};
\draw (U1) -- (U2);
\draw (U2) -- (U3);
\draw (U3) -- (U4);
\draw (U4) -- (D1);
\draw (U4) -- (D2);
\draw[color=blue, <->, bend left=45, shorten >=3pt, shorten <=3pt] (D1) to (D2);
\node at (current bounding box.south) [yshift=-8pt] {Figure 3: Type $D_{d+1}$};
\end{tikzpicture}
\hspace{0.3cm}
\begin{tikzpicture}[thick, baseline=(current bounding box.center), scale=0.8]
\node (L1) at (0,0) [draw, circle, fill=white, scale=0.6, minimum size=5mm, inner sep=0pt, label={above:$1$}] {};
\node (L2) at (0.9,0) [draw, circle, fill=white, scale=0.6, minimum size=5mm, inner sep=0pt, label={above:$2$}] {};
\node (U1) at (1.8,0.6) [draw, circle, fill=white, scale=0.6, minimum size=5mm, inner sep=0pt, label={above:$3$}] {};
\node (U2) at (2.7,0.6) [draw, circle, fill=white, scale=0.6, minimum size=5mm, inner sep=0pt, label={above:$4$}] {};
\node (D1) at (1.8,-0.6) [draw, circle, fill=white, scale=0.6, minimum size=5mm, inner sep=0pt, label={below:$5$}] {};
\node (D2) at (2.7,-0.6) [draw, circle, fill=white, scale=0.6, minimum size=5mm, inner sep=0pt, label={below:$6$}] {};
\draw (L1) -- (L2);
\draw (L2) -- (U1);
\draw (U1) -- (U2);
\draw (L2) -- (D1);
\draw (D1) -- (D2);
\draw[color=blue, <->, bend left=30, shorten >=3pt, shorten <=3pt] (U1) to (D1);
\draw[color=blue, <->, bend left=30, shorten >=3pt, shorten <=3pt] (U2) to (D2);
\node at (current bounding box.south) [yshift=-8pt] {Figure 4: Type $E_6$};
\end{tikzpicture}
\hspace{0.3cm}
\raisebox{-0.25cm}{
\begin{tikzpicture}[thick, baseline=(current bounding box.center), scale=0.8]
\node (U1) at (0,0) [draw, circle, fill=white, scale=0.6, minimum size=5mm, inner sep=0pt, label={above:$2$}] {};
\node (D1) at (0.9,0) [draw, circle, fill=white, scale=0.6, minimum size=5mm, inner sep=0pt, label={above:$1$}] {};
\node (U2) at (1.8,0.6) [draw, circle, fill=white, scale=0.6, minimum size=5mm, inner sep=0pt, label={below:$3$}] {};
\node (U3) at (1.8,-0.6) [draw, circle, fill=white, scale=0.6, minimum size=5mm, inner sep=0pt, label={below:$4$}] {};
\draw (U1) -- (D1);
\draw (D1) -- (U2);
\draw (D1) -- (U3);
\draw[color=blue, ->, bend left=30, shorten >=3pt, shorten <=3pt] (U1) to (U2);
\draw[color=blue, ->, bend left=30, shorten >=3pt, shorten <=3pt] (U2) to (U3);
\draw[color=blue, ->, bend left=30, shorten >=3pt, shorten <=3pt] (U3) to (U1);
\node at (current bounding box.south) [yshift=-8pt] {Figure 5: Type $D_4$};
\end{tikzpicture}}

\end{center}

After Section~\ref{sec:generalframework}, we will focus on the cases associated with Figures 1-3, in which the order of $\tau$ is always $2$. Since we will frequently use descent from $\mathbb{F}_{q^2}$ to $\mathbb{F}_q$, we sometimes write $\mathbf{G}_{q^2}:=\mathbf{G}\times_{\mathbb{F}_q}\mathbb{F}_{q^2}$ for short.

 \subsection{Non-split groups of Lie type}\label{subsec:nonsplit-groups}
Let $r$ be the semisimple rank of the $\mathrm{G}$.
Let $Q=\mathrm{Hom}(T,\mathbb{G}_m)$ be the weight lattice of $\mathrm{G}$. Denote by $\Phi$ and $\Delta=\{\alpha_1,\alpha_2,\ldots,\alpha_{r}\}\subset\Phi$ the root system and the set of simple roots (with respect to $T\subset B$), respectively. Let $W = N_{\mathrm{G}}(T)/T$ be the Weyl group of $\mathrm{G}$ associated with $T$. The conjugation action of $N_{\mathrm{G}}(T)$ on $T$ induces a right action of $W$ on $Q$ by $(\alpha \cdot w) (t)=\alpha(\dot{w}t\dot{w}^{-1})$ for $\alpha\in Q, w\in W, t\in T$, where $\dot{w}\in N_{\mathrm{G}}(T)$ is a representative of $w$. Extending this action on the real vector space $E:=\mathbb{R}Q$, we identify $W$ as the reflection group of $E$ generated by simple reflections $s_i:=s_{\alpha_i} (i=1,2,\ldots,r)$. Denote $S=\{s_1, s_2,\ldots,s_r\}$.

Take the Frobenius $F=F_q\circ \tau$ with nontrivial $\tau$. The $F$-fixed point subgroup $\mathrm{G}^F$ of $\mathrm{G}$ is called a non-split group (or twisted group) of Lie type. Similarly, we denote the $F$-fixed point subgroups of $B$ and $T$ by $B^F$ and $T^F$, respectively. Since $T$ is $F$-invariant, we can define an $F$-action on $W = N_{\mathrm{G}}(T)/T$ by $F(gT) = F(g)T$. The associated $F$-fixed point subgroup is denoted by $W^F$. It is known (cf. \cite[Theorem 24.10]{MT11}) that $(B^F, N_{\mathrm{G}}(T)^F)$ is a $BN$-pair in $\mathrm{G}^F$ with the Weyl group $W^F$, which implies the Bruhat decomposition for $\mathrm{G}^F$ (refer to \cite[Proposition 4.4.1]{DM20}):
\begin{equation}\label{Bruhat}
    \mathrm{G}^F = \bigsqcup_{w\in W^F} B^F\dot{w} B^F,\quad \mbox{with cardinality $\left|B^F\dot{w}B^F\right| = q^{\ell(w)}\left|B^F\right|$},
\end{equation}
where $\ell(w)$ is the length function of $W$ (NOT of $W^F$).
Moreover, since the standard Frobenius map $F_q$ acts trivially on $W$, $W^F=W^\tau$ is also the $\tau$-fixed point subgroup.

Write $\{\mathcal{I}_1,\ldots, \mathcal{I}_d\}$ for the set of all $\tau$-orbits in $\Delta$. The parabolic subgroup $W_i$ of $W$ generated by the reflections $s_\alpha (\alpha \in \mathcal{I}_i)$ has a unique longest element $t_i$. It is known (refer to \cite[Lemma~C.1]{MT11}) that $W^F$ can be identified as the Weyl group associated with the folding Dynkin diagram, whose simple reflections are $\{t_1,t_2,\ldots, t_d\}$. We list these data in the following table.
\begin{table}[h!]
    \centering
    \caption{}
    \begin{tabular}{|c|c|c|}
        \hline
        $W$ & $W^F$ & $t_1, t_2,\ldots,t_d$ \\
        \hline
        $A_{2d}$ & $B_d$ & $s_1s_{2d},s_2s_{2d-1},\ldots, s_{d-1}s_{d+2}, s_ds_{d+1}s_d$ \\
        \hline
        $A_{2d-1}$ & $B_d$ & $s_1s_{2d-1},s_2s_{2d-2},\ldots,s_{d-1}s_{d+1},s_d$ \\
        \hline
        $D_{d+1}$ & $B_d$ & $s_1,s_2,\ldots,s_{d-1},s_ds_{d+1}$ \\
        \hline
        $E_6$ & $F_4$ & $s_1, s_2, s_3s_5, s_4s_6$ \\
        \hline
        $D_4$ & $G_2$ & $s_1, s_2s_3s_4$ \\
        \hline
    \end{tabular}
\label{tab1}
\end{table}

The $\tau$-action on $\Delta$ can be naturally extended to $Q$ (and hence to $E=\mathbb{R}Q$). As a formal expression, the $\tau$-action of $Q$ is given by (refer to \cite[Proposition 22.2]{MT11})
$$(\tau \alpha) (t)=\alpha(\tau t) \quad \mbox{for $\alpha\in Q=\mathrm{Hom}(T,\mathbb{G}_m)$ and $t\in T$}.$$
Let $Q^\tau\subset Q$ denote the sublattice of $\tau$-fixed points in $Q$, which is $W^F$-invariant because for any $w\in W^F=W^\tau$, $\alpha\in Q^\tau$ and $t\in T$,
\begin{align}\label{WFonQF}
(\alpha\cdot w)(t) &= (\tau\alpha\cdot w)(t)=\tau \alpha(\dot{w}t\dot{w}^{-1})=\alpha(\tau(\dot{w}t\dot{w}^{-1}))\\
&=\alpha(\dot{w}(\tau t)\dot{w}^{-1})=(\alpha\cdot w)(\tau t)=(\tau (\alpha\cdot w))(t). \nonumber
\end{align}
Furthermore, it is known (cf. \cite[Lemma~C.3]{MT11}) that 
\begin{equation}\label{eq:ti=s}
t_i|_{Q^{\tau}}=s_{\beta_i}|_{Q^{\tau}}\quad (\forall i=1,2,\ldots,d), \quad \mbox{where $\beta_i:=\sum_{\alpha\in \mathcal{I}_i}\alpha$}.
\end{equation}

We denote by $\Lambda$ the set of all $W^F$-orbits in $Q^\tau$.

\begin{lem}
    For each $W^F$-orbit $\gamma\in\Lambda$, there exists a unique element $\mathbf{i}_\gamma \in \gamma$ such that $(\mathbf{i}_\gamma, \beta_i)\leq0$ ($\forall i=1,2,\ldots,d$), where $(\cdot,\cdot)$ is the standard bilinear form on $E\simeq \mathbb{R}^D$ induced from the Killing form.
\end{lem}
\begin{proof}
    Thanks to \eqref{eq:ti=s}, the sublattice $Q^\tau$ affords an action of the Weyl group $W^F$ with simple reflections $s_{\beta_i}$ ($i=1,2,\ldots,d$). Thus, there is a unique anti-dominant weight $\mathbf{i}_\gamma$ in each $W^F$-orbit $\gamma\in\Lambda$, saying $(\mathbf{i}_\gamma, \beta_i)\leq0$ ($\forall i=1,2,\ldots,d$).
\end{proof}

For $\gamma\in\Lambda$, let $W_\gamma^F$ denote the parabolic subgroup of $W^F$ generated by $\{ t_i \mid 1\leq i\leq d, \mathbf{i}_\gamma \cdot t_i= \mathbf{i}_\gamma\}$, and let $P_\gamma = \bigsqcup_{w\in W_\gamma^F} B^F\dot{w} B^F$ be a parabolic subgroup of $\mathrm{G}^F$ associated with $\gamma$. In particular, if $\gamma$ is a regular $W^F$-orbit, then $W_\gamma^F=\{\id\}$ is trivial and $P_\gamma=B^F$. By definition, the parabolic subgroup $P_\gamma$ is $F$-stable.

\subsection{Hecke algebras and Schur algebras}\label{sec:Hecke}
Let $\mathbf{u},\mathbf{v}$ be two indeterminates. Since the folding Dynkin diagram induced by a nontrivial $\tau$ yields a root system with both long and short roots, the Hecke algebra for $W^F$ comes with a two-parameter formulation, which is a $\mathbb{Z}[\mathbf{u}^{\pm1},\mathbf{v}^{\pm1}]$-algebra $\mathbb{H}$ generated by $T_1,T_2,\ldots,T_d$ subject to the following:
\begin{itemize}
    \item quadratic relations:
        $$(T_i+1)(T_i-\mathbf{v}^2)=0 \ (i < d), \quad (T_d+1)(T_d-\mathbf{u}^2)=0$$ if $W^F$ is of type $B_d/C_d$;
        $$(T_i+1)(T_i-\mathbf{v}^2)=0 \ (i=1,2), \quad  (T_i+1)(T_i-\mathbf{u}^2)=0\ (i=3,4)$$ if $W^F$ is of type $F_4$;
        $$(T_1+1)(T_1-\mathbf{v}^2)=0, \quad  (T_2+1)(T_2-\mathbf{u}^2)=0 $$ if $W^F$ is of type $G_2$.
    \item braid relations:
        \begin{gather*}T_i T_j = T_j T_i \ (\mid i-j \mid > 1),\\
        T_iT_{i+1}T_i = T_{i+1}T_iT_{i+1} \ (i < d-1), \quad (T_{d-1}T_d)^2 = (T_dT_{d-1})^2
        \end{gather*} if $W^F$ is of type $B_d/C_d$;
        \begin{gather*}
        T_iT_j=T_jT_i \ (|i-j|>1),\\
        T_iT_{i+1}T_i = T_{i+1}T_iT_{i+1} \ (i = 1,3), \quad (T_2T_3)^2 = (T_3T_2)^2
        \end{gather*} if $W^F$ is of type $F_4$;
        $$(T_1T_2)^3 = (T_2T_1)^3$$ if $W^F$ is of type $G_2$.
\end{itemize}
As usual, for $w\in W^F$ we denote $T_w=T_{i_1}\ldots T_{i_l}$ if $w=t_{i_1}\ldots t_{i_l}$ is a reduced form in $W^F$. It is known that $\{T_w~|~w\in W^F\}$ forms a $\mathbb{Z}[\mathbf{u}^{\pm1},\mathbf{v}^{\pm1}]$-basis of $\mathbb{H}$.

Fix a finite $W^F$-invariant subset $Q_{\ff}\subset Q^\tau$. Denote the set of $W^F$-orbits in $Q_{\ff}$ by
    $$\Lambda_{\ff} = \{W^F\mbox{-orbits~in~}Q_{\ff}\}\subset\Lambda.
    $$
Let $\mathbb{T}_\ff=\bigoplus_{\gamma\in\Lambda_\ff}x_\gamma\mathbb{H}$ with $x_\gamma=\sum_{w\in W_\gamma^F}T_w$, which carries a natural right $\mathbb{H}$-action. The two-parameter Schur algebra is defined as
\begin{equation}\label{def:bbSf}
\mathbb{S}_\ff=\mathrm{End}_{\mathbb{H}}(\mathbb{T}_\ff)=\bigoplus_{\gamma,\nu\in\Lambda_\ff}\mathrm{Hom}_{\mathbb{H}}(x_\nu\mathbb{H},x_\gamma\mathbb{H}),
\end{equation}
which depends on the choice of $\Lambda_\ff$.

 Let $\mathcal{D}_{\gamma\nu}$ (resp. $\mathcal{D}_\gamma$) be the set of minimal length representatives of the double coset $W^F_\gamma\setminus W^F/W^F_\nu$ (resp. the left coset $W^F_\gamma\setminus W^F$).
It is a standard statement that, as free $\mathbb{Z}[\mathbf{u}^{\pm1},\mathbf{v}^{\pm1}]$-modules, $\mathbb{T}_\ff$ has a basis $\{x_\gamma T_w~|~\gamma\in\Lambda_{\ff}, w\in\mathcal{D}_\gamma\}$ and $\mathbb{S}_{\ff}$ has a basis $\{e_{\gamma\nu}^{w} ~|~ \gamma,\nu\in\Lambda_\ff, w\in\mathcal{D}_{\gamma\nu}\}$, where
$$e_{\gamma\nu}^{w}\in \mathrm{Hom}_{\mathbb{H}}(x_\nu\mathbb{H},x_\gamma\mathbb{H}): \quad x_\nu\to\sum_{w'\in W_\gamma^Fw W_\nu^F}T_{w'}.$$

\subsection{Geometric construction in general}\label{subsec:General}
    Denote 
    $$
        \mathscr{B} = \mathrm{G}^F/B^F,\quad \mathscr{F}_\gamma=\mathrm{G}^F/P_\gamma \quad\mbox{and}\quad \mathscr{F}_{\ff} = \bigsqcup_{\gamma\in\Lambda_{\ff}} \mathscr{F}_\gamma.
    $$
The diagonal $\mathrm{G}^F$-action on $\mathscr{B\times B}$ induces a $\mathrm{G}^F$-action on the space of $\mathbb{Z}$-valued functions on $\mathscr{B\times B}$ by
    $$g\cdot f(xB^F, yB^F) = f(g^{-1}xB^F, g^{-1}yB^F)$$ for any function $f$ on $\mathscr{B\times B}$, $g\in \mathrm{G}^F$ and $(xB^F, yB^F)\in\mathscr{B\times B}$. 
    Denote $$\mathcal{H}:=\mathbb{Z}_{\mathrm{G}^F}[\mathscr{B\times B}]$$ the $\mathbb{Z}$-algebra of $\mathrm{G}^F$-invariant $\mathbb{Z}$-valued functions on $\mathscr{B\times B}$, whose convolution product $*$ is defined by
    \begin{equation}\label{equ:Hecke-convolution-Borel}
        f * g (B^F, \dot{w}B^F) = \sum_{xB^F\in \mathrm{G}^F/B^F} f(B^F, xB^F) g(xB^F, \dot{w}B^F).
    \end{equation}
The $\mathrm{G}^F$-orbits on $\mathscr{B\times B}$ match the double coset decomposition of $\mathrm{G}^F$ of $B^F$, hence can be indexed by $W^F$ (independent of $q$) thanks to the Bruhat decomposition \eqref{Bruhat}. Therefore, $\mathcal{H}$ has a basis $\{e_w|w\in W^F\}$, where $e_w$ is the characteristic function of the $\mathrm{G}^F$-orbit on $\mathscr{B\times B}$ associated with $w\in W^F$.
    
    Moreover, we can introduce a $\mathbb{Z}$-algebra $\mathcal{S}_\ff:=\mathbb{Z}_{\mathrm{G}^F}[\mathscr{F}_\ff\times \mathscr{F}_\ff]$ and a $(\mathcal{S}_\ff,\mathcal{H})$-bimodule $\mathcal{T}_\ff:=\mathbb{Z}_{\mathrm{G}^F}[\mathscr{F}_\ff\times \mathscr{B}]$ with similar convolution products. 
   
   The $\mathrm{G}^F$-orbits on $\mathscr{F}_\gamma\times \mathscr{F}_\nu$ (resp., on $\mathscr{F}_\gamma\times \mathscr{B}$) are indexed by $\mathcal{D}_{\gamma\nu}$ (resp. $\mathcal{D}_\gamma$), also independent of $q$. Thus, $\mathcal{S}_\ff$ (resp. $\mathcal{T}_\ff$) has a basis $\{e_{\gamma,w,\nu}|\gamma,\nu\in\Lambda_\ff, w\in\mathcal{D}_{\gamma\nu}\}$ (resp. $\{e_{\lambda,w}|\lambda\in\Lambda_\ff,w\in\mathcal{D}_\lambda\}$), where $e_{\gamma,w,\nu}$ and $e_{\lambda,w}$ are the corresponding characteristic functions.

Let $v$ and $\q$ be indeterminates with $\q=v^2$. Denote $\mathcal{A}=\mathbb{Z}[v,v^{-1}]$. 
It is a well-known property of these convolution products that there exist polynomials in $\mathbb{Z}[v]$ such that, for every odd prime power $q$, the structure constants for the above basis elements are given by evaluating these polynomials at $\q=v^2=q$. Thus, we can introduce $\mathcal{A}$-algebras 
\begin{equation*}
\mathbf{H}:=\mathcal{A}_{\mathrm{G}^F}[\mathscr{B\times B}]\quad \text{and}\quad \mathbf{S}_\ff:=\mathcal{A}_{\mathrm{G}^F}[\mathscr{F}_\ff\times \mathscr{F}_\ff],
\end{equation*}
and an $(\mathbf{S}_\ff,\mathbf{H})$-bimodule 
\begin{equation*}
\mathbf{T}_\ff:=\mathcal{A}_{\mathrm{G}^F}[\mathscr{F}_\ff\times \mathscr{B}],
\end{equation*}
with bases $\{e_w\}, \{e_{\gamma,w,\nu}\}$ and $\{e_{\lambda,w}\}$, respectively. Their multiplication and action structure constants are given by the above polynomials in $\mathbb{Z}[v]$.

The following theorem tells us that $\mathbf{S}_\ff$, $\mathbf{T}_\ff$ and $\mathbf{H}$ are, respectively, the specializations of $\mathbb{S}_\ff$, $\mathbb{T}_\ff$ and $\mathbb{H}$ with the corresponding unequal parameters.

\begin{thm}[Geometric Schur duality] \label{thm:special}
   We have the following commutative diagram: 
    \begin{eqnarray*}
\begin{array}{ccccc}
  \mathbb{S}_\ff &\circlearrowright
 & \T_{\ff} & \circlearrowleft  \; & \mathbb{H}
  \\
  \downarrow\;\; & & \downarrow & & \downarrow
 \\
\mathbf{S}_{\ff}     & \circlearrowright
& \mathbf{T}_{\ff} &  \circlearrowleft  \; & \mathbf{H}
\end{array}
 \end{eqnarray*}
 where the arrows $\downarrow$ are given, respectively, by 
  \begin{equation*}
  e_{\gamma\nu}^w\mapsto e_{\gamma,w,\nu},\quad x_{\gamma}T_w \mapsto e_{\gamma,w}, \quad T_w\mapsto e_w 
  \end{equation*}
  with the following specialization:  
 \begin{itemize}
 \item $\mathbf{u}^2=\q^3,\mathbf{v}^2=\q^2$ for Figure 1;
\item
    $\mathbf{u}^2=\q,\mathbf{v}^2=\q^2$ for Figure 2;
\item        $\mathbf{u}^2=\q^2,\mathbf{v}^2=\q$ for Figure 3;
\item        $\mathbf{u}^2=\q^2,\mathbf{v}^2=\q$ for Figure 4;
\item        $\mathbf{u}^2=\q^3,\mathbf{v}^2=\q$ for Figure 5.
\end{itemize} 
\end{thm}
\begin{proof}
Since $|B^F\dot{w}B^F/B^F|=q^{\ell(w)}$ by \eqref{Bruhat}, it is a theorem due to Iwahori and Matsumoto (cf. \cite[\S8.4.6]{GP00}) that $\mathbf{H}$ is a specialization of $\mathbb{H}$ with the corresponding unequal parameters by identifying $T_w \rightarrow e_w$ for each $w\in W^F$. This is proven by checking the quadratic and braid relations for the generators $e_{t_1},\ldots, e_{t_d}$ directly. Similarly, we can also directly check that, as a right $\mathbf{H}$-module, $\mathbf{T}_\ff$ is a specialization of the $\mathbb{H}$-module $\mathbb{T}_\ff$  with the corresponding unequal parameters by identifying $x_\lambda T_w\rightarrow e_{\lambda,w}$.

The proof of the corresponding specialization $\mathbb{S}_\ff\to\mathbf{S}_\ff$ is standard but cumbersome. We omit it here and refer to the proof of \cite[Theorem~13.15]{DDPW08}. The proof for the validity of $\mathbb{T}_\ff$ (as an $\mathbb{S}_\ff$-module) to $\mathbf{T}_\ff$ (as an $\mathbf{S}_\ff$-module) is almost the same.
\end{proof}

\begin{rem}\label{rem:S=Special}
In the following, we focus on the cases of Figures 1, 2 and 3, for which we will choose special $\Lambda_\ff$ related to the corresponding $N$-step isotropic flag varieties (see the next section). Under this choice of $\Lambda_\ff$, the two-parameter quantum Schur algebras defined in \eqref{def:bbSf} are precisely the iSchur algebras $\mathbb{S}_{n,d}^\jmath$ and $\mathbb{S}_{n,d}^\imath$ studied in \cite{BWW18,LL21}. By the above theorem, we obtain geometric realizations of these iSchur algebras at the three specializations corresponding to Figures 1, 2 and 3, respectively.
\end{rem}

\subsection{Revisit Frobenius morphism}\label{subsec:frobenius-revisited}
We conclude this section by revisiting Frobenius morphisms and their orbits, which will be used in subsequent sections.

\begin{lem}\label{connected}
    The intersection of any two parabolic subgroups of a connected reductive group $\mathrm{G}=\mathbf{G}(\overline{\mathbb{F}_q})$ is connected. In particular, all stabilizers for the diagonal action of $\mathrm{G}$ on a product of two partial flag varieties are connected.
\end{lem}
\begin{proof}
The first assertion follows from
\cite[Corollary~4.5]{BT65}, applied over
$\overline{\mathbb{F}_q}$.
The stabilizer of $(g_1P,g_2Q)\in \mathrm{G}/P\times \mathrm{G}/Q$ is
$g_1Pg_1^{-1}\cap g_2Qg_2^{-1}$, which is connected by the first assertion.
\end{proof}

\begin{prop}\label{prop:rational-orbits}
Let $\mathrm{G}=\mathbf{G}(\overline{\mathbb{F}_q})$ equip with a Frobenius $F$. Let $P,Q\subset \mathrm{G}$ be $F$-stable parabolic subgroups. Put $X=\mathrm{G}/P\times \mathrm{G}/Q$, then for any $x\in X^F$,
$$
(\mathrm{G}\cdot x)^F=\mathrm{G}^F\cdot x.
$$
In particular, two $F$-fixed points of $X$ lie in the same $\mathrm{G}^F$-orbit if and only if they lie in the same geometric $\mathrm{G}$-orbit.\end{prop}
\begin{proof} Let $x,y\in X^F$ and suppose that $y=g\cdot x$ for some $g\in \mathrm{G}$. The stabilizer $H=\mathrm{G}_x$ is connected by Lemma~\ref{connected} and is $F$-stable because $F(x)=x$. Since $F(y)=y$, we have
$g^{-1}F(g)\in H$. By Lang's theorem (cf. \cite[Theorem~17.96]{Mi17}), there exists $h\in H$ such that $g^{-1}F(g)=hF(h)^{-1}$. Consequently, $F(gh)=gh$ and $gh\cdot x=g\cdot x=y$. Thus $y\in \mathrm{G}^F\cdot x$, proving the assertion. 
\end{proof}

Recall $q=p^r$. Let $\mathbf{X}$ be a variety over $\mathbb{F}_q$. We write
$F_\mathbf{X}: \mathbf{X} \to \mathbf{X}$ for the absolute Frobenius morphism, which is the identity on the underlying topological space and acts on functions by $f \mapsto f^p$. Its $r$-th iterate
$F_{\mathbf{X},q} := F_\mathbf{X}^r$
is called the $q$-Frobenius morphism of $\mathbf{X}$.

If $\mathbf{X}$ is a split reductive group $\mathbf G$, this morphism induces the standard Frobenius $F_q$ on $\mathrm{G}=\mathbf G(\overline{\mathbb{F}_q})$.
In fact, for $g\in \mathrm{G}$ and $f\in \mathbb{F}_q[\mathbf G]$, its action is
$$
g(f)\mapsto g(f^q)=g(f)^q.
$$
Assume that there is a twisted Frobenius $F=F_q\circ\tau$ on $\mathrm{G}$. Since $\tau$ is defined over $\mathbb{F}_q$ and has finite order,
Galois descent gives a $\mathbb{F}_q$-form $\mathbf G_0$
of $\mathbf G\times_{\mathbb{F}_q} {\overline{\mathbb{F}_q}}$; see
\cite[Theorem~3.43]{Mi17}.
Under the identification
$\mathbf G_0\times_{\mathbb{F}_q}{\overline{\mathbb{F}_q}}\simeq\mathbf G\times_{\mathbb{F}_q}{\overline{\mathbb{F}_q}}$,
the $q$-Frobenius of $\mathbf G_0$
induces $F$ on $\mathrm{G}$.
In particular,
$$
\mathbf G_0(\mathbb{F}_q)=\mathrm{G}^F.
$$
Let $P,Q$ be $F$-stable parabolic subgroups of $\mathrm{G}$.
They descend to parabolic subgroups
$\mathbf P_0,\mathbf Q_0$ of $\mathbf G_0$.
Set
$$
\pmb{\mathscr{F}}_0=\mathbf G_0/\mathbf P_0
       \times\mathbf G_0/\mathbf Q_0,
\qquad
\mathscr{F}=\pmb{\mathscr{F}}_0(\overline{\mathbb{F}_q}).
$$
The $q$-Frobenius of $\pmb{\mathscr{F}}_0$ induces an action $F$ on $\mathscr{F}$ such that $\pmb{\mathscr{F}}_0(\mathbb{F}_q)=\mathscr{F}^F$.
For a $\mathbb{F}_q$-point $x\in\pmb{\mathscr{F}}_0(\mathbb F_q)$, let $\boldsymbol{\mathcal{O}}_x$ be the associated locally closed $\mathbb{F}_q$-subvariety of $\mathscr{F}_0$ (with reduced structure) whose underlying topological space is $\mathbf{G}_0\cdot x$. We have
$$
\boldsymbol{\mathcal O}_x({\mathbb F_q})=\mathcal (\boldsymbol{\mathcal O}_x(\overline{\mathbb{F}_q}))^F
=
(\mathrm{G}\cdot x)^F=\mathbf{G}_0(\mathbb{F}_q)\cdot x,
$$ where the last equality is due to Proposition \ref{prop:rational-orbits}.

\section{General unitary groups and minus-type orthogonal groups}
\label{sec:UandO}
In this section, we provide specific information corresponding to Figures 1, 2 and 3, for which the restricted Weyl groups $W^F$ are all of type B and the associated non-split groups are the general unitary groups and the minus-type orthogonal groups.

 \subsection{General unitary group}\label{ex:U_D}
        Fix $D=2d+1$ or $2d\in\mathbb{N}$. Let $\mathbf U_D$ and $\mathbf{GL}_D$, respectively, be the unitary group and general linear group over $\mathbb F_q$.
Choose a suitable isomorphism
$$
\mathbf{U}_{D}\times_{\mathbb{F}_q} \overline{\mathbb{F}_q}
\simeq \mathbf{GL}_{D}\times_{\mathbb{F}_q}\overline{\mathbb{F}_q}.$$
It identifies $\mathbf{U}_D(\overline{\mathbb{F}_q})$ with $\mathrm{G}=\mathbf{GL}_D(\overline{\mathbb{F}_q})$.
Under this identification, the $q$-Frobenius of
$\mathbf{U}_D$ induces $F=F_q\circ\tau$ on $\mathrm{G}$, where
        \begin{equation}\label{def:J}
        \tau: \mathrm{G}\to \mathrm{G},\quad g \mapsto Jg^{\mathrm{-tr}}J\quad\mbox{with}\quad
        J = {\tiny\begin{pmatrix}
                &&& 1\\
                && 1\\
                &\begin{sideways}$\ddots$\end{sideways}\\
                1
            \end{pmatrix}}.
            \end{equation}
        It preserves the upper triangular Borel subgroup $B$ and the diagonal maximal torus $T$, and corresponds to the nontrivial graph automorphism shown in Figure 1 if $D=2d+1$ or in Figure 2 if $D=2d$. 
        
        The $F$-fixed point subgroup $\mathrm{G}^F =\mathbf{U}_D(\mathbb{F}_q)=\mathrm{U}_{D}(q)$ is the general unitary group consisting of invertible $D\times D$ matrices over $\mathbb{F}_{q^2}$ leaving invariant the non-degenerate Hermitian form
        \begin{equation}\label{form1}
        \langle\cdot,\cdot\rangle: \mathbb{F}_{q^2}^{D}\times \mathbb{F}_{q^2}^{D}\to \mathbb{F}_{q^2},\qquad \big\langle\sum_{i=1}^{D} x_iv_i,\sum_{i=1}^{D} y_iv_i\big\rangle=\sum_{i=1}^{D} x_iy_{D+1-i}^q,
        \end{equation}
        where $\{v_i~|~1\leq i\leq D\}$ denotes the standard basis of $\mathbb{F}_{q^2}^{D}$.
\begin{rem}
Although the elements of $\mathrm{U}_D(q)$ are represented by matrices
over $\mathbb{F}_{q^2}$, we have $\mathrm{U}_D(q)=\mathbf{U}_D(\mathbb{F}_q)$.
Since $\tau^2=1$ and $\tau$ commutes with $F_q$,
we have $F^2=F_{q^2}$.
Thus $\mathbf{U}_D$ becomes split over $\mathbb{F}_{q^2}$,
and we obtain
\begin{equation}\label{eq:Uq2=GLq2}
\mathbf{U}_D\times_{\mathbb{F}_q} \mathbb{F}_{q^2} \simeq \mathbf{GL}_D\times_{\mathbb{F}_q} \mathbb{F}_{q^2}=\mathbf{GL}_{D,q^2}.
\end{equation}
\end{rem}
        
    Let $\epsilon_i\in Q$ ($i=1,2,\ldots, D$) be the weights satisfying $$\epsilon_i(\mathrm{diag}(\lambda_1,\ldots,\lambda_{D})) =\lambda_i.$$ Via the standard Euclidean form, $\epsilon_1,\ldots,\epsilon_{D}$ become an orthonormal basis of $\mathbb{R}^{D}\simeq\bigoplus_{i=1}^{D}\mathbb{R}\epsilon_i$.

    The set of simple roots of $\mathrm{G}$ is $\Delta=\{\alpha_i=\epsilon_i-\epsilon_{i+1}~|~i=1,2,\ldots,D-1\}$. Let us extend the graph automorphism 
    $$\tau: \Delta\to\Delta, \quad \alpha_i\mapsto \alpha_{D-i} \quad (i=1,2,\ldots,D-1)$$ 
    to the automorphism
    $$\tau: \mathbb{R}^{D}\to\mathbb{R}^{D}, \quad \epsilon_i\mapsto -\epsilon_{D+1-i}\quad (i=1,2,\ldots,D).$$

    The preimage $\dot{s}_i\in N_{\mathrm{G}}(T)$ of the simple generators $s_i=s_{\alpha_i}\in W$ $(i=1,2\ldots,D-1)$ can be chosen as 
    $$\dot{s}_i = {\tiny\begin{pmatrix}
        I_{i-1}\\
        &  & 1\\
        & 1 & \\
        &&& I_{D-1-i}
    \end{pmatrix}}.$$
    Clearly, $F(\dot{s}_i)=\dot{s}_{D-i}$ for $i=1,2,\ldots,D-1$.
    
    Recall the generators $t_1,t_2, \ldots, t_d$ of $W^F$ listed in Table~\ref{tab1}, whose preimage $\dot{t}_i$ can be chosen as 
    $$
    \dot{t}_i={\tiny\begin{pmatrix}
        I_{i-1} & & \\
        &  & 1 \\
        & 1 &  \\
        & & & I_{D-2i-2} \\
        & & & &  & 1 \\
        & & & & 1 &  \\
        & & & & & & I_{i-1}
    \end{pmatrix}}\quad(1\le i<d), \mbox{and} 
    $$
    $$\dot{t}_d=
    {\tiny\begin{pmatrix}
        I_{d-1} & & \\
        & & & 1 \\
        & & 1 \\
        & 1 \\
        & & & & I_{d-1}
    \end{pmatrix}} \quad\mbox{if $D=2d+1$}; \quad 
    \dot{t}_d=
    {\tiny\begin{pmatrix}
        I_{d-1} &  \\
        & &  1 \\
        &  1 \\
        & & & I_{d-1}
    \end{pmatrix}} \quad\mbox{if $D=2d$}.
    $$
   These matrices generate a faithful permutation representation of the type $B_d$ Weyl group on the set $\{1,2,\ldots, D\}$.
\subsection{Hermitian isotropic flag varieties}\label{HermitF_D}
    Fix $N=2n+1$ with $n\in\mathbb{N}$. Take $Q_{\ff}$ as
        $$
            Q_n := \{\sum_{i=1}^d a_i(\epsilon_i - \epsilon_{D+1-i}) \mid a_i\in\mathbb{Z}, -n\le a_i \le n, \forall i\},
        $$ which lies in $Q^\tau$ and is clearly $W^F$-invariant.
        The set $\Lambda_{\ff}$ of $W^F$-orbits on $Q_n$ corresponds to the set of weak compositions of $d$ into $n+1$ parts
        \begin{equation}\label{def:Lambda}
            \Lambda_{n+1,d} = \{\gamma = (\gamma_0, \gamma_1, \cdots, \gamma_{n})\mid \sum_{i=0}^{n}\gamma_i = d, \gamma_i\in\mathbb{Z}_{\geq 0}\},
        \end{equation}
        where an orbit $\gamma = (\gamma_0, \cdots, \gamma_{n})\in\Lambda_{n+1,d}$ consists of all weights $\sum_{i=1}^{d}a_i(\epsilon_i - \epsilon_{D+1-i})$ such that $\gamma_k = \#\{1\leq i\leq d \mid \left|a_i\right| = k\}$. The unique anti-dominant weight $\mathbf{i}_\gamma = \sum_{i=1}^d a_i(\epsilon_i -\epsilon_{D+1-i})$ in this orbit satisfies $a_1\leq a_2\leq\cdots\leq a_d\leq 0$. For $\gamma\in\Lambda_{n+1,d}$, the associated parabolic subgroup $W_\gamma^F$ is generated by
        $$
        \{t_i \mid 1\leq i\leq d, i\neq\gamma_{n}, \gamma_{n-1}+\gamma_{n},\cdots,\gamma_1+\cdots+\gamma_{n}\}.
        $$
               
        Set $V_{(i)} = \mathrm{Span}\{ v_1,\cdots, v_i \}$ for $i=1,2,\ldots,D$. The parabolic subgroup $P_\gamma= \bigsqcup_{w\in W_\gamma^F}B^F\dot{w}B^F$ stabilizes the $N$-step flag
        \begin{align*}
            &\mathfrak{f}_\gamma := \\
            &\left\{\begin{array}{l}
            (0 = V_{(0)} \subseteq V_{(\gamma_{n})} \subseteq \cdots \subseteq V_{(d-\gamma_0)}\subset V_{(d+1+\gamma_0)}\subseteq \cdots \subseteq V_{(D)} = \mathbb{F}_{q^2}^{D}), \ \mbox{if}~ D = 2d+1;\\
        (0 = V_{(0)} \subseteq V_{(\gamma_{n})}\subseteq \cdots \subseteq V_{(d-\gamma_0)}\subseteq V_{(d+\gamma_0)}\subseteq \cdots\subseteq V_{(D)} = \mathbb{F}_{q^2}^{D}), \ \mbox{if}~ D = 2d.
        \end{array}
        \right.
        \end{align*}
        In both cases, $\mathscr{F}_{\ff}=\bigsqcup_{\gamma\in\Lambda_{\ff}} \mathrm{G}^F/P_\gamma$ is naturally identified
with the following set of Hermitian isotropic flags:
        \begin{equation}\label{equ:N-step-flags}
            \mathscr{F}^\jmath_D=\{\mathfrak{f} = (0 = \mathfrak{f}_0 \subseteq \mathfrak{f}_1 \subseteq\cdots\subseteq \mathfrak{f}_{N} = \mathbb{F}_{q^2}^{D}) \mid \mathfrak{f}_i = \mathfrak{f}_j^\perp\ \text{if}\ i+j = N\},
        \end{equation}
        where the orthogonal complement is defined with respect to the Hermitian form defined in \eqref{form1}.
Furthermore, we are also interested in the following set of Hermitian isotropic flags:
\begin{align}\label{eq:N-step-flags-i}
\sF_D^{\imath}=\{\mathfrak{f}=(\mathfrak{f}_i)\in\sF^\jmath_D\mid\mathfrak{f}_n\ \text{is maximal isotropic}\}\subset \sF^\jmath_D.
\end{align} 
These sets are the $\mathbb{F}_q$-rational points of the
corresponding disjoint unions of partial flag varieties 
$\pmb{\mathscr{F}}_D^\jmath=\bigsqcup_{\gamma\in\Lambda_\ff}\mathbf{U}_D/\mathbf{P}_\gamma$ and $\pmb{\mathscr{F}}_D^\imath=\bigsqcup_{\gamma\in\Lambda_{\ff},\gamma_0=0}\mathbf{U}_D/\mathbf{P}_\gamma$
of
$\mathbf{U}_D$, where $\mathbf{P}_\gamma$ is the parabolic subgroup of $\mathbf{U}_{D}$ with $\mathbf{P}_\gamma(\mathbb{F}_q)=P_\gamma$.

For a pair $(\mathfrak{f},\mathfrak{f}')\in \mathscr{F}^\jmath_D\times\mathscr{F}^\jmath_D$, we associate an $N\times N$ matrix $A=(a_{ij})$ with
    \begin{align*}
        a_{ij} &= \dim \frac{\mathfrak{f}_{i-1}+\mathfrak{f}_i\cap \mathfrak{f}'_j}{\mathfrak{f}_{i-1}+\mathfrak{f}_i\cap \mathfrak{f}'_{j-1}}\\
        &=\dim\mathfrak{f}_i\cap \mathfrak{f}'_j-\dim\mathfrak{f}_{i-1}\cap \mathfrak{f}'_j-\dim\mathfrak{f}_i\cap \mathfrak{f}'_{j-1}+\dim\mathfrak{f}_{i-1}\cap \mathfrak{f}'_{j-1}
        \qquad (\forall 1\leq i,j\leq N).
    \end{align*}
    Note that $\sum_ja_{ij} = \dim \mathfrak{f}_i/\mathfrak{f}_{i-1}$ and $\sum_{i,j}a_{ij} = \sum_i\dim \mathfrak{f}_i/\mathfrak{f}_{i-1} = \dim\mathfrak{f}_N - \dim\mathfrak{f}_0 = D$. By Zassenhaus Lemma,
    \begin{equation}\label{eq:Zassenhaus}
        \frac{\mathfrak{f}_{i-1}+\mathfrak{f}_i\cap \mathfrak{f}'_j}{\mathfrak{f}_{i-1}+\mathfrak{f}_i\cap \mathfrak{f}'_{j-1}}\simeq \frac{\mathfrak{f}'_{j-1} + \mathfrak{f}'_j\cap \mathfrak{f}_i}{\mathfrak{f}'_{j-1} + \mathfrak{f}'_j\cap \mathfrak{f}_{i-1}},
    \end{equation}
    which implies $\sum_ia_{ij} = \dim \mathfrak{f}'_j/\mathfrak{f}'_{j-1}$. Furthermore, the conditions $\mathfrak{f}_i = \mathfrak{f}_{N-i}^\perp$ and $\mathfrak{f}'_j = {\mathfrak{f}'}_{N-j}^\perp$ yield
    \begin{align*}
        \dim \mathfrak{f}_i\cap \mathfrak{f}'_j &= \dim \mathfrak{f}_{N-i}^\perp\cap {\mathfrak{f}'}_{N-j}^{\perp} = \dim (\mathfrak{f}_{N-i} + \mathfrak{f}'_{N-j})^\perp\\
        &= D - \dim \mathfrak{f}_{N-i} - \dim \mathfrak{f}'_{N-j} + \dim\mathfrak{f}_{N-i}\cap\mathfrak{f}'_{N-j},
    \end{align*}
    which establishes the central-symmetry $a_{ij} = a_{N+1-i,N+1-j}$. 

In addition, suppose $(\mathfrak{f},\mathfrak{f}')\in \mathscr{F}^\imath_D\times\mathscr{F}^\imath_D$. Since the subspace $\mathfrak{f}_{n+1}\cap \mathfrak{f}'_j\subset \mathfrak{f}'_j$ is always isotropic for $j\leq n$, we have $a_{n+1,j}=\dim \mathfrak{f}_{n}/\mathfrak{f}_{n}=0 (\forall j\leq n)$. Then $a_{n+1,j}=0 (\forall j>n+1)$ by central-symmetry. Similarly, $a_{i,n+1}=0$ ($\forall i\neq n+1$) by \eqref{eq:Zassenhaus}. Finally, $$a_{n+1,n+1}\leq\dim \mathfrak{f}_{n+1}/\mathfrak{f}_{n}=\left\{\begin{array}{ll}
0, & \text{if $D=2d$},\\
1, & \text{if $D=2d+1$},
\end{array}\right.$$
which implies $a_{n+1,n+1}=0$ if $D=2d$; and $a_{n+1,n+1}=1$ if $D=2d+1$ because in this case $a_{n+1,n+1}$ must be odd due to central-symmetry.

Consequently, the map $(\mathfrak{f},\mathfrak{f}')\mapsto(a_{ij})$ sends $\mathscr{F}^\jmath_D\times\mathscr{F}^\jmath_D$ and 
    $\mathscr{F}^\imath_D\times\mathscr{F}^\imath_D$ to the sets $\Theta_D^\jmath$ and $\Theta_D^\imath$, respectively, where
    \begin{align*}
       \Theta_D^\jmath&:= \{A = (a_{ij})\in\mathrm{Mat}_{N\times N}(\mathbb{N}) \mid \sum_{i,j}a_{ij} =D,\ a_{ij} = a_{N+1-i,N+1-j}\},\\
     \Theta_{D}^\imath&:=
        \big\{A=(a_{ij})\in\Theta^\jmath_D ~\big|~\begin{array}{c} a_{n+1,n+1}=D-2d,\\
      a_{n+1,j}=a_{i,n+1}=0\ \text{for}\ i,j\neq n+1
      \end{array}
      \big\}
     \end{align*}

\begin{lem}\label{lem:bijections}
  For $\mathfrak{c}=\imath,\jmath$, there is a bijection
        \begin{equation*}
            \mathrm{U}_D(q)\setminus \mathscr{F}^\mathfrak{c}_D\times\mathscr{F}^\mathfrak{c}_D\longleftrightarrow\Theta_D^\mathfrak{c},\quad (\mathfrak{f},\mathfrak{f}')\mapsto(a_{ij}),
        \end{equation*}
        where $(\mathfrak{f},\mathfrak{f}')$ is a representative of the corresponding $\mathrm{U}_D(q)$-orbit.
         \end{lem}
    \begin{proof}
The assignment $(a_{ij})$ is invariant on each $\mathrm{U}_D(q)$-orbit in $\mathscr{F}^\mathfrak{c}_D\times\mathscr{F}^\mathfrak{c}_D$. Hence, the map is well-defined.  

For injectivity, let two pairs of flags share the same associated
matrix. They lie in the same $\mathrm{GL}_D$-orbit by \cite[\S 1.1]{BLM90}.
Since both pairs are fixed by the Frobenius $F$,
Proposition~\ref{prop:rational-orbits} shows that they lie in the
same orbit of $\mathrm{GL}_D^F=\mathrm{U}_D(q)$. Hence the map is injective.

Finally, let us show the surjectivity.
For each $(a_{ij})\in \Theta_D^\mathfrak{c}$, we need to find a pair $(\mathfrak{f},\mathfrak{f}')\in\mathscr{F}^\mathfrak{c}_D\times\mathscr{F}^\mathfrak{c}_D$ such that its associated matrix is $(a_{ij})$. Let $\{e^{(t)}_{ij}\mid 1\le i,j\le N, a_{ij}\neq0, 1\le t\le a_{ij} \}$ be a basis of $\mathbb{F}^{D}_{q^2}$
    such that $\langle e^{(t_1)}_{ij},e^{(t_2)}_{kl}\rangle=\delta_{t_1,t_2}\delta_{(i,j),(N+1-k,N+1-l)}$. Let $\{B_{ij}\}_{i,j=1}^N$ be some $\mathbb{F}_{q^2}$-vector spaces such that$$B_{ij}=\begin{cases}
        0, &\text{if}\ a_{ij}=0,\\
        \mathrm{Span}\{e^{(1)}_{ij},\cdots,e^{(a_{ij})}_{ij}\}, &\text{if}\ a_{ij}\neq0.
    \end{cases} $$
 
 Let $\mathfrak{f}=(0=\mathfrak{f}_0\subseteq\mathfrak{f}_1\subseteq\cdots\subseteq\mathfrak{f}_N=\mathbb{F}_{q^2}^{D})$ and $\mathfrak{f}'=(0=\mathfrak{f}'_0\subseteq\mathfrak{f}'_1\subseteq\cdots\subseteq\mathfrak{f}'_N=\mathbb{F}_{q^2}^{D})$, where $\mathfrak{f}_i:=\bigoplus_{1\le k\le i; 1\le l\le N}B_{kl}$ and $\mathfrak{f}'_j:=\bigoplus_{1\le l\le j;1\le k\le N}B_{kl}$, $1\le i,j\le N$. 
 It is clear that $\mathfrak{f}_i=\mathfrak{f}_{N-i}^\perp$ and $\mathfrak{f}'_j=\mathfrak{f}'^{\perp}_{N-j}$. Moreover, $\mathfrak{f}_n$ and $\mathfrak{f}'_n$ are both maximal isotropic if $(a_ij)\in\Theta_D^\imath$. So $(\mathfrak{f},\mathfrak{f}')\in\mathscr{F}^\mathfrak{c}_D\times\mathscr{F}^\mathfrak{c}_D$. The associated matrix of $(\mathfrak{f},\mathfrak{f}')$ is $(\dim B_{ij})=(a_{ij})$.
\end{proof}

For $A\in\Theta_D^\mathfrak{c}$ ($\mathfrak{c}=\imath,\jmath$), let $\mathcal{O}_A$ denote the corresponding $\mathrm{U}_D(q)$-orbit in $\mathscr{F}^\mathfrak{c}_D\times\mathscr{F}^\mathfrak{c}_D$.

\subsection{Minus-type orthogonal group} \label{ex:GO_D}

Let $\mathbf{O}_{2d+2}$ (resp. $\mathbf{O}^-_{2d+2}$) be the orthogonal group (resp., minus-type orthogonal group) defined over $\mathbb{F}_q$. 
Take a standard basis $\{v_1, v_2\cdots, v_{2d+2}\}$ of $\overline{\mathbb{F}_q}^{2d+2}$. The group $\mathrm{O}_{2d+2}=\mathbf{O}_{2d+2}(\overline{\mathbb{F}_q})$ of geometric points is the isometry group of the standard bilinear form
\begin{equation}\label{standarform}
(\cdot,\cdot): \overline{\mathbb{F}_q}^{2d+2}\times \overline{\mathbb{F}_q}^{2d+2} \to \overline{\mathbb{F}_q}, \quad (\sum_{i=1}^{2d+2}x_iv_i,\sum_{i=1}^{2d+2}y_iv_i)=\sum_{i=1}^{2d+2}x_iy_{2d+3-i},
\end{equation}
while the group $\mathrm{O}_{2d+2}^-=\mathbf{O}_{2d+2}^-(\overline{\mathbb{F}_q})$ of geometric points is the isometry group of another non-degenerate bilinear form $\langle\cdot,\cdot\rangle: \overline{\mathbb{F}_q}^{2d+2}\times \overline{\mathbb{F}_q}^{2d+2} \to \overline{\mathbb{F}_q}$,
  \begin{align}  \label{def:quad-form}
\big\langle\sum_{i=1}^{2d+2} x_iv_i,\sum_{i=1}^{2d+2} y_iv_i\big\rangle =\sum_{i=1}^{d} (x_iy_{2d+3-i}+x_{2d+3-i}y_i)+x_{d+1}y_{d+1}-\omega x_{d+2}y_{d+2},
    \end{align}
    where $\omega$ is a non-square element in $\mathbb{F}_q$.
It is a very basic fact that $\mathrm{O}_{2d+2}\simeq\mathrm{O}_{2d+2}^-$. However, the finite groups of $\mathbb{F}_q$-points $\mathrm{O}_{2d+2}(q)$ and $\mathrm{O}_{2d+2}^-(q)$ are not isomorphic to each other because the two bilinear forms are not equivalent over $\mathbb{F}_q$ when $\sqrt{\omega}\not\in\mathbb{F}_q$.

Note that $\mathrm{O}_{2d+2}$ is not connected. To adapt the framework established in Section~\ref{sec:generalframework}, we take its identity component $\mathrm{G}=\mathrm{SO}_{2d+2}=\mathrm{O}_{2d+2}^{\circ}$. In particular, the root data, Weyl group and parabolic subgroups below are those of $\mathrm{G}$. 

Take the maximal torus $T$ and the Borel subgroup $B$ of $\mathrm{G}=\mathrm{SO}_{2d+2}$ as
\begin{gather*}
T = \left\{\mathrm{diag}(\lambda_1,\cdots, \lambda_{d+1}, \lambda_{d+1}^{-1}, \cdots, \lambda_1^{-1})\right\},\\
B=\left\{\begin{pmatrix}
        Y & Y J X\\
        0 & J Y^{-\mathrm{tr}}J
    \end{pmatrix}\middle| \begin{array}{c}X,Y\in \mathrm{Mat}_{d+1}(\overline{\mathbb{F}_q}),\\X^{\mathrm{tr}}=-X,\ \det Y\ne0, \ Y\text{ upper triangular}\end{array}\right\},
\end{gather*}    
    where $J$ is defined in \eqref{def:J} with size $(d+1)\times(d+1)$.
The element
    \begin{equation}\label{def:g}
        g:={\tiny\begin{pmatrix}
            I_{d} \\
            &  & 1\\
            & 1 & \\
            &&& I_{d}
        \end{pmatrix}}\in\mathrm{O}_{2d+2}
    \end{equation}
    normalizes both $T$ and $B$. The map 
    \begin{equation}\label{def:FO2d+2}
    F = F_q\circ \mathrm{Ad}_g: \mathrm{O}_{2d+2}\to \mathrm{O}_{2d+2}, \quad x\mapsto gF_q(x)g^{-1}
    \end{equation}
    preserves $\mathrm{G}$. Its restriction to $\mathrm{G}$, also denoted by $F$, is the twisted Frobenius morphism involving the nontrivial graph automorphism shown in Figure 3. Here $\det g=-1$, but conjugation by $g$ defines an automorphism of $\mathrm{G}$ preserving $T$ and $B$. It is known that 
    \begin{equation}\label{F_q res}
    \mathrm{O}_{2d+2}^-(q)=(\mathrm{O}_{2d+2}^-)^{F_q}\simeq (\mathrm{O}_{2d+2})^F.
    \end{equation}
Here the isomorphism $(\mathrm{O}_{2d+2}^-)^{F_q}\simeq (\mathrm{O}_{2d+2})^F$ induced from $\mathrm{Ad}_{P^{-1}}$ on $\mathrm{O}_{2d+2}$; see the following commutative diagram:
      $$  \begin{tikzcd}
\mathrm{O}_{2d+2} \arrow[d, "\mathrm{Ad}_{P^{-1}}"] \arrow[r, "F_q\circ\mathrm{Ad}_g"] & \mathrm{O}_{2d+2} \arrow[d, "\mathrm{Ad}_{P^{-1}}"] \\
P^{-1}\mathrm{O}_{2d+2} P \arrow[r, "F_q"]                                          & P^{-1}\mathrm{O}_{2d+2}P
\end{tikzcd}$$
with
 $$P={\tiny\begin{pmatrix}
                I_{d}&&&\\
                & \alpha&\alpha\cdot\sqrt{\omega}\\
                &\frac{1}{2\alpha}&\frac{-\sqrt{\omega}}{2\alpha}\\
                &&&I_{d}
            \end{pmatrix}}, \quad \text {a congruent matrix from $J_{2d+2}$ to ${\tiny\begin{pmatrix}
                &&&J_{d}\\
                & 1\\
                && -\omega\\
                J_{d}&&&
            \end{pmatrix}}$},$$
            where $\alpha\in \mathbb{F}_{q^2}$ satisfies $\alpha^{q+1}=\frac{1}{2}$.
In particular, $F_q(P)\cdot P^{-1}=g$ is the matrix defined in \eqref{def:g}.

\begin{rem}
We remark here that $F_q$ on second term of \eqref{F_q res} denotes the restriction of the standard Frobenius on $\mathrm{GL}_{2d+2}$ on $\mathrm{O}^-_{2d+2}$.
It acts by raising each matrix entry to its $q$-th power. It should not be confused with the standard Frobenius
on $\mathrm O_{2d+2}$ in the sense of \S \ref{Frobenius}.
\end{rem}

Since $\sqrt{\omega}\in\mathbb{F}_{q^2}$, the two bilinear forms $(\cdot,\cdot)$ and $\langle\cdot,\cdot\rangle$ become equivalent over $\mathbb{F}_{q^2}$. Therefore, the base change to $\mathbb{F}_{q^2}$ gives 
\begin{equation}\label{eq:Oq2=Oq2}
\mathbf{O}_{2d+2}^{-}\times_{\mathbb{F}_q} \mathbb{F}_{q^2} \simeq\mathbf{O}_{2d+2}\times_{\mathbb{F}_q} \mathbb{F}_{q^2}=\mathbf{O}_{2d+2,q^2}.
\end{equation}
We also denote by $\mathbf{SO}_{2d+2}^-$ the special orthogonal group of minus-type defined over $\mathbb{F}_q$, which is the identity component of $\mathbf{O}^-_{2d+2}$. Let $\mathrm{SO}^-_{2d+2}(q)$ be its $\mathbb{F}_q$-rational points. There is an isomorphism $\mathrm{SO}_{2d+2}^-(q)\simeq(\mathrm{SO}_{2d+2})^F$, which is also induced from $\mathrm{Ad}_{P^{-1}}$ on $\mathrm{SO}_{2d+2}$. 

Let $\epsilon_i\in Q$ ($i=1,2,\ldots,d+1$) be the weights satisfying $$\epsilon_i(\mathrm{diag}(\lambda_1,\ldots,\lambda_{d+1},\lambda_{d+1}^{-1},\ldots,\lambda_1^{-1}))=\lambda_i,$$ 
    which form an orthonormal basis of $\mathbb{R}^{d+1}\simeq \bigoplus_{i=1}^{d+1}\mathbb{R}\epsilon_i$ via the standard bilinear form induced from the Killing form. The set of simple roots is $\Delta=\{\alpha_1=\epsilon_1-\epsilon_2,\ldots,\alpha_d=\epsilon_{d}-\epsilon_{d+1},\alpha_{d+1}=\epsilon_{d}+\epsilon_{d+1}\}$. Extend the graph automorphism $\tau: \Delta\to\Delta$ that switches $\alpha_d$ and $\alpha_{d+1}$ to the automorphism 
    $\tau: \mathbb{R}^{d+1}\to\mathbb{R}^{d+1}$ that switches $\epsilon_{d+1}$ and $-\epsilon_{d+1}$.

    The preimage $\dot{s}_i\in N_{\mathrm{G}}(T)$ of $s_i=s_{\alpha_i}\in W (i = 1, 2, \ldots, d+1)$ can be chosen as
        \begin{align*}
            & \dot{s}_i = {\tiny \begin{pmatrix}
                I_{i-1}\\
                & 0 & 1\\
                & 1 & 0\\
                &&& I_{2d-2i}\\
                &&&& 0 & 1\\
                &&&& 1 & 0\\
                &&&&&& I_{i-1}
            \end{pmatrix}} \quad (i\leq d); \quad \dot{s}_{d+1} = {\tiny \begin{pmatrix}
                I_{d-1}\\
                &&& 1\\
                &&&& 1\\
                & 1 \\
                && 1\\
                &&&&&& I_{d-1}
            \end{pmatrix}}.
        \end{align*}
It is clear that $F(s_i) = s_i$ ($i<d$), $F(s_d) = s_{d+1}$ and $F(s_{d+1}) = s_{d}$.

The group $W^F$ is generated by $t_1= s_1,\cdots, t_{d-1}= s_{d-1}$ and $t_{d}= s_{d}s_{d+1}$. In particular, $\epsilon_i\cdot t_d= \epsilon_i$ for $i<d$ and $\epsilon_{i}\cdot t_d = -\epsilon_{i}$ for $i=d,d+1$.
        
\subsection{Orthogonal isotropic flag varieties}\label{O-flag}
We take $Q_{\ff}$ as
        $$
            Q'_n = \{\sum_{i=1}^{d} a_i\epsilon_i \mid a_i\in\mathbb{Z}, -n\le a_i\le n\},
        $$
        which lies in $Q^\tau$ and is $W^F$-invariant. The set $\Lambda_\ff$ of $W^F$-orbits on $Q'_n$ corresponds to the set $\Lambda_{n+1,d}$ defined in \eqref{def:Lambda} again, for which an orbit $\gamma = (\gamma_0,\cdots,\gamma_{n})\in\Lambda_{n+1,d}$ consists of all weights $\sum_{i=1}^{d}a_i\epsilon_i$ such that $\gamma_k = \#\{1\leq i\leq d~|\left|a_i\right| = k\}$. The unique anti-dominant weight $\mathbf{i}_\gamma = \sum_{i=1}^d a_i\epsilon_i$ in this orbit satisfies $a_1\leq a_2\leq\cdots\leq a_d\leq 0$.
        For $\gamma \in\Lambda_{n+1,d}$, the associated parabolic group $W^F_\gamma$ is generated by
        $$
        \{ t_i\mid 1\leq i\leq d, i\neq\gamma_{n}, \gamma_{n-1} + \gamma_{n}, \cdots, \gamma_1+\cdots+\gamma_{n}\}.
        $$

        Up to $\mathrm{Ad}_{P^{-1}}$, we now identify $\mathrm{G}^F$ with $\mathrm{SO}_{2d+2}^-(q)$ and transport the subgroups $P_\gamma$ accordingly. Since $P$ acts non-trivially only on $\mathrm{Span}\{v_{d+1},v_{d+2}\}$, it preserves all subspaces in the flags $\mathfrak{f}_\gamma$ below.
        Set $\mathfrak{e}_i = \mathrm{Span}\{ v_1,\cdots, v_i \}$ for $i=1,2,\ldots,2d+2$. The parabolic group $P_\gamma$ is the stabilizer in $\mathrm{SO}_{2d+2}^-(q)$ of the $N$-step flag
        \begin{align*}
            \mathfrak{f}_\gamma = (0 = \mathfrak{e}_0 \subseteq \mathfrak{e}_{\gamma_{n}} \subseteq \cdots \subseteq \mathfrak{e}_{d-\gamma_0}\subseteq \mathfrak{e}_{d+2+\gamma_0}\subseteq \cdots \subseteq \mathfrak{e}_{2d+2-\gamma_{n}} \subseteq \mathfrak{e}_{2d+2} = \mathbb{F}_{q}^{2d+2}).
        \end{align*}
        Therefore, $\mathscr{F}_{\ff}$ is isomorphic to the set of isotropic $N$-step flags
        $$
           \mathscr{F}^\jmath_O=\{\mathfrak{f} = (0 = \mathfrak{f}_0 \subseteq \cdots\subseteq \mathfrak{f}_n \overset{\geq 2}{\subset} \mathfrak{f}_{n+1} \cdots\subseteq \mathfrak{f}_{N} = \mathbb{F}_{q}^{2d+2}) \mid \mathfrak{f}_i = \mathfrak{f}_j^\perp\ \text{if}\ i+j = N\},
        $$
        where the orthogonal complement is defined with respect to the quadratic form \eqref{def:quad-form}.
        For each $\mathfrak{f}\in\mathscr{F}^\jmath_O$, since the restriction of the quadratic form on $\mathfrak{f}_n^\perp/\mathfrak{f}_n$ is nondegenerate and $\dim \mathfrak{f}_n^\perp/\mathfrak{f}_n\geq2$, we can choose $w\in \mathfrak{f}_n^\perp$ with $\langle w,w\rangle\ne0$. The reflection $s_w(x)=x-\frac{2\langle x,w\rangle}{\langle w,w\rangle} w$ is defined over $\mathbb{F}_q$, has determinant $-1$ and fixes $\mathfrak{f}_n$ pointwise; hence it stabilizes $\mathfrak{f}$. Thus $\mathrm{O}_{2d+2}^-(q)\mathfrak{f}=\mathrm{SO}_{2d+2}^-(q)\mathfrak{f}$ ($\forall \mathfrak{f}\in\mathscr{F}^\jmath_O$), by a stabilizer argument similar to that in \cite[Lemma~4.4]{DLZ25}.

In addition, we are also interested in the following set of isotropic $N$-step flags
\begin{align}\label{eq:N-step-itypeO}
\sF_O^{\imath}=\{\mathfrak{f}=(\mathfrak{f}_i)\in\sF^\jmath_O\mid\mathfrak{f}_n\ \text{is maximal isotropic}\}\subset \sF^\jmath_O.
\end{align} 
As before, we denote the corresponding $\mathbb{F}_q$-varieties of $\mathscr{F}^\jmath_O$ and $\mathscr{F}^\imath_O$ by bold letters $\pmb{\mathscr{F}}^\jmath_O$ and $\pmb{\mathscr{F}}^\imath_O$, respectively. 

     For $(\mathfrak{f},\mathfrak{f}')\in\sF^\jmath_O\times\sF^\jmath_O$, we define a matrix $$(a_{ij})_{N\times N}:= \bigg(\dim\frac{\mathfrak{f}_{i-1}+\mathfrak{f}_i\cap\mathfrak{f}'_j}{\mathfrak{f}_{i-1}+\mathfrak{f}_i\cap\mathfrak{f}'_{j-1}}\bigg)_{N\times N}.$$

\begin{lem}\label{ann}
    For any $(\mathfrak{f},\mathfrak{f}')\in\sF^\jmath_O\times\sF^\jmath_O$, the central entry  $a_{n+1,n+1}\ge 2$.
\end{lem}
\begin{proof}
    We write $a=\dim\mathfrak{f}_n$, $b=\dim\mathfrak{f}'_n$, $t=\dim\mathfrak{f}_n\cap\mathfrak{f}'_n$ and $r=\mathrm{rank}(\langle\cdot,\cdot\rangle|_{\mathfrak{f}_n\times\mathfrak{f}'_n})$. Since $\mathfrak{f}_n+\mathfrak{f}_{n+1}\cap\mathfrak{f}_{n}'$ is isotropic, we have $$d\geq\dim(\mathfrak{f}_n+\mathfrak{f}_{n+1}\cap\mathfrak{f}_{n}')=\dim\mathfrak{f}_n+\dim(\mathfrak{f}'_n\cap\mathfrak{f}_{n+1})-\dim\mathfrak{f}_n\cap\mathfrak{f}_{n}'=a+b-r-t.$$
    Moreover, we have
    \begin{align*}
\dim(\mathfrak{f}_n+\mathfrak{f}_{n+1}\cap\mathfrak{f}_{n+1}')&=\dim\mathfrak{f}_n+\dim\mathfrak{f}'_{n+1}\cap\mathfrak{f}_{n+1}-\dim\mathfrak{f}_n\cap\mathfrak{f}_{n+1}'\\
&=a+(2d+2-(a+b-t))-(a-r)=2d+2-a-b+t+r,
\end{align*}
    where we use $\dim\mathfrak{f}'_{n+1}\cap\mathfrak{f}_{n+1}=\dim(\mathfrak{f}_n+\mathfrak{f}_n')^\perp=2d+2-(a+b-t)$ and $\dim\mathfrak{f}_n\cap\mathfrak{f}'_{n+1}=a-r$.
    So we have $a_{n+1,n+1}=2(d+1-a-b+r+t)\ge2$.
\end{proof}

Denote 
\begin{align*}
\Theta_O^\jmath&:=\{A=(a_{ij})\in\Mat_{N\times N}(\mathbb{N})\mid\sum a_{ij}=2d+2, a_{ij}=a_{N+1-i,N+1-j}, a_{n+1,n+1}\ge 2\},\\
\Theta_O^\imath&:=\{A=(a_{ij})\in\Theta^\jmath_O\mid a_{n+1,j}=a_{i,n+1}=0\ \text{for}\ i,j\neq n+1,\ a_{n+1,n+1}=2\}.
\end{align*}
\begin{lem}\label{biject D} 
In $\sF^\mathfrak{c}_O\times\sF^\mathfrak{c}_O$ ($\mathfrak{c}=\imath,\jmath$), the $\mathrm{SO}^-_{2d+2}(q)$-orbits and the $\mathrm{O}^-_{2d+2}(q)$-orbits coincide. Moreover, the common orbit set is in bijection with $\Theta_O^\mathfrak{c}$:
$$\mathrm{SO}^-_{2d+2}(q)\backslash\sF^\mathfrak{c}_O\times\sF^\mathfrak{c}_O=\mathrm{O}^-_{2d+2}(q)\backslash\sF^\mathfrak{c}_O\times\sF^\mathfrak{c}_O\longleftrightarrow\Theta_O^\mathfrak{c},\qquad (\mathfrak{f},\mathfrak{f}')\mapsto(a_{ij})_{N\times N},$$
  where $(\mathfrak{f},\mathfrak{f}')$ is a representative of the corresponding orbit.
\end{lem}
\begin{proof}
By Lemma \ref{ann}, the map $(\mathfrak{f},\mathfrak{f}')\mapsto(a_{ij})_{N\times N}$ is well-defined at least for $\mathfrak{c}=\jmath$. Below
we explain that it is also well-defined for $\mathfrak{c}=\imath$. For $(\mathfrak{f},\mathfrak{f}')\in\sF_O^{\imath}\times\sF_O^{\imath}$, we claim
\begin{equation}\label{claim:bij}
(a_{ij}):=(\dim\frac{\mathfrak{f}_{i-1}+\mathfrak{f}_i\cap\mathfrak{f}'_j}{\mathfrak{f}_{i-1}+\mathfrak{f}_i\cap\mathfrak{f}'_{j-1}})\in\Theta_O^{\imath}.
\end{equation}
Note that in this case $\mathfrak{f}_{n}+\mathfrak{f}_{n+1}\cap\mathfrak{f}'_n$ is maximal isotropic. In particular, $\dim(\mathfrak{f}_{n}+\mathfrak{f}_{n+1}\cap\mathfrak{f}'_n)=d$. By the proof of Lemma \ref{ann}, the central entry $a_{n+1,n+1}=2$. 
    Consider $a_{n+1,j}=\dim\frac{\mathfrak{f}_{n}+\mathfrak{f}_{n+1}\cap\mathfrak{f}'_j}{\mathfrak{f}_{n}+\mathfrak{f}_{n+1}\cap\mathfrak{f}'_{j-1}}$ with $j\neq n+1$. By the central symmetry of the matrix $(a_{ij})$, we may assume $j\le n$. In such a situation, $\mathfrak{f}_{n}+\mathfrak{f}_{n+1}\cap\mathfrak{f}'_j$ and $\mathfrak{f}_{n}+\mathfrak{f}_{n+1}\cap\mathfrak{f}'_{j-1}$ are both 
    maximal isotropic. Hence $a_{n+1,j}=0$ for $j\neq n+1$. Moreover, by Zassenhaus Lemma, we also have $a_{i,n+1}=0$ for $i\neq n+1$. Hence, we have shown the validity of \eqref{claim:bij}, and then have a well-defined map $\mathrm{O}^-_{2d+2}(q)\backslash\sF_O^{\imath}\times\sF_O^{\imath}\rightarrow\Theta_O^\imath$ via $(\mathfrak{f},\mathfrak{f}')\mapsto (a_{ij})$.

For injectivity, take two pairs of flags with
the same associated matrix $A\in\Theta_O^\jmath$.
Under the identification induced by $\mathrm{Ad}_P$,
both pairs of flags are fixed by $F$.
Hence they are defined over $\mathbb F_{q^2}$,
since $F^2=F_{q^2}$.
By \cite[Corollary~3.7]{DLZ25}, applied over $\mathbb F_{q^2}$,
they lie in the same $\mathrm{O}_{2d+2}({q^2})$-orbit.
Moreover, $a_{n+1,n+1}\ge 2$, so this orbit is a single
$\mathrm{SO}_{2d+2}({q^2})$-orbit by
\cite[Proposition~4.5]{DLZ25}.
Applying Proposition~\ref{prop:rational-orbits} to
$\mathrm{G}=\mathrm{SO}_{2d+2}$ with the Frobenius $F$,
we conclude that both pairs lie in the same
$\mathrm{G}^F=\mathrm{SO}_{2d+2}^-(q)$-orbit.
This proves the injectivity for the
$\mathrm{SO}_{2d+2}^-(q)$-orbit map.
Since $\mathrm{SO}_{2d+2}^-(q)\subset\mathrm{O}_{2d+2}^-(q)$,
the two pairs also lie in the same
$\mathrm{O}_{2d+2}^-(q)$-orbit.
Thus the orbit maps for both groups are injective.

By a procedure similar to that of Lemma~\ref{lem:bijections}, the map is surjective no matter for $\mathrm{SO}^-_{2d+2}(q)$-orbits or for $\mathrm{O}^-_{2d+2}(q)$-orbits.
The proof is complete.
\end{proof}   

For $A\in\Theta^\mathfrak{c}_O$, let $\mathcal{O}_A$ denote the $\mathrm{O}_{2d+2}^{-}(q)$-orbit in $\sF^\mathfrak{c}_O\times\sF^\mathfrak{c}_O$ associated to $A$ .

\subsection{Base change}\label{sec:basechange}
Denote $\widetilde{\pmb{\sF}}^\mathfrak{c}_{\square}=\pmb{\sF}^\mathfrak{c}_{\square}\times_{\mathbb{F}_q}\mathbb{F}_{q^2}$, ($\square=D$ or $O$, $\mathfrak{c}=\imath$ or $\jmath$).  

For $\square=D$, after base change to $\mathbb{F}_{q^2}$, the Hermitian form \eqref{form1} no longer imposes isotropic conditions. So the requirements that $\mathfrak{f}_i = \mathfrak{f}_{N-i}^\perp$ and that $\mathfrak{f}_n$ is maximal isotropic only survive as the dimension constraints $\dim \mathfrak{f}_i+\dim \mathfrak{f}_{N-i}=D$ and $\dim\mathfrak{f}_n=d$, respectively.

For $\square=O$, after base change to $\mathbb{F}_{q^2}$, the bilinear form \eqref{def:quad-form} becomes equivalent to the standard one \eqref{standarform}. So the requirement that $\mathfrak{f}_n$ is maximal isotropic reduces to the condition $\dim\mathfrak{f}_n=d$ only.
 
In summary, we have 
\begin{align*}
\widetilde{\mathscr{F}}^\jmath_D:=\widetilde{\pmb{\mathscr{F}}}^\jmath_D(\mathbb{F}_{q^2})&=\{\mathfrak{f} = (0 = \mathfrak{f}_0 \subseteq \mathfrak{f}_1 \subseteq\cdots\subseteq \mathfrak{f}_{N} = \mathbb{F}_{q^2}^{D}) \mid \dim\mathfrak{f}_i +\dim\mathfrak{f}_{N-i}=D,\ \forall i\},\\
\widetilde{\mathscr{F}}^\imath_D:=\widetilde{\pmb{\mathscr{F}}}^\imath_D(\mathbb{F}_{q^2})&=\{\mathfrak{f}\in \widetilde{\mathscr{F}}^\jmath_D\mid \dim \mathfrak{f}_n=d\},\\
\widetilde{\mathscr{F}}^\jmath_O:=\widetilde{\pmb{\mathscr{F}}}^\jmath_O(\mathbb{F}_{q^2})&=\{\mathfrak{f} = (0 = \mathfrak{f}_0 \subseteq \cdots\subseteq \mathfrak{f}_n \overset{\geq 2}{\subset} \mathfrak{f}_{n+1} \cdots\subseteq \mathfrak{f}_{N} = \mathbb{F}_{q^2}^{2d+2}) \mid \mathfrak{f}_i = \mathfrak{f}_j^\perp\ \text{if}\ i+j = N\},
\\
\widetilde{\mathscr{F}}^\imath_O:=\widetilde{\pmb{\mathscr{F}}}^\imath_O(\mathbb{F}_{q^2})&=\{\mathfrak{f}\in \widetilde{\mathscr{F}}^\jmath_O\mid\dim\mathfrak{f}_n=d\}.
\end{align*}

Thanks to \eqref{eq:Uq2=GLq2} and \eqref{eq:Oq2=Oq2}, $\widetilde{\sF}^\mathfrak{c}_D\times\widetilde{\sF}^\mathfrak{c}_D$ admits a diagonal $\mathrm{GL}_{D}(q^2)$-action, while $\widetilde{\sF}^\mathfrak{c}_O\times\widetilde{\sF}^\mathfrak{c}_O$ admits a diagonal $\mathrm{O}_{2d+2}(q^2)$-action and a diagonal $\mathrm{SO}_{2d+2}(q^2)$-action.

 Set 
 \begin{align*}
 \widetilde{\Theta}^\jmath_D&:=\big\{A=(a_{ij})\in\Mat_{N\times N}(\mathbb{N}) ~\big|~ \sum_{i,j=1}^N a_{ij}=D,\begin{array}{l}\mathrm{ro}(A)_i=\mathrm{ro}(A)_{N+1-i},\\ \mathrm{co}(A)_i=\mathrm{co}(A)_{N+1-i}\end{array}\big\},\\
\widetilde{\Theta}^\imath_D&:=\big\{A=(a_{ij})\in\widetilde{\Theta}^\jmath_D ~\big|~ 
\mathrm{ro}(A)_{n+1}=\mathrm{co}(A)_{n+1}=D-2d
\big\},
 \\
 \widetilde{\Theta}^\jmath_O&:=\big\{A=(a_{ij})\in\Mat_{N\times N}(\mathbb{N}) ~\big|~ 
 \sum_{i,j=1}^{N} a_{ij}=2d+2,
 \begin{array}{c}
 a_{ij}=a_{N+1-i,N+1-j},\\
 \mathrm{ro}(A)_{n+1},\mathrm{co}(A)_{n+1}\ge2
 \end{array}\big\},\\
 \widetilde{\Theta}^\imath_O&:=\{A\in\widetilde{\Theta}^\jmath_O\mid \mathrm{ro}(A)_{n+1}=\mathrm{co}(A)_{n+1}=2\}.
 \end{align*}
It is clear that $\Theta_\square^\mathfrak{c}\subset\widetilde{\Theta}^\mathfrak{c}_\square$.
The following lemma tells us that $\widetilde{\Theta}^\mathfrak{c}_D$ (resp. $\widetilde{\Theta}^\mathfrak{c}_O$) is an index set of the $\mathrm{GL}_{D}(q^2)$-orbits (resp. $\mathrm{O}_{2d+2}(q^2)$-orbits) in $\widetilde{\sF}^\mathfrak{c}_D\times \widetilde{\sF}^\mathfrak{c}_D$ (resp. $\widetilde{\sF}^\mathfrak{c}_O\times \widetilde{\sF}^\mathfrak{c}_O$). 
\begin{lem}
There are bijections
    $$\mathrm{GL}_{D}(q^2)\backslash\widetilde{\sF}^\mathfrak{c}_D\times\widetilde{\sF}^\mathfrak{c}_D\longleftrightarrow\widetilde{\Theta}_D^\mathfrak{c} \quad
    \text{and} \quad \mathrm{O}_{2d+2}(q^2)\backslash\widetilde{\sF}^\mathfrak{c}_O\times\widetilde{\sF}^\mathfrak{c}_O\longleftrightarrow\widetilde{\Theta}_O^\mathfrak{c}.$$
\end{lem}
\begin{proof}
These are essentially the known results in \cite{BLM90} for $\mathrm{GL}_D$ and in \cite{DLZ25} for $\mathrm{O}_{2d+2}$. 

The condition $\mathrm{ro}(A)_i=\mathrm{ro}(A)_{N+1-i}, \mathrm{co}(A)_i=\mathrm{co}(A)_{N+1-i}$ in $\widetilde{\Theta}_D^\jmath$ comes from the requirement $\dim\mathfrak{f}_i +\dim\mathfrak{f}_{N-i}=D$, while the condition $ \mathrm{ro}(A)_{n+1},\mathrm{co}(A)_{n+1}\ge2$ in $\widetilde{\Theta}_O^\jmath$ comes from the requirement $\mathfrak{f}_n \overset{\geq 2}{\subset} \mathfrak{f}_{n+1}$. 
The extra conditions in $\widetilde{\Theta}_D^\imath$ and $\widetilde{\Theta}_O^\imath$ reflect the requirement $\dim\mathfrak{f}_n=d$.
\end{proof}

By the above lemma, we denote by $\widetilde{\mathcal{O}}_A$ the $\mathrm{GL}_D(q^2)$-orbit in $\widetilde{\sF}^\mathfrak{c}_D\times\widetilde{\sF}^\mathfrak{c}_D$ (resp. $\mathrm{O}_{2d+2}(q^2)$-orbit in $\widetilde{\sF}^\mathfrak{c}_O\times\widetilde{\sF}^\mathfrak{c}_O$) associated with $A\in\widetilde{\Theta}^\mathfrak{c}_D$ (resp., $\widetilde{\Theta}^\mathfrak{c}_O$). Denote by $\widetilde{\boldsymbol{\mathcal{O}}}_A$ (resp., $\boldsymbol{\mathcal{O}}_A$) the corresponding $\mathbf{GL}_{D,q^2}$ or $\mathbf{O}_{2d+2,q^2}$ (resp., $\mathbf{U}_{D}$ or $\mathbf{O}^-_{2d+2}$)-orbit in $\widetilde{\pmb{\sF}}^\mathfrak{c}_\square\times\widetilde{\pmb{\sF}}^\mathfrak{c}_\square$ (resp., ${\pmb{\sF}}^\mathfrak{c}_\square\times{\pmb{\sF}}^\mathfrak{c}_\square$) associated with $\widetilde{\mathcal{O}}_A$ (resp., $\mathcal{O}_A$), so that $\widetilde{\boldsymbol{\mathcal{O}}}_A(\mathbb{F}_{q^2})=\widetilde{\mathcal{O}}_A$ (resp., $\boldsymbol{\mathcal{O}}_A(\mathbb{F_q})=\mathcal{O}_A$).
 We have the following result:
\begin{lem}
For $A\in \Theta_\square^\mathfrak{c}$ ($\square=D$ or $O$, $\mathfrak{c}=\imath$ or $\jmath$), we have $\widetilde{\boldsymbol{\mathcal{O}}}_A\simeq\boldsymbol{\mathcal{O}}_A\times_{\mathbb{F}_q}\mathbb{F}_{q^2}$.
\end{lem}
\begin{proof}
Put $k=\mathbb{F}_q$ and $K=\mathbb{F}_{q^2}$.
Let $\mathbf{G}$ be the corresponding algebraic group
$\mathbf{U}_D$ or $\mathbf{O}_{2d+2}^{-}$ over $\mathbb{F}_q$.
Choose $x=(\mathfrak{f},\mathfrak{f}')\in\boldsymbol{\mathcal{O}}_A(\mathbb{F}_q)$
and set $\mathbf{H}=\operatorname{Stab}_{\mathbf{G}}(x)$.
Then $\boldsymbol{\mathcal{O}}_A\simeq\mathbf{G}/\mathbf{H}$, and
$$
\boldsymbol{\mathcal{O}}_A\times_{\mathbb{F}_q} \mathbb{F}_{q^2}
\simeq(\mathbf{G}/\mathbf{H})_{q^2}
\simeq\mathbf{G}_{q^2}/\mathbf{H}_{q^2},
$$ where the last isomorphism is due to the compatibility
of fppf quotients with base change \cite[Tag~0DTF]{Stacks}.
Let $\widetilde{x}$ be the $\mathbb{F}_{q^2}$-point of $\boldsymbol{\mathcal{O}}_A\times_{\mathbb{F}_q} \mathbb{F}_{q^2}$ induces by $x$, we have $\mathrm{Stab}_{\mathbf{G}_{q^2}}(\widetilde{x})\simeq \mathbf{H}_{q^2}$. Thus we can identify
$\boldsymbol{\mathcal{O}}_A\times_{\mathbb{F}_q} \mathbb{F}_{q^2}$ with the
$\mathbf{G}_{q^2}$-orbit of $\widetilde{x}$.
Since the associated matrix of $\widetilde{x}$ remains $A$,
this orbit is precisely $\widetilde{\boldsymbol{\mathcal{O}}}_A$.
\end{proof}

Note that $a_{n+1,n+1}\neq 0$ for each $A=(a_{ij})\in\Theta^\mathfrak{c}_O$ and $a_{n+1,n+1}=0$ for each $A\in\widetilde{\Theta}^\mathfrak{c}_O\setminus\Theta^\mathfrak{c}_O$.
The following is an $\mathbb{F}_{q^2}$-variety version of \cite[Proposition~4.5]{DLZ25}.
\begin{lem}\label{lem:2SO-orbits}
For $A\in\Theta^\mathfrak{c}_O$, the $\mathbf{O}_{2d+2,q^2}$-orbit $\widetilde{\boldsymbol{\mathcal{O}}}_A$ is also a $\mathbf{SO}_{2d+2,q^2}$-orbit, and hence it is geometrically connected. For $A\in\widetilde{\Theta}^\mathfrak{c}_O\setminus\Theta^\mathfrak{c}_O$, the $\mathbf{O}_{2d+2,q^2}$-orbit $\widetilde{\boldsymbol{\mathcal{O}}}_A$ splits into two $\mathbf{SO}_{2d+2,q^2}$-orbits. As a consequence, the $\mathrm{O}_{2d+2}(q^2)$-orbit $\widetilde{\mathcal{O}}_A$ splits into two $\mathrm{SO}_{2d+2}(q^2)$-orbits.
\end{lem}
\begin{proof}
Put $\mathbf{O}=\mathbf{O}_{2d+2,q^2}$, $\mathbf{SO}=\mathbf{SO}_{2d+2,q^2}$. For $A\in\widetilde{\Theta}_O^\mathfrak{c}$, choose $x\in\widetilde{\boldsymbol{\mathcal{O}}}_A(\mathbb{F}_{q^2})$ and let $\mathbf{H}_x=\operatorname{Stab}_\mathbf{O}(x)$.
The orbit map $g\mapsto g\cdot x$ induces an isomorphism
$\mathbf{O}/\mathbf{H}_x\xrightarrow{\ \sim\ }\widetilde{\boldsymbol{\mathcal O}}_A$.
We have $\mathbf{O}=\mathbf{SO}\sqcup \mathbf{SO}s$
for a fixed reflection $s\in \mathbf{O}(\mathbb{F}_{q^2})$ with $\det(s)=-1$.
For $y\in\widetilde{\boldsymbol{\mathcal O}}_A(\mathbb{F}_{q^2})$ we write $\mathbf{SO}\cdot y\subseteq \widetilde{\boldsymbol{\mathcal O}}_A$
for its $\mathbf{SO}$-orbit, which is geometrically connected since it is an image
of the geometrically connected group $\mathbf{SO}$.

\smallskip
\noindent\textit{Case $a_{n+1,n+1}\neq0$.}
By \cite[Proposition~4.5]{DLZ25}, there exists $h\in \mathbf{SO}(\mathbb{F}_{q^2})$ with $s\cdot x=h\cdot x$.
Then $t=h^{-1}s\in \mathbf{H}_x(\mathbb{F}_{q^2})$ satisfies $\det(t)=-1$.
Hence $\mathbf{O}=\mathbf{SO}\sqcup \mathbf{SO}t$, and
\[
\widetilde{\boldsymbol{\mathcal O}}_A=\mathbf{SO}\cdot x\,\cup\,\mathbf{SO}t\cdot x
            =\mathbf{SO}\cdot x\,\cup\,\mathbf{SO}\cdot(t\cdot x)
            =\mathbf{SO}\cdot x .
\]
In particular, $\widetilde{\boldsymbol{\mathcal O}}_A$ is geometrically connected.

\smallskip
\noindent\textit{Case $a_{n+1,n+1}=0$.}
We first show that $\mathbf{H}_x\subseteq \mathbf{SO}$.
Let $k/\mathbb{F}_{q^2}$ be a finite extension, and let $g\in \mathbf{H}_x(k)$.
If $\det(g)=-1$, then $sg\in \mathbf{SO}(k)$ and $(sg)x=sx$, so $sx\in \mathbf{SO}(k)\cdot x$.
This contradicts \cite[Proposition~4.5]{DLZ25} after applied over $k$.
Thus $\mathbf{H}_x(k)\subseteq \mathbf{SO}(k)$ for every finite extension $k/\mathbb{F}_{q^2}$,
and hence $\mathbf{H}_x(\overline{\mathbb{F}_q})\subseteq \mathbf{SO}(\overline{\mathbb{F}_q})$.
Since $\mathbf{SO}=\mathbf{O}^\circ$ is open in $\mathbf{O}$,
$\mathbf{H}_x\cap \mathbf{SO}$ is an open subscheme of $\mathbf{H}_x$ containing all of its geometric points;
thus $\mathbf{H}_x\cap\mathbf{SO}=\mathbf{H}_x$, i.e.\ $\mathbf{H}_x\subseteq\mathbf{SO}$.

Consequently, $\det$ is invariant under right translation by $\mathbf{H}_x$.
It therefore descends to a $\mathbb{F}_{q^2}$-morphism
$$
\delta\colon\widetilde{\boldsymbol{\mathcal O}}_A\simeq \mathbf{O}/\mathbf{H}_x\longrightarrow\{\pm1\},
\qquad g\mathbf{H}_x\longmapsto\det(g).
$$
The fibers
$$
\delta^{-1}(1)=\mathbf{SO}\cdot x
\quad\text{and}\quad
\delta^{-1}(-1)=\mathbf{SO}s\cdot x=\mathbf{SO}\cdot(s\cdot x)
$$
are open and closed $\mathbb{F}_{q^2}$-subvarieties of $\widetilde{\boldsymbol{\mathcal O}}_A$.
Both are nonempty and geometrically connected. Therefore
$$
\widetilde{\boldsymbol{\mathcal O}}_A=\mathbf{SO}\cdot x\ \sqcup\ \mathbf{SO}\cdot(s\cdot x)
$$
is the decomposition of $\widetilde{\boldsymbol{\mathcal O}}_A$ into geometrically connected components.
\end{proof}

\subsection{Frobenius action on the orbits}\label{subsec:frobenius-orbits}
Obviously, the standard Frobenius $F_q$ and the automorphism $\tau$ are also well-defined on $\mathrm{GL}_{D}(q^2)$ and $\mathrm{O}_{2d+2}(q^2)$, hence so is the Frobenius $F=F_q\circ\tau$. It follows from \eqref{eq:Uq2=GLq2} and \eqref{eq:Oq2=Oq2} that
$$\mathrm{U}_D(q)=(\mathrm{GL}_D(q^2))^F \quad \text{and} \quad \mathrm{O}^-_{2d+2}(q)=(\mathrm{O}_{2d+2}(q^2))^F.$$
The Frobenius $F$ can be further defined on $\widetilde{\sF}^\mathfrak{c}_\square$ ($\square=D$ or $O$, $\mathfrak{c}=\imath$ or $\jmath$) naturally, and hence diagonally on $\widetilde{\sF}^\mathfrak{c}_\square\times\widetilde{\sF}^\mathfrak{c}_\square$. It induces an involution, still denoted by $F$ and called Frobenius, on the set of orbits $\mathrm{GL}_{D}(q^2)\setminus\widetilde{\sF}^\mathfrak{c}_D\times\widetilde{\sF}^\mathfrak{c}_D$ or $\mathrm{O}_{2d+2}(q^2)\setminus\widetilde{\sF}^\mathfrak{c}_O\times\widetilde{\sF}^\mathfrak{c}_O$.

\subsubsection{Case of $\mathrm{GL}_{D}(q^2)$}
For any $A=(a_{ij})\in\widetilde{\Theta}^\mathfrak{c}_D$, we denote $A^{\vartheta}=(a_{ij}^{\vartheta})$ with $a_{ij}^{\vartheta}=a_{N+1-i,N+1-j}$. It is obvious that $A^{\vartheta}$ is also in $\widetilde{\Theta}^\mathfrak{c}_D$. Moreover,
\begin{equation}\label{eq:A=Atheta}
A=A^\vartheta \quad\text{if and only if}\quad A\in\Theta^\mathfrak{c}_D.
\end{equation}

\begin{prop}\label{F act on O_D}
    For $A=(a_{ij})\in\widetilde{\Theta}^\mathfrak{c}_D$, the Frobenius $F$ on $\mathrm{GL}_{D}(q^2)\setminus\widetilde{\sF}^\mathfrak{c}_D\times\widetilde{\sF}^\mathfrak{c}_D$ swaps the orbits $\widetilde{\mathcal{O}}_A$ and $\widetilde{\mathcal{O}}_{A^{\vartheta}}$. 
    As a consequence, the subset of $F$-invariant orbits is precisely 
    $\{\widetilde{\mathcal{O}}_A\mid A\in\Theta_D^\mathfrak{c}\}$.
\end{prop}
\begin{proof}
Recall that the Frobenius $F$ on $\mathrm{GL}_{D}(q^2)$ maps $g\in\mathrm{GL}_{D}(q^2)$ to $F(g)=(g')^{-1}$, where $g':=JF_q(g^{tr})J$ and satisfies $\langle g(v),v'\rangle=\langle v,g'(v')\rangle$ for any $v,v'\in V$. Here $\langle-,-\rangle$ is the nondegenerate Hermitian form on $\mathbb{F}_{q^2}^{D}$ defined by \eqref{form1}. 
    
    For $\mathfrak{f}=(0=\mathfrak{f}_0\subset \mathfrak{f}_1\subset \mathfrak{f}_2\subset\cdots\subset \mathfrak{f}_i\subset\cdots\subset \mathfrak{f}_{N-1}\subset \mathfrak{f}_N=\mathbb{F}_{q^2}^{D})\in\widetilde{\sF}^\mathfrak{c}_D$,
    we denote 
    $$\mathfrak{f}^{\vartheta}=(0=\mathfrak{f}_N^\perp\subset \mathfrak{f}_{N-1}^\perp\subset \mathfrak{f}_{N-2}^\perp\subset\cdots\subset \mathfrak{f}_{N-i}^\perp\subset\cdots\subset \mathfrak{f}_{1}^\perp\subset \mathfrak{f}_0^\perp=\mathbb{F}_{q^2}^{D}),$$
    which is a flag in $\widetilde{\sF}^\mathfrak{c}_D$, too.
    
    Let $P_\mathfrak{f}\subset\mathrm{GL}_D(q^2)$ be the stabilizer of $\mathfrak{f}$. It is clear that $F(P_\mathfrak{f})$ is the stabilizer of $F(\mathfrak{f})$ in $\mathrm{GL}_D(q^2)$. We claim $F(P_\mathfrak{f})=P_{\mathfrak{f}^\vartheta}$.
    To prove the claim, it suffices to consider the special case $\mathfrak{f}=(0\subset \mathfrak{f}_1\subset \mathbb{F}_{q^2}^{D})$ because of $P_{\mathfrak{f}}=\cap_i P_{\mathfrak{f}_i}$. In this case, 
    \begin{align*}
    g\in P_\mathfrak{f} \iff g(\mathfrak{f}_1)=\mathfrak{f}_1 &\iff \langle g(\mathfrak{f}_1),\mathfrak{f}_1^\perp\rangle=0 \\
    &\iff \langle\mathfrak{f}_1,g'(\mathfrak{f}_1^\perp)\rangle=0 \iff g'\in P_{\mathfrak{f}^{\vartheta}}\iff F(g)\in P_{\mathfrak{f}^{\vartheta}},
    \end{align*}
by which we prove the claim $F(P_\mathfrak{f})=P_{\mathfrak{f}^{\vartheta}}$, and hence 
\begin{equation}\label{eq:Ff=ftheta}
F(\mathfrak{f})=\mathfrak{f}^{\vartheta}.
\end{equation}
 Take $(\mathfrak{f},\mathfrak{f}')\in\widetilde{\mathcal{O}}_A$ with $\mathfrak{f}=(\mathfrak{f}_i), \mathfrak{f}'=(\mathfrak{f}_i')\in\widetilde{\sF}^\mathfrak{c}_D$. It follows from \eqref{eq:Ff=ftheta} that 
    $F(\mathfrak{f},\mathfrak{f}')=(\mathfrak{f}^{\vartheta},{\mathfrak{f}'}^{\vartheta})$.
    Denote $$r_{ij}=\dim(\mathfrak{f}_i\cap \mathfrak{f}'_j)\quad \text{and}\quad r'_{ij}=\dim(\mathfrak{f}_i^\perp\cap {\mathfrak{f}'_j}^\perp).$$
Since $\mathfrak{f}_i^\perp\cap {\mathfrak{f}'_j}^\perp=(\mathfrak{f}_i+\mathfrak{f}_j')^\perp$, we have 
\begin{equation}\label{eq:rij=rij}
r'_{ij}=D-\dim(\mathfrak{f}_i)-\dim(\mathfrak{f}_j')+r_{ij}.
\end{equation}
Compute
\begin{align*}
a_{N+1-i,N+1-j}&=\dim\frac{\mathfrak{f}_{N-i}+\mathfrak{f}_{N+1-i}\cap \mathfrak{f}'_{N+1-j}}{\mathfrak{f}_{N-i}+\mathfrak{f}_{N+1-i}\cap \mathfrak{f}'_{N-j}}\\
&=r_{N+1-i,N+1-j}-r_{N-i,N+1-j}-r_{N+1-i,N-j}+r_{N-i,N-j}\\
&=r'_{N+1-i,N+1-j}-r'_{N-i,N+1-j}-r'_{N+1-i,N-j}+r'_{N-i,N-j}\qquad \text{by \eqref{eq:rij=rij}}\\
&=\dim\frac{\mathfrak{f}^\perp_{N-i+1}+\mathfrak{f}^\perp_{N-i}\cap {\mathfrak{f}'}_{N-j}^\perp}{\mathfrak{f}^\perp_{N-i+1}+\mathfrak{f}^\perp_{N-i}\cap {\mathfrak{f}'}_{N-j+1}^\perp}=\dim\frac{\mathfrak{f}^\vartheta_{i-1}+\mathfrak{f}^\vartheta_{i}\cap {\mathfrak{f}'}_{j}^\vartheta}{\mathfrak{f}^\vartheta_{i-1}+\mathfrak{f}^\vartheta_{i}\cap {\mathfrak{f}'}_{j-1}^\vartheta},
\end{align*}
which implies $F(\mathfrak{f},\mathfrak{f}')\in\widetilde{\mathcal{O}}_{A^\vartheta}$. Hence, the Frobenius $F$ swaps $\widetilde{\mathcal{O}}_{A}$ and $\widetilde{\mathcal{O}}_{A^\vartheta}$.

As a consequence, by \eqref{eq:A=Atheta}, we find that the set of $F$-invariant $\mathrm{GL}_{D}(q^2)$-orbits in $\widetilde{\sF}^\mathfrak{c}_D\times\widetilde{\sF}^\mathfrak{c}_D$ is $\{\widetilde{\mathcal{O}}_A\mid A\in\Theta_D^\mathfrak{c}\}$.
\end{proof}

The following is the $\mathbb{F}_{q^2}$-varieties version of the above proposition.
\begin{cor}\label{F-var_D}
    For $A=(a_{ij})\in\widetilde{\Theta}^\mathfrak{c}_D$, the Frobenius $F$ of $\widetilde{\pmb{\sF}}^\mathfrak{c}_D\times\widetilde{\pmb{\sF}}^\mathfrak{c}_D$ swaps the orbits $\widetilde{\boldsymbol{\mathcal{O}}}_A$ and $\widetilde{\boldsymbol{\mathcal{O}}}_{A^{\vartheta}}$ and the set of $\mathbf{U}_{D,q^2}$-orbits on $\widetilde{\pmb{\sF}}^\mathfrak{c}_D\times\widetilde{\pmb{\sF}}^\mathfrak{c}_D$ are $\{\widetilde{\boldsymbol{\mathcal{O}}}_A\mid A\in\widetilde{\Theta}_D^\mathfrak{c}\}$ precisely. 
    Therefore, the subset of $F$-invariant orbits is precisely 
    $\{\widetilde{\boldsymbol{\mathcal{O}}}_A\mid A\in\Theta_D^\mathfrak{c}\}$.
\end{cor}
\begin{proof}
   Let $x\in\widetilde{\boldsymbol{\mathcal{O}}}_A(\mathbb{F}_{q^2})$. By Proposition \ref{F act on O_D}, $F(x)\in\widetilde{{\mathcal{O}}}_{A^{\vartheta}}=\widetilde{\boldsymbol{\mathcal{O}}}_{A^{\vartheta}}(\mathbb{F}_{q^2})$. We have $F(\widetilde{\boldsymbol{\mathcal{O}}}_{A})=F(\mathbf{U}_{D,q^2}\cdot x)=\mathbf{U}_{D,q^2}\cdot F(x)=\widetilde{\boldsymbol{\mathcal{O}}}_{A^{\vartheta}}$ as desired.
   By the Bruhat decomposition in the split case, $\bigsqcup_{A\in\widetilde{\Theta}^\mathfrak{c}_D}\widetilde{\boldsymbol{\mathcal{O}}}_A(\mathbb{F}_{q^2})$ exhausts all the $\mathbb{F}_{q^2}$-rational points of $\widetilde{\pmb{\sF}}^\mathfrak{c}_D\times\widetilde{\pmb{\sF}}^\mathfrak{c}_D$, from which the last assertion follows.
\end{proof}

\subsubsection{Case of $\mathrm{O}_{2d+2}(q^2)$}
Let us turn to study the Frobenius $F$ on $\mathrm{O}_{2d+2}(q^2)\setminus\widetilde{\sF}^\mathfrak{c}_O\times\widetilde{\sF}^\mathfrak{c}_O$.

\begin{prop}\label{O-orbits}
   \begin{itemize}
   \item[(1).] For each $A\in\widetilde{\Theta}^\mathfrak{c}_O$, the $\mathrm{O}_{2d+2}(q^2)$-orbit $\widetilde{\mathcal{O}}_A$ is $F$-invariant. 
   \item[(2).] For $A\in\widetilde{\Theta}^\mathfrak{c}_O\setminus \Theta^\mathfrak{c}_O$, the Frobenius $F$ swaps the two $\mathrm{SO}_{2d+2}(q^2)$-orbits in $\widetilde{\mathcal{O}}_A$. 
   \end{itemize}
\end{prop}
\begin{proof}
The Frobenius $F$ on $\mathrm{O}_{2d+2}(q^2)$ is $F=\mathrm{Ad}_g\circ F_q$, where $g$ is the matrix defined in \eqref{def:g}. Since $g\in \mathrm{O}_{2d+2}(q^2)$, each $\mathrm{O}_{2d+2}(q^2)$-orbit in $\widetilde{\sF}_O^\mathfrak{c}\times\widetilde{\sF}_O^\mathfrak{c}$ is $F$-invariant. 

By Lemma~\ref{lem:2SO-orbits}, for $A\in\widetilde{\Theta}^\mathfrak{c}_O\setminus \Theta^\mathfrak{c}_O$, the $\mathrm{O}_{2d+2}(q^2)$-orbit $\widetilde{\mathcal{O}}_A$ splits into two $\mathrm{SO}_{2d+2}(q^2)$-orbits. Since $g\in\mathrm{O}_{2d+2}(q^2)$ but $g\not\in\mathrm{SO}_{2d+2}(q^2)$ and the standard Frobenius $F_q$ preserve each $\mathrm{SO}_{2d+2}(q^2)$-orbit, the $F$-action on $\widetilde{\mathcal{O}}_A$ must be the swapping of these two components. The proof is complete. 
\end{proof}
The following corollary is an $\mathbb{F}_{q^2}$-variety version of the above proposition. We omit the proof, since it is similar to that of Corollary~\ref{F-var_D} using Proposition~\ref{O-orbits}.
\begin{cor}\label{F-var_O}
    \begin{itemize}
   \item[(1).] For each $A\in\widetilde{\Theta}^\mathfrak{c}_O$, the $\mathbf{O}_{2d+2,q^2}$-orbit $\widetilde{\boldsymbol{\mathcal{O}}}_A$ is $F$-invariant. 
   \item[(2).] For $A\in\widetilde{\Theta}^\mathfrak{c}_O\setminus \Theta^\mathfrak{c}_O$, the Frobenius $F$ swaps the two $\mathbf{SO}_{2d+2,q^2}$-orbits in $\widetilde{\boldsymbol{\mathcal{O}}}_A$. 
   \end{itemize}
\end{cor}

\subsection{Bruhat order} \label{subsec:bruhatorder}
The Bruhat order $\le$ on $\Theta^\mathfrak{c}_\square$ is defined as follows: for $A,B\in\Theta^\mathfrak{c}_\square$,
\begin{equation*}
A\le B \quad\text{if and only if}\quad \mathcal{O}_A\subseteq\overline{\mathcal{O}}_B,
\end{equation*}
where $\overline{\mathcal{O}}_B$ denotes the $\mathbb{F}_q$-rational points of the Zariski closure, in the ambient variety over $\overline{\mathbb{F}_q}$, of the orbit subvariety associated with $\mathcal{O}_B$. Clearly, the orders on $\Theta^\imath_\square$ and $\Theta^\jmath_\square$ are compatible with the inclusion $\Theta^\imath_\square\subset\Theta^\jmath_\square$. 

In general, Bruhat orders defined via $\mathrm{SO}$-orbits and $\mathrm{O}$-orbits are different. Nevertheless, by Lemma~\ref{biject D} (or more precisely, by \cite[Proposition 4.5]{DLZ25} that $\boldsymbol{\mathcal{O}}_A\times_{\mathbb{F}_q}\overline{\mathbb{F}_q}$ is an $\mathrm{SO}_{2d+2}$-orbit if $A\in \Theta^\mathfrak{c}_O$), the above definition is independent of whether we take $\mathrm{G}=\mathrm{O}_{2d+2}$ or $\mathrm{SO}_{2d+2}$.

For $1\leq i,j\leq N$, let $E_{ij}$ be the $N\times N$ matrix whose $(i,j)$-th entry is $1$ and the others are zeros. Denote
    \begin{equation}\label{def:Eijtheta}
        E_{ij}^\theta = E_{ij} + E_{N+1-i,N+1-j},
    \end{equation}
    whose $(i,j)$-entry is $2$ if $i=j=n+1$ and is $1$ otherwise.

\begin{lem}\label{lem:order-minimal}
For $A\in\Theta^\mathfrak{c}_\square$, if $A-E^\theta_{i,i+1}$ or $A-E^\theta_{i+1,i}$ is diagonal for $A\in\Theta^\mathfrak{c}_\square$, then $A$ is minimal under the Bruhat order in $\Theta^\mathfrak{c}_\square$. 
\end{lem}
\begin{proof}
    This follows by the definition of Bruhat order and \cite[Lemma 3.6]{BLM90}, \cite[Lemma 4.6.1]{FL15} (here we recall that for $A\in\Theta^\jmath_O$, $\mathcal{O}_A\times_{\mathbb{F}_q}\overline{\mathbb{F}_q}$ is an $\mathrm{SO}_{2d+2}$-orbit by \cite[Proposition 4.5]{DLZ25}, hence we can apply \cite[Lemma 4.6.1]{FL15}).
\end{proof}

\subsection{Numbers of isotropic subspaces}
   In the next section, we shall count the numbers of certain one-dimensional subspaces to derive multiplication formulas for Schur algebras, for which we need some lemmas.

First, we consider Hermitian forms $x_1y_1^{q}+\cdots+x_ry_r^{q}$ on $\mathbb{F}_{q^2}^{D}$ of rank $r$, which is non-degenerate when $r=D$ and degenerate when $r<D$. We denote \begin{equation}\label{nom:phi}
  \phi(D,\q) := \frac{(\q^{D+1} - (-1)^{D+1})(\q^{D} - (-1)^{D})}{\q^2 - 1}.
   \end{equation}
   \begin{lem}\label{lem:count1}
   (1) Let $\mathbb{F}_{q^2}^{D}$ be equipped with a non-degenerate Hermitian form.
   The number of its one-dimensional isotropic subspaces is $\phi(D-1,q)$.\\
(2) Let $\mathbb{F}_{q^2}^{D}$ be equipped with a degenerate Hermitian form of rank $D-1$, then the number of its one-dimensional isotropic subspaces is $q^2\phi(D-2,q) + 1$.
   \end{lem}
   \begin{proof}
   The first statement is a reformulation of \cite[Theorem 8.1]{BC66}, while the second is a specialization of \cite[Corollary of Theorem 8.1]{BC66}. 
   \end{proof}
    

Next, we consider the bilinear form 
\begin{equation}\label{eq:bilinear}
\sum_{i=1}^{d} (x_iy_{2d+3-i}+x_{2d+3-i}y_i)+x_{d+1}y_{d+1}-\omega x_{d+2}y_{d+2}
\end{equation} on $\mathbb{F}^{2d+2}_{q}$ (the non-degenerate case) or on $\mathbb{F}^{2d+3}_{q}$ (a degenerate case with co-rank $1$). For an odd number $D=2d+1$, we denote $$\psi(D,\q):=\frac{\q^{D} - 1}{\q - 1}-\q^{\frac{D-1}{2}}=\frac{\q^{2d+1} - 1}{\q - 1}-\q^{d}.$$

\begin{lem}\label{1-isotrop D}
    \begin{itemize}
        \item [(1)] Let $\mathbb{F}^{2d+2}_{q}$ be equipped with the bilinear form \eqref{eq:bilinear}. The number of its one-dimensional isotropic subspaces is $\psi(2d+1,q)$.
        \item[(2)] Let $\mathbb{F}_{q}^{2d+3}$ be equipped with the bilinear form \eqref{eq:bilinear} with corank $1$. The number of its one-dimensional isotropic subspaces is $q\psi(2d+1,q)+1$.
    \end{itemize}
\end{lem}
    \begin{proof}
        The first statement follows from \cite[Theorem 6.32 \& Definition 6.22]{LN97}. For the second statement, it follows from the fact that the number of solutions $(x_1,\cdots,x_{2d+3})$ of $2\sum_{1\le i\le d }x_ix_{2d+3-i}+x_{d+1}^2-\omega x_{d+2}^2=0$ over $\mathbb{F}_q^{2d+3}$ is $q(q-1)\psi(2d+1,q)+q$.
    \end{proof}

\section{Geometric realization of $\mathrm{i}$Schur algebras}
    \label{sec:structure_Sj_U(D)}
In this section, we give a geometric realization of the iSchur algebras. We establish explicit multiplication formulas and introduce their standard bases. Finally, we show that these geometric algebras are specializations of the generic two-parameter iSchur algebras studied in \cite{LL21}.

\subsection{iSchur algebras}\label{flag F_D}
Recall
Lemma~\ref{lem:bijections} and Lemma~\ref{biject D} on the description of $\mathrm{U}_D(q)$-orbits and $\mathrm{O}_{2d+2}^{-}(q)$-orbits on double isotropic flag varieties. The parameterizations of these orbits are independent of the finite fields $\mathbb{F}_q$. 

Write 
$$\mathfrak{c}=\imath, \jmath, \quad \text{and} \quad \square=D,O.$$
Recall $\mathcal{A}=\mathbb{Z}[v,v^{-1}]$. We denote
    $$
        \mathbf{S}_{n,d}^{\mathfrak{c},D} := \mathcal{A}_{\mathrm{U}_D(q)}(\mathscr{F}^\mathfrak{c}_D\times\mathscr{F}^\mathfrak{c}_D) \quad\text{and}\quad \mathbf{S}_{n,d}^{\mathfrak{c},O} := \mathcal{A}_{\mathrm{O}_{2d+2}^{-}(q)}(\mathscr{F}^\mathfrak{c}_O\times\mathscr{F}^\mathfrak{c}_O)=\mathcal{A}_{\mathrm{SO}_{2d+2}^{-}(q)}(\mathscr{F}^\mathfrak{c}_O\times\mathscr{F}^\mathfrak{c}_O),
    $$ 
where the equality $\mathcal{A}_{\mathrm{O}_{2d+2}^{-}(q)}(\mathscr{F}^\mathfrak{c}_O\times\mathscr{F}^\mathfrak{c}_O)=\mathcal{A}_{\mathrm{SO}_{2d+2}^{-}(q)}(\mathscr{F}^\mathfrak{c}_O\times\mathscr{F}^\mathfrak{c}_O)$ comes from Lemma~\ref{biject D}.
    
    For $A\in\Theta_\square^\mathfrak{c}$, let $e_A$ be the characteristic function of the orbit $\mathcal{O}_A$. Lemmas~\ref{lem:bijections} \& \ref{biject D} imply that $\mathbf{S}_{n,d}^{\mathfrak{c},\square}$ is a free $\mathcal{A}$-module with a basis $\{e_A\mid A\in\Theta_\square^\mathfrak{c}\}$.
    
Applying the general framework given in \S\ref{subsec:General} to $\mathrm{G}=\mathrm{GL}_D$ and $\mathrm{SO}_{2d+2}$, we know that  $\mathbf{S}_{n,d}^{\mathfrak{c},\square}$ is an $\mathcal{A}$-algebra under the convolution product. We call it an iSchur algebra. 

By definition, it is obvious that $\mathbf{S}_{n,d}^{\imath,\square}$ is a subalgebra of $\mathbf{S}_{n,d}^{\jmath,\square}$.

\subsection{Multiplication formulas for $\mathbf{S}_{n,d}^{\mathfrak{c},D}$}\label{subsec:mfD}
Since $\mathbf{S}_{n,d}^{\imath,D}$ is a subalgebra of $\mathbf{S}_{n,d}^{\jmath,D}$. It is enough to focus on the case of $\mathfrak{c}=\jmath$.

For $A=(a_{ij})\in\Theta_D^\jmath$, denote 
$$\mathrm{ro}(A)=(\sum_{j=1}^N a_{1j},\sum_{j=1}^N a_{2j},\ldots,\sum_{j=1}^N a_{Nj})\quad\mbox{and}\quad \mathrm{co}(A)=(\sum_{i=1}^N a_{i1},\sum_{i=1}^N a_{i2},\ldots,\sum_{i=1}^N a_{iN}).$$ It is obvious that $e_B * e_A=0$ if $\mathrm{ro}(A) \neq \mathrm{co}(B)$ for $A,B\in\Theta_D^\jmath$. 
In the sequel, $e_A$ is understood to be $0$ if $A\not\in\Theta_D^\jmath$, to simplify the statement of multiplication formulas.

As Remark~\ref{rem:S=Special} indicates, the following lemma may be derived directly by specializing the multiplication formulas given in \cite{LL21}. However, in order to gain more familiarity with the geometric constructions, we offer a proof via point-counting over finite fields.
 
 Recall matrices $E_{ij}^\theta$ ($i,j=1,2,\ldots,N$) in \eqref{def:Eijtheta}.
\begin{prop}\label{lem:characteristic_function_multiplication}
    Let $A=(a_{ij}),B, C\in\Theta_D^\jmath$ be such that $\mathrm{ro}(A) = \mathrm{co}(B)=\mathrm{co}(C)$ and $B-E_{h,h+1}^\theta, C-E_{h+1,h}^\theta$ are diagonal for some $1\leq h\leq n$. Here are some multiplication formulas:
    \begin{align*}
    e_B * e_A =& \sum_{1\leq k\leq N}
                \q^{2\sum_{j>k}a_{hj}}\frac{\q^{2(1+a_{hk})}-1}{\q^2-1}e_{A+E_{hk}^\theta-E_{h+1,k}^\theta},\\
    \mathop{e_C*e_A}_{(h\neq n)} = &\sum_{1\leq k\leq N}
                \q^{2\sum_{j<k}a_{h+1,j}}\frac{\q^{2(1+a_{h+1,k})}-1}{\q^2-1}e_{A-E_{hk}^\theta+E_{h+1,k}^\theta},\\
     \mathop{e_C * e_A}_{(h=n)} =& \sum_{1\leq k< n+1}\q^{2\sum_{j<k}a_{n+1,j}}\frac{\q^{2(1+a_{n+1,k})}-1}{\q^2-1}e_{A-E_{nk}^\theta+E_{n+1,k}^\theta}\\
                &+ \q^{2\sum_{j<n+1}a_{n+1,j}}\phi(1+a_{n+1,n+1},\q)e_{A-E_{n,n+1}^\theta+E_{n+1,n+1}^\theta}\\
                &+ \sum_{n+1< k\leq N}\q^{1+2\sum_{j<k}a_{n+1,j}}\frac{\q^{2(1+a_{n+1,k})}-1}{\q^2-1}e_{A-E_{nk}^\theta+E_{n+1,k}^\theta}.           
    \end{align*}
    \end{prop}
    \begin{proof}

    It suffices to derive the formulas for $\q=q$ by the definition of the convolution product. 
    
    The verification of the first two formulas is similar to that of \cite[Lemma~3.2]{BLM90}, since each subspace of an isotropic space is naturally isotropic. In the following, we focus on the third formula. Let $A'=(a'_{ij})\in \Theta_D^\jmath$ and take $(\mathfrak{f},\mathfrak{f}')\in \mathcal{O}_{A'}$ so that
        $$
            a'_{ij} = \dim \frac{\mathfrak{f}_{i-1}+\mathfrak{f}_i\cap \mathfrak{f}'_j}{\mathfrak{f}_{i-1}+\mathfrak{f}_i\cap \mathfrak{f}'_{j-1}}.
        $$
        The convolution
        $$
            e_C * e_A(\mathfrak{f},\mathfrak{f}')=\sum_{\mathfrak{f}''\in \mathscr{F}_D}e_C(\mathfrak{f},\mathfrak{f}'')e_A(\mathfrak{f}'',\mathfrak{f}')
        $$
        requires that $\mathfrak{f}''$ is obtained from $\mathfrak{f}$ by replacing $\mathfrak{f}_n$ with an isotropic subspace $\mathfrak{s}\in\mathbb{F}_{q^2}^{D}$ satisfying $\mathfrak{f}_n \overset{1}{\subset} \mathfrak{s}\subset\mathfrak{f}_{n+1}=\mathfrak{f}_n^\perp$. Since $\mathfrak{f}$ and $\mathfrak{f}''$ differ only at the $n$ and $(n+1)$-th subspace, it must be $a_{ij}=a'_{ij}$ unless $i=n,n+1,n+2$.
        
For $1\leq k\leq N$, denote
            $$G_k = \{ \mathfrak{s}\  \mbox{is isotropic} ~|~ \mathfrak{f}_n\overset{1}{\subset}\mathfrak{s}\subset \mathfrak{f}_n^\perp,  \mathfrak{s} \cap \mathfrak{f}'_j = \mathfrak{f}_n \cap \mathfrak{f}'_j ~ \ (\forall j < k), \mathfrak{s}\cap \mathfrak{f}'_j \neq \mathfrak{f}_n \cap \mathfrak{f}'_j ~\ (\forall j \geq k)\}.$$
An isotropic subspace $\mathfrak{s}$ with $\mathfrak{f}_n\overset{1}{\subset}\mathfrak{s}\subset \mathfrak{f}_n^\perp$ lies in $G_k$ if and only if        
\begin{equation}          
        \mathfrak{s} + \mathfrak{f}'_j \neq \mathfrak{f}_n + \mathfrak{f}'_j ~\ (\forall j < k), \quad \mathfrak{s} + \mathfrak{f}'_j = \mathfrak{f}_n + \mathfrak{f}'_j ~\ (\forall j \geq k). \label{condition-n-2}
        \end{equation}
Taking orthogonal complements on both sides of (\ref{condition-n-2}), we have
        \begin{equation}\label{condition-n-3}
            \mathfrak{s}^\perp \cap \mathfrak{f}'_j=\mathfrak{f}_n^\perp \cap \mathfrak{f}'_j  ~(j \leq N-k) \quad \mbox{and} \quad  \mathfrak{s}^\perp \cap \mathfrak{f}'_j\neq \mathfrak{f}_n^\perp \cap \mathfrak{f}'_j ~(j > N-k).
        \end{equation}

For $\mathfrak{f}''$ contributing to the convolution with
$\mathfrak{f}''_n=\mathfrak{s}\in G_k$, we obtain
\[
a'_{nj}=a_{nj}-\delta_{j,k},\qquad
a'_{n+1,j}
=a_{n+1,j}+\delta_{j,k}+\delta_{j,N+1-k}
\quad(1\leq j\leq N).
\]
Moreover, $a'_{ij}=a_{ij}$ for $i<n$.
The remaining entries are determined by central symmetry.
Hence
\[
A'=A-E_{nk}^{\theta}+E_{n+1,k}^{\theta}.
\] Thus
        $$
            e_C * e_A =\sum_{1\leq k\leq N,a_{nk}\geq 1}\# G_k\cdot e_{A-E_{nk}^\theta+E_{n+1,k}^\theta}.
        $$
        The proof reduces to computing $\#G_k$.

        For $k<n+1$, the situation is similar to \cite[Lemma~3.2]{BLM90} again. 
        
        For $k=n+1$, consider $\Psi_{n+1}:\mathfrak{s}\mapsto \mathfrak{s}+\mathfrak{s}^\perp\cap \mathfrak{f}'_n$, mapping $G_{n+1}$ onto 
\begin{equation}\label{equ:G'_n+1_size}
            G'_{n+1} = \{\mathfrak{t}~\mbox{is~isotropic~}\mid  \mathfrak{f}_n+\mathfrak{f}_n^\perp\cap \mathfrak{f}_n'\overset{1}{\subset} \mathfrak{t}\subset \mathfrak{f}_n+\mathfrak{f}_n^\perp\cap \mathfrak{f}'_{n+1}\}.
        \end{equation}
        Since
        $$
            \dim\frac{\mathfrak{f}_n+\mathfrak{f}_n^\perp\cap \mathfrak{f}'_{n+1}}{\mathfrak{f}_n+\mathfrak{f}_n^\perp\cap \mathfrak{f}_n'} = \dim\frac{\mathfrak{s}+\mathfrak{s}^\perp\cap \mathfrak{f}'_{n+1}}{\mathfrak{s}+\mathfrak{s}^\perp\cap \mathfrak{f}_n'} + 2 = a_{n+1,n+1} + 2,
        $$
        and $\mathfrak{t}$ corresponds to isotropic lines in this quotient, it follows from Lemma~\ref{lem:count1} that
        \begin{align*}
            \# G'_{n+1} =\phi(a_{n+1,n+1}+1,q)
        \end{align*}
        For each $\mathfrak{t}\in G'_{n+1}$, the fiber size $\#\Psi^{-1}_{n+1}(\mathfrak{t})=q^{2\sum_{j<n+1}a_{n+1,j}}$ is determined by the same arguments used in the proof of \cite[Lemma~3.2(b)]{BKLW18}. Therefore,
        \begin{align*}
            \# G_{n+1} =\#\Psi_{n+1}^{-1}(\mathfrak{t})\cdot\# G'_{n+1}
            = q^{2\sum_{j<n+1}a_{n+1,j}}\phi(a_{n+1,n+1}+1,q).
        \end{align*}

        Finally, consider $k> n+1$. Define the map $\Psi_k:\mathfrak{s}\mapsto (\mathfrak{s}+\mathfrak{s}^\perp\cap \mathfrak{f}'_i)_{1\leq i\leq n}$ on $G_k$, whose image set (denoted by $G'_k$) consists of the isotropic flags $\mathfrak{e}=(\mathfrak{e}_i)_{1\leq i\leq n}$ in $\mathbb{F}_{q^2}^{D}$ satisfying
        \begin{itemize}
            \item $\mathfrak{e}_i$ is isotropic for $1\leq i\leq n$, with $\mathfrak{f}_n\subseteq \mathfrak{e}_1\subseteq \mathfrak{e}_2\subseteq\cdots\subseteq \mathfrak{e}_n$,
            \item $\mathfrak{f}_n+\mathfrak{f}_n^\perp\cap \mathfrak{f}'_i\overset{1}{\subset} \mathfrak{e}_i$ for $1\leq i\leq N-k$,
            \item $\dim \mathfrak{e}_i=\dim (\mathfrak{f}_n+\mathfrak{f}_n^\perp\cap \mathfrak{f}'_i)$ and $(\mathfrak{e}_i\cap(\mathfrak{f}_n+\mathfrak{f}_n^\perp\cap \mathfrak{f}'_i)) \overset{1}{\subset} \mathfrak{e}_i$ for $N-k+1\leq i\leq n$,
            \item $\mathfrak{e}_1\not\subseteq \mathfrak{f}_n+\mathfrak{f}_n^\perp\cap \mathfrak{f}'_{k-1}$.
        \end{itemize}
        For each $\mathfrak{e}\in G'_k$, the fiber $\Psi_k^{-1}(\mathfrak{e}) \simeq \mathbb{F}_{q^2}^{a_{n+1,1}}$ (cf. \cite[(3.6)]{BKLW18}), and hence its size         \begin{equation}\label{equ:fiber-Psi}
            \# \Psi_k^{-1}(\mathfrak{e}) = q^{2a_{n+1,1}}. 
        \end{equation}

Let $I_k$ be the set of all flags $U=(U_{N-k+1}\subseteq\cdots\subseteq U_n)$ satisfying
        \begin{itemize}
            \item $\mathfrak{f}_n+\mathfrak{f}_n^\perp\cap \mathfrak{f}'_{N-k}\subset U_{N-k+1}\overset{1}{\subset} \mathfrak{f}_n+\mathfrak{f}_n^\perp\cap \mathfrak{f}'_{N-k+1}$ and
            \item $\mathfrak{f}_n+\mathfrak{f}_n^\perp\cap \mathfrak{f}'_{i-1}\not\subset U_i\overset{1}{\subset} \mathfrak{f}_n+\mathfrak{f}_n^\perp\cap \mathfrak{f}'_i$ for $N-k+2\leq i\leq n$.
            \end{itemize}
Decompose $G'_k = \bigsqcup_{U\in I_k}G'_{k,U}$ with $$G'_{k,U}=\{\mathfrak{e}\in G'_k\mid \mathfrak{e}_i\cap(\mathfrak{f}_n+\mathfrak{f}_n^\perp\cap \mathfrak{f}'_i)=U_i,~\forall N-k+1\leq i\leq n\}.$$
The subspace $\mathfrak{e}_n$ of a flag $\mathfrak{e}\in G'_{k,U}$ must satisfy
        $$
            U_n\subset \mathfrak{e}_n\subset U_n^\perp,\quad \mathfrak{e}_n\not\subset \mathfrak{f}_n+\mathfrak{f}_n^\perp\cap \mathfrak{f}'_{n+1}.
        $$
        Compute
        \begin{gather*}
            \dim U_n^\perp/U_n = \dim\frac{\mathfrak{s}+\mathfrak{f}_n^\perp\cap \mathfrak{f}'_{n+1}}{\mathfrak{f}_n+\mathfrak{s}^\perp\cap \mathfrak{f}'_n} = \dim \frac{\mathfrak{s}+\mathfrak{s}^\perp\cap \mathfrak{f}'_{n+1}}{\mathfrak{s}+\mathfrak{s}^\perp\cap \mathfrak{f}'_n}+2 = a_{n+1,n+1}+2,\\
            \dim(\mathfrak{f}_n+\mathfrak{f}_n^\perp\cap \mathfrak{f}'_{n+1})/U_n = \dim \frac{\mathfrak{s}+\mathfrak{s}^\perp\cap \mathfrak{f}'_{n+1}}{\mathfrak{s}+\mathfrak{s}^\perp\cap \mathfrak{f}'_n} +1 = a_{n+1,n+1}+1.
        \end{gather*}
        Since $(\mathfrak{f}_n+\mathfrak{f}_n^\perp\cap \mathfrak{f}'_{n+1})^\perp=\mathfrak{f}_n+\mathfrak{f}_n^\perp\cap \mathfrak{f}'_n\overset{1}{\supset} U_n$, the number of valid $\mathfrak{e}_n$ is equal to the number of one-dimensional isotropic subspaces in $U_n^\perp/U_n$ excluding those in $(\mathfrak{f}_n+\mathfrak{f}_n^\perp\cap \mathfrak{f}'_{n+1})/U_n$. Note that the co-rank of the Hermitian form in $(\mathfrak{f}_n+\mathfrak{f}_n^\perp\cap \mathfrak{f}'_{n+1})/U_n$ is $1$. Thanks to Lemma~\ref{lem:count1}, the number of choices of $\mathfrak{e}_n$ is
\begin{equation}\label{equ:choices_of_W_n}
\phi(a_{n+1,n+1}+1,q)-(q^2\phi(a_{n+1,n+1}-1,q)+1)=q^{2a_{n+1,n+1}+1}.
\end{equation}

For a fixed $\mathfrak{e}_n$, the number of possible $\mathfrak{e}_{n-1}$ is $q^{2\dim U_n/U_{n-1}}=q^{2a_{n+1,n}}$. Inductively, we obtain
        \begin{equation}\label{G'pU}
            \# G'_{k,U} = q \prod_{1\leq i\leq n}q^{2a_{n+1,i+1}}.
        \end{equation}
 Finally, follow \cite[(3.8)]{BKLW18} to obtain the remaining combinatorics:
        \begin{equation}\label{Ip}
            \# I_k=\frac{q^{2(1+a_{n+1,k})}-1}{q^2-1}\prod_{n+2\leq i\leq k-1}q^{2a_{n+1,i}}.
        \end{equation}

        Combining \eqref{equ:fiber-Psi}, \eqref{G'pU}, and \eqref{Ip}, we have
        $$
            \# G_k = q^{2a_{n+1,1}}\# G'_k = q^{2a_{n+1,1}}\# I_k \# G'_{k,U} =\frac{q^{2(1+a_{n+1,k})}-1}{q^2-1}\cdot q\cdot\prod_{1\leq i<k}q^{2a_{n+1,i}}.   
        $$
     \end{proof}

    Consider a pair of matrices $A$ and $B$ such that $A-E_{h,h+1}^\theta$ and $B-RE_{h,h+1}^\theta$ ($R\in\mathbb{Z}^{>0}$) are diagonal and $\mathrm{co}(A) = \mathrm{ro}(B)$. We have
    $$e_A * e_{B} = \frac{\q^{2(1+R)}-1}{\q^2 - 1}e_{C}, \quad \text{where $C-(R+1)E_{h,h+1}^\theta$ is diagonal}.$$
    This extends Proposition~\ref{lem:characteristic_function_multiplication} to the following corollary.

    To simplify the notation, for $t,a\in\mathbb{N}$ with $a\geq t$, we denote
\begin{equation}\label{equ:[t]_and_[N,t]}
[t]_{\q} = \frac{\q^{t} - 1}{\q - 1}, \quad [t]_{\q}!=[t]_\q[t-1]_\q\cdots[1]_\q\quad\mbox{and}\quad
     \left[\begin{array}{c}a\\t\end{array}\right]_\q = \frac{[a]_\q!}{[t]_\q![a-t]_\q!}.
\end{equation}

\begin{cor}\label{cor:characteristic_function_multiplication_R}
    Take $1\leq h\leq n$ and $R\in\mathbb{N}$. If $A=(a_{ij}),B, C\in\Theta_D^\jmath$ satisfy that $\mathrm{ro}(A) = \mathrm{co}(B)=\mathrm{co}(C)$ and $B-RE_{h,h+1}^\theta, C-RE_{h+1,h}^\theta$ are diagonal, then
\begin{align*}
e_B * e_A = &\sum_{\bt}\q^{2\sum_{j>k}a_{hj}t_k}\prod_{k=1}^N \left[\begin{array}{c}a_{hk}+t_k\\t_k\end{array}\right]_{\q^2} e_{A+\sum_{k=1}^Nt_k(E_{hk}^\theta-E_{h+1,k}^\theta)},\\
\mathop{e_C*e_A}_{(h\neq n)} = &\sum_{\bt}\q^{2\sum_{j<k}a_{h+1,j}t_k}\prod_{k=1}^N 
\left[\begin{array}{c}a_{h+1,k}+t_k\\t_k\end{array}\right]_{\q^2}
e_{A-\sum_{k=1}^Nt_k(E_{hk}^\theta-E_{h+1,k}^\theta)},\\
\mathop{e_C*e_A}_{(h=n)} = & \sum_{\bt} \q^{ 2 \sum_{j<k} a_{n+1, j} t_k+2\sum_{N+1-j<k<j} t_kt_j + \sum_{k>n+1} t_k^2} \prod_{k=1}^n\frac{[a_{n+1,k}+t_k+t_{N+1-k}]_{\q^2}!}{[a_{n+1,k}]_{\q^2}![t_k]_{\q^2}![t_{N+1-k}]_{\q^2}!}\\
 &\qquad\qquad \cdot \frac{\prod_{i=1}^{t_{n+1}} \phi(2i-1+a_{n+1,n+1},\q)}{[t_{n+1}]_{\q^2}!} e_{A- \sum_{k=1}^N t_k( E^{\theta}_{nk} - E^{\theta}_{n+1,k})},
\end{align*}
where $\bt=(t_1,\cdots,t_N)$ runs over $\mathbb{N}^N$ with $t_1+\cdots+t_N= R$. 
    \end{cor}
\begin{proof}
        Using the formulas in Proposition~\ref{lem:characteristic_function_multiplication}, the above formulas can be obtained by induction on $R$. We omit the details, which is almost the same as the proof of \cite[Proposition~3.3]{BKLW18}. 
\end{proof}

%
    \subsection{Normalization for $\mathbf{S}_{n,d}^{\mathfrak{c},D}$}\label{subsec:normalization-D}
For $A=(a_{ij})\in\Theta_D^\mathfrak{c}$ and $(\mathfrak{f},\mathfrak{f}')\in\mathcal{O}_A$, let $G_{\mathfrak{f},\mathfrak{f}'}$ be the stabilizer of $(\mathfrak{f},\mathfrak{f}')$ in $\mathbf{U}_D$.

\begin{lem}\label{lem:d(A)}
Let $A=(a_{ij})\in\Theta_D^\mathfrak{c}$ and $(\mathfrak{f},\mathfrak{f}')\in\mathcal{O}_A$.
           \begin{itemize}
            \item [(1)] The dimension of the stabilizer $G_{\mathfrak{f},\mathfrak{f}'}$ (as a variety over $\mathbb{F}_q$) is
            \begin{equation}\label{eq:dimoff}
\dim G_{\mathfrak{f},\mathfrak{f}'}=\sum_{i\ge k, j\ge l}a_{ij}a_{kl}.
            \end{equation}
           
            \item[(2)] The dimension of $\boldsymbol{\mathcal{O}}_A$ (as a variety over $\mathbb{F}_q$), denoted by $d(A)$, is 
            $$d(A)=\sum_{i< k\ \text{or}\ j<l}a_{ij}a_{kl}.$$
        \end{itemize}
   \end{lem}
   \begin{proof}
       It suffices to compute the dimensions after the base change to $\overline{\mathbb{F}_q}$. Since $\mathbf{GL}_{D}\times_{\mathbb{F}_q}\overline{\mathbb{F}_q}\simeq \mathbf{U}_D\times_{\mathbb{F}_q}\overline{\mathbb{F}_q}$, the problem is reduced to the case of type $A$. Thus, the lemma follows by \cite[Lemma 2.2]{BLM90}.  
   \end{proof}
  
    \begin{rem}
    
   Following \cite[\S3.4]{BKLW18}, choose a decomposition
\[
\mathbb F_{q^2}^{D}=\bigoplus_{i,j=1}^{N}Z_{ij},
\qquad
\dim_{\mathbb F_{q^2}}Z_{ij}=a_{ij},
\]
such that
\[
\mathfrak f_r=\bigoplus_{i\le r,j}Z_{ij},
\qquad
\mathfrak f'_s=\bigoplus_{i,j\le s}Z_{ij}.
\]
Order the summands lexicographically. Choose bases such that
the given Hermitian form has a block anti-diagonal Gram
matrix $Q$, with identity blocks on the anti-diagonal.
Then the $\mathbb F_q$-Lie algebra of $G_{\mathfrak f,\mathfrak f'}$
consists of the $\mathbb F_{q^2}$-linear block maps
$x_{(i,j),(k,l)}:Z_{ij}\to Z_{kl}$ satisfying
    \begin{itemize}
        \item[(a)] $x_{(i,j),(k,l)} = 0$ unless $i\geq k$ and $j\geq l$;
        \item[(b)] $\bar{x}_{(i,j),(k,l)} = -x^T_{(N+1-k,N+1-l),(N+1-i,N+1-j)}$, $\forall i,j,k,l\in[1,N]$,
    \end{itemize} where $\bar x=(x_{st}^q)$ denotes entry-wise conjugation of $x=(x_{st})$
under the nontrivial automorphism of
$\mathbb F_{q^2}/\mathbb F_q$.
   Condition (b) can be expressed in matrix form as the standard condition $X^TQ+Q\bar{X}=O$ for the unitary Lie algebra. When $i+k =j+l = N+1$, Condition (b) implies $\bar{x}_{(i,j),(k,l)} = -x^T_{(i,j),(k,l)}$. Thus we obtain another formula for $\dim G_{\mathfrak{f},\mathfrak{f}'}$, distinct from \eqref{eq:dimoff}: 
    $$
      \dim G_{\mathfrak{f},\mathfrak{f}'}=  2\sum_{\substack{i\geq k,j\geq l\\i+k<N+1}} a_{ij}a_{kl} + 2\sum_{\substack{i\geq k,j\geq l\\i+k=N+1,j+l<N+1}} a_{ij}a_{kl}+\sum_{i\geq n+1,j\geq n+1}a_{ij}^2.
    $$ Although these two expressions look different, they are definitely equal. Essentially, this equality follows from the fact that $\dim_{\mathbb{F}_q}(\mathrm{Lie}(G_{\mathfrak{f},\mathfrak{f}'}))=\dim_{\overline{\mathbb{F}_q}}(G_{\mathfrak{f},\mathfrak{f}'}\times_{\mathbb{F}_q}\overline{\mathbb{F}_q})$, and the latter $G_{\mathfrak{f},\mathfrak{f}'}\times_{\mathbb{F}_q}\overline{\mathbb{F}_q}$ is the stabilizer of the base-changed pair $(\mathfrak{f}\otimes\overline{\mathbb{F}_q},\mathfrak{f}'\otimes\overline{\mathbb{F}_q})$ in $\mathbf{GL}_{D}\times_{\mathbb{F}_q}\overline{\mathbb{F}_q}\simeq \mathbf{U}_D\times_{\mathbb{F}_q}\overline{\mathbb{F}_q}$.
    \end{rem}

    For $A\in\Theta_D^\mathfrak{c}$, denote by $r(A)$ the dimension of the image of $\boldsymbol{\mathcal{O}}_A$ under the first projection $\pmb{\mathscr{F}}^\mathfrak{c}_D\times\pmb{\mathscr{F}}^\mathfrak{c}_D\to\pmb{\mathscr{F}}^\mathfrak{c}_D$. It equals $d(B)$ for diagonal $B$ with $b_{ii}=\sum_j a_{ij}$. Applying Lemma~\ref{lem:d(A)}, we have
    \begin{equation*}
        r(A) =\sum\limits_{\substack{1\le i<k\le N\\1\le j,l\le N}}a_{ij}a_{kl}.
    \end{equation*}
Hence
    \begin{equation}\label{d-a}
        d(A)-r(A) =\sum_{i\ge k,j<l}a_{ij}a_{kl}.
    \end{equation}
    
    Introduce the standard basis
    \begin{equation*}
        [A]=v^{-d(A)+r(A)}e_A,\qquad\forall A\in\Theta_D^\mathfrak{c}.
    \end{equation*}
    The set $\{[A]\mid A\in\Theta_D^\mathfrak{c}\}$ forms an $\mathcal{A}$-basis of $\mathbf{S}_{n,d}^{\mathfrak{c},D}$, called the standard basis of $\mathbf{S}_{n,d}^{\mathfrak{c},D}$.

 \subsection{Multiplication formulas for $\mathbf{S}_{n,d}^{\mathfrak{c},O}$}   
   We focus on the case $\mathfrak{c}=\jmath$, since $\mathbf{S}_{n,d}^{\imath,O}$ is a subalgebra of $\mathbf{S}_{n,d}^{\jmath,O}$.
   The multiplication formulas for $\mathbf{S}_{n,d}^{\jmath,O}$ differ only slightly from those for $\mathbf{S}_{n,d}^{\jmath,D}$ derived in \S\ref{subsec:mfD}. By Lemma \ref{1-isotrop D}, the size of $G'_{n+1}$ in \eqref{equ:G'_n+1_size} is replaced by
    $$
        \psi(1+a_{n+1,n+1})=\frac{q^{a_{n+1,n+1}+1}-1}{q-1}-q^{\frac{a_{n+1,n+1}}{2}},
    $$
the number of choices of $\mathfrak{e}_n$ in \eqref{equ:choices_of_W_n} is replaced by
    $$
        \psi(1+a_{n+1,n+1})-q\psi(-1+a_{n+1,n+1})-1 = q^{a_{n+1,n+1}},
    $$
    and the field $\mathbb{F}_{q^2}$ is replaced by $\mathbb{F}_q$. Therefore, we can obtain the following multiplication formulas similar to Proposition~\ref{lem:characteristic_function_multiplication} and Corollary~\ref{cor:characteristic_function_multiplication_R}.
    \begin{prop}\label{lem:characteristic_function_multiplication_O(D)}
        \begin{itemize}
            \item[(a)] For $A,B\in\Theta^\jmath_O$ such that $\mathrm{ro}(A) = \mathrm{co}(B)$ and $B-E_{h,h+1}^\theta$ is a diagonal matrix for some $h\in[1,n]$, we have
            $$
                e_B * e_A = \sum_{\substack{1\leq k\leq N}}
                \q^{\sum_{j>k}a_{hj}}\frac{\q^{1+a_{hk}}-1}{\q-1}e_{A+E_{hk}^\theta-E_{h+1,k}^\theta}.
            $$
            
            \item[(b)] For $A,C\in\Theta^\jmath_O$ such that $\mathrm{ro}(A) = \mathrm{co}(C)$ and $C-E_{h+1,h}^\theta$ is a diagonal matrix for some $h\in[1,n-1]$, we have
            $$
                e_C * e_A = \sum_{\substack{1\leq k\leq N}}
                \q^{\sum_{j<k}a_{h+1,j}}\frac{\q^{1+a_{h+1,k}}-1}{\q-1}e_{A-E_{hk}^\theta+E_{h+1,k}^\theta}.
            $$
    
            \item[(c)] For $A,C\in\Theta^\jmath_O$ such that $\mathrm{ro}(A) = \mathrm{co}(C)$ and $C-E_{n+1,n}^\theta$ is a diagonal matrix, we have
            \begin{align*}
                e_C * e_A = \sum_{\substack{1\leq k\leq N, k\neq n+1}}\q^{\sum_{j<k}a_{n+1,j}}\frac{\q^{1+a_{n+1,k}}-1}{\q-1}e_{A-E_{nk}^\theta+E_{n+1,k}^\theta}\\
                + \q^{\sum_{j<n+1}a_{n+1,j}}\psi(1+a_{n+1,n+1})e_{A-E_{n,n+1}^\theta+E_{n+1,n+1}^\theta}.
            \end{align*}
        \end{itemize}
    \end{prop}

    \begin{cor}\label{cor:characteristic_function_multiplication_R_O(D)}
        Take $1\leq h\leq n$ and $R\in\mathbb{N}$. If $A=(a_{ij}),B, C\in\Theta^\jmath_O$ satisfy that $\mathrm{ro}(A) = \mathrm{co}(B)=\mathrm{co}(C)$ and $B-RE_{h,h+1}^\theta, C-RE_{h+1,h}^\theta$ are diagonal, then
            \begin{align*}
           e_B * e_A &= \sum_{\bt}\q^{\sum_{j>u}a_{hj}t_u}\prod_{u=1}^N\left[\begin{array}{c}a_{h,u}+t_u\\t_u\end{array}\right]_{\q}e_{A+\sum_{u=1}^Nt_u(E_{hu}^\theta-E_{h+1,u}^\theta)}, \\
            \mathop{e_C*e_A}_{(h\neq n)}& = \sum_{\bt}\q^{\sum_{j<u}a_{h+1,j}t_u}\prod_{u=1}^N\left[\begin{array}{c}a_{h+1,u}+t_u\\t_u\end{array}\right]_{\q} e_{A-\sum_{u=1}^Nt_u(E_{hu}^\theta-E_{h+1,u}^\theta)},
            \\
             \mathop{e_C*e_A}_{(h=n)}  &= \sum_{\bt} \q^{\sum_{j<k} a_{n+1, j} t_k+\sum_{N+1-j<k<j} t_kt_j + \sum_{k>n+1} \frac{t_k(t_k-1)}{2}} \prod_{k=1}^n\frac{[a_{n+1,k}+t_k+t_{N+1-k}]_\q!}{[a_{n+1,k}]_\q![t_k]_\q![t_{N+1-k}]_\q!}\\
 &\qquad\qquad \cdot \frac{\prod_{i=1}^{t_{n+1}} \psi(2i-1+a_{n+1,n+1})}{[t_{n+1}]_\q!} e_{A- \sum_{k=1}^N t_k( E^{\theta}_{nk} - E^{\theta}_{n+1,k})},   
            \end{align*}
       where $\bt=(t_1,\cdots,t_N)$ runs over $\mathbb{N}^N$ with $t_1+\cdots+t_N= R$.     
    \end{cor}
\subsection{Normalization for $\mathbf{S}_{n,d}^{\mathfrak{c},O}$}\label{subsec:normalization-O}
   For $A=(a_{ij})\in\Theta^\mathfrak{c}_O$, let us compute the dimension of the $\mathbf{O}^-_{2d+2}$-orbit $\boldsymbol{\mathcal{O}}_A$. Let $(\mathfrak{f},\mathfrak{f}')\in\mathcal{O}_A$.
   \begin{lem}\label{lem:d(A) type D}
        \begin{itemize}
            \item [(1)] The dimension of the stabilizer $G_{\mathfrak{f},\mathfrak{f}'}$ of $(\mathfrak{f},\mathfrak{f}')$ in $\mathbf{O}^-_{2d+2}$ is $$\dim_{{\mathbb{F}_{q}}}G_{\mathfrak{f},\mathfrak{f}'}=\frac{1}{2}\bigg(\sum\limits_{i\ge k, j\ge l}a_{ij}a_{kl}-\sum\limits_{i,j\ge n+1}a_{ij}\bigg).$$
            \item[(2)] The dimension of $\boldsymbol{\mathcal{O}}_A$, denoted by $d(A)$, is given by 
            $$d(A)=\frac{1}{2}\bigg(\sum\limits_{i< k\ \text{or}\ j<l}a_{ij}a_{kl}-\sum\limits_{i< n+1\ \text{or}\ j<n+1}a_{ij}\bigg).$$
        \end{itemize}
   \end{lem}
   \begin{proof}
       Since $\mathbf{O}^-_{2d+2}\times_{\mathbb{F}_q}{\overline{\mathbb{F}_q}}\simeq\mathbf{O}_{2d+2}\times_{\mathbb{F}_q}{\overline{\mathbb{F}_q}}$, the problem is reduced to the case of the split type $D$, and the lemma then follows from \cite[Lemma 4.5.1]{FL15}.  
   \end{proof}
 For $A\in\Theta^\mathfrak{c}_O$, denote the dimension of the image of $\boldsymbol{\mathcal{O}}_A$ under the first projection $\pmb{\mathscr{F}}^\mathfrak{c}_O\times\pmb{\mathscr{F}}^\mathfrak{c}_O\to\pmb{\mathscr{F}}^\mathfrak{c}_O$ by $r(A)$, which is equal to $d(B)$ for the diagonal $B$ with $b_{ii}=\sum_j a_{ij}$. Applying Lemma~\ref{lem:d(A) type D}, we have
    \begin{equation*}
        r(A) =\frac{1}{2}\bigg(\sum\limits_{\substack{1\le i<k\le N\\1\le j,l\le N}}a_{ij}a_{kl}-\sum\limits_{\substack{1\le i<n+1\\1\le j\le N}}a_{ij}\bigg).
    \end{equation*}
Hence
    \begin{equation}\label{d-a type D}
        d(A)-r(A) =\frac{1}{2}\bigg(\sum\limits_{i\ge k,j<l}a_{ij}a_{kl}-\sum\limits_{i\ge n+1>j}a_{ij}\bigg).
    \end{equation}
    
    Introduce the standard basis
    \begin{equation*}
        [A]=v^{-d(A)+r(A)}e_A,\qquad\forall A\in\Theta^\mathfrak{c}_O.
    \end{equation*}
    The set $\{[A]\mid A\in\Theta^\mathfrak{c}_O\}$ forms an $\mathcal{A}$-basis of $\mathbf{S}_{n,d}^{\mathfrak{c},O}$, called the standard basis of $\mathbf{S}_{n,d}^{\mathfrak{c},O}$.

\subsection{Specialization}\label{subsec:specialization}
The set $\Xi_{n,d}$ of matrices used in \cite{LL21} coincides with our $\Theta^\jmath_{2d+1}$, up to the shift of the matrix indices from $\{-n,\ldots,n\}$ to $\{1, 2,\ldots, N=2n+1\}$. That is, $a_{ij}$ in \cite{LL21} corresponds to our $a_{n+i+1,n+j+1}$. To adapt this index shift, let us update some notation defined therein as follows:
    \begin{itemize}
    \item $I_\mathfrak{a}:=(\{n+1\}\times[n+2,N])\sqcup([n+2,N]\times[1,N])$;
    \item $I_\mathfrak{c}:=I_\mathfrak{a}\sqcup\{(n+1,n+1)\}$;
    \item for $A\in\Theta^\jmath_{2d+1}$, $$\hat{\ell}_\mathfrak{c}(A):=\frac{1}{2}\big(\sum_{\substack{n+1\le x\\n+1>y}}+\sum_{\substack{n+1\ge x\\n+1<y}}\big)a_{xy} \quad\text{and}\quad \hat{\ell}_\mathfrak{a}(A):=\frac{1}{2}\bigg(\sum_{(i,j)\in I_\mathfrak{c}}\big(\sum_{\substack{x\le i\\y>j}}+\sum_{\substack{x\ge i\\y<j}}\big)a_{ij}^{\natural\natural}a_{xy}\bigg),$$
where $a_{n+1,n+1}^{\natural\natural}=\frac{1}{2}(a_{n+1,n+1}-3)$ and $a_{ij}^{\natural\natural}=a_{ij}$ if $(i,j)\in I_{\mathfrak{a}}$.
\end{itemize}
With the above new notation, the standard basis elements of $\mathbb{S}_{n,d}^\jmath$, defined in \cite{LL21}, are $$[A]_{LL}:=\mathbf{u}^{-\hat{\ell}_\mathfrak{c}(A)}\mathbf{v}^{-\hat{\ell}_\mathfrak{a}(A)}e_A,\quad (A\in\Theta^\jmath_{2d+1}).$$
The following lemma is obtained by a straightforward computation.
\begin{lem}\label{eq=}
    For $A=(a_{ij})\in\Mat_{N\times N}(\mathbb{Z})$ with $a_{ij}=a_{N+1-i,N+1-j}$, we have $$\frac{1}{2}\big(\sum\limits_{\substack{x\le n+1\\y>n+1}}+\sum\limits_{\substack{x\ge n+1\\y<n+1}}\big)a_{n+1,n+1}a_{xy}+\sum\limits_{(i,j)\in I_\mathfrak{a}}\big(\sum\limits_{\substack{x\le i\\y>j}}+\sum\limits_{\substack{x\ge i\\y<j}}\big)a_{ij}a_{xy}=\sum\limits_{i\ge x,j<y}a_{ij}a_{xy}.$$ In particular, the above equation holds for any $A\in\Theta^\mathfrak{c}_\square$.
\end{lem}
We denote both sides of the above equation by $L(A)$.

For $A\in \Theta^\mathfrak{c}_{2d+1}$, it is clear that $A-E_{n+1,n+1}\in \Theta^{\mathfrak{c}}_{2d}$ and $A+E_{n+1,n+1}\in \Theta^{\mathfrak{c}}_{O}$.

\begin{thm} \label{thm:specialofS}
\begin{itemize}
    \item [(1).] The algebra $\mathbf{S}^{\mathfrak{c},2d+1}_{n,d}$ is a specialization of $\mathbb{S}^\mathfrak{c}_{n,d}$ at $\mathbf{u}^2=\q^3$ and $\mathbf{v}^2=\q^2$. Moreover, the specialization is compatible with the standard bases in the sense that $[A]_{LL}\mapsto [A]$.
    \item[(2).]  The algebra $\mathbf{S}^{\mathfrak{c},2d}_{n,d}$ is a specialization of $\mathbb{S}^\mathfrak{c}_{n,d}$ at $\mathbf{u}^2=\q$ and $\mathbf{v}^2=\q^2$. Moreover, the specialization is compatible with the standard bases in the sense that $[A]_{LL}\mapsto [A-E_{n+1,n+1}]$.
    \item[(3).] The algebra $\mathbf{S}^{\mathfrak{c},O}_{n,d}$ is a specialization of $\mathbb{S}^\mathfrak{c}_{n,d}$ at $\mathbf{u}^2=\q^2$ and $\mathbf{v}^2=\q$. Moreover, the specialization is compatible with the standard bases in the sense that $[A]_{LL}\mapsto [A+E_{n+1,n+1}]$. 
\end{itemize}
\end{thm}
\begin{proof}
We can directly check that the multiplication formulas given in Corollary \ref{cor:characteristic_function_multiplication_R} and Corollary \ref{cor:characteristic_function_multiplication_R_O(D)} above are specializations of those shown in \cite[Proposition 3.2.2]{LL21}. More precisely, 
\begin{align*}
\mathbb{S}^\mathfrak{c}_{n,d}&\twoheadrightarrow\mathbf{S}_{n,d}^{\mathfrak{c},2d+1}\quad\text{via}\ e_A\mapsto e_A,\ \mathbf{u}^2\mapsto \q^3,\ \mathbf{v}^2\mapsto \q^2,\\ \mathbb{S}^\mathfrak{c}_{n,d}&\twoheadrightarrow\mathbf{S}_{n,d}^{\mathfrak{c},2d}\quad\text{via}\ e_A\mapsto e_{A-E_{n+1,n+1}},\ \mathbf{u}^2\mapsto \q,\ \mathbf{v}^2\mapsto \q^2,\\
\mathbb{S}^\mathfrak{c}_{n,d}&\twoheadrightarrow\mathbf{S}_{n,d}^{\mathfrak{c},O}\quad\text{via}\ e_A\mapsto e_{A+E_{n+1,n+1}},\ \mathbf{u}^2\mapsto \q^2,\ \mathbf{v}^2\mapsto \q.
\end{align*}

Next, let us show that the above three specializations are compatible with the standard bases.
 For (1), by Lemma \ref{eq=}, we have $$2\hat{\ell}_\mathfrak{a}(A)=d(A)-r(A)-3\hat{\ell}_\mathfrak{c}(A),$$ and hence $[A]_{LL}$ maps to $[A]$ by definition. 
 For (2), let $A':=(a'_{ij})=A-E_{n+1,n+1}$, and note that $d(A')-r(A')=L(A')$. Using Lemma \ref{eq=}, we compute
\begin{align*}
2\hat{\ell}_{\mathfrak{a}}(A)&=L(A')-2\hat{\ell}_{\mathfrak{c}}(A)+\bigg(\sum_{\substack{(i,j)\in I_{\mathfrak{a}}\\j<n+1\le i}}a_{ij}+\sum_{\substack{(i,j)\in I_{\mathfrak{a}}\\i\le n+1<j}}a_{ij}\bigg)
\\&=d(A')-r(A')-2\hat{\ell}_\mathfrak{c}(A)+\hat{\ell}_\mathfrak{c}(A)=d(A')-r(A')-\hat{\ell}_\mathfrak{c}(A).
\end{align*}
Therefore, $[A]_{LL}$ maps to $[A-E_{n+1,n+1}]$ by definition.
 For (3), let $A'':=(a_{ij}'')=A+E_{n+1,n+1}$, and note that $d(A'')-r(A'')=\frac{1}{2}\big(L(A'')-\hat{\ell}_\mathfrak{c}(A)\big)$.
By \eqref{d-a type D} and Lemma~\ref{eq=}, we calculate 
\begin{align*}
\hat{\ell}_{\mathfrak{a}}(A)&=\frac{1}{2}L(A'')-2\hat{\ell}_{\mathfrak{c}}(A)-\frac{1}{2}\bigg(\sum_{\substack{(i,j)\in I_{\mathfrak{a}}\\j<n+1\le i}}a_{ij}+\sum_{\substack{(i,j)\in I_{\mathfrak{a}}\\i\le n+1<j}}a_{ij}\bigg)\\
    &=d(A'')-r(A'')-2\hat{\ell}_\mathfrak{c}(A).
\end{align*}
Hence, $[A]_{LL}$ maps to $[A+E_{n+1,n+1}]$ by definition. The proof is complete.
\end{proof}

\section{Canonical basis and integrality}
\label{sec:canonicalintegral}
In this section, we introduce the canonical basis of $\mathbf{S}^{\mathfrak{c},\square}_{n,d}$ via Kazhdan-Lusztig polynomials. We modify the extra generator $\mathbf{t}$ to be the shadow of the intersection cohomology complex on its support variety. This modified generator helps us obtain an integral form of $\mathbf{S}_{n,d}^{\mathfrak{c},\square}$.

\subsection{Intersection cohomology complex} \label{subsec:ICC}
Let $\mathbf{G}$ be a split connected reductive group over $\mathbb{F}_q$ with two parabolic subgroups $\mathbf{P}$ and $\mathbf{Q}$.
For a diagonal $\mathbf{G}$-orbit $\boldsymbol{\mathcal{O}}$ in $\mathbf{G}/\mathbf{P}\times\mathbf{G}/\mathbf{Q}$, let $\mathrm{IC}_{\boldsymbol{\mathcal{O}}}$ be the
$\ell$-adic intersection cohomology complex with coefficients in $\overline{\mathbb{Q}}_{\ell}$ on the closure of $\boldsymbol{\mathcal{O}}$.
For a $\mathbf{G}$-variety $\mathbf{X}$, we say a complex $\mathcal{F}\in D^b_{\mathrm{m},\mathbf{G}}(\mathbf{X})$ pointwise {\em $*$-even}, if for each geometric point $x_0$ of $\mathbf{X}$,  the $i$-th cohomology group $\mathsf{H}^i(\mathcal{F}_{x_0})=0$ unless $i$ is even; say {\em $*$-odd} if $\mathcal{F}[1]$ is $*$-even.
\begin{lem}\label{even tate}
    The intersection cohomology complex $\mathrm{IC}_{\boldsymbol{\mathcal{O}}}$ is pointwise
    pure, $*$-even or $*$-odd (depending on the parity of $\dim\boldsymbol{\mathcal{O}}$). The stalk of $\mathrm{IC}_{\boldsymbol{\mathcal{O}}}$ at each $\mathbb{F}_q$-rational point of $\overline{\boldsymbol{\mathcal{O}}}$ is of Tate type.
\end{lem}
\begin{proof}
    Consider the first projection $\boldsymbol{\mathcal{O}}\to \mathbf{G}/\mathbf{P}$. We have $\overline{\boldsymbol{\mathcal{O}}}\simeq\mathbf{G}\times^{\mathbf{P}} {X}$, where ${X}$ is some generalized Schubert variety in $\mathbf{G}/\mathbf{Q}$. Let $i:{X}\to\overline{\boldsymbol{\mathcal{O}}}, x\mapsto[e,x]$ be the closed embedding of the fiber on $e\mathbf{P}/\mathbf{P}$, where $e$ is the unit of $\mathbf{G}$.  Let $\mathrm{IC}_{X}$ be the intersection cohomology complex of ${X}$. By induction equivalence \cite[Proposition 6.2.12]{Ac21}, we have $i^*\mathrm{IC}_{\boldsymbol{\mathcal{O}}}\simeq\mathrm{IC}_X$ (up to a shift). By the $\mathbf{G}$-equivariance of $\mathrm{IC}_{\boldsymbol{\mathcal{O}}}$,
    it is enough to show that $\mathrm{IC}_X$ has the properties stated in the proposition.
    
Let $\mathbf{B}\subset\mathbf{Q}\subset\mathbf{G}$ be a $\mathbb{F}_q$-Borel subgroup.
Let $W$ and $W_Q$ be the Weyl groups of $\mathbf{G}$ and $\mathbf{Q}$, respectively. 
    Without loss of generality, we assume $X=\overline{\mathbf{B}\cdot w\mathbf{Q}/\mathbf{Q}}$, where $w\in W$ is the maximal representative of a right coset $wW_Q$. Put $Y=\overline{\mathbf{B}\cdot w\mathbf{B}/\mathbf{B}}$. We have a smooth fibration $\pi:Y\to X$, and hence $\pi^*(\mathrm{IC}_X)=\mathrm{IC}_Y$ up to a shift. Since $Y$ is an ordinary Schubert variety,
$\mathrm{IC}_Y$, with its canonical split Frobenius
structure, is pointwise pure,
$\ast$-even or $\ast$-odd and pointwise Tate at $\mathbb{F}_q$-rational points
(cf. \cite[Theorem~2.2.1]{dCHL18}).
For any $x\in X(\overline{\mathbb{F}_q})$, we can choose $y\in Y(\overline{\mathbb{F}_q})$ with $\pi(y)=x$.  Since $\pi$ is also surjective on the $\mathbb{F}_q$-rational points, for any $x_0\in X(\mathbb{F}_q)$, we can choose $y_0\in Y(\mathbb{F}_q)$ with $\pi(y_0)=x_0$. We have $(\mathrm{IC}_Y)_y=(\pi^*(\mathrm{IC}_X))_y=(\mathrm{IC}_X)_x$ (resp., $(\mathrm{IC}_Y)_{y_0}=(\pi^*(\mathrm{IC}_X))_{y_0}=(\mathrm{IC}_X)_{x_0}$) up to a shift, the lemma follows.
\end{proof}

As before, we write $\square$ for $D$ or $O$, and $\mathfrak{c}$ for $\imath$ or $\jmath$. For $A\in\Theta_\square^\mathfrak{c}$, recall that the orbit $\boldsymbol{\mathcal{O}}_A$ is the $\mathbb{F}_q$-variety associated with the geometric Frobenius $F$ on $\widetilde{\boldsymbol{\mathcal{O}}}_A$. Denote by $\boldsymbol{\mathcal{O}}'_A$ the $\mathbb{F}_q$-variety associated with the split Frobenius $F_q$ on $\widetilde{\boldsymbol{\mathcal{O}}}_A$.
The orbits $\widetilde{\boldsymbol{\mathcal{O}}}_A$, $\boldsymbol{\mathcal{O}}_A$ and $\boldsymbol{\mathcal{O}}'_A$ are all locally closed, smooth, geometrically connected. This is clear for the case of $\square=D$ since $\mathbf{GL}_D$ is geometrically connected. For the case of $\square=O$, it follows from the geometrically connectedness of $\mathbf{SO}_{2d+2}$ because these $\mathbf{O}$-orbits are all $\mathbf{SO}$-orbits by Lemmas~\ref{biject D} \& \ref{lem:2SO-orbits}. Therefore, we can define the intersection cohomology complexes $\widetilde{\mathrm{IC}}_A$, $\mathrm{IC}_A$ and $\mathrm{IC}'_A$ as follows. 

Let $\ell$ be a prime number coprime to $q$.
Let \begin{equation}\label{normal-IC}\widetilde{\mathrm{IC}}_A:=\mathrm{IC}(\widetilde{\boldsymbol{\mathcal{O}}}_A,\underline{\overline{\mathbb{Q}}_\ell})[-\dim\widetilde{\boldsymbol{\mathcal{O}}}_A]\end{equation} be the ($\ell$-adic) shifted intersection cohomology complex on the closure of $\widetilde{\boldsymbol{\mathcal{O}}}_A$, whose restriction to $\widetilde{\boldsymbol{\mathcal{O}}}_A$ is the constant sheaf $\underline{\overline{\mathbb{Q}}_\ell}_{\widetilde{\boldsymbol{\mathcal{O}}}_A}$. Similarly, set $$\mathrm{IC}_A:=\mathrm{IC}(\boldsymbol{\mathcal{O}}_A,\underline{\overline{\mathbb{Q}}_\ell})[-\dim\boldsymbol{\mathcal{O}}_A] \quad\text{and}\quad \mathrm{IC}_A':=\mathrm{IC}(\boldsymbol{\mathcal{O}}_A',\underline{\overline{\mathbb{Q}}_\ell})[-\dim\boldsymbol{\mathcal{O}}_A'],$$ whose base changes to $\mathbb F_{q^2}$ are canonically isomorphic to $\widetilde{\mathrm{IC}}_A$.
There are two canonical Weil structures $\vartheta:F^*(\widetilde{\mathrm{IC}}_A)\simeq\widetilde{\mathrm{IC}}_A$ and $\vartheta':F_q^*(\widetilde{\mathrm{IC}}_A)\simeq\widetilde{\mathrm{IC}}_A$ that correspond to $\mathrm{IC}_A$ and $\mathrm{IC}_A'$, respectively.

Applying Lemma~\ref{even tate} directly, we obtain the following proposition.
\begin{prop}\label{prop:ICC-stalk}
The intersection cohomology complex $\widetilde{\mathrm{IC}}_A$ is pointwise $*$-even, pure of weight $0$. The stalks of $\mathsf H^i(\mathrm{IC}'_A)$ at
$\mathbb F_q$-rational points are of Tate type.
\end{prop}


\subsection{KL polynomials and Canonical basis}\label{subsec:canonical-basis}
Take $A,B\in\Theta_\square^\mathfrak{c}$. Let $x_0\in\boldsymbol{\mathcal{O}}_{B}(\mathbb{F}_q)$ and $x_0'\in\boldsymbol{\mathcal{O}}_{B}'(\mathbb{F}_q)$ be two $\mathbb{F}_q$-rational points, which we regard as $\mathbb{F}_{q^2}$-rational points of $\widetilde{\boldsymbol{\mathcal{O}}}_{B}$ canonically. 

Each $\mathsf{H}^i(\widetilde{\mathrm{IC}}_A|_{\widetilde{\mathcal{O}}_{B}})$ is a geometrically constant sheaf. It is clear for the case of $\square=D$ since the stabilizer in $\mathrm{GL}_D$ of each geometric point in $\widetilde{\boldsymbol{\mathcal{O}}}_{B}$ is connected by Lemma~\ref{connected}. For the case of $\square=O$, we again use the fact that each $\widetilde{\boldsymbol{\mathcal{O}}}_B$ is also a $\mathbf{SO}_{2d+2,q^2}$-orbit, the stabilizer in $\mathrm{SO}_{2d+2}$ of $x\in\widetilde{\boldsymbol{\mathcal{O}}}_B(\overline{\mathbb{F}_q})$ is isomorphic to $P\cap w Q w^{-1}$ for two standard parabolic subgroups $P, Q$ of $\mathrm{SO}_{2d+2}$, which is also connected by Lemma~\ref{connected}.

We have explained in \S\ref{subsec:ICC} that $\widetilde{\boldsymbol{\mathcal{O}}}_{B}$ is geometrically connected. So we have canonical isomorphisms 
\begin{equation}\label{iso2}
\mathsf{H}^i(\widetilde{\mathrm{IC}}_A)_{x_0}\simeq\Gamma\big(\widetilde{\boldsymbol{\mathcal{O}}}_{B},\mathsf{H}^i(\widetilde{\mathrm{IC}}_A|_{\widetilde{\boldsymbol{\mathcal{O}}}_{B}})\big)\simeq\mathsf{H}^i(\widetilde{\mathrm{IC}}_A)_{x_0'}.
\end{equation}
Here and below, $\Gamma$ denotes geometric global sections, taken after base change to $\overline{\mathbb {F}_q}$.
Under the isomorphisms \eqref{iso2}, we identify the action $\vartheta_{x_0}:\mathsf{H}^i(\widetilde{\mathrm{IC}}_A)_{x_{0}}\to\mathsf{H}^i(\widetilde{\mathrm{IC}}_A)_{x_{0}}$ (resp., $\vartheta'_{x_0'}:\mathsf{H}^i(\widetilde{\mathrm{IC}}_A)_{x'_{0}}\to\mathsf{H}^i(\widetilde{\mathrm{IC}}_A)_{x'_{0}}$) with the action of $F$ (resp., $F_q$) on $\Gamma\big(\widetilde{\boldsymbol{\mathcal{O}}}_{B},\mathsf{H}^i(\widetilde{\mathrm{IC}}_A|_{\widetilde{\boldsymbol{\mathcal{O}}}_{B}})\big)$.
Introduce the Kazhdan-Lusztig polynomial \begin{align}\label{PA}
P_{B,A}:=&\sum\limits_{i\in\mathbb{Z}}(-1)^i\mathrm{tr}\bigg(F;\Gamma\big(\widetilde{\boldsymbol{\mathcal{O}}}_{B},\mathsf{H}^i(\widetilde{\mathrm{IC}}_{A}|_{\widetilde{\boldsymbol{\mathcal{O}}}_{B}})\big)\bigg)q^{-i/2}(-v)^i\cdot v^{-d(A)+d(B)}\\ \nonumber
=&\sum\limits_{i\in\mathbb{Z}}(-1)^i\mathrm{tr}\big(\vartheta_{x_0};\mathsf{H}^i(\widetilde{\mathrm{IC}}_{A})_{x_0}\big)q^{-i/2}(-v)^i\cdot v^{-d(A)+d(B)}.\end{align} 
Here we use geometric global sections to compare the $F$-action and $F_q$-action on a common vector space. 
We have $P_{A,A}=1$ because $\widetilde{\mathrm{IC}}_A$ descends to $\mathrm{IC}(\boldsymbol{\mathcal{O}}_A,\underline{\overline{\mathbb{Q}}_\ell})[-\dim{\boldsymbol{\mathcal{O}}}_A]$ on $\pmb{\sF}^\mathfrak{c}_\square\times\pmb{\sF}^\mathfrak{c}_\square$.
But unlike the case of equal parameters (cf. \cite[\S 3.6 (3.27)]{BKLW18}), we only have $P_{B,A}\in v^{-1}\mathbb{Z}[v^{-1}]$ if $B<A$. The reason is as follows.

Thanks to Proposition~\ref{prop:ICC-stalk}, the eigenvalues of $F_q$ on $\Gamma\big(\widetilde{\boldsymbol{\mathcal{O}}}_{B},\mathsf{H}^i(\widetilde{\mathrm{IC}}_{A}|_{\widetilde{\boldsymbol{\mathcal{O}}}_{B}})\big)$ are all equal to $q^{i/2}$, by which we have \begin{equation}\label{P_A,B}
P_{B,A}=\sum_{i\in\mathbb{Z}}\mathrm{tr}\bigg(\tau;\Gamma\big(\widetilde{\boldsymbol{\mathcal{O}}}_{B},\mathsf{H}^i(\widetilde{\mathrm{IC}}_{A}|_{\widetilde{\boldsymbol{\mathcal{O}}}_{B}})\big)\bigg)v^{i-d(A)+d(B)}.
\end{equation} Since $\tau$ is an involution, $P_{B,A}$ may have negative coefficients. Now, by \cite[\S 1.4]{BLM90} and \eqref{P_A,B}, we have $P_{B,A}\in v^{-1}\mathbb{Z}[v^{-1}]$ for $B<A$.

 For $A\in\Theta_\square^\mathfrak{c}$, we define $$\{A\}=\sum_{B\in\Theta_\square^\mathfrak{c},B\le A}P_{B,A}[B].$$
 Then $\mathbf{B}_{n,d}^{\mathfrak{c},\square}:=\{\{A\}\mid A\in\Theta_\square^\mathfrak{c}\}$ forms an $\mathcal{A}$-basis for $\mathbf{S}^{\mathfrak{c},\square}_{n,d}$, called the canonical basis.

\subsection{Generators}\label{generator}

\subsubsection{$\jmath$-version}\label{j-version}

For $i=1,2,\ldots,n$ and $\square=D,O$, we define $$\mathbf{e}_i=\sum_{B}[B]=\sum_B\{B\},\quad \mathbf{f}_i=\sum_{C}[C]=\sum_{C}\{C\}\in\mathbf{S}^{\jmath,\square}_{n,d},$$
where the sums run over $B,C\in\Theta_\square^\jmath$ such that $B-E^\theta_{i+1,i}$ and $C-E^\theta_{i,i+1}$ are diagonal. Note that $[B]=\{B\}$ and $[C]=\{C\}$ for these $B$ and $C$ because they are all minimal under the Bruhat order in $\Theta_\square^\jmath$ by Lemma~\ref{lem:order-minimal}.

Let $\Theta_{\square}^{\jmath,\mathrm{diag}}$ be the set of all diagonal matrices in $\Theta_\square^\jmath$. For $\lambda=\text{diag}\{\lambda_1,\cdots,\lambda_N\}\in\Theta_\square^{\jmath,\mathrm{diag}}$, we rewrite $1_\lambda=e_\lambda\in\mathbf{S}^{\jmath,\square}_{n,d}$. For $a=1,2,\ldots,n+1$, we define $$\mathbf{d}_a^{\pm1}=\sum_{\lambda\in\Theta_D^{\jmath,\mathrm{diag}}}v^{\mp2\lambda_a}1_{\lambda}\quad\text{if}\quad \square=D, \quad\text{and} \quad \sum_{\lambda\in\Theta_{O}^{\jmath,\mathrm{diag}}}v^{\mp\lambda_a}1_{\lambda}\quad\text{if}\quad \square=O.$$

The following lemma is an analog of \cite[Corollary~3.13]{BKLW18}.
\begin{lem}
The elements $\mathbf{e}_i, \mathbf{f}_i$ ($1\leq i\leq n$) and $\mathbf{d}_a$ ($1\leq a\leq n+1$) generate the $\mathbb{Q}(v)$-algebra ${}_{\mathbb{Q}}\mathbf{S}^{\jmath,\square}_{n,d}$, ($\square=D,O$).
\end{lem}
\begin{proof}
Similarly to \cite[Theorem~3.10]{BKLW18}, we can prove that the elements $[\lambda+RE_{i,i+1}^\theta]$, $[\lambda+RE_{i+1,i}^\theta]$ ($R\in\mathbb{N}$, $1\leq i\leq n$) generate the whole $\A$-algebra $\mathbf{S}^{\jmath,\square}_{n,d}$, where each $\lambda=\mathrm{diag}(\lambda_1,\ldots,\lambda_N)$ is a diagonal matrix such that $\lambda+RE_{i,i+1}$ or $\lambda+RE_{i+1,i}^\theta\in\Theta^\jmath_\square$. 

Using multiplication formulas derived in Proposition~\ref{lem:characteristic_function_multiplication} \& \ref{lem:characteristic_function_multiplication_O(D)} with normalization, we can check directly that
$$[\lambda+(R-1)E_{i,i}^\theta+E_{i,i+1}^\theta][\lambda+E_{i+1,i+1}^\theta+(R-1)E_{i,i+1}^\theta]=k_1[\lambda+RE_{i,i+1}^\theta]+\ \text{lower terms}, \quad\text{and}$$
$$[\lambda+(R-1)E_{i+1,i+1}^\theta+E_{i+1,i}^\theta][\lambda+E_{i,i}^\theta+(R-1)E_{i+1,i}^\theta]=k_2[\lambda+RE_{i+1,i}^\theta]+\ \text{lower terms},$$ 
where $k_1,k_2\in\mathbb{Q}(v)$ are nonzero, and the lower terms are understood with respect to the Bruhat order defined in \S\ref{subsec:bruhatorder} (see \cite{LL21} for a direct description of the Bruhat order in terms of the entries of the matrices in $\Theta_{\square}^\jmath$).
Therefore, the elements $[\lambda+E_{i,i+1}^\theta]$ and $[\lambda+E_{i+1,i}^\theta]$ ($1\leq i\leq n$) are enough to generate the whole $\mathbb{Q}(v)$-algebra ${}_{\mathbb{Q}}\mathbf{S}^{\jmath,\square}_{n,d}$. Since $[\lambda+E_{i,i+1}^\theta]$ and $[\lambda+E_{i+1,i}^\theta]$ can be generated by $\mathbf{e}_i, \mathbf{f}_i$ ($1\leq i\leq n$) and $\mathbf{d}_a$ ($1\leq a\leq n+1$) via a standard Vandermonde determinant argument, the proof is complete.
\end{proof}

\begin{prop}
   The following relations hold in ${}_\mathbb{Q}\mathbf{S}^{\jmath,D}_{n,d}$: for $1\leq i,j\leq n$ and $1\leq a,b\leq n+1$,
\begin{gather}
\left\{\begin{aligned}\label{common}
&\mathbf{d}_a\mathbf{d}_a^{-1}=1=\mathbf{d}_a^{-1}\mathbf{d}_a, \quad \mathbf{d}_a\mathbf{d}_b=\mathbf{d}_b\mathbf{d}_a,\\
&\mathbf{d}_{n+1}\mathbf{e}_{i}\mathbf{d}_{n+1}^{-1}=v^{-4\delta_{n,i}}\mathbf{e}_i, \quad \mathbf{d}_{n+1}\mathbf{f}_{i}\mathbf{d}_{n+1}^{-1}=v^{4\delta_{n,i}}\mathbf{f}_i,\\
&\mathbf{d}_a\mathbf{e}_i\mathbf{d}_a^{-1}=v^{2\delta_{a,i}-2\delta_{a,i+1}}\mathbf{e}_i,\quad
\mathbf{d}_a\mathbf{f}_i\mathbf{d}_a^{-1}=v^{2\delta_{a,i+1}-2\delta_{a,i}}\mathbf{f}_i \quad(a\neq n+1),\\
&\mathbf{e}_i\mathbf{f}_j-\mathbf{f}_j\mathbf{e}_i=\delta_{i,j}\frac{\mathbf{d}_i\mathbf{d}_{i+1}^{-1}-\mathbf{d}_i^{-1}\mathbf{d}_{i+1}}{v^2-v^{-2}} \quad\big((i,j)\neq(n,n)\big), \\
&\mathbf{e}_i^2\mathbf{e}_j+\mathbf{e}_j\mathbf{e}_i^2=\llbracket2\rrbracket_{v^2}\mathbf{e}_i\mathbf{e}_j\mathbf{e}_i,\quad \mathbf{f}_i^2\mathbf{f}_j+\mathbf{f}_j\mathbf{f}_i^2=\llbracket2\rrbracket_{v^2}\mathbf{f}_i\mathbf{f}_j\mathbf{f}_i \quad (|i-j|=1),\\
&\mathbf{e}_i\mathbf{e}_j=\mathbf{e}_j\mathbf{e}_i,\quad \mathbf{f}_i\mathbf{f}_j=\mathbf{f}_j\mathbf{f}_i \quad (|i-j|>1),\end{aligned}\right.\\
\mathbf{e}_n^2\mathbf{f}_n+\mathbf{f}_n\mathbf{e}_n^2=\llbracket2\rrbracket_{v^2}\big(\mathbf{e}_n\mathbf{f}_n\mathbf{e}_n-\mathbf{e}_n(v^3\mathbf{d}_n\mathbf{d}_{n+1}^{-1}+v^{-3}\mathbf{d}_n^{-1}\mathbf{d}_{n+1})\big),\\
\mathbf{f}_n^2\mathbf{e}_n+\mathbf{e}_n\mathbf{f}_n^2=\llbracket2\rrbracket_{v^2}\big(\mathbf{f}_n\mathbf{e}_n\mathbf{f}_n-(v^3\mathbf{d}_n\mathbf{d}_{n+1}^{-1}+v^{-3}\mathbf{d}_n^{-1}\mathbf{d}_{n+1})\mathbf{f}_n\big).
\end{gather}
    \end{prop}
\begin{proof}
The above relations except the last two are identical to the case of equal parameters. Hence, they can be verified as the same as in \cite{BLM90} and \cite{BKLW18}. 
In the following, we present the details for the penultimate relation, while the last one is similar.

Using Proposition~\ref{lem:characteristic_function_multiplication}, a direct computation brings us
        \begin{align*}
            \e_n^2\f_n 1_\lambda =& v^{-4\lambda_{n+1} - 2\lambda_n + 2}(v^4 + 1)\big(e_{\lambda - 2E_{nn}^\theta + 2E_{n+1,n}^\theta - E_{n+1, n+1}^\theta + E_{n, n+1}^\theta}\\
            &+ \frac{(v^{2\lambda_{n+1}} - (-1)^{\lambda_{n+1}})(v^{2\lambda_{n+1} - 2} - (-1)^{\lambda_{n+1} - 1})}{v^4 - 1} e_{\lambda - E_{nn}^\theta + E_{n+1, n}^\theta}\big),\\
            \f_n\e_n^2 1_\lambda =& v^{-4\lambda_{n+1} - 2\lambda_n - 2}(v^4 + 1)\big(e_{\lambda - 2E_{nn}^\theta + E_{n+1,n}^\theta + E_{n,n+2}^\theta}\\
            &+ e_{\lambda - 2E_{nn}^\theta + 2E_{n+1,n}^\theta - E_{n+1, n+1}^\theta + E_{n, n+1}^\theta} + \frac{v^{4\lambda_n - 4} - 1}{v^4 - 1}e_{\lambda - E_{nn}^\theta + E_{n+1, n}^\theta}\big),\\
            \e_n\f_n\e_n 1_\lambda =& v^{-4\lambda_{n+1} - 2\lambda_n}\big(e_{\lambda - 2E_{nn}^\theta + E_{n+1,n}^\theta + E_{n,n+2}^\theta} \\
            &+ (v^4 + 1)e_{\lambda - 2E_{nn}^\theta + 2E_{n+1,n}^\theta - E_{n+1, n+1}^\theta + E_{n, n+1}^\theta}\\
            &+ (v^{4\lambda_{n+1} + 2} + v^4\frac{(v^{2\lambda_{n+1}} - (-1)^{\lambda_{n+1}})(v^{2\lambda_{n+1} - 2} - (-1)^{\lambda_{n+1} - 1})}{v^4 - 1} \\
            &+ \frac{v^{4\lambda_n} - 1}{v^4 - 1}) e_{\lambda - E_{nn}^\theta + E_{n+1, n}^\theta}\big).
        \end{align*}
Thus,
        \begin{align*}
            (\e_n^2\f_n - \llbracket 2\rrbracket_{v^2} \e_n\f_n\e_n + \f_n\e_n^2) 1_\lambda &= - (v^2 + v^{-2})(v^{-2\lambda_n + 2} + v^{-4\lambda_{n+1} + 2\lambda_n - 4})e_{\lambda - E_{nn}^\theta + E_{n+1, n}^\theta}\\
            &= - \llbracket 2\rrbracket_{v^2}(v^{2\lambda_{n+1} - 2\lambda_n + 3} + v^{-2\lambda_{n+1} + 2\lambda_n - 3})\e_n1_\lambda \\
            &= - \llbracket 2\rrbracket_{v^2} \e_n (v^3\bd_n\bd_{n+1}^{-1} + v^{-3}\bd_n^{-1}\bd_{n+1})1_{\lambda},
        \end{align*}
        by which we get the penultimate relation since $\sum_{\lambda\in\Theta_{D}^{\jmath,\mathrm{diag}}}1_{\lambda}$ is the identity element in ${}_\mathbb{Q}\mathbf{S}_{n,d}^{\jmath,D}$. The proof is complete.
\end{proof}

Similarly, we can use the formulas in Proposition~\ref{lem:characteristic_function_multiplication_O(D)} to prove the following.
\begin{prop}
   The following relations hold in ${}_\mathbb{Q}\mathbf{S}^{\jmath,O}_{n,d}$: for $1\leq i,j\leq n$ and $1\leq a,b\leq n+1$,
\begin{gather}
\left\{
\begin{aligned}\label{common_D}
&\mathbf{d}_a\mathbf{d}_a^{-1}=1=\mathbf{d}_a^{-1}\mathbf{d}_a, \quad \mathbf{d}_a\mathbf{d}_b=\mathbf{d}_b\mathbf{d}_a,\\
&\mathbf{d}_{n+1}\mathbf{e}_{i}\mathbf{d}_{n+1}^{-1}=v^{-2\delta_{n,i}}\mathbf{e}_i, \quad \mathbf{d}_{n+1}\mathbf{f}_{i}\mathbf{d}_{n+1}^{-1}=v^{2\delta_{n,i}}\mathbf{f}_i,\\
&\mathbf{d}_a\mathbf{e}_i\mathbf{d}_a^{-1}=v^{\delta_{a,i}-\delta_{a,i+1}}\mathbf{e}_i,\quad
\mathbf{d}_a\mathbf{f}_i\mathbf{d}_a^{-1}=v^{\delta_{a,i+1}-\delta_{a,i}}\mathbf{f}_i \quad(a\neq n+1),\\
&\mathbf{e}_i\mathbf{f}_j-\mathbf{f}_j\mathbf{e}_i=\delta_{i,j}\frac{\mathbf{d}_i\mathbf{d}_{i+1}^{-1}-\mathbf{d}_i^{-1}\mathbf{d}_{i+1}}{v-v^{-1}} \quad\big((i,j)\neq(n,n)\big), \\
&\mathbf{e}_i^2\mathbf{e}_j+\mathbf{e}_j\mathbf{e}_i^2=\llbracket2\rrbracket_{v}\mathbf{e}_i\mathbf{e}_j\mathbf{e}_i,\quad \mathbf{f}_i^2\mathbf{f}_j+\mathbf{f}_j\mathbf{f}_i^2=\llbracket2\rrbracket_{v}\mathbf{f}_i\mathbf{f}_j\mathbf{f}_i \quad (|i-j|=1),\\
&\mathbf{e}_i\mathbf{e}_j=\mathbf{e}_j\mathbf{e}_i,\quad \mathbf{f}_i\mathbf{f}_j=\mathbf{f}_j\mathbf{f}_i \quad (|i-j|>1),\end{aligned}\right.\\
\mathbf{e}_n^2\mathbf{f}_n+\mathbf{f}_n\mathbf{e}_n^2=\llbracket2\rrbracket_{v}\big(\mathbf{e}_n\mathbf{f}_n\mathbf{e}_n-\mathbf{e}_n(v\mathbf{d}_n\mathbf{d}_{n+1}^{-1}+v^{-1}\mathbf{d}_n^{-1}\mathbf{d}_{n+1})\big),\\
\mathbf{f}_n^2\mathbf{e}_n+\mathbf{e}_n\mathbf{f}_n^2=\llbracket2\rrbracket_{v}\big(\mathbf{f}_n\mathbf{e}_n\mathbf{f}_n-(v\mathbf{d}_n\mathbf{d}_{n+1}^{-1}+v^{-1}\mathbf{d}_n^{-1}\mathbf{d}_{n+1})\mathbf{f}_n\big).
\end{gather}
\end{prop}

\subsubsection{$\imath$-version} \label{subsec:t}
For $i=1,2,\ldots, n-1$ and $\square=D,O$, we denote $$\mathbf{e}_i=\sum_{B}[B]=\sum_{B}\{B\},\quad \mathbf{f}_i=\sum_{C}[C]=\sum_{C}\{C\}\in\mathbf{S}_{n,d}^{\imath,\square},$$
where the sums run over $B,C\in\Theta_\square^{\imath}$ such that $B-E^\theta_{i+1,i}$ and $C-E^\theta_{i,i+1}$ are diagonal.

Denote by $\Theta_{\square}^{\imath,\mathrm{diag}}$ the set of all diagonal matrices in $\Theta_\square^{\imath}$.
It is clear that for $\lambda=\text{diag}\{\lambda_1,\cdots,\lambda_{N}\}\in\Theta_{2d}^{\imath,\mathrm{diag}}$ (resp., $\Theta_{2d+1}^{\imath,\mathrm{diag}}$ and $\Theta_O^{\imath,\mathrm{diag}}$), the central entry $\lambda_{n+1}=0$ (resp., $1$ and $2$). 
For $a=1,2,\ldots,n$, define $$\mathbf{d}_a^{\pm1}:=\sum_{\lambda\in\Theta_D^{\imath,\mathrm{diag}}}v^{\mp2\lambda_a}1_\lambda\quad\text{if}\quad \square=D, \quad\text{and}\quad \sum_{\lambda\in\Theta_O^{\imath,\mathrm{diag}}}v^{\mp\lambda_a}1_\lambda\quad\text{if}\quad \square=O.$$

Finally, in the spirit of \cite[(7.3.25)]{LL21} we define
\begin{gather}\label{def:t}
\mathbf{t}:=
\left\{
\begin{array}{l}
\sum\limits_{\lambda\in \Theta_{2d+1}^{\imath,\mathrm{diag}}}\big([\lambda-E^\theta_{n,n}+E^\theta_{n,n+2}]+v^{-2\lambda_n}\frac{v^3-v^{-3}}{v^2-v^{-2}} 1_\lambda\big)\in{}_{\mathbb{Q}}\mathbf{S}^{\imath,2d+1}_{n,d},\\
\sum\limits_{\lambda\in\Theta_{2d}^{\imath,\mathrm{diag}}}\big([\lambda-E^\theta_{n,n}+E^\theta_{n,n+2}]+v^{-2\lambda_n}\frac{v-v^{-1}}{v^2-v^{-2}} 1_\lambda\big)\in{}_{\mathbb{Q}}\mathbf{S}^{\imath,2d}_{n,d},\\
\sum\limits_{\lambda\in\Theta_{O}^{\imath,\mathrm{diag}}}\big([\lambda-E^\theta_{n,n}+E^\theta_{n,n+2}]+v^{-\lambda_n}\frac{v^2-v^{-2}}{v-v^{-1}} 1_\lambda\big)\in\mathbf{S}^{\imath,O}_{n,d}.
\end{array}
\right.
\end{gather}

\begin{lem}
The elements $\mathbf{e}_i, \mathbf{f}_i$ ($1\leq i\leq n-1$), $\mathbf{t}$ and $\mathbf{d}_a$ ($1\leq a\leq n$) generate the $\mathbb{Q}(v)$-algebra ${}_{\mathbb{Q}}\mathbf{S}^{\imath,\square}_{n,d}$, ($\square=D,O$).
\end{lem}
\begin{proof}
Similarly to \cite[Corollary~5.8]{BKLW18}, we know that $[\lambda+E_{i,i+1}^\theta], [\lambda+E_{i+1,i}^\theta], [\lambda+E_{n,n+2}^\theta]$ ($1\leq i\leq n-1$) generate the $\mathbb{Q}(v)$-algebra ${}_{\mathbb{Q}}\mathbf{S}^{\imath,\square}_{n,d}$, where each $\lambda$ is a diagonal matrix such that $\lambda+E_{i,i+1}^\theta, \lambda+E_{i+1,i}^\theta$ or $\lambda+E_{n,n+2}^\theta\in\Theta_\square^\imath$. Then a standard Vandermonde determinant argument gives us the desired result.
\end{proof}

\begin{prop}
 The following relations hold in ${}_\mathbb{Q}\mathbf{S}^{\imath,D}_{n,d}$: for $1\leq i,j\leq n-1$ and $1\leq a,b\leq n$, the relations \eqref{common} and the relations
 \begin{equation}\label{i_D}
\mathbf{d}_a\mathbf{t}=\mathbf{t}\mathbf{d}_a,\quad
\mathbf{e}_i\mathbf{t}=\mathbf{t}\mathbf{e}_i,\quad\mathbf{f}_i\mathbf{t}=\mathbf{t}\mathbf{f}_i \quad (i\neq n-1);
\end{equation}
\begin{equation}\label{it2}
\left\{\begin{aligned}
&\mathbf{e}^2_{n-1}\mathbf{t}-\llbracket2\rrbracket_{v^2}\mathbf{e}_{n-1}\mathbf{t}\mathbf{e}_{n-1}+\mathbf{t}\mathbf{e}^2_{n-1}=0,\quad \mathbf{f}^2_{n-1}\mathbf{t}-\llbracket2\rrbracket_{v^2}\mathbf{f}_{n-1}\mathbf{t}\mathbf{f}_{n-1}+\mathbf{t}\mathbf{f}^2_{n-1}=0;
\\
&\mathbf{t}^2\mathbf{e}_{n-1}-\llbracket2\rrbracket_{v^2}\mathbf{t}\mathbf{e}_{n-1}\mathbf{t}+\mathbf{e}_{n-1}\mathbf{t}^2=\mathbf{e}_{n-1},\quad \mathbf{t}^2\mathbf{f}_{n-1}-\llbracket2\rrbracket_{v^2}\mathbf{t}\mathbf{f}_{n-1}\mathbf{t}+\mathbf{f}_{n-1}\mathbf{t}^2=\mathbf{f}_{n-1}.
\end{aligned}\right.
\end{equation}
   \end{prop}

\begin{prop}
 The following relations hold in ${}_\mathbb{Q}\mathbf{S}^{\imath,O}_{n,d}$: for $1\leq i,j\leq n-1$ and $1\leq a,b\leq n$, the relations \eqref{common_D} and the relations 
 \begin{gather}
 \left\{\begin{aligned}\label{i_O}
&\mathbf{d}_a\mathbf{t}=\mathbf{t}\mathbf{d}_a,\quad \mathbf{e}_i\mathbf{t}=\mathbf{t}\mathbf{e}_i,\quad\mathbf{f}_i\mathbf{t}=\mathbf{t}\mathbf{f}_i \quad (i\neq n-1);\\
&\mathbf{e}^2_{n-1}\mathbf{t}-\llbracket2\rrbracket_{v}\mathbf{e}_{n-1}\mathbf{t}\mathbf{e}_{n-1}+\mathbf{t}\mathbf{e}^2_{n-1}=0,\quad \mathbf{f}^2_{n-1}\mathbf{t}-\llbracket2\rrbracket_{v}\mathbf{f}_{n-1}\mathbf{t}\mathbf{f}_{n-1}+\mathbf{t}\mathbf{f}^2_{n-1}=0;\\
    &\mathbf{t}^2\mathbf{e}_{n-1}-\llbracket2\rrbracket_{v}\mathbf{t}\mathbf{e}_{n-1}\mathbf{t}+\mathbf{e}_{n-1}\mathbf{t}^2=\mathbf{e}_{n-1},\quad \mathbf{t}^2\mathbf{f}_{n-1}-\llbracket2\rrbracket_{v}\mathbf{t}\mathbf{f}_{n-1}\mathbf{t}+\mathbf{f}_{n-1}\mathbf{t}^2=\mathbf{f}_{n-1}.\end{aligned}\right.
    \end{gather}
\end{prop}

\subsection{Support of $\mathbf{t}$}
We have seen that for the iSchur algebras ${}_{\mathbb{Q}}\mathbf{S}^{\imath,\square}_{n,d}$, there exists an extra generator $\mathbf{t}$, which behaves very differently from the Chevalley generators $\mathbf{e}_i$ and $\mathbf{f}_i$ in the geometric construction. In the case of equal parameters, \cite[\S 5.2 \& 5.3]{BKLW18} pointed out that $\mathbf{t}$ is the shadow of a perverse sheaf supported on a closed subvariety (called the support variety) in the double $N$-step flag variety, and each connected component of this support variety is rationally smooth. Thus, each idempotent component of $\mathbf{t}$ was defined as the shadow of a constant sheaf (up to a shift) on the corresponding connected component of the support variety.

Things become very different and more subtle in the case of unequal parameters, because the support variety of $\mathbf{t}$ is no longer rationally smooth in general. 
The $\mathbb{F}_q$-rational points of the support variety of $\mathbf{t}\in\mathbf{S}_{n,d}^{\imath,\square}$ is
\begin{equation}\label{s(t)}
S^\square(\mathbf{t}):=\{(\mathfrak{f},\mathfrak{f}')\in\sF^\imath_\square\times\sF^\imath_\square\mid\dim(\mathfrak{f}_n\cap\mathfrak{f}'_n)\ge d-1, \mathfrak{f}_j=\mathfrak{f}'_j,\forall j\in[1,n-1]\}.
\end{equation}
Denote by $\boldsymbol{S}^\square(\mathbf{t})$ the corresponding support variety of $\mathbf{t}$ over $\mathbb{F}_q$.
Every $\mathbb{F}_q$-rational singular point of $\boldsymbol{S}^\square(\mathbf{t})$ lies on the union of diagonals of $\sF^\imath_\square\times\sF^\imath_\square$.
From a categorification perspective, the best choice for the definition of $\mathbf{t}$ should be the shadow of the intersection cohomology complex (up to a shift) on $\boldsymbol{S}^\square(\mathbf{t})$. However, we will see later that the element $\mathbf{t}$ defined in \S\ref{subsec:t} fails to meet this requirement in the case of $\square=D$.


\subsection{Computation via Frobenius trace}
To adjust the element $\mathbf{t}$, we need to determine the Kazhdan-Lusztig polynomials of the IC complexes on $\boldsymbol{S}^\square(\mathbf{t})$.
Unlike the case of equal parameters (and also the split case in geometric perspective), we need to additionally calculate the traces of the Frobenius morphism on the cohomologies of stalks of the IC complexes at the $\mathbb{F}_q$-rational points of the singular locus of $\boldsymbol{S}^\square(\mathbf{t})$. The eigenvalues of the Frobenius morphism on these cohomologies of stalks are $\pm q^{i/2}$ (involving the diagram involution), while in the split case, these eigenvalues are $q^{i/2}$ precisely. So it is more difficult to compute Frobenius traces in the case of unequal parameters.

We first explain that we can reduce the calculation to the case of $n=1$.
Let us focus on $\boldsymbol{S}^{2d}(\mathbf{t})$ to explain the reduction, while the cases of $\boldsymbol{S}^{2d+1}(\mathbf{t})$ and $\boldsymbol{S}^O(\mathbf{t})$ are similar.
For $\mathbf{\lambda}=(\lambda_1,\cdots,\lambda_n)\in\mathbb{N}^n$ with $\sum_{1\le i\le n}\lambda_i=d$, let $\pmb{\sF}_{2d}^{\imath}(\mathbf{\lambda})$ be the  connected component of $\pmb{\sF}_{2d}^{\imath}$. The set of its $\mathbb{F}_q$-rational points, denoted by $\mathscr{F}^\imath_{2d}(\lambda)$, consists of those $\mathfrak{f}=(\mathfrak{f}_i)_{0\le i\le N}$ such that $\dim\mathfrak{f}_i/\mathfrak{f}_{i-1}=\lambda_i$ for each $1\le i\le n$. 
Denote $$\boldsymbol{S}^{2d}_{\mathbf{\lambda}}(\mathbf{t})=S^{2d}(\mathbf{t})\cap(\pmb{\sF}_{2d}^{\imath}(\mathbf{\lambda})\times\pmb{\sF}_{2d}^{\imath}(\mathbf{\lambda})),$$ 
which is a connected component of $\boldsymbol{S}^{2d}(\mathbf{t})$.
Fix ${\mathfrak{f}}=({\mathfrak{f}}_i)_{0\le i\le N}\in\sF_{2d}^{\imath}(\mathbf{\lambda})$. Consider the projection of the first factor $\boldsymbol{S}^{2d}_{\mathbf{\lambda}}(\mathbf{t})\to\pmb{\sF}_{2d}^{\imath}(\mathbf{\lambda})$. Denote fiber of $\mathfrak{f}$ by $\pmb{\mathscr{S}}(\mathfrak{f})$, whose $\mathbb{F}_q$-rational points is $$\mathscr{S}(\mathfrak{f}):=\{V\mid \dim (V\cap\mathfrak{f}_n)\ge d-1, \mathfrak{f}_{n-1}\subset V\subset\mathfrak{f}_{n+1}, V\ \text{is maximal isotropic}\}.$$
Let $\mathbf{P}_\mathfrak{f}\subset \mathbf{U}_{2d}$ be the parabolic subgroup that stabilizes $\mathfrak{f}$.
 We have an isomorphism of $\mathbb{F}_q$-varieties: $$\boldsymbol{S}^{2d}_{\mathbf{\lambda}}(\mathbf{t})\simeq\mathbf{U}_{2d}\times^{\mathbf{P}_\mathfrak{f}}\pmb{\mathscr{S}}(\mathfrak{f}).$$ By the induction equivalence \cite[Proposition 6.2.12]{Ac21} and the $\mathbf{U}_{2d}$-equivariance of $\mathrm{IC}_{\boldsymbol{S}^{2d}_{\mathbf{\lambda}}(\mathbf{t})}$, we only need to calculate the intersection cohomology and the Frobenius traces of the fiber $\pmb{\mathscr{S}}(\mathfrak{f})$. Noting that $\pmb{\mathscr{S}}(\mathfrak{f})$ depends only on $\mathcal{V}=\mathfrak{f}_{n+1}/\mathfrak{f}_{n-1}$ and the restriction of the nondegenerate Hermitian form on $\mathfrak{f}_{N}=\mathbb{F}_{q^2}^{2d}$ to $\mathcal{V}=\mathbb{F}_{q^2}^{2\lambda_n}$ is also a nondegenerate Hermitian form, and we have 
$$\mathscr{S}(\mathfrak{f})\simeq\{V\mid V\subset\mathcal{V}\ \text{is maximal isotropic}, \dim V\cap(\mathfrak{f}_{n}/\mathfrak{f}_{n-1})\ge\lambda_{n}-1\}.$$
Hence the problem can be reduced to the case of $n=1$.
\subsubsection{Calculation for $\boldsymbol{S}^{2d}(\mathbf{t})$}\label{cal U-2d}
For $\lambda_{n}=1$, the variety $\boldsymbol{S}_\lambda^{2d}(\mathbf{t})$ is smooth.
In this subsection, we assume $n=1$ and $d=\lambda_n\ge2$, so that the $\mathbb{F}_q$-rational points of $\boldsymbol{S}_\lambda^{2d}(\mathbf{t})$ is
$$S^{2d}(\mathbf{t}):=\{(V,V')\mid \dim V=\dim V'=d,\dim(V\cap V')\ge d-1, V,V'\subset\mathbb{F}_{q^2}^{2d}\ \text{are isotropic}\}.$$
Denote $$\widetilde{S}^{2d}(\mathbf{t})=\boldsymbol{S}^{2d}(\mathbf{t})\times_{\mathbb{F}_q}\overline{\mathbb{F}_q}\subset\mathrm{Gr}(d,2d)\times\mathrm{Gr}(d,2d),$$ we have $$\widetilde{S}^{2d}(\mathbf{t})=\{(V,V')\mid \dim V=\dim V'=d,\dim(V\cap V')\ge d-1, V,V'\subset\overline{\mathbb{F}_{q}}^{2d}\}.$$

Recall the Frobenius morphism defined in \S\ref{ex:U_D}:
$$F:\mathrm{GL}_{2d}\to\mathrm{GL}_{2d}, \quad g\mapsto JF_q(g^\mathrm{-tr})J.$$
Let $\mathbb{V}=\overline{\mathbb{F}_q}^{2d}$ and  let $\mathrm{Gr}_{2d}:=\bigsqcup_{0\le k\le 2d}\mathrm{Gr}(k,2d)$ be the variety of all linear subspaces of $\mathbb{V}$, which is also regarded as the variety of all maximal parabolic subgroups of $\mathrm{GL}_{2d}$. The Frobenius morphism $F$ induces a Frobenius morphism on $\mathrm{Gr}_{2d}$, which we denote by $F$ again. 

Equip $\mathbb{V}$ with the $q$-sesquilinear form
$$
\langle v,v'\rangle=v^{\mathrm{tr}}JF_q(v'),
$$
which is linear in the first factor, $q$-semilinear in the second, and restricts to the Hermitian form \eqref{form1} on $\mathbb{F}_{q^2}^{2d}$. For $g\in\mathrm{GL}_{2d}$, put $g'=JF_q(g^{\mathrm{tr}})J$. Then
$$
\langle g'(v),v'\rangle=\langle v,g(v')\rangle.
$$
For a linear subspace $V$ of $\mathbb{V}$, define its left orthogonal complement by
$$
V^{\perp}:=\{v\in\mathbb{V}\mid\langle v,V\rangle=0\}.
$$
For subspaces defined over $\mathbb{F}_{q^2}$, this agrees with the usual Hermitian orthogonal complement.
\begin{lem}\label{F-map}
    If $V\in\mathrm{Gr}_{2d}$ is a linear subspace of $\mathbb{V}$, then $F(V)=V^{\perp}$.
\end{lem}
\begin{proof}
Since $F(g)=(g')^{-1}$, the preceding identity gives
$$
\langle F(g)(v),g(v')\rangle=\langle v,v'\rangle.
$$
Consequently, $(gV)^\perp=F(g)(V^\perp)$. Let $E_r$ be the span of the first $r$ standard basis vectors. The action of $F$ on the standard parabolic subgroup associated with $E_r$ gives $F(E_r)=E_{2d-r}=E_r^\perp$. Writing $V=gE_r$, we obtain
$$
F(V)=F(g)F(E_r)=F(g)(E_r^\perp)=(gE_r)^\perp=V^\perp.
$$
\end{proof}
\begin{rem}
Since $F^2=F_{q^2}$, Lemma~\ref{F-map} gives $(V^\perp)^\perp=F_{q^2}(V)$. Thus $(V^\perp)^\perp=V$ if and only if $V$ is defined over $\mathbb{F}_{q^2}$.
\end{rem}
Fix a maximal isotropic subspace $\mathring{V}\subset\mathbb{F}^{2d}_{q^2}$. Let $V=\mathring{V}\otimes_{\mathbb{F}_{q^2}}\overline{\mathbb{F}_q}$ so that $F(V)=V^{\perp}=V$.
Consider the projection of the first factor $p_1:\widetilde{S}^{2d}(\mathbf{t})\to\mathrm{Gr}(d,2d)$, the fiber of $V$ is $$\mathscr{S}(V):=\{W\in\mathrm{Gr}(d,2d)\mid \dim(W\cap V)\ge d-1\}.$$ Denote the $\mathbb{F}_q$-rational model of $\mathscr{S}(V)$ relating to $F$ by $\pmb{\mathscr{S}}(\mathring{V})$, whose $\mathbb{F}_q$-rational points  is $$\mathscr{S}(\mathring{V}):=\{\mathring{W}\ \text{is maximal isotropic in}\  \mathbb{F}^{2d}_{q^2}\mid \dim(\mathring{W}\cap\mathring{V})\ge d-1\}.$$
Let $P=\{g\in \mathrm{GL}_{2d}~|~gV=V\}$ be the parabolic subgroup that stabilizes $V$. There is an isomorphism $$\mathrm{GL}_{2d}\times^{P}\mathscr{S}(V)\simeq \widetilde{S}^{2d}(\mathbf{t}),$$ 
which is exactly defined over $\mathbb{F}_q$, i.e., $\mathbf{U}_{2d}\times^{\mathring{\mathbf{P}}}\pmb{\mathscr{S}}(\mathring{V})\simeq \boldsymbol{S}^{2d}(\mathbf{t})$. Here $\mathring{\mathbf{P}}$ is the parabolic subgroup of $\mathbf{U}_{2d}$ that stabilizes $\mathring{V}$.

We now construct a resolution of singularities of $\mathscr{S}(V)$.
Denote $$\widetilde{\mathscr{S}}(V)=\{(L\overset{1}{\subset}W\overset{1}{\subset}L')\mid W\in\mathscr{S}(V), L\subseteq W\cap V, (W+V)\subseteq L'\}.$$ Using $(W\cap V)^\perp=W^\perp+V$ and Lemma \ref{F-map}, the Frobenius morphism on $\mathrm{Gr}_{2d}$ induces an endomorphism \begin{equation}\label{tilde-F}
\widetilde{F}:\widetilde{\mathscr{S}}(V)\to\widetilde{\mathscr{S}}(V),\quad (L\subset W\subset L')\mapsto ((L')^\perp\subset W^\perp\subset L^\perp),\end{equation}
which gives an $\mathbb{F}_q$-rational structure on $\widetilde{\mathscr{S}}(V)$. Denote by $\widetilde{\pmb{\mathscr{S}}}(\mathring{V})$ the corresponding $\mathbb{F}_q$-form of $\widetilde{\mathscr{S}}(V)$, whose $\mathbb{F}_q$-rational points is $$\{({L}\overset{1}{\subset}{W}\overset{1}{\subset}{{L}}^\perp)\mid {W}\in\mathscr{S}(\mathring{V}), {L}\subseteq{W}\cap\mathring{V}\},$$ 
where $L^\perp$ denotes the orthogonal complement with respect to the nondegenerate Hermitian form on $\mathbb{F}^{2d}_{q^2}$.
Let \begin{equation}
    \pi:\widetilde{\mathscr{S}}(V)\to\mathscr{S}(V),\quad (L\subset W\subset L')\mapsto W
\end{equation} be the obvious projection. We have the following commutative diagram: \begin{equation}\label{2d-F}
\begin{tikzcd}
\widetilde{\mathscr{S}}(V) \arrow[r, "\pi"] \arrow[d, "\widetilde{F}"'] & \mathscr{S}(V) \arrow[d, "F"] \\
\widetilde{\mathscr{S}}(V) \arrow[r, "\pi"]                             & \mathscr{S}(V)               
\end{tikzcd}.
\end{equation}
The variety ${\mathscr{S}}(V)$ has a stratification with two strata $O_0$ and $O_1$, where
$O_0:=\{V\}$ is the unique singularity of ${\mathscr{S}}(V)$
and $$O_1:=\{W\in\mathrm{Gr}(d,2d)\mid \dim(W\cap V)=d-1\}=\mathscr{S}(V)\backslash O_0$$ is an open stratum.
\begin{lem}\label{lem:piresolution}
    The map $\pi$ is a stratified resolution of $\mathscr{S}(V)$, whose restriction to $\pi^{-1}(O_1)$ is an isomorphism. Moreover, the fiber over $\{V\}$ is isomorphic to $\mathbb{P}^{d-1}\times\mathbb{P}^{d-1}$.
\end{lem}
\begin{proof}
    Consider the forgetful map $$f:\widetilde{\mathscr{S}}(V)\to\mathrm{Gr}(d-1,V)\times\mathrm{Gr}(1,\mathbb{V}/V),\ (L\subset W\subset L')\mapsto (L,L'/V),$$ which is a fibration. The fiber over $(L,L'/V)$ is $$\{W\in\mathrm{Gr}(d,2d)\mid L\overset{1}{\subset} W\overset{1}{\subset} L'\}\simeq\mathrm{Gr}(1,L'/L)\simeq\mathbb{P}^1.$$
    We have that $\widetilde{\mathscr{S}}(V)$ is a smooth projective variety of dimension $2d-1$. When $W\in O_1$, we have $\dim(W\cap V)=d-1$ and $\dim(W+V)=d+1$, which makes the choice of $(L,L')$ unique, and hence $\pi|_{\pi^{-1}(O_1)}$ is an isomorphism.
    Since $\mathscr{S}(V)$ is also projective (in fact it is isomorphic to some generalized Schubert variety), the map $\pi$ is a resolution of singularity of $\mathscr{S}(V)$. 
    The fiber over $\{V\}$ is $$\{(L,L')\mid L\overset{1}{\subset} V, V\overset{1}{\subset}L'\subset\mathbb{V}\}\simeq\mathrm{Gr}(d-1,V)\times\mathrm{Gr}(1,\mathbb{V}/V)\simeq\mathbb{P}^{d-1}\times\mathbb{P}^{d-1}.$$ 
    This completes the proof.
\end{proof}

\begin{lem}\label{Fq form}
For $d\geq 2$, the $\mathbb{F}_q$-forms of $\mathbb{P}^{d-1}_{\overline{\mathbb{F}_q}}\times
\mathbb{P}^{d-1}_{\overline{\mathbb{F}_q}}$ are
$$
\mathbb{P}^{d-1}_{\mathbb{F}_q}\times\mathbb P^{d-1}_{\mathbb{F}_q}
\quad\text{and}\quad
\operatorname{Res}_{\mathbb{F}_{q^2}/\mathbb{F}_q}
\bigl(\mathbb{P}^{d-1}_{\mathbb{F}_{q^2}}\bigr).
$$
 The latter is the Weil restriction form. After base change to $\overline{\mathbb{F}_q}$, the Frobenius swaps the two geometric factors of $\operatorname{Res}_{\mathbb{F}_{q^2}/\mathbb{F}_q}
\bigl(\mathbb{P}^{d-1}_{\mathbb{F}_{q^2}}\bigr)\times_{\mathbb{F}_q}\overline{\mathbb{F}_q}$. 
\end{lem}
\begin{proof}
Write $\overline{X}=\mathbb{P}^{d-1}_{\overline{\mathbb{F}_q}}\times
\mathbb{P}^{d-1}_{\overline{\mathbb{F}_q}}$. Its automorphism group is 
$$
\operatorname{Aut}(\overline{X})
\simeq
(\operatorname{PGL}_d(\overline{\mathbb{F}_q})\times\operatorname{PGL}_d(\overline{\mathbb{F}_q}))\rtimes\mathfrak{S}_2,
$$
where $\mathfrak{S}_2$ permutes the two factors. The $\mathbb{F}_q$-forms of
$\overline{X}$ are classified by the first Galois cohomology $H^1\!\left(\mathbb{F}_q,\operatorname{Aut}(\overline{X})\right)$.

The projection to $\mathfrak S_2$ induces
$H^1\!\left(\mathbb F_q,\operatorname{Aut}(\overline X)\right)
\rightarrow H^1(\mathbb F_q,\mathfrak S_2)$.
The cohomology $H^1(\mathbb F_q,\mathfrak S_2)$ has exactly two elements, corresponding to the standard and twisted Frobenius homomorphisms. 

For either class, the corresponding twisted identity component is a connected algebraic group over $\mathbb{F}_q$, i.e.,
$$
\operatorname{PGL}_{d,\mathbb{F}_{q}}\times\operatorname{PGL}_{d,\mathbb{F}_{q}}
\quad\text{or}\quad
\operatorname{Res}_{\mathbb{F}_{q^2}/\mathbb{F}_q}(\operatorname{PGL}_{d,\mathbb{F}_{q^2}}),
$$
whose first Galois cohomology is trivial by Lang's theorem. Thus each of the two
classes in $H^1(\mathbb F_q,\mathfrak S_2)$ has a unique lift.
Therefore, there are exactly two $\mathbb{F}_q$-forms: the split form $\mathbb{P}^{d-1}_{\mathbb{F}_q}\times\mathbb P^{d-1}_{\mathbb{F}_q}$
and the non-split form $\operatorname{Res}_{\mathbb F_{q^2}/\mathbb{F}_q}
\bigl(\mathbb{P}^{d-1}_{\mathbb F_{q^2}}\bigr)$ for which the Frobenius swaps the two geometric factors.
\end{proof}

Let $\sigma$ be the morphism of $\mathbb{P}^{d-1}\times\mathbb{P}^{d-1}$ that swaps the two factors. Let $F_q$ be the standard Frobenius morphism of $\mathbb{P}^{d-1}\times\mathbb{P}^{d-1}$ associated with $\mathbb{P}^{d-1}_{\mathbb{F}_q}\times\mathbb{P}^{d-1}_{\mathbb{F}_q}$.
\begin{lem}\label{F action}
    Under the isomorphism $\pi^{-1}(\{V\})\simeq\mathbb{P}^{d-1}\times\mathbb{P}^{d-1}$, the restriction of $\widetilde{F}$ on this fiber is just $F_q\circ\sigma$.
\end{lem}
\begin{proof}
    This follows from Lemma~\ref{F-map} and Lemma~\ref{Fq form}.
\end{proof}

We write $Y_{d-1}=\mathbb{P}^{d-1}\times\mathbb{P}^{d-1}$ for short. 
\begin{cor}\label{trace}
   For $0\le i\le 2d-2$, the trace of $\widetilde{F}$ acting on $H^{2i}_{\mathrm{\acute{e}t}}(Y_{d-1})$ is $q^i$ for even $i$, and $0$ for odd $i$.
\end{cor}
\begin{proof}
We have
$$
H^{\bullet}_{\mathrm{\acute{e}t}}(\mathbb{P}^{d-1})
\simeq
\overline{\mathbb{Q}}_\ell[\alpha]/(\alpha^d),
$$
where $\alpha\in H^{2}_{\mathrm{\acute{e}t}}(\mathbb{P}^{d-1})$ is the hyperplane class. By the Künneth formula,
$$
H^{\bullet}_{\mathrm{\acute{e}t}}(Y_{d-1})
\simeq
\overline{\mathbb{Q}}_\ell[\alpha,\beta]/(\alpha^d,\beta^d),
$$
where $\alpha$ and $\beta$ are the pullbacks of the hyperplane class via the projections onto the first and second factors, respectively. Lemma~\ref{F action} gives $\sigma(\alpha)=\beta$. Hence $H^{2i}_{\mathrm{\acute{e}t}}(Y_{d-1})$ has a basis $\{\alpha^k\beta^{i-k}\mid\max(0,i-d+1)\le k\le\min(i,d-1)\}$, and $F_q$ acts on each basis vector by multiplication by $q^i$. The trace is therefore $q^i$ if $i$ is even, and $0$ otherwise.
\end{proof}

We carry out the following calculations over $\overline{\mathbb F_q}$.
The commutative diagram~\eqref{2d-F} shows that $\pi$ descends
to an $\mathbb F_q$-morphism
\[
\widetilde{\pmb{\mathscr S}}(\mathring V)
\longrightarrow
\pmb{\mathscr S}(\mathring V).
\]
The constant sheaf on $\widetilde{\mathscr S}(V)$ is obtained
by base change from the constant sheaf on
$\widetilde{\pmb{\mathscr S}}(\mathring V)$.
Similarly, the IC complex on $\mathscr S(V)$ is obtained
by base change from the IC complex on
$\pmb{\mathscr S}(\mathring V)$.
We equip both complexes with the canonical Weil structures.
All stalks and fiber cohomology groups are over $\overline{\mathbb{F}_q}$.
Proper base change and Verdier duality are compatible
with these Weil structures.
In particular, self-duality is understood in the category
of Weil complexes.
The same convention applies to the calculations in
\S\ref{cal-2d+1} and \S\ref{cal-O}.

Consider the intersection cohomology complex $\mathrm{IC}_{\widetilde{\mathscr{S}}(V)}(\frac{2d-1}{2})=\underline{\overline{\mathbb{Q}}_\ell}_{\widetilde{\mathscr{S}}(V)}[2d-1](\frac{2d-1}{2})$ on $\widetilde{\mathscr{S}}(V)$. It is self-dual and pure of weight $0$.
By Lemma~\ref{lem:piresolution} and BBDG Decomposition Theorem, we have $$\pi_*\underline{\overline{\mathbb{Q}}_\ell}_{\widetilde{\mathscr{S}}(V)}[2d-1](\frac{2d-1}{2})\simeq\mathrm{IC}_{\mathscr{S}(V)}(\frac{2d-1}{2})\oplus\mathcal{S},$$ where the first summand is the intersection cohomology complex on $\mathscr{S}(V)$ which is self-dual and pure of weight 0, and $\mathcal{S}$ is a self-dual semisimple complex, pure of weight $0$, supported on $O_0$.
By proper base change, we have
$$(\pi_*\underline{\overline{\mathbb{Q}}_\ell}_{\widetilde{\mathscr{S}}(V)}[2d-1](\frac{2d-1}{2}))|_{O_0}\simeq \mathrm{R}\Gamma(\underline{\overline{\mathbb{Q}}_\ell}_{\mathbb{P}^{d-1}\times\mathbb{P}^{d-1}}[2d-1](\frac{2d-1}{2})).$$
The following table gives the stalks on each stratum, where rows are indexed by strata, and columns by cohomological degree.
$$\begin{tabular}{|c|c|c|c|c|c|c|}
       \hline
    \text{fiber} & \text{stratum}&$-2d+1$&$\cdots$&$-2d+1+2i$&$\cdots$&$2d-3$ \\
    \hline
     $\{\mathrm{pt}\}$&$O_1$&$\overline{\mathbb{Q}}_\ell(\frac{2d-1}{2})$&$\cdots$&$0$&$\cdots$&$0$\\
     $Y_{d-1}$&$O_0$&$H_{\mathrm{\acute{e}t}}^{0}(Y_{d-1})(\frac{2d-1}{2})$&$\cdots$&$H_{\mathrm{\acute{e}t}}^{2i}(Y_{d-1})(\frac{2d-1}{2})$&$\cdots$&$H_{\mathrm{\acute{e}t}}^{4d-4}(Y_{d-1})(\frac{2d-1}{2})$\\
     \hline
\end{tabular}$$ 

From the table, we read that for $ 0\le j\le d-2$,
\begin{equation}\label{iso-int} 
\mathsf{H}^{2j+1}(\mathrm{IC}_{\mathscr{S}(V)}|_{O_0})=0,\ \mathsf{H}^{2j+1}(\mathcal{S}|_{O_0})=H_{\mathrm{\acute{e}t}}^{2d+2j}(Y_{d-1})(\frac{2d-1}{2}).\end{equation} 

\begin{prop}
    For $0\le j\le d-1$, we have 
    $$H^{2d-2j-2}_{\mathrm{\acute{e}t}}(Y_{d-1})(\frac{2d-1}{2})\simeq\mathsf{H}^{-2j-1}(\mathrm{IC}_{\mathscr{S}(V)}(\frac{2d-1}{2})|_{O_0})\oplus {H}_{\mathrm{\acute{e}t}}^{2d+2j}(Y_{d-1})^\vee(-\frac{2d-1}{2}).$$ Moreover,
$$\mathrm{tr}\big({F};\mathsf{H}^{-2j-1}(\mathrm{IC}_{\mathscr{S}(V)}(\frac{2d-1}{2})|_{O_0})\big)=(-1)^{d+j-1}q^{-\frac{2j+1}{2}}.$$
\end{prop}
\begin{proof}
    By Poincar\'e-Verdier duality and \eqref{iso-int}, we have 
\begin{align*}
    H^{2d-2j-2}_{\mathrm{\acute{e}t}}(Y_{d-1})(\frac{2d-1}{2})&\simeq\mathsf{H}^{-2j-1}(\pi_*\underline{\overline{\mathbb{Q}}_\ell}_{\widetilde{\mathscr{S}}(V)}[2d-1](\frac{2d-1}{2})|_{O_0})
    \\&\simeq\mathsf{H}^{-2j-1}(\mathrm{IC}_{\mathscr{S}(V)}(\frac{2d-1}{2})|_{O_0})\oplus \mathsf{H}^{-2j-1}(\mathcal{S}|_{O_0})\\&\simeq
    \mathsf{H}^{-2j-1}(\mathrm{IC}_{\mathscr{S}(V)}(\frac{2d-1}{2})|_{O_0})\oplus \mathsf{H}^{2j+1}(\mathcal{S}|_{O_0})^\vee\\&\simeq
    \mathsf{H}^{-2j-1}(\mathrm{IC}_{\mathscr{S}(V)}(\frac{2d-1}{2})|_{O_0})\oplus {H}_{\mathrm{\acute{e}t}}^{2d+2j}(Y_{d-1})^\vee(-\frac{2d-1}{2}).
\end{align*}
Note that the possible eigenvalues of $\widetilde{F}$ are only $\pm q^{k/2}$. By Corollary \ref{trace}, we have $\mathrm{tr}\big(\widetilde{F};H^{2d+2j}_{\mathrm{\acute{e}t}}(Y_{d-1})^\vee(-\frac{2d-1}{2})\big)=\delta_{d+j,even}\cdot q^{-\frac{2j+1}{2}}$ and $\mathrm{tr}\big(\widetilde{F};H^{2d-2j-2}_{\mathrm{\acute{e}t}}(Y_{d-1})(\frac{2d-1}{2})\big)=\delta_{d-j-1,even}\cdot q^{-\frac{2j+1}{2}}$.
By the above formula, we have 
\begin{align*}
&\mathrm{tr}\big({F};\mathsf{H}^{-2j-1}(\mathrm{IC}_{\mathscr{S}(V)}(\frac{2d-1}{2})|_{O_0})\big)\\&=\mathrm{tr}\big(\widetilde{F};H^{2d-2j-2}_{\mathrm{\acute{e}t}}(Y_{d-1})(\frac{2d-1}{2})\big)-\mathrm{tr}\big(\widetilde{F};H^{2d+2j}_{\mathrm{\acute{e}t}}(Y_{d-1})^\vee(-\frac{2d-1}{2})\big)\\&=(\delta_{d-j-1,even}-\delta_{d+j,even})\cdot q^{-\frac{2j+1}{2}}=(-1)^{d+j-1}q^{-\frac{2j+1}{2}}.
\end{align*}
\end{proof}

From the above trace computations, we define the normalized Kazhdan-Lusztig polynomial
\begin{equation}\label{KL-2d}f_{2d,\mathbf{\lambda}}(v)=\sum\limits_{j=0}^{\lambda_n-1}(-1)^{\lambda_n+j-1}v^{-2j-1}=v^{-2\lambda_n+1}-\frac{v^{-2\lambda_{n}+3}+(-1)^{\lambda_{n}}v}{v^2+1}.\end{equation}

\subsubsection{Calculation for $\boldsymbol{S}^{2d+1}(\mathbf{t})$}\label{cal-2d+1}
Similarly to $\boldsymbol{S}^{2d}(\mathbf{t})$, the calculation for $\boldsymbol{S}^{2d+1}(\mathbf{t})$ can also be reduced to the case of $n=1$. 
Therefore, we still fix $n=1$ in this subsection. Now, the $\mathbb{F}_q$-rational points of $\boldsymbol{S}^{2d+1}(\mathbf{t})$ is
 \begin{align*}
 &S^{2d+1}(\mathbf{t})=\\
 &\{(V\subset V^\perp,V'\subset{V'}^\perp)\mid \dim V=\dim V'=d,\dim(V\cap V')\ge d-1, V,V'\subset\mathbb{F}^{2d+1}_{q^2}\ \},
 \end{align*}
Note that $\boldsymbol{S}^{2d+1}(\mathbf{t})$ is not smooth unless $d=1$. In the following, we assume $d=\lambda_n\ge2$.

Denote $$\mathcal{F}l_{d,d+1}:=\mathrm{Fl}(d,d+1;\overline{\mathbb{F}_q}^{2d+1})=\{\mathcal{F}=(0\subset F_1\subset F_2\subset\overline{\mathbb{F}_q}^{2d+1})\mid \dim F_1=d,\dim F_2=d+1\}$$ and
$$\widetilde{S}^{2d+1}(\mathbf{t}):=S^{2d+1}(\mathbf{t})\times_{\mathbb{F}_q}\overline{\mathbb{F}_q}\subset\mathcal{F}l_{d,d+1}\times\mathcal{F}l_{d,d+1},$$
we have \begin{align*}
  \widetilde{S}^{2d+1}(t)
={}&\bigl\{(\mathcal F,\mathcal F')
\in \mathcal{F}l_{d,d+1}\times\mathcal{F}l_{d,d+1}
\;\bigm|
\dim(F_1\cap F'_1)\ge d-1,\quad
\dim(F_2\cap F'_2)\ge d
\bigr\}.
 \end{align*}
Fix a maximal isotropic subspace $\mathring{V}$ of $\mathbb{F}_{q^2}^{2d+1}$, and set $\mathring{\mathcal{F}}=(0\subset \mathring{V}\subset\mathring{V}^\perp\subset\mathbb{F}^{2d+1}_{q^2})$.
Let $\mathcal{F}=(0\subset V\subset V^\perp\subset\overline{\mathbb{F}_q}^{2d+1})$ be the lift of $\mathring{\mathcal{F}}$ to $\mathcal{F}l_{d,d+1}$, and $P_\mathcal{F}$ be the corresponding parabolic subgroup in $\mathrm{GL}_{2d+1}$ that stabilizes $\mathcal{F}$. Consider the projection of the first factor $\widetilde{S}^{2d+1}(\mathbf{t})\to\mathcal{F}l_{d,d+1}$. The fiber over $\{\mathcal{F}\}$ is $$\mathscr{S}(\mathcal{F})=\{(W\subset W')\mid \dim W=d=\dim W'-1, \dim W\cap V\ge d-1,\dim W'\cap V^\perp\ge d\}.$$ 
We have an $\mathbb{F}_q$-rational isomorphism of varieties 
$\mathrm{GL}_{2d+1}\times^{P_\mathcal{F}}\mathscr{S}(\mathcal{F})\simeq\widetilde{S}^{2d+1}(\mathbf{t})$, which descends to an isomorphism of $\mathbb{F}_q$-varieties $\mathbf{U}_{2d+1}\times^{\mathbf{P}_{\mathring{\mathcal{F}}}}\pmb{\mathscr{S}}(\mathring{\mathcal{F}})\simeq\boldsymbol{S}^{2d+1}(\mathbf{t})$, where $\pmb{\mathscr{S}}(\mathring{\mathcal{F}})$ is the $\mathbb{F}_q$-form of $\mathscr{S}(\mathcal{F})$ and $\mathbf{P}_{\mathring{\mathcal{F}}}$ is the parabolic subgroup of $\mathbf{U}_{2d+1}$ that stabilizes $\mathring{\mathcal{F}}$. As before, we only need to deal with $\mathscr{S}(\mathcal{F})$.
Denote $$\widetilde{\mathscr{S}}(\mathcal{F})=\{(L\overset{1}{\subset} W\subset W'\overset{1}{\subset} L')\mid (W\subset W')\in\mathscr{S}(\mathcal{F}),L\subseteq W\cap V,(W'+V^\perp)\subseteq L'\}.$$
Let $F$ be the Frobenius morphisms on $\mathscr{S}(\mathcal{F})$, and $\widetilde{F}$ be the Frobenius of $\widetilde{\mathscr{S}}(\mathcal{F})$ defined similar to \eqref{tilde-F}.
Let \begin{equation}
    \pi:\widetilde{\mathscr{S}}(\mathcal{F})\to\mathscr{S}(\mathcal{F}),\quad (L\subset W\subset W'\subset L')\mapsto (W\subset W')
\end{equation} be the obvious projection. We have a commutative diagram: \begin{equation}
\begin{tikzcd}
\widetilde{\mathscr{S}}(\mathcal{F}) \arrow[r, "\pi"] \arrow[d, "\widetilde{F}"'] & \mathscr{S}(\mathcal{F}) \arrow[d, "F"] \\
\widetilde{\mathscr{S}}(\mathcal{F}) \arrow[r, "\pi"]                             & \mathscr{S}(\mathcal{F})       
\end{tikzcd}.
\end{equation}
Denote
\begin{align*}
O_{0,0}&:=\{\mathcal{F}\},\\
O_{1,1}&:=\{(W\subset W')\in\mathscr{S}(\mathcal{F})\mid\dim W\cap V=d-1, \dim W'\cap V^{\perp}=d\},\\
O_{0,1}&:=\{(V\subset W')\in\mathscr{S}(\mathcal{F})\mid \dim W'\cap V^\perp=d\},\\
O_{1,0}&:=\{(W\subset V^\perp)\in\mathscr{S}(\mathcal{F})\mid \dim W\cap V=d-1\}=F(O_{0,1}).
\end{align*}
Reset $\mathbb{V}:=\overline{\mathbb{F}_q}^{2d+1}$.
We compute the Zariski closures 
\begin{align*}
\overline{O_{0,1}}=O_{0,1}\sqcup O_{0,0}\simeq\mathrm{Gr}(1,\mathbb{V}/V)\simeq\mathbb{P}^{d} \quad \text{and}\quad
\overline{O_{1,0}}=O_{1,0}\sqcup O_{0,0}\simeq\mathrm{Gr}(d,V^\perp)\simeq\mathbb{P}^d.
\end{align*}
Therefore, we obtain a stratification of the variety $\mathscr{S}(\mathcal{F})$ with four strata $O_{1,1}$, $O_{0,1}$, $O_{1,0}$ and $O_{0,0}$.

\begin{lem}
 The map $\pi$ is a stratified resolution of $\mathscr{S}(\mathcal{F})$, whose restriction to $\pi^{-1}(O_{1,1})$ is an isomorphism. Moreover, the fiber over $\{\mathcal{F}\}$ is isomorphic to $\mathbb{P}^{d-1}\times\mathbb{P}^{d-1}$; the restrictions of $\pi$ to $\pi^{-1}(O_{0,1})$ and $\pi^{-1}(O_{1,0})$ are both fibrations with fibers isomorphic to $\mathbb{P}^{d-1}$.   
\end{lem}
\begin{proof}
    Consider the forgetful map $$f:\widetilde{\mathscr{S}}(\mathcal{F})\to\mathrm{Gr}(d-1,V)\times\mathrm{Gr}(1,\mathbb{V}/V^\perp),\quad (L\subset W\subset W'\subset L')\mapsto (L,L'/V^\perp),$$ which is a fibration. The fiber over $(L,L'/V^\perp)$ is $\{(W\overset{1}{\subset} W')\mid L\overset{1}{\subset} W\overset{1}{\subset} W'\overset{1}{\subset}L'\}$, which is isomorphic to the complete flag variety of type $A_2$ and whose dimension is $3$. We find that $\widetilde{\mathscr{S}}(\mathcal{F})$ is a smooth projective variety of dimension $2d+1$. When $(W\subset W')\in O_{1,1}$, the dimensions $\dim(W\cap V)=d-1$ and $\dim(W'+V^\perp)=d+2$ force the choice of $(L,L')$ to be unique, and hence $\pi|_{\pi^{-1}(O_{1,1})}$ is an isomorphism. The fibers of $\pi$ over $O_{0,1}$ and $O_{1,0}$ are isomorphic to $\mathrm{Gr}(d-1,V)\simeq\mathbb{P}^{d-1}$ and $\mathrm{Gr}(1,\mathbb{V}/V^\perp)\simeq\mathbb{P}^{d-1}$, respectively.
    Since $\mathscr{S}(\mathcal{F})$ is also projective (in fact it is isomorphic to some generalized Schubert variety), the map $\pi$ is a resolution of the singularity $\{\mathcal{F}\}$ of $\mathscr{S}(\mathcal{F})$. 
    Finally, the fiber over $\{\mathcal{F}\}$ is $\{(L,L')\mid L\overset{1}{\subset} V, V^\perp\overset{1}{\subset}L'\subset\mathbb{V}\}\simeq\mathrm{Gr}(d-1,V)\times\mathrm{Gr}(1,\mathbb{V}/V^\perp)\simeq\mathbb{P}^{d-1}\times\mathbb{P}^{d-1}$. The proof is complete.
\end{proof}

Consider the intersection cohomology complex $\mathrm{IC}_{\widetilde{\mathscr{S}}(\mathcal{F})}(\frac{2d+1}{2})=\underline{\overline{\mathbb{Q}}_\ell}_{\widetilde{\mathscr{S}}(\mathcal{F})}[2d+1](\frac{2d+1}{2})$ on $\widetilde{\mathscr{S}}(\mathcal{F})$, which is pure of weight $0$ and self-dual.
By BBDG Decomposition Theorem, we have 
$$\pi_*\mathrm{IC}_{\widetilde{\mathscr{S}}(\mathcal{F})}(\frac{2d+1}{2})\simeq\mathrm{IC}_{\mathscr{S}(\mathcal{F})}(\frac{2d+1}{2})\oplus\mathcal{S}_{0,1}\oplus\mathcal{S}_{1,0}\oplus\mathcal{S}_{0,0}.$$ Here $\mathrm{IC}_{\mathscr{S}(\mathcal{F})}(\frac{2d+1}{2})$ is the intersection cohomology complex on $\mathscr{S}(\mathcal{F})$, which is pure of weight $0$ and self-dual; $\mathcal{S}_{0,1}$ (resp., $\mathcal{S}_{1,0}$ and $\mathcal{S}_{0,0}$) is the semisimple complex supported on $\overline{O_{0,1}}$ (resp., $\overline{O_{1,0}}$ and $\{\mathcal{F}\}$). In particular, $\mathcal{S}_{0,0}$ is pure of weight $0$ and self-dual.
By proper base change, we have 
$$\big(\pi_*\mathrm{IC}_{\widetilde{\mathscr{S}}(\mathcal{F})}(\frac{2d+1}{2})\big)|_{O_{0,0}}\simeq\mathrm{R}\Gamma\big(\underline{\overline{\mathbb{Q}}_\ell}_{Y_{d-1}}[2d+1](\frac{2d+1}{2})\big).$$
Since $F(O_{0,1})=O_{1,0}$, we have $F^*(\mathcal{S}_{0,1})=\mathcal{S}_{1,0}.$ The map $F$ swaps the two factors of $\mathcal{S}_{0,1}|_{O_{0,0}}\oplus\mathcal{S}_{1,0}|_{O_{0,0}}$, so $\mathcal{S}_{0,1}\oplus\mathcal{S}_{1,0}$ does not contribute to the Frobenius trace on $\mathsf{H}^\bullet((\pi_*\mathrm{IC}_{\widetilde{\mathscr{S}}(\mathcal{F})}(\frac{2d+1}{2}))|_{O_{0,0}})$. The following table gives some information on the cohomology of stalks, concerning only the strata $O_{1,1}$ and $O_{0,0}$:
$$\begin{tabular}{|c|c|c|c|c|c|c|}
    \hline
    \text{fiber} & \text{stratum} &$-2d-1$&$\cdots$&$-2d-1+2i$&$\cdots$&$2d-5$ \\
    \hline
     $\{\mathrm{pt}\}$&$O_{1,1}$&$\overline{\mathbb{Q}}_\ell(\frac{2d+1}{2})$&$\cdots$&$0$&$\cdots$&$0$\\
     $Y_{d-1}$&$O_{0,0}$&$H_{\mathrm{\acute{e}t}}^{0}(Y_{d-1})(\frac{2d+1}{2})$&$\cdots$&$H_{\mathrm{\acute{e}t}}^{2i}(Y_{d-1})(\frac{2d+1}{2})$&$\cdots$&$H_{\mathrm{\acute{e}t}}^{4d-4}(Y_{d-1})(\frac{2d+1}{2})$\\
     \hline
\end{tabular}.$$ 

The table shows that for $1\le j\le d-2$,
\begin{equation}
    \mathsf{H}^{2j-1}(\mathrm{IC}_{\mathscr{S}(\mathcal{F})}(\frac{2d+1}{2})|_{O_{0,0}})=0,\ \mathsf{H}^{2j-1}(\mathcal{S}_{0,0}\oplus\mathcal{S}_{0,1}\oplus\mathcal{S}_{1,0})|_{O_{0,0}}=H^{2d+2j}_{\mathrm{\acute{e}t}}(Y_{d-1})(\frac{2d+1}{2}).
\end{equation}
\begin{prop}
    For $0\le j\le d$ with $j\neq d-1$, we have $$\mathrm{tr}\big(F;\mathsf{H}^{-2j-1}(\mathrm{IC}_{{\mathscr{S}}(\mathcal{F})}(\frac{2d+1}{2})|_{O_{0,0}})\big)=(-1)^{d+j}q^{-\frac{2j+1}{2}}.$$ 
    For the remaining
 case $j=d-1$, we have $$\mathrm{tr}\big(F;\mathsf{H}^{-2d+1}(\mathrm{IC}_{{\mathscr{S}}(\mathcal{F})}(\frac{2d+1}{2})|_{O_{0,0}})\big)=0.$$
\end{prop}
\begin{proof}
By the above table and the self-duality of $\mathcal{S}_{0,0}$, we have $\mathsf{H}^{k}(\mathcal{S}_{0,0})=0$, for $k=-2d-1,-2d+1,-2d+3$.
For $j=d,d-1,d-2$, we have 
$$\mathrm{tr}\big(F;\mathsf{H}^{-2j-1}(\mathrm{IC}_{{\mathscr{S}}(\mathcal{F})}(\frac{2d+1}{2})|_{O_{0,0}})\big)=\mathrm{tr}\big(\widetilde{F};H^{2d-2j}_{\mathrm{\acute{e}t}}(Y_{d-1})(\frac{2d+1}{2})\big)=\delta_{d-j,even}q^{-\frac{2j+1}{2}}.$$
 Note that the only possible eigenvalues of $\widetilde{F}$ on $H^i_{\mathrm{\acute{e}t}}(Y_{d-1})$ are $\pm q^{\frac{i}{2}}$. For $0\le j\le d-3$, by Poincar\'e-Verdier duality and Corollary \ref{trace}, we have
\begin{align*}
&\mathrm{tr}\big(F;\mathsf{H}^{-2j-1}((\mathcal{S}_{0,0}\oplus\mathcal{S}_{0,1}\oplus\mathcal{S}_{1,0})|_{O_{0,0}})\big)=\mathrm{tr}\big(F;\mathsf{H}^{-2j-1}(\mathcal{S}_{0,0}|_{O_{0,0}})\big)\\&=\mathrm{tr}\big(F;\mathsf{H}^{2j+1}(\mathcal{S}_{0,0}|_{O_{0,0}})^\vee\big)=\mathrm{tr}\big(\widetilde{F};H^{2d+2j+2}_{\mathrm{\acute{e}t}}(Y_{d-1})^\vee(-\frac{2d+1}{2})\big)=\delta_{d+j+1,even}\cdot q^{-\frac{2j+1}{2}}.
\end{align*}
Consequently, \begin{align*}
&\mathrm{tr}\big(F;\mathsf{H}^{-2j-1}(\mathrm{IC}_{\mathscr{S}(\mathcal{F})}(\frac{2d+1}{2})|_{O_{0,0}})\big)\\=&\mathrm{tr}\big(\widetilde{F};{H}_{\mathrm{\acute{e}t}}^{2d-2j}(Y_{d-1})(\frac{2d+1}{2})\big)-\mathrm{tr}\big(F;\mathsf{H}^{-2j-1}((\mathcal{S}_{0,0}\oplus\mathcal{S}_{0,1}\oplus\mathcal{S}_{1,0})|_{O_{0,0}})\big)\\=&\delta_{d-j,even}\cdot q^{-\frac{2j+1}{2}}-\delta_{d+j+1,even}\cdot q^{-\frac{2j+1}{2}}=(-1)^{d+j}q^{-\frac{2j+1}{2}}.
\end{align*}
\end{proof}

From the above trace computations, we define the normalized Kazhdan-Lusztig polynomial \begin{equation}\label{KL-2d+1}f_{2d+1,\mathbf{\lambda}}(v)=v^{-2\lambda_n-1}+\sum\limits_{j=0}^{\lambda_{n}-2}(-1)^{\lambda_n-j}v^{-2j-1}=v^{-2\lambda_n-1}+\frac{v^{-2\lambda_n+3}+(-1)^{\lambda_n}v}{v^2+1}.\end{equation}

\subsubsection{Calculation for $\boldsymbol{S}^O(\mathbf{t})$}\label{cal-O}
As before, we assume $n=1$ and first consider $d=\lambda_n>1$, so that the $\mathbb{F}_q$-rational points of $\boldsymbol{S}^O(\mathbf{t})$ is  \begin{align*}
S^O(\mathbf{t})=\left\{(V,V')\middle | \begin{array}{c}
\dim V=\dim V'=d, \dim(V\cap V')\ge d-1, \\
V,V'\subset\mathbb{F}^{2d+2}_q\ \text{are maximal isotropic}
\end{array}\right\}.
\end{align*}
Denote $\widetilde{S}^O(\mathbf{t})=\boldsymbol{S}^O(\mathbf{t})\times_{\mathbb{F}_q}\overline{\mathbb{F}_q}\subset\mathrm{OGr}(d,2d+2)\times\mathrm{OGr}(d,2d+2)$. Fix a maximal isotropic subspace $V\subset\mathbb{F}^{2d+2}_q$, and denote by ${E}$ its lift in $\mathrm{OGr}(d,2d+2)$. Let $P_{{E}}$ be the corresponding parabolic subgroup  of ${E}$ in $\mathrm{O}_{2d+2}$, and $\mathscr{S}({E})$ the fiber over $\{{E}\}$ of the first projection $\widetilde{S}^O(\mathbf{t})\to\mathrm{OGr}(d,2d+2)$. We have an $\mathbb{F}_q$-rational isomorphism of varieties $\mathrm{O}_{2d+2}\times^{P_{{E}}}\mathscr{S}({E})\simeq\widetilde{S}^O(\mathbf{t})$.
It suffices to deal with the fiber over $\{E\}$
$$\mathscr{S}({E})=\{W\in\mathrm{OGr}(d,2d+2)\mid \dim W\cap E\ge d-1\}.$$
Denote $$\widetilde{\mathscr{S}}(E)=\{(L\overset{1}{\subset}W)\mid W\in\mathscr{S}(E)\mid L\subset W\cap E\}.$$
Let $F$ and $\widetilde{F}$ be the Frobenius morphisms of $\mathscr{S}(E)$ and $\widetilde{\mathscr{S}}(E)$ associated with the obvious $\mathbb{F}_q$-rational structures, respectively.
Let \begin{equation}    
\pi:\widetilde{\mathscr{S}}(E)\to\mathscr{S}(E),\quad (L\subset W)\mapsto W
\end{equation} be the obvious projection, which satisfies the following commutative diagram \begin{equation}
\begin{tikzcd}
\widetilde{\mathscr{S}}(E) \arrow[r, "\pi"] \arrow[d, "\widetilde{F}"'] & \mathscr{S}(E) \arrow[d, "F"] \\
\widetilde{\mathscr{S}}(E) \arrow[r, "\pi"]                             & \mathscr{S}(E)       
\end{tikzcd}.
\end{equation}
The variety $\mathscr{S}(E)$ has a stratification with two strata $O_{1}$ and $O_0$, where $O_{0}$ is the single point $\{E\}$, and $O_{1}=\mathscr{S}(E)\backslash O_0$.

\begin{lem}
    The map $\pi$ is a stratified resolution of $\mathscr{S}(E)$. The restriction of $\pi$ to $\pi^{-1}(O_1)$ is an isomorphism. Moreover, the fiber over $\{E\}$ is isomorphic to $\mathbb{P}^{d-1}$.
\end{lem}
\begin{proof}    
Consider the forgetful map $$f:\widetilde{\mathscr{S}}(E)\to\mathrm{Gr}(d-1,E),\quad (L\subset W)\mapsto L,$$ which is a fibration. The fiber over $L$ is $$\{W\in\mathrm{OGr}(d,2d+2)\mid L\subset W\}\simeq\mathrm{OGr}(1,L^{\perp}/L)\simeq\mathbb{P}^1\times\mathbb{P}^1.$$ So we know that $\widetilde{\mathscr{S}}(E)$ is a smooth projective variety of dimension $d+1$. When $W\in O_1$, we have $\dim(W\cap E)=d-1$, which makes the choice of $L$ unique. Hence, $\pi|_{\pi^{-1}(O_1)}$ is an isomorphism.
Since $\mathscr{S}(E)$ is also projective (in fact it is isomorphic to some generalized Schubert variety), the map $\pi$ is a resolution of $\mathscr{S}(E)$.
The fiber over $\{E\}$ is $\{L\mid L\overset{1}{\subset} E\}\simeq\mathrm{Gr}(d-1,E)\simeq\mathbb{P}^{d-1}.$
\end{proof}
Consider the intersection cohomology complex $\mathrm{IC}_{\widetilde{\mathscr{S}}(E)}(\frac{d+1}{2})=\underline{\overline{\mathbb{Q}}_\ell}_{\widetilde{\mathscr{S}}(E)}[d+1](\frac{d+1}{2})$ on $\widetilde{\mathscr{S}}(E)$, which is pure of weight $0$ and self-dual. By BBDG Decomposition Theorem, we have
$$\pi_*\mathrm{IC}_{\widetilde{\mathscr{S}}(E)}(\frac{d+1}{2})\simeq\mathrm{IC}_{\mathscr{S}(E)}(\frac{d+1}{2})\oplus \mathcal{S},$$ where $\mathrm{IC}_{\mathscr{S}(E)}(\frac{d+1}{2})$ is the intersection cohomology complex on $\mathscr{S}(E)$, pure of weight $0$; $\mathcal{S}$ is a semisimple complex supported on $O_0$, which is also pure of weight $0$ and self-dual.
By proper base change, we have $$\big(\pi_*\mathrm{IC}_{\widetilde{\mathscr{S}}(E)}(\frac{d+1}{2})\big)|_{O_0}\simeq\mathrm{R}\Gamma\big(\underline{\overline{\mathbb{Q}}_\ell}_{\mathbb{P}^{d-1}}[d+1](\frac{d+1}{2})\big).$$
The fiber is $\mathrm{Gr}(d-1,V)\simeq\mathbb{P}^{d-1}_{\mathbb{F}_q}$ over $\mathbb{F}_q$, so $\widetilde{F}$ acts on its cohomology as the standard Frobenius.
The following table gives us information about the cohomology of stalks:
$$\begin{tabular}{|c|c|c|c|c|c|c|}
\hline
    \text{fiber} & \text{stratum} &$-d-1$&$\cdots$&$-d-1+2i$&$\cdots$&$d-3$ \\
    \hline
     $\{\mathrm{pt}\}$&$O_{1}$&$\overline{\mathbb{Q}}_\ell(\frac{d+1}{2})$&$\cdots$&$0$&$\cdots$&$0$\\
     $\mathbb{P}^{d-1}$&$O_{0}$&$\overline{\mathbb{Q}}_\ell(\frac{d+1}{2})$&$\cdots$&$\overline{\mathbb{Q}}_\ell(\frac{d+1-2i}{2})$&$\cdots$&$\overline{\mathbb{Q}}_\ell(-\frac{d-3}{2})$\\
     \hline
\end{tabular}.$$ 
Here $0\le i\le d-1$, and the other cohomological degrees vanish. Moreover, we have $\mathsf{H}^i(\mathrm{IC}_{\mathscr{S}(E)})|_{O_0}=0$ for $i\ge0$.
By the self-duality of $\mathcal{S}$, we can read from the above table that $$\mathsf{H}^i(\mathcal{S}|_{O_{0}})=\overline{\mathbb{Q}}_\ell(\frac{-i}{2})\quad \text{for}\quad  i=-d+3,-d+5,\ldots,d-3,$$ and $\mathsf{H}^i(\mathcal{S}|_{O_0})=0$ for other $i$. Thus, we get the following result.
\begin{prop}
    For $d\ge2$, we have $\mathsf{H}^{i}\big(\mathrm{IC}_{\mathscr{S}(E)}(\frac{d+1}{2})|_{O_0}\big)=\overline{\mathbb{Q}}_\ell(-\frac{i}{2})$ for $i=-d-1,-d+1$, and these stalk cohomology groups vanish in all other degrees. Moreover, for $i=-d-1,-d+1$, $$\mathrm{tr}\big(F;\mathsf{H}^{i}(\mathrm{IC}_{\mathscr{S}(E)}(\frac{d+1}{2})|_{O_0})\big)=q^{\frac{i}{2}}.$$
\end{prop}
Finally, when $\lambda_n=1$, the reduction gives $d=1$. In this case $\pi$ is an isomorphism and $\mathscr{S}(E)=\mathrm{OGr}(1,4)\simeq\mathbb{P}^1\times\mathbb{P}^1$ is smooth. Its normalized IC stalk at $E$ is $\overline{\mathbb{Q}}_\ell(1)$ in degree $-2$, with Frobenius trace $q^{-1}$, so the normalized Kazhdan-Lusztig coefficient is $v^{-2}$.

For the definition of $\mathbf{t}'$ in the next subsection, we retain the coefficient
\begin{equation}\label{KL-O}g_{2d+2,\mathbf{\lambda}}(v)=v^{-\lambda_{n}}(v+v^{-1}).\end{equation}
This is the normalized Kazhdan-Lusztig coefficient when $\lambda_n\ge2$. When $\lambda_n=1$, it is $1+v^{-2}$, whereas the IC coefficient is $v^{-2}$. Thus, writing $A_\lambda=\lambda-E^\theta_{n,n}+E^\theta_{n,n+2}$, in this case we have
$$[A_\lambda]+g_{2d+2,\lambda}(v)1_\lambda=\{A_\lambda\}+1_\lambda.$$
Hence this component is realized by the direct sum of the normalized orbit-closure IC object and the diagonal object, rather than by a single IC object.

\subsection{Integral form}\label{subsec:integral-form}
Using the coefficients $f_{D,\lambda}(v)$ and $g_{2d+2,\lambda}(v)$ defined in \eqref{KL-2d}, \eqref{KL-2d+1} and \eqref{KL-O}, we introduce
\begin{gather}\label{def:tprime}
\mathbf{t}':=\left\{
\begin{array}{l}
\sum\limits_{\lambda\in \Theta_{D}^{\imath,\mathrm{diag}}}\big([\lambda-E^\theta_{n,n}+E^\theta_{n,n+2}]+f_{D,\lambda}(v)1_\lambda\big)\in\mathbf{S}^{\imath,D}_{n,d},\\ \sum\limits_{\lambda\in\Theta_{O}^{\imath,\mathrm{diag}}}\big([\lambda-E^\theta_{n,n}+E^\theta_{n,n+2}]+g_{2d+2,\lambda}(v)1_\lambda\big)\in\mathbf{S}^{\imath,O}_{n,d}.
\end{array}\right.
\end{gather}

Recall the generator $\mathbf{t}$ in \eqref{def:t}. We check directly that
\begin{align}\label{eq:difftt}
\mathbf{t}'=
\left\{\begin{array}{ll}
\mathbf{t}+\sum\limits_{\lambda\in\Theta^{\imath,\mathrm{diag}}_{2d+1}}(-1)^{\lambda_n}(v+v^{-1})^{-1}1_\lambda, &\text{for}\quad\mathbf{S}^{\imath,2d+1}_{n,d},\\
\mathbf{t}-\sum\limits_{\lambda\in\Theta^{\imath,\mathrm{diag}}_{2d}}(-1)^{\lambda_n}(v+v^{-1})^{-1}1_\lambda, &\text{for}\quad\mathbf{S}^{\imath,2d}_{n,d},\\
\mathbf{t}, &\text{for}\quad \mathbf{S}^{\imath,O}_{n,d}.
\end{array}\right.
\end{align}
Thus, $\mathbf{t}'\in\mathbf{S}^{\imath,D}_{n,d}$ satisfies the relations \eqref{i_D} (after changing $\mathbf{t}$ to $\mathbf{t}'$), but does not satisfy the relations \eqref{it2}; while $\mathbf{t}'=\mathbf{t}\in\mathbf{S}^{\imath,O}_{n,d}$ definitely satisfies the relations \eqref{i_O}.

\begin{rem}\label{rem:tprime-relations}
For the case of $\mathbf{S}^{\imath,D}_{n,d}$, assume $n\ge2$ and set
$$
\mathbf{j}=\sum_{\lambda\in\Theta_D^{\imath,\mathrm{diag}}}(-1)^{\lambda_n}1_\lambda.
$$
The following relations hold:
\begin{align*}
\mathbf{e}_{n-1}^2\mathbf{t}'-\llbracket2\rrbracket_{v^2}\mathbf{e}_{n-1}\mathbf{t}'\mathbf{e}_{n-1}+\mathbf{t}'\mathbf{e}_{n-1}^2
&=(-1)^{D+1}(v+v^{-1})\mathbf{e}_{n-1}^2\mathbf{j},\\
\mathbf{f}_{n-1}^2\mathbf{t}'-\llbracket2\rrbracket_{v^2}\mathbf{f}_{n-1}\mathbf{t}'\mathbf{f}_{n-1}+\mathbf{t}'\mathbf{f}_{n-1}^2
&=(-1)^{D+1}(v+v^{-1})\mathbf{f}_{n-1}^2\mathbf{j},\\
\mathbf{e}_{n-1}\mathbf{t}'^2-\llbracket2\rrbracket_{v^2}\mathbf{t}'\mathbf{e}_{n-1}\mathbf{t}'+\mathbf{t}'^2\mathbf{e}_{n-1}
&=(-1)^{D+1}(v+v^{-1})(\mathbf{e}_{n-1}\mathbf{t}'-\mathbf{t}'\mathbf{e}_{n-1})\mathbf{j},\\
\mathbf{f}_{n-1}\mathbf{t}'^2-\llbracket2\rrbracket_{v^2}\mathbf{t}'\mathbf{f}_{n-1}\mathbf{t}'+\mathbf{t}'^2\mathbf{f}_{n-1}
&=(-1)^{D+1}(v+v^{-1})(\mathbf{f}_{n-1}\mathbf{t}'-\mathbf{t}'\mathbf{f}_{n-1})\mathbf{j}.
\end{align*}
In fact, the formula \eqref{eq:difftt} gives $\mathbf{t}'=\mathbf{t}+(-1)^{D+1}(v+v^{-1})^{-1}\mathbf{j}$. Expanding \eqref{it2} with $$\mathbf{j}\mathbf{e}_{n-1}=-\mathbf{e}_{n-1}\mathbf{j}, \quad\mathbf{j}\mathbf{f}_{n-1}=-\mathbf{f}_{n-1}\mathbf{j},\quad \mathbf{j}\mathbf{t}=\mathbf{t}\mathbf{j},\quad \mathbf{j}^2=1$$ yields these identities.
\end{rem}

\begin{lem}
    The following two sets 
    $$\mathscr{G}_D:=\{\mathbf{e}_i,\mathbf{f}_i\ (1\le i\le n-1), \mathbf{t}, \mathbf{d}_a\ (1\le a\le n)\}\subset{}_{\mathbb{Q}}\mathbf{S}^{\imath,D}_{n,d},$$
    $$\mathscr{G}'_D:=\{\mathbf{e}_i,\mathbf{f}_i\ (1\le i\le n-1), \mathbf{t}', \mathbf{d}_a\ (1\le a\le n)\}\subset{}_{\mathbb{Q}}{}\mathbf{S}^{\imath,D}_{n,d}$$
    generate the same $\mathbb{Q}(v)$-algebra $_{\mathbb{Q}}\mathbf{S}^{\imath,D}_{n,d}$.
\end{lem}
\begin{proof}
    The only difference between $\mathscr{G}_D$ and $\mathscr{G}_D'$ are the elements $\mathbf{t}$ and $\mathbf{t}'$. These two elements differ by some scale 
    of $1_\lambda\ (\lambda\in\Theta^{\imath,\mathrm{diag}}_D)$ with coefficients in $\mathbb{Q}(v)$. By a standard Vandermonde determinant argument, $\mathbf{d}_a\ (1\le a\le n)$ can generate $\mathbb{Q}(v)1_\lambda\ (\lambda\in\Theta^{\imath,\mathrm{diag}}_D)$ over $\mathbb{Q}(v)$. 
\end{proof}
We also define $\mathscr{G}_O':=\{\mathbf{e}_i, \mathbf{f}_i\ (1\le i\le n-1), \mathbf{t}'(=\mathbf{t}), \mathbf{d}_a\ (1\le a\le n)\}\subset\mathbf{S}^{\imath,O}_{n,d}$.

By \S\ref{generator}, the elements $\mathbf{e}_i$, $\mathbf{f}_i$ and $\mathbf{d}_a^{\pm1}$
belong to $\mathbf{S}^{\imath,\square}_{n,d}$ for $\square=D,O$ and
the same holds for $\mathbf{t}'$ by our construction. Thus, we obtain an integral form of ${}_{\mathbb{Q}}\mathbf{S}^{\imath,\square}_{n,d}$ in the following proposition. In the next section, we will see that these generators behave well in the sense of categorification.
\begin{prop}
The algebra generated by $\mathscr{G}_\square'$ over $\mathcal{A}$ is an integral form of ${}_{\mathbb{Q}}\mathbf{S}^{\imath,\square}_{n,d}$.
\end{prop}
This integral form is not the iSchur algebra $\mathbf{S}^{\imath,\square}_{n,d}$, but a proper subalgebra of $\mathbf{S}^{\imath,\square}_{n,d}$. The idempotents corresponding to diagonal matrices cannot be obtain from these generators, since the Vandermonde determinant argument no longer makes sense over $\mathbb Z$.

\section{Categorification} \label{sec:categorification}

In this section, we provide a monoidal categorification of $\mathbf{S}^{\mathfrak{c},\square}_{n,d}$ by the category $\mathrm{Semis}^{\mathfrak{c},W}_\square$ of all descent data of the $F$-fixed semisimple complexes. Then we apply Verdier duality to get a bar involution on $\mathbf{S}^{\mathfrak{c},\square}_{n,d}$.

\subsection{Revisit $\mathbf{SO}_{2d+2,q^2}$-orbits}\label{subsec:so-orbits}
In order to give a unified treatment of the categorifications of $\mathbf{S}_{n,d}^{\mathfrak{c},D}$ and $\mathbf{S}_{n,d}^{\mathfrak{c},O}$, we take $\mathbf{G}=\mathbf{U}_D$ or $\mathbf{SO}^{-}_{2d+2}$ (NOT $\mathbf{O}^{-}_{2d+2}$ as before), defined over $\mathbb{F}_q$. Thus, in this section, we shall write $$\square=D \text{ or } SO \quad \text{(NOT $O$ as before)}.$$

Let us clarify some notions for the case of $\mathbf{SO}^{-}_{2d+2}$. All notions for $\square=O$ are still valid for $SO$ (cf. Lemma~\ref{biject D}) except $\widetilde{\Theta}^\jmath_{O}$ (cf. Lemma~\ref{lem:2SO-orbits}). That is, we rewrite 
$$
\sF_{SO}^{\mathfrak{c}}:=\sF_O^{\mathfrak{c}},\quad \pmb{\sF}_{SO}^{\mathfrak{c}}:=\pmb{\sF}_O^{\mathfrak{c}}, \quad \Theta_{SO}^\mathfrak{c}:=\Theta_{O}^\mathfrak{c} ,\quad
\widetilde{\sF}_{SO}^{\mathfrak{c}}:=\widetilde{\sF}_O^{\mathfrak{c}},\quad \widetilde{\pmb{\sF}}_{SO}^{\mathfrak{c}}:=\widetilde{\pmb{\sF}}_O^{\mathfrak{c}}, \quad
\mathbf{S}_{n,d}^{\mathfrak{c},SO}:=\mathbf{S}_{n,d}^{\mathfrak{c},O},
$$ but denote
\begin{align}\label{def:tildeThetaSO}
 \widetilde{\Theta}^\mathfrak{c}_{SO}&:=\{A,A^\vartheta \mid A\in \widetilde{\Theta}^\mathfrak{c}_O \} \neq \widetilde{\Theta}^\mathfrak{c}_{O},
 \end{align}
where we use $A$ and $A^\vartheta$ to label the two $\mathrm{SO}_{2d+2}(q^2)$ (resp., $\mathbf{SO}_{2d+2,q^2}$)-orbits in the $\mathrm{O}_{2d+2}(q^2)$ (resp., $\mathbf{O}_{2d+2,q^2}$)-orbit $\widetilde{\mathcal{O}}_A$ (resp., $\widetilde{\boldsymbol{\mathcal{O}}}_A$) if $A\in\widetilde{\Theta}^\mathfrak{c}_O\setminus \Theta^\mathfrak{c}_O$, and adopt the convention that $A=A^\vartheta$ if $A\in\Theta^\mathfrak{c}_O$.

By abuse of notation, we let $\widetilde{\mathcal{O}}_A$ (resp., $\widetilde{\boldsymbol{\mathcal{O}}}_A$) denote to the corresponding $\mathrm{SO}_{2d+2}(q^2)$ (resp., $\mathbf{SO}_{2d+2,q^2}$)-orbit if $A$ is regarded as an element of $\widetilde{\Theta}^\mathfrak{c}_{SO}$. Note that when $A\in\widetilde{\Theta}^\mathfrak{c}_O\setminus \Theta^\mathfrak{c}_O$, it does not coincide with the corresponding $\mathrm{O}_{2d+2}(q^2)$ (resp., $\mathbf{O}_{2d+2,q^2}$)-orbit denoted by the same notation. For $A\in\widetilde{\Theta}^\mathfrak{c}_{SO}$, we also write $\widetilde{\mathrm{IC}}_A$ for the intersection
complex of its closure, normalized as in \eqref{normal-IC}.

Moreover, let us regard $\vartheta$ as an involution of $\widetilde{\Theta}^\mathfrak{c}_{SO}$, i.e. $(A^\vartheta)^\vartheta=A$ for $A\in\widetilde{\Theta}^\mathfrak{c}_{SO}$. 
Then we can rewrite the known results (i.e., Propositions~\ref{F act on O_D} \& \ref{O-orbits} and Corollaries~\ref{F-var_D}~\&~\ref{F-var_O}) about $F$-action on obits in a unified form as follows.
\begin{prop}
The Frobenius $F$ satisfies
$F(\widetilde{\boldsymbol{\mathcal{O}}}_A)=\widetilde{\boldsymbol{\mathcal{O}}}_{A^\vartheta}$, and hence
$F(\widetilde{\mathcal{O}}_A)=\widetilde{\mathcal{O}}_{A^\vartheta}$ for $A\in\widetilde{\Theta}^\mathfrak{c}_\square$.
\end{prop}

\subsection{Semisimple complexes}\label{subsec:semisimple-complexes}

Recall $\mathbf{G}_{q^2}:=\mathbf{G}\times_{\mathbb{F}_q}\mathbb{F}_{q^2}$.  Let $\mathrm{Semis}^{\mathfrak{c}}_\square$ be the subcategory of semisimple complexes of $D^b_{\mathrm{m},\mathbf{G}_{q^2}}(\widetilde{\pmb{\sF}}_\square^\mathfrak{c}\times\widetilde{\pmb{\sF}}_\square^\mathfrak{c})$ generated by $\widetilde{\mathrm{IC}}_A\langle n\rangle$ ($A\in\widetilde{\Theta}_\square^\mathfrak{c}$ and $n\in\mathbb{Z}$). In particular, any object in $\mathrm{Semis}^{\mathfrak{c}}_\square$ is pure of weight $0$.
The Frobenius $F$ acts naturally on $\mathrm{Semis}^{\mathfrak{c}}_\square$. We denote by $\mathrm{Semis}^{\mathfrak{c},F}_\square$ the full subcategory of $\mathrm{Semis}^{\mathfrak{c}}_\square$ consisting of the objects fixed by $F$, where an object $\mathcal{F}$ is said to be fixed by $F$ if $F^*(\mathcal{F})\simeq \mathcal{F}$.
Let $\mathrm{Semis}^{\mathfrak{c},W}_\square$ be the category of all descent data $(\mathcal{F},\varphi)$ where $\mathcal{F}\in\mathrm{Semis}^{\mathfrak{c},F}_\square$ and $\varphi:F^*(\mathcal{F})\xrightarrow{\sim} \mathcal{F}$ is an isomorphism satisfying the cocycle condition $\varphi\circ F^*(\varphi)=\Psi_{\mathcal{F}}$, where $\Psi_{\mathcal{F}}:(F^*)^2(\mathcal{F})\simeq\mathcal{F}$ is the canonical $\mathbb{F}_{q^2}$-Weil structure of $\mathcal{F}$. 
A morphism $u:(\mathcal{F},\varphi)\to(\mathcal{F}',\varphi')$ is a descent-compatible morphism $u:\mathcal{F}\to\mathcal{F}'$ in $\mathrm{Semis}^{\mathfrak{c}}_\square$, saying $u\circ\varphi=\varphi'\circ F^*(u)$.

The following lemma shows that descent data can be split according to perverse degrees.
\begin{lem}\label{prop:descent-splitting}
Every $(K,\varphi)\in\mathrm{Semis}^{\mathfrak{c},W}_\square$ admits an isomorphism
$$
  (K,\varphi)\simeq\bigoplus_i
  \bigl({}^p\!H^i(K)[-i],{}^p\!H^i(\varphi)[-i]\bigr).
$$
In particular, if $M\in\mathrm{Semis}^{\mathfrak{c}}_\square$ lies in a single perverse degree and $m\neq n$, any descent on $M\langle m\rangle\oplus M\langle n\rangle$ splits as a direct sum of descents:
$$
  \bigl(M\langle m\rangle\oplus M\langle n\rangle,\varphi\bigr)
  \simeq (M\langle m\rangle,\varphi_m)
  \oplus(M\langle n\rangle,\varphi_n).
$$
\end{lem}
\begin{proof}
Since $F^*$ is perverse $t$-exact, the perverse truncations inherit descent data. Let $b$ be the largest perverse degree of $K$, put $N={}^p\!H^b(K)[-b]$, and denote its induced descent datum by $\varphi_N$. The truncation map $p:K\to N$ is descent-compatible and, by semisimplicity, has a section $s:N\to K$ in the underlying $\mathbb{F}_{q^2}$-Weil category. The cocycle condition and naturality of the prescribed Weil structures make
$$
  s\mapsto\varphi\circ F^*(s)\circ\varphi_N^{-1}
$$
an involution on $\Hom(N,K)$ in this underlying category. Thus
$$
  s_{\mathrm{av}}=\tfrac12\bigl(s+\varphi\circ F^*(s)\circ\varphi_N^{-1}\bigr)
$$
is descent-compatible. Moreover, $p\circ\varphi=\varphi_N\circ F^*(p)$ and $p\circ s=\operatorname{id}_N$ give
$$
  p\circ s_{\mathrm{av}}
  =\tfrac12\bigl(\operatorname{id}_N+\varphi_N\circ F^*(p\circ s)\circ\varphi_N^{-1}\bigr)
  =\operatorname{id}_N.
$$
Thus $s_{\mathrm{av}}$ is still a section of $p$. The truncation triangle therefore splits compatibly with descent. Induction on the number of nonzero perverse cohomology sheaves proves the first assertion. The second follows since the two shifted summands lie in distinct perverse degrees.
\end{proof}

Now we can give a classification of descent data in the following proposition.
\begin{prop}\label{prop:main}
Up to normalized shifts of each summand, write
$$
  M=\bigoplus_{A\in \widetilde{\Theta}_\square^\mathfrak{c}}\widetilde{\mathrm{IC}}_A^{\oplus m_A}
$$
for an object of $\mathrm{Semis}^{\mathfrak{c}}_\square$. All direct sums are finite, and the multiplicities and the condition in \textup{(1)} below are understood separately for each normalized shift.

\medskip
\noindent
\textup{(1)} The object $M$ admits an $F$-descent if and only if
$$
  m_A=m_{A^{\vartheta}}\quad\text{for all }A\in \widetilde{\Theta}_\square^\mathfrak{c}.
$$

\medskip
\noindent
\textup{(2)} Suppose $A\in\Theta_\square^\mathfrak{c}$. For every
$m\geq1$, the isomorphism classes of $F$-descents of $\widetilde{\mathrm{IC}}_A^{\oplus m}$ are
indexed by $r=0,1,\ldots,m$.  
If $\phi_+:F^*(\widetilde{\mathrm{IC}}_A)\xrightarrow{\sim}\widetilde{\mathrm{IC}}_A$ satisfies $\mathrm{tr}(\phi_+,\widetilde{\mathrm{IC}}_A|_{x_0})=1$ for $x_0\in{\mathcal{O}}_A$, then 
the class indexed by $r$ is
represented by $$
  \Phi_r=\phi_+\otimes\operatorname{diag}(I_r,-I_{m-r}),
$$
and the descended object is isomorphic to
$$
  (\widetilde{\mathrm{IC}}_A^+)^{\oplus r}\oplus(\widetilde{\mathrm{IC}}_A^-)^{\oplus(m-r)},
$$
where $\widetilde{\mathrm{IC}}_A^+=(\widetilde{\mathrm{IC}}_A,\phi_+)$ and $\widetilde{\mathrm{IC}}_A^-=(\widetilde{\mathrm{IC}}_A,-\phi_+)$.

In particular, an $F$-invariant IC complex of an individual orbit closure has exactly two nonisomorphic descents.

\medskip
\noindent
\textup{(3)} Let $A\in\widetilde{\Theta}_\square^\mathfrak{c}$ with $A\not\in\Theta_\square^\mathfrak{c}$.
For $m\geq1$, the sum $\widetilde{\mathrm{IC}}_A^{\oplus m}\oplus \widetilde{\mathrm{IC}}_{A^{\vartheta}}^{\oplus m}$
has a unique $F$-descent up to an isomorphism. Once we choose compatible
identifications $F^*(\widetilde{\mathrm{IC}}_A)\cong \widetilde{\mathrm{IC}}_{A^{\vartheta}}$, this unique $F$-descent is represented
by the swap descent
$$
  \Phi_{\mathrm{orb}}=
  \begin{pmatrix}
    0&\Psi_{\widetilde{\mathrm{IC}}_A}\otimes I_m\\
    I_m&0
  \end{pmatrix},  
$$where $\Psi_{\widetilde{\mathrm{IC}}_A}: (F^*)^2(\widetilde{\mathrm{IC}}_A)\to \widetilde{\mathrm{IC}}_A$ is the standard $\mathbb{F}_{q^2}$-Weil structure of $\widetilde{\mathrm{IC}}_A$.

For $m=1$, the descended object
$$
  P_{\{A,A^{\vartheta}\}}
  :=\left(\widetilde{\mathrm{IC}}_A\oplus \widetilde{\mathrm{IC}}_{A^{\vartheta}},
  \begin{pmatrix}0&\Psi_{\widetilde{\mathrm{IC}}_A}\\1&0\end{pmatrix}\right)
$$
is indecomposable in $\mathrm{Semis}^{\mathfrak{c},W}_\square$. For any $m\geq1$, the
unique descent is $P_{\{A,A^{\vartheta}\}}^{\oplus m}$. 

\medskip
\noindent
\textup{(4)} Every descended object has a decomposition of the following form, with normalized shifts suppressed as above:
$$
  \bigoplus_{A\in\Theta_\square^\mathfrak{c}}
  \left((\widetilde{\mathrm{IC}}_A^+)^{\oplus a_A}\oplus
        (\widetilde{\mathrm{IC}}_A^-)^{\oplus b_A}\right)
  \;\oplus\!
  \bigoplus_{\substack{\{A,A^{\vartheta}\}\\A\in\widetilde{\Theta}_\square^\mathfrak{c}\backslash\Theta_\square^\mathfrak{c}}}
  P_{\{A,A^{\vartheta}\}}^{\oplus c_{\{A,A^{\vartheta}\}}}.
$$
The second sum runs over the two-element $F$-orbits, with each orbit counted once (so $\{A,A^{\vartheta}\}$ and $\{A^{\vartheta},A\}$ count as one term). After recording the normalized shift of each summand, all multiplicities are uniquely determined.
\end{prop}
\begin{proof}
By Lemma~\ref{prop:descent-splitting}, we may work separately in each perverse degree and with each $F$-orbit of simple constituents. The assertion is the multiplicity-space form of the descent formalism of Lusztig--Vogan~\cite[\S2.1-\S2.2]{LV14}; we record the short linear-algebra argument.

Assume $F^*(\widetilde{\mathrm{IC}}_A)\simeq \widetilde{\mathrm{IC}}_A$ and fix the positive descent $\phi_{\widetilde{\mathrm{IC}}_A}:F^*(\widetilde{\mathrm{IC}}_A)\xrightarrow{\sim}\widetilde{\mathrm{IC}}_A$. Since $\widetilde{\mathrm{IC}}_A$ is a shift of a simple perverse sheaf, we have
$$
\End(\widetilde{\mathrm{IC}}_A)=\Ql.
$$
Write $\widetilde{\mathrm{IC}}_A^{\oplus m}=\widetilde{\mathrm{IC}}_A\otimes V$ with $\dim V=m$. Every descent datum is therefore of the form
$$
(\widetilde{\mathrm{IC}}_A^{\oplus m},\phi_{\widetilde{\mathrm{IC}}_A}\otimes P),\quad P\in\mathrm{GL}_m(\overline{\mathbb{Q}}_\ell).$$ 
The descent condition is equivalent to $P^2=1$, which implies that $P$ is semisimple Since $\Ql$ has characteristic zero, and
$$
V=V^+\oplus V^-, 
$$where $V^\pm$ are the eigen-spaces of $P$ with eigenvalues $\pm1$, respectively.
Hence
$$
(\widetilde{\mathrm{IC}}_A^{\oplus m},\Phi)\cong (\widetilde{\mathrm{IC}}_A^+)^{\oplus \dim V^+}\oplus(\widetilde{\mathrm{IC}}_A^-)^{\oplus \dim V^-},
$$
which proves (2).

Now suppose that $F^*(\widetilde{\mathrm{IC}}_A)\not\simeq \widetilde{\mathrm{IC}}_A$, and identify $\widetilde{\mathrm{IC}}_{A^{\vartheta}}=F^*\widetilde{\mathrm{IC}}_A$. Since $\widetilde{\mathrm{IC}}_A$ and $F^*\widetilde{\mathrm{IC}}_A$ are nonisomorphic shifts of simple perverse sheaves in the same perverse degree, any descent datum on $\widetilde{\mathrm{IC}}_A^{\oplus m}\oplus(F^*\widetilde{\mathrm{IC}}_A)^{\oplus m}$ has the form
$$
\Phi=\begin{pmatrix}0&\theta\otimes P\\ \operatorname{id}_{F^*\widetilde{\mathrm{IC}}_A}\otimes Q&0\end{pmatrix},\qquad P,Q\in\GL_m(\Ql),
$$
where $\theta=\Psi_{\widetilde{\mathrm{IC}}_A}:(F^*)^2\widetilde{\mathrm{IC}}_A\to \widetilde{\mathrm{IC}}_A$. The cocycle condition gives $PQ=QP=I_m$, and $T=\operatorname{diag}(I_m,P)$ satisfies $T\Phi=\Phi_{\mathrm{orb}}F^*(T)$. Hence every such descent is isomorphic to the stated swap descent. For $m=1$, its endomorphism ring is $\Ql$, so it is indecomposable. This proves (3).

The permutation of IC constituents under $F^*$ and the constructions above prove (1). Applying (2) and (3) in each perverse degree proves (4), including uniqueness of the multiplicities at each normalized shift.
\end{proof}

\subsection{Convolution} \label{subsec:convolution}
Let $X_1=X_2=X_3=\widetilde{\pmb{\sF}}_\square^\mathfrak{c}$, and let $\widetilde{p}_{ij}:X_{1}\times X_2\times X_3\to X_i\times X_j$ be the obvious projections.
Introduce the convolution
\begin{align*}\star:D^b_{\mathrm{m},\mathbf{G}_{q^2}}(\widetilde{\pmb{\sF}}_\square^\mathfrak{c}\times\widetilde{\pmb{\sF}}_\square^\mathfrak{c})\times D^b_{\mathrm{m},\mathbf{G}_{q^2}}(\widetilde{\pmb{\sF}}_\square^\mathfrak{c}\times\widetilde{\pmb{\sF}}_\square^\mathfrak{c}) &\to D^b_{\mathrm{m},\mathbf{G}_{q^2}}(\widetilde{\pmb{\sF}}_\square^\mathfrak{c}\times\widetilde{\pmb{\sF}}_\square^\mathfrak{c}),\\
(\mathcal{F},\mathcal{F}') &\mapsto\mathcal{F}\star\mathcal{F}':=R\widetilde{p}_{13*}\big(R\widetilde{p}_{12}^*(\mathcal{F})\otimes^{\mathbb{L}} R\widetilde{p}_{23}^*(\mathcal{F}')\big).
\end{align*}
The following result may be known to experts. We give a proof in Appendix~\ref{app:mixed-convolution} for self-contain.
\begin{lem}\label{prop:semis-convolution}
   The subcategory $\mathrm{Semis}^{\mathfrak{c}}_\square$ is closed under the convolution $\star$. 
\end{lem}
\begin{proof}
See Appendix~\ref{app:proof-semis-convolution}.
\end{proof}

Since $F$ commutes with $\widetilde{p}_{ij}$, we have the following corollary.
\begin{cor}\label{closed conv}
    The pullback functor $F^*$ is monoidal with respect to the convolution $\star$. That is, for any $\mathcal{F},\mathcal{F}'\in\mathrm{Semis}^{\mathfrak{c}}_\square,$ we have a canonical isomorphism $$F^*(\mathcal{F}\star\mathcal{F}')\simeq F^{*}(\mathcal{F})\star F^*(\mathcal{F}').$$
    Moreover, the subcategory $\mathrm{Semis}^{\mathfrak{c},F}_\square$ is closed under the convolution $\star$.
\end{cor}
For $(\mathcal{F},\varphi),(\mathcal{F}',\varphi')\in\mathrm{Semis}^{\mathfrak{c},W}_\square$, we define $$\varphi\star\varphi':=R\widetilde{p}_{13*}\big(R\widetilde{p}_{12}^*(\varphi)\otimes^{\mathbb{L}} R\widetilde{p}_{23}^*(\varphi')\big)\in\Hom(F^{*}(\mathcal{F}\star\mathcal{F}'),\mathcal{F}\star\mathcal{F}'),$$ 
which is clearly an isomorphism.

\begin{prop}\label{conv}
    There is a well-defined convolution map $$\diamond:\mathrm{Semis}^{\mathfrak{c},W}_\square\times\mathrm{Semis}^{\mathfrak{c},W}_\square\to\mathrm{Semis}^{\mathfrak{c},W}_\square \quad \text{via} \quad (\mathcal{F},\varphi)\diamond(\mathcal{F}',\varphi'):=(\mathcal{F}\star\mathcal{F}',\varphi\star\varphi').$$
\end{prop}
\begin{proof}
    By Corollary \ref{closed conv}, it remains to show that $\varphi\star\varphi'$ satisfies the cocycle condition.
    We give a straightforward calculation:
    \begin{align*}
        &\varphi\star\varphi'\circ F^*(\varphi\star\varphi')\\&=R\widetilde{p}_{13*}\big(R\widetilde{p}_{12}^*(\varphi)\otimes^{\mathbb{L}} R\widetilde{p}_{23}^*(\varphi')\big)\circ F^*\bigg(R\widetilde{p}_{13*}\big(R\widetilde{p}_{12}^*(\varphi)\otimes^{\mathbb{L}} R\widetilde{p}_{23}^*(\varphi')\big)\bigg)\\
    &=R\widetilde{p}_{13*}\big(R\widetilde{p}_{12}^*(\varphi)\otimes^{\mathbb{L}} R\widetilde{p}_{23}^*(\varphi')\big)\circ R\widetilde{p}_{13*}\big(R\widetilde{p}_{12}^*(F^*(\varphi))\otimes^{\mathbb{L}} R\widetilde{p}_{23}^*(F^*(\varphi'))\big)\\&=R\widetilde{p}_{13*}\big(R\widetilde{p}_{12}^*(\varphi\circ F^*(\varphi))\otimes^{\mathbb{L}} R\widetilde{p}_{23}^*(\varphi'\circ F^*(\varphi'))\big)\\
    &=R\widetilde{p}_{13*}\big(R\widetilde{p}_{12}^*(\Psi_\mathcal{F})\otimes^{\mathbb{L}} R\widetilde{p}_{23}^*(\Psi_{\mathcal{F}'})\big)\\
        &=\Psi_{\mathcal{F}\star\mathcal{F}'},
    \end{align*}
    where we use the canonical isomorphism $(F^{*})^2(\mathcal{F})\star(F^*)^2(\mathcal{F}')\simeq(F^*)^2(\mathcal{F}\star\mathcal{F}')$ for the last equality.
    The proof is complete.
\end{proof}
\begin{cor}
    The category $(\mathrm{Semis}^{\mathfrak{c},W}_\square,\diamond)$ is an additive monoidal category. Its indecomposable objects, up to normalized shifts and isomorphisms, are $\widetilde{\mathrm{IC}}^{\pm}_A$ and $P_{\{B,B^{\vartheta}\}}$, where $A\in\Theta_\square^\mathfrak{c}$ and $\{B,B^\vartheta\}$ runs over the two-element $F$-orbits in $\widetilde{\Theta}_\square^\mathfrak{c}\setminus\Theta_\square^\mathfrak{c}$. Every object admits a finite direct-sum decomposition into these irreducible objects, uniquely up to permutations and isomorphisms.
\end{cor}
\begin{proof}
    It follows from Propositions~\ref{prop:main} \& \ref{conv}.
\end{proof}
\subsection{Monoidal categorification}\label{subsec:monoidal-categorification}

For an additive category $\mathcal{C}$, let $K^\oplus_0(\mathcal{C})$ denote its split Grothendieck group.

The convolution on $\mathrm{Semis}^{\mathfrak{c},W}_\square$ gives a convolution map on $K_0^\oplus(\mathrm{Semis}^{\mathfrak{c},W}_\square)$, which is also written by $\diamond$.
For $(\mathcal{F},\varphi)$, we use the sheaf-function correspondence to give a function $\chi_{\mathcal{F},\varphi}$ on $\sF_\square^\mathfrak{c}\times\sF_\square^\mathfrak{c}$ as follows.
For $x_0\in\sF_\square^\mathfrak{c}\times\sF_\square^\mathfrak{c}$, 
we have a map $\varphi_{x_0}:(F^{*}(\mathcal{F}))_{x_{0}}\to \mathcal{F}_{x_0}$. Since $F(x_0)=x_0$, this map becomes $\varphi_{x_0}:\mathcal{F}_{x_0}\to\mathcal{F}_{x_0}$. For $(\mathcal{F},\varphi)\in\mathrm{Semis}^{\mathfrak{c},W}_{\square}$, we define $$\chi_{(\mathcal{F},\varphi)}:\sF_\square^\mathfrak{c}\times\sF_\square^\mathfrak{c}\to\mathbb{Z}[q^{\pm1/2}],\ \text{via}\quad\chi_{(\mathcal{F},\varphi)}(x_{0})=\sum\limits_{i\in\mathbb{Z}}(-1)^{i}\mathrm{tr}(\varphi_{x_0};\mathsf{H}^{i}(\mathcal{F})_{x_0}).$$
\begin{lem}\label{A}
    For $A\in\Theta^\mathfrak{c}_{\square}$, we have
    $\chi_{\widetilde{\mathrm{IC}}^+_A\langle d(A)-r(A)\rangle}=\{A\}|_{v=-q^{1/2}}$.
\end{lem}
\begin{proof}
We need to check the value at the points in $\overline{\mathcal{O}_A}$. Suppose $x_0\in\mathcal{O}_{A'}\subset\overline{\mathcal{O}_A}$, we have $r(A)=r(A')$.
    One has \begin{align*}
        \chi_{\widetilde{\mathrm{IC}}^+_A\langle d(A)-r(A)\rangle}&=\sum\limits_{i\in\mathbb{Z}}(-1)^i\mathrm{tr}\big(\varphi_{x_0};\mathsf{H}^i(\widetilde{\mathrm{IC}}^+_A\langle d(A)-r(A)\rangle)_{x_0}\big)\\&=\sum\limits_{i\in\mathbb{Z}}(-1)^i\mathrm{tr}\big(\varphi_{x_0};\mathsf{H}^i(\widetilde{\mathrm{IC}}^+_A)_{x_0}\big)(-q^{1/2})^{-d(A)+r(A)}.
    \end{align*}
    On the other hand, since $[A'](x_0)|_{v=-q^{1/2}}=(-q^{1/2})^{-d(A')+r(A')}=(-q^{1/2})^{-d(A')+r(A)}$, by \eqref{PA}, we have \begin{align*}
        \{A\}(x_0)|_{v=-q^{1/2}}&=P_{A',A}|_{v=-q^{1/2}}\cdot [A'](x_0)|_{v=-q^{1/2}}\\&=\sum_{i\in\mathbb{Z}}(-1)^i\mathrm{tr}\big(\varphi_{x_0};\mathsf{H}^i(\widetilde{\mathrm{IC}}^+_A)_{x_0}\big)(-q^{1/2})^{-d(A)+r(A)}.
    \end{align*}
   Therefore, $\chi_{\widetilde{\mathrm{IC}}^+_A\langle d(A)-r(A)\rangle}(x_0)=\{A\}(x_0)|_{v=-q^{1/2}}$, and the lemma follows.
\end{proof}
\begin{prop}\label{Sf}
    The map $$\chi:K_0^{\oplus}(\mathrm{Semis}^{\mathfrak{c},W}_\square)\to\mathbf{S}^{\mathfrak{c},\square}_{n,d}|_{v=-q^{1/2}},\quad [(\mathcal{F},\varphi)]\mapsto \chi_{(\mathcal{F},\varphi)}$$ is a surjective algebra homomorphism.
\end{prop}
\begin{proof}
For $(\mathcal{F},\varphi)\in\mathrm{Semis}^{\mathfrak{c},W}_\square$, since $\mathcal{F}$ is $\mathbf{G}_{q^2}$-equivariant, the image $\chi_{(\mathcal{F},\varphi)}$ is a $\mathbf{G}(\mathbb{F}_q)$-equivariant function on $\sF_\square^\mathfrak{c}\times\sF_\square^\mathfrak{c}$. 
    For $x_0=(\mathfrak{f},\mathfrak{f}'')\in\sF_\square^\mathfrak{c}\times\sF_\square^\mathfrak{c}$, we fix an isomorphism $\widetilde{\pmb{\sF}}_\square^\mathfrak{c}\simeq\{\mathfrak{f}\}\times\widetilde{\pmb{\sF}}_\square^\mathfrak{c}\times\{\mathfrak{f}''\}$. By proper base change, Grothendieck-Lefschetz trace formula and the functoriality of sheaf-function correspondence, we compute 
    \begin{align*}
&\chi_{(\mathcal{F},\varphi)\diamond(\mathcal{F}',\varphi')}(\mathfrak{f},\mathfrak{f}'')=\sum\limits_{i\in\mathbb{Z}}(-1)^i\mathrm{tr}\big(\varphi\star\varphi'|_{x_0};\mathsf{H}^i(\mathcal{F}\star\mathcal{F}')|_{x_0}\big)\\=&\sum\limits_{i\in\mathbb{Z}}(-1)^i\mathrm{tr}\bigg(\mathrm{R}\Gamma\big(\widetilde{p}_{12}^*(\varphi)|_{\widetilde{\pmb{\sF}}_\square^\mathfrak{c}}\otimes^\mathbb{L} \widetilde{p}_{23}^*(\varphi')|_{\widetilde{\pmb{\sF}}_\square^\mathfrak{c}}\big);H^i_{{\mathrm{\acute{e}t}}}\big(\widetilde{p}_{12}^*(\mathcal{F})|_{\widetilde{\pmb{\sF}}_\square^\mathfrak{c}}\otimes^\mathbb{L} \widetilde{p}_{23}^*(\mathcal{F}')|_{\widetilde{\pmb{\sF}}_\square^\mathfrak{c}}\big)\bigg)\\=&\sum\limits_{\mathfrak{f}'\in{\sF}_\square^\mathfrak{c}}\chi_{(\widetilde{p}^*_{12}(\mathcal{F})|_{\widetilde{\pmb{\sF}}_\square^\mathfrak{c}}\otimes^\mathbb{L} \widetilde{p}_{23}^*(\mathcal{F}')|_{\widetilde{\pmb{\sF}}_\square^\mathfrak{c}},\widetilde{p}_{12}^*(\varphi)|_{\widetilde{\pmb{\sF}}_\square^\mathfrak{c}}\otimes^\mathbb{L} \widetilde{p}_{23}^*(\varphi')|_{\widetilde{\pmb{\sF}}_\square^\mathfrak{c}})}(\mathfrak{f},\mathfrak{f}',\mathfrak{f}'')\\=&\sum\limits_{\mathfrak{f}'\in{\sF}^\mathfrak{c}_\square}\chi_{(\widetilde{p}_{12}^*(\mathcal{F})|_{\widetilde{\pmb{\sF}}_\square^\mathfrak{c}},\widetilde{p}_{12}^*(\varphi)|_{\widetilde{\pmb{\sF}}_\square^\mathfrak{c}})}(\mathfrak{f},\mathfrak{f}',\mathfrak{f}'')\cdot\chi_{(\widetilde{p}_{23}^*(\mathcal{F}')|_{\widetilde{\pmb{\sF}}_\square^\mathfrak{c}},\widetilde{p}_{23}^*(\varphi')|_{\widetilde{\pmb{\sF}}_\square^\mathfrak{c}})}(\mathfrak{f},\mathfrak{f}',\mathfrak{f}'')\\=&\sum\limits_{\mathfrak{f}'\in{\sF}^\mathfrak{c}_\square}\chi_{(\mathcal{F},\varphi)}(\mathfrak{f},\mathfrak{f}')\cdot\chi_{(\mathcal{F}',\varphi')}(\mathfrak{f}',\mathfrak{f}'')=\big(\chi_{(\mathcal{F},\varphi)}*\chi_{(\mathcal{F}',\varphi')}\big)(\mathfrak{f},\mathfrak{f}'')=\big(\chi_{(\mathcal{F},\varphi)}*\chi_{(\mathcal{F}',\varphi')}\big)(x_0).
    \end{align*}
    Hence, the map $\chi: K_0^{\oplus}(\mathrm{Semis}^{\mathfrak{c},W}_\square)\to\mathbf{S}^{\mathfrak{c},\square}_{n,d}|_{v=-q^{1/2}}$ is an algebra homomorphism.
    Moreover, by a shift version of Lemma~\ref{A}, we have $\chi_{\widetilde{\mathrm{IC}}^+_A}=v^{d(A)-r(A)}\{A\}|_{v=-q^{1/2}}$ for $A\in\Theta_\square^\mathfrak{c}$, which implies the surjectivity of $\chi$. The proof is complete.
\end{proof}

For $s\geq1$, put
$$
v_s=-(q^{1/2})^s.
$$
For $M\in\mathrm{Semis}^{\mathfrak{c},W}_\square$, write $M_{\mathrm{sp}}$
for the corresponding complex with its canonical split $F_q$-Weil
structure. Let $M_s$ denote the analogue of $M$ with respect to the
Frobenius $F_q^{\,s}\circ\tau$. This notation extends to direct sums,
normalized shifts and convolutions. Write $\chi_s$ for the
corresponding Frobenius trace map.

\begin{thm}\label{cat-main}
    There is a unique surjective algebra homomorphism, called the characteristic homomorphism,
    $$\mathrm{ch}:K_0^{\oplus}(\mathrm{Semis}^{\mathfrak{c},W}_\square)\to\mathbf{S}^{\mathfrak{c},\square}_{n,d}$$ such that
    $$\mathrm{ch}([\widetilde{\mathrm{IC}}^+_{A}\langle n_A\rangle])=-\mathrm{ch}([\widetilde{\mathrm{IC}}^-_{A}\langle n_A\rangle])=v^{-n_A}\mathrm{ch}([\widetilde{\mathrm{IC}}^+_{A}])=v^{d(A)-r(A)-n_A}\{A\}$$
    $$\text{and}\quad \mathrm{ch}([P_{\{B,B^\vartheta\}}\langle n\rangle])=0,$$
    where $A\in\Theta_\square^\mathfrak{c}$, $B\in\widetilde{\Theta}_\square^\mathfrak{c}\backslash\Theta_\square^\mathfrak{c}$ and $n_A,n\in\mathbb Z$.
    Moreover, the kernel is generated by $[\widetilde{\mathrm{IC}}^+_A\oplus\widetilde{\mathrm{IC}}^-_A]$ and $[P_{\{B,B^\vartheta\}}]$ up to some normalized shifts.
\end{thm}
\begin{proof}
Put $E_A^\pm=\widetilde{\mathrm{IC}}_A^\pm\langle d(A)-r(A)\rangle$. With $v[L]=[L\langle-1\rangle]$, Proposition~\ref{prop:main} gives an $\mathcal{A}$-basis of $K_0^\oplus(\mathrm{Semis}^{\mathfrak{c},W}_\square)$ (regarded as an $\mathcal A$-algebra) consisting of
$$
[E_A^\pm]\quad\text{and}\quad [P_\omega],
$$
where $A$ ranges over $\Theta_\square^\mathfrak{c}$ and $\omega=\{B,B^{\vartheta}\}$ ranges over the two-element $F$-orbits. Define the surjective $\mathcal{A}$-linear map
$$
\mathrm{ch}:K_0^\oplus(\mathrm{Semis}^{\mathfrak{c},W}_\square)\to\mathbf{S}_{n,d}^{\mathfrak{c},\square},\qquad
[E_A^\pm]\mapsto\pm\{A\},\quad
[P_\omega]\mapsto0.
$$
This gives the stated formulas on normalized shifts.

To prove multiplicativity, take $M,N$ of the form $E_A^\pm$ or $P_\omega$.
By Propositions~\ref{prop:main} \& \ref{conv}, write
$$
M\diamond N\simeq
\bigoplus_{C,n,\varepsilon=\pm}
\big(E_C^\varepsilon\langle n\rangle\big)^{\oplus a^\varepsilon_{C,n}}
{}\oplus\bigoplus_{\omega,n}
\big(P_\omega\langle n\rangle\big)^{\oplus b_{\omega,n}}.
$$
The split convolution decomposes as
\begin{equation*}
M_{\mathrm{sp}}\star N_{\mathrm{sp}}\simeq
\bigoplus_{C,n}(E_C^+)_{\mathrm{sp}}\langle n\rangle\otimes V_{C,n}
{}\oplus\bigoplus_{\omega,n}
\big((P_\omega)_{\mathrm{sp}}\langle n\rangle\big)^{\oplus b_{\omega,n}}.
\end{equation*}
In this decomposition, denote the induced action of $\tau$ on $V_{C,n}$ by $T_{C,n}$, while the induced action $\Phi_{C,n}$ of $F_q$ on $V_{C,n}$ is trivial. Keeping the $\tau$-structures
fixed as $s$ varies, the twisted Frobenius
$F_q^{\,s}\tau$ induces $\Phi_{C,n}^sT_{C,n}=T_{C,n}$ on $V_{C,n}$.
Thus, after factoring out $v_s^{-n}$, its coefficient under $\chi_s$ is
$\operatorname{tr}(T_{C,n};V_{C,n})=a^+_{C,n}-a^-_{C,n}$, independent of $s$.
Exchange summands have zero trace.
Putting $a_C(v)=\sum_n\operatorname{tr}(T_{C,n};V_{C,n})v^{-n}$,
Lemma~\ref{A} with $q$ replaced by $q^s$ and Proposition~\ref{Sf} give
\begin{align*}
\mathrm{ch}([M\diamond N])|_{v=v_s}
&=\sum_C a_C(v_s)\{C\}|_{v=v_s}
=\chi_s(M_s\diamond N_s)\\
&=\chi_s(M_s)*\chi_s(N_s)
=\bigl(\mathrm{ch}([M])\mathrm{ch}([N])\bigr)|_{v=v_s}.
\end{align*}
Write
$$
\mathrm{ch}([M\diamond N])-\mathrm{ch}([M])\mathrm{ch}([N])
=\sum_C h_C(v)\{C\}, \quad \text{with}\quad h_C(v)\in\mathcal A.
$$
The specialized canonical basis is
unitriangular over the standard basis, so $h_C(v_s)=0$ for every
$C$ and $s$. Infinitely many distinct $v_s$ force $h_C=0$
(cf.~\cite[Lemma~4.6]{LV14}). Multiplicativity follows by
$\mathcal A$-bilinearity.

Finally, comparing coefficients in the canonical basis gives
$$
\ker(\mathrm{ch})
=\sum_A\mathcal{A}([E_A^+]+[E_A^-])
+\sum_\omega\mathcal{A}[P_\omega],
$$
as asserted (cf.~\cite[\S2.3]{LV14}). Uniqueness follows since the
stated formulas prescribe the images of all indecomposable classes.
\end{proof}

\subsection{Refined covering and positivity}
\label{subsec:covering-refinement}

To keep track of the signs of descent data, we introduce a parameter
$\varpi$ with $\varpi^2=1$ and put
$$\mathcal A^\varpi=\mathcal A[\varpi]/(\varpi^2-1).$$
The sign twist
\[
\Pi(M,\varphi)=(M,-\varphi)
\]
is an additive autoequivalence commuting with shifts, and
$\Pi^2=\mathrm{id}$. Since the descent datum of a convolution is
induced by those of its factors, we have
\[
\Pi(M\diamond N)
\simeq(\Pi M)\diamond N
\simeq M\diamond(\Pi N).
\]
Thus $K_0^\oplus(\mathrm{Semis}^{\mathfrak c,W}_\square)$ is an $\mathcal A^\varpi$-algebra via
\[
v[M]=[M\langle-1\rangle],
\qquad
\varpi[M]=[\Pi M].
\]

For $A\in\Theta^\mathfrak c_\square$, set
\[
E_A^\pm
=\widetilde{\mathrm{IC}}_A^\pm
 \langle d(A)-r(A)\rangle,
\]
The sign twist exchanges $E_A^+$ and $E_A^-$ and fixes the
isomorphism class of $P_\omega$.
Hence
\[
[E_A^-]=\varpi [E_A^+],
\qquad
\varpi [P_\omega]=[P_\omega].
\]

With these notions, we can rewrite some results on the generating set of $K_0^\oplus(\mathrm{Semis}^{\mathfrak c,W}_\square)$ (cf. Proposition~\ref{prop:main}) and the characteristic homomorphism $\mathrm{ch}$ (cf. Theorem~\ref{cat-main}) as follows.
\begin{lem}
As an $\mathcal A^\varpi$-module,
\begin{equation}\label{eq:decompofK}
K_0^\oplus(\mathrm{Semis}^{\mathfrak c,W}_\square)\simeq
\bigoplus_{A\in\Theta^\mathfrak c_\square}
\mathcal A^\varpi [E_A^+]
\;\oplus\;
\bigoplus_\omega
\bigl(\mathcal A^\varpi/(\varpi-1)\bigr)[P_\omega].
\end{equation}
Moreover, the map $\mathrm{ch}$ induces an isomorphism of
$\mathcal A$-algebras
$$
K_0^\oplus(\mathrm{Semis}^{\mathfrak c,W}_\square)\Big/\mathrm{Ker}(\mathrm{ch})
\xrightarrow{\ \sim\ }
\mathbf S_{n,d}^{\mathfrak c,\square},
\qquad
[E_A^+]\longmapsto\{A\},
$$
with $\mathrm{Ker}(\mathrm{ch})=(\varpi+1)K_0^\oplus(\mathrm{Semis}^{\mathfrak c,W}_\square)+
      \sum_\omega\mathcal A [P_\omega]$.
\end{lem}

The generating set $\{[E_A^+]\}_A\cup\{[P_\omega]\}_\omega$ of $K_0^\oplus(\mathrm{Semis}^{\mathfrak c,W}_\square)$ admits a positivity property in the following sense.
\begin{thm}[Positivity]\label{thm:positivity}
For any $x,y\in\{[E_A^+]\}_A\cup\{[P_\omega]\}_\omega$, the product has the form
\[
xy=\sum_A h^\varpi_{x,y;A}(v)\,[E_A^+]
   +\sum_\omega r_{x,y;\omega}(v)\,[P_\omega],
\]
where $h^\varpi_{x,y;C}(v)
\in\mathbb N[v,v^{-1}]+\varpi\,\mathbb N[v,v^{-1}]$ and $r_{x,y;\omega}(v)\in\mathbb N[v,v^{-1}]$.

Consequently, for $A,B\in\Theta^\mathfrak c_\square$,
\[
\{A\}\{B\}
=\sum_C h^\varpi_{[E_A^+],[E_B^+];C}(v)\big|_{\varpi=-1}\,\{C\}.
\]
\end{thm}

\begin{proof}
By Propositions~\ref{prop:main} and~\ref{conv}, a convolution of
objects from $\{E_A^+\}_A\cup\{P_\omega\}_\omega$ decomposes into
normalized shifts of positive, negative and exchange objects, with
nonnegative integral multiplicities. A summand $E_A^+\langle n\rangle$,
$E_A^-\langle n\rangle$ or $P_\omega\langle n\rangle$ contributes
$v^{-n}[E_A^+]$, $\varpi v^{-n}[E_A^+]$ or $v^{-n}[P_\omega]$, respectively. This proves the positivity of the coefficients of $xy$ as desired. 
Thus the structure constants of canonical basis of $\mathbf S_{n,d}^{\mathfrak c,\square}$ are obtained by
specializing $\varpi=-1$.
\end{proof}

\subsection{Bar involution}\label{subsec:bar-involution}
As an application of the above categorification, we can construct a bar involution on $\mathbf{S}^{\mathfrak{c},\square}_{n,d}$ via Verdier duality.

Take $Y_1=Y_2=Y_3=\pmb{\sF}^\mathfrak{c}_\square$. Let ${p}_{ij}:Y_{1}\times Y_2\times Y_3\to Y_i\times Y_j$ be the obvious projections.
Introduce the convolution
\begin{align*}
\odot:D^b_{\mathrm{m},\mathbf{G}}({\pmb{\sF}}^\mathfrak{c}_\square\times{\pmb{\sF}}^\mathfrak{c}_\square)\times D^b_{\mathrm{m},\mathbf{G}}({\pmb{\sF}}^\mathfrak{c}_\square\times{\pmb{\sF}}^\mathfrak{c}_\square) &\to D^b_{\mathrm{m},\mathbf{G}}({\pmb{\sF}}^\mathfrak{c}_\square\times{\pmb{\sF}}^\mathfrak{c}_\square),\\ (\mathcal{F},\mathcal{F}') &\mapsto\mathcal{F}\odot\mathcal{F}':=R{p}_{13*}\big(R{p}_{12}^*(\mathcal{F})\otimes^{\mathbb{L}} R{p}_{23}^*(\mathcal{F}')\big).
\end{align*}

To simplify the discussion on Verdier duality, we henceforth work in the descended category. By the $\mathbb F_{q^2}/\mathbb F_q$-Galois descent equivalence, we identify the category $\mathrm{Semis}^{\mathfrak{c},W}_\square$ of descent data with the corresponding full subcategory of the $\mathbf{G}$-equivariant derived category $D^b_{\mathrm{m},\mathbf{G}}({\pmb{\sF}}^\mathfrak{c}_\square\times{\pmb{\sF}}^\mathfrak{c}_\square)$ on the descended double flag variety ${\pmb{\sF}}^\mathfrak{c}_\square\times{\pmb{\sF}}^\mathfrak{c}_\square$. Since the Frobenius $F$ preserves the convolution correspondence, proper base change identifies the convolution $\diamond$ on $\mathrm{Semis}^{\mathfrak{c},W}_\square$ with the usual convolution $\odot$ defined above on the descended double flag variety.

The variety $\pmb{\sF}_\square^\mathfrak{c}$ is not connected, whose connected components are partial flag varieties of $\mathbf{G}$. Write $\pmb{\sF}_\square^\mathfrak{c}=\bigsqcup_\gamma{\pmb{\sF}}_\gamma$, a disjoint union of some partial flag varieties of $\mathbf{G}$. Write $d_\gamma$ for the dimension of $\pmb{\sF}_\gamma$.
For $\eta,\eta'\in\{\gamma,\delta,\varepsilon\}$, let $\DD_{\eta,\eta'}$ denote the usual Verdier duality on
$\pmb{\sF}_\eta\times\pmb{\sF}_{\eta'}$.
Define the dimension-normalized Verdier duality by
$$
\nDD_{\eta,\eta'}(\mathcal M)
:=
\DD_{\eta,\eta'}(\mathcal M)
\langle-2d_\eta\rangle,\qquad\text{for}\quad \mathcal M\in D^b_{\mathrm{m},\mathbf{G}}(\pmb{\sF}_\eta\times\pmb{\sF}_{\eta'}).
$$

\begin{lem}\label{Verdier}
Let $\mathcal K\in
D^b_{\mathrm{m},\mathbf G}(\pmb{\sF}_\gamma\times\pmb{\sF}_\delta)$ and $\mathcal L\in
D^b_{\mathrm m, \mathbf G}(\pmb{\sF}_\delta\times\pmb{\sF}_\varepsilon)$
be two equivariant intersection cohomology complexes (up to shifts and Tate twists) associated with
closures of diagonal $\mathrm{G}$-orbits. 
There exists a canonical isomorphism
$$
\nDD_{\gamma,\varepsilon}
\left(
\mathcal K\odot\mathcal L
\right)
\cong
\nDD_{\gamma,\delta}(\mathcal K)
\odot
\nDD_{\delta,\varepsilon}(\mathcal L).
$$
equivalently, the usual Verdier duality satisfies
$$
\DD_{\gamma,\varepsilon}
\left(
\mathcal K\odot\mathcal L
\right)
\cong
\left(
\DD_{\gamma,\delta}(\mathcal K)
\odot
\DD_{\delta,\varepsilon}(\mathcal L)
\right)
\langle-2d_\delta\rangle.
$$
\end{lem}

\begin{proof}
Let
$Y=\pmb{\sF}_\gamma\times\pmb{\sF}_\delta\times\pmb{\sF}_\varepsilon$.
Consider the middle-diagonal embedding
$$
\iota:\ Y
\to
(\pmb{\sF}_\gamma\times\pmb{\sF}_\delta)
\times
(\pmb{\sF}_\delta\times\pmb{\sF}_\varepsilon), \quad
(A,B,C)
\mapsto
((A,B),(B,C)),
$$
which satisfies
$$
p_{12}^*\mathcal K
{\otimes}^{\mathbb L}
p_{23}^*\mathcal L
\cong
\iota^*(\mathcal K\boxt\mathcal L).
$$
The projections of every diagonal
$\mathrm{G}$-orbit in the two adjacent double partial flag varieties onto
$\pmb{\sF}_\delta$ are smooth homogeneous fibrations. Thus $\iota$ is
transverse to the product orbit stratification and is therefore
non-characteristic for $\DD_{\gamma,\delta}(\mathcal K)
\boxtimes
\DD_{\delta,\varepsilon}(\mathcal L)$.
Since $\iota$ is a regular closed embedding of codimension $d_\delta$ and is non-characteristic with respect to
$\DD_{\gamma,\delta}(\mathcal K)
\boxtimes
\DD_{\delta,\varepsilon}(\mathcal L)$,
the non-characteristic pullback formula, together with the standard purity isomorphism, gives
\begin{align*}
\iota^!
\left(
\DD_{\gamma,\delta}(\mathcal K)
\boxtimes
\DD_{\delta,\varepsilon}(\mathcal L)
\right)
\cong
\iota^*
\left(
\DD_{\gamma,\delta}(\mathcal K)
\boxtimes
\DD_{\delta,\varepsilon}(\mathcal L)
\right)
[-2d_\delta](-d_\delta);
\end{align*}
see \cite[Proposition~8.13(1)]{Sa17}.

Using
$\DD_Y\iota^*
\cong
\iota^!\DD$ and 
$\DD(\mathcal K\boxt\mathcal L)
\cong
\DD_{\gamma,\delta}(\mathcal K)
\boxt
\DD_{\delta,\varepsilon}(\mathcal L)$,
we obtain
$$
\DD_Y
\left(
p_{12}^*\mathcal K
{\otimes}^{\mathbb L}
p_{23}^*\mathcal L
\right)
\cong
\left(
p_{12}^*\DD_{\gamma,\delta}(\mathcal K)
{\otimes}^{\mathbb L}
p_{23}^*\DD_{\delta,\varepsilon}(\mathcal L)
\right)
\langle-2d_\delta\rangle.
$$
Since $\pmb{\sF}_\delta$ is proper, the projection $p_{13}$ is also proper. The compatibility of Verdier duality with proper pushforward therefore yields
$$
\DD_{\gamma,\varepsilon}
\left(
\mathcal K\odot\mathcal L
\right)
\cong
\left(
\DD_{\gamma,\delta}(\mathcal K)
\odot
\DD_{\delta,\varepsilon}(\mathcal L)
\right)
\langle-2d_\delta\rangle.
$$
The normalized formula follows immediately from the definition of $\nDD_{\eta,\eta'}$. The resulting compatibility with convolution is analogous to
\cite[Lemma~3.2]{LY20}; see also \cite[\S2.2]{W17} for the convolution formalism on products of partial flag varieties.
\end{proof}

We define the normalized Verdier duality on $D^b_{\mathrm{m},\mathbf{G}}(\pmb{\sF}_\square^\mathfrak{c}\times\pmb{\sF}^\mathfrak{c}_{\square})$ by $\mathfrak{D}_\square^\mathfrak{c}:=\bigoplus_{\gamma,\delta}\mathfrak{D}_{\gamma,\delta}$. Clearly, we have $\mathfrak{D}_\square^\mathfrak{c}\circ\mathfrak{D}_\square^\mathfrak{c}\simeq \mathrm{id}$ and the subcategory $\mathrm{Semis}^{\mathfrak{c},W}_{\square}$ is stable under $\mathfrak{D}_\square^\mathfrak{c}$.
By abuse of notation, we still use $\widetilde{\mathrm{IC}}^\pm_A$ and $P_{\{B,B^\vartheta\}}$ to denote the descended objects, where $A\in\Theta_\square^\mathfrak{c}$ and $B\in\widetilde{\Theta}_\square^\mathfrak{c}\backslash\Theta_\square^\mathfrak{c}$. We write $d(B)$ for the dimension of $\widetilde{\boldsymbol{\mathcal{O}}}_B$ and $r(B)$ for the dimension of the image of $\widetilde{\boldsymbol{\mathcal{O}}}_B$ under the first projection $\widetilde{\pmb{\sF}}_\square^\mathfrak{c}\times\widetilde{\pmb{\sF}}_\square^\mathfrak{c}$. The normalized Verdier duality induces an involution on $K^{\oplus}_0(\mathrm{Semis}^{\mathfrak{c},W}_\square)$, which is also denoted by $\mathfrak{D}_\square^\mathfrak{c}$.
\begin{prop}\label{D_C}
    For $A\in\Theta_\square^\mathfrak{c}$, we have $$\mathfrak{D}_\square^\mathfrak{c}\big(\widetilde{\mathrm{IC}}^\pm_A\langle d(A)-r(A)+n\rangle\big)=\widetilde{\mathrm{IC}}^\pm_A\langle d(A)-r(A)-n\rangle.$$ In particular, $$\mathfrak{D}_\square^\mathfrak{c}\big(\underline{\overline{\mathbb{Q}}_\ell}_{\Delta\pmb{\sF}_\square^\mathfrak{c}}\langle n\rangle\big)=\underline{\overline{\mathbb{Q}}_\ell}_{\Delta \pmb{\sF}_\square^\mathfrak{c}}\langle -n\rangle,$$ where $\Delta \pmb{\sF}_\square^\mathfrak{c}$ is the diagonal of $\pmb{\sF}_\square^\mathfrak{c}\times\pmb{\sF}_\square^\mathfrak{c}$.
    For $B\in\widetilde{\Theta}_\square^\mathfrak{c}\backslash\Theta_\square^\mathfrak{c}$, we have $$\mathfrak{D}_\square^\mathfrak{c}\big(P_{\{B,B^\vartheta\}}\langle d(B)-r(B)+n\rangle\big)=P_{\{B,B^\vartheta\}}\langle d(B)-r(B)-n\rangle.$$
    As a consequence, $\mathrm{ker}(\mathrm{ch})$ is stable under $\mathfrak{D}_\square^\mathfrak{c}$.
    \end{prop}

\begin{proof}
    Suppose $\boldsymbol{\mathcal{O}}_A\subset\pmb{\sF}_\gamma\times\pmb{\sF}_\delta$, then $r(A)=d_\gamma$. By the standard Verdier self-duality of the weight-zero intersection
cohomology complexes with their positive and negative Weil structures,
we have
\begin{align*}
\mathfrak{D}_\square^\mathfrak{c}\big(\widetilde{\mathrm{IC}}^\pm_A\langle d(A)-r(A)+n\rangle\big)=&\mathbb{D}_{\gamma,\delta}\big(\widetilde{\mathrm{IC}}^\pm_A\langle d(A)\rangle\langle -r(A)+n\rangle\big)\langle-2r(A)\rangle\\=&\widetilde{\mathrm{IC}}^\pm_A\langle d(A)\rangle\langle r(A)-n\rangle\langle-2r(A)\rangle\\=&\widetilde{\mathrm{IC}}^\pm_A\langle d(A)-r(A)-n\rangle.
\end{align*}
 In particular, when $\boldsymbol{\mathcal{O}}_A$ is the diagonal copy of $\pmb{\sF}_\gamma\times\pmb{\sF}_\gamma$, we have 
$\mathfrak{D}_\square^\mathfrak{c}\big(\widetilde{\mathrm{IC}}_A\langle n\rangle\big)=\widetilde{\mathrm{IC}}_A\langle-n\rangle$, and hence $$\mathfrak{D}_\square^\mathfrak{c}\big(\underline{\overline{\mathbb{Q}}_\ell}_{\Delta\pmb{\sF}_\square^\mathfrak{c}}\langle n\rangle\big)=\underline{\overline{\mathbb{Q}}_\ell}_{\Delta \pmb{\sF}_\square^\mathfrak{c}}\langle -n\rangle.$$
For $B\in\widetilde{\Theta}_\square^\mathfrak{c}\backslash\Theta_\square^\mathfrak{c}$, the complex $P_{\{B,B^\vartheta\}}\langle d(B)\rangle$, equipped
with its standard $\mathbb{F}_{q^2}$-Weil structure, is Verdier self-dual and admits
a unique descent to $\mathbb{F}_q$ up to isomorphism. Since Verdier
duality commutes with finite \'etale base change, its descended object is also Verdier self-dual. Consequently,
$$\mathfrak{D}_\square^\mathfrak{c}\big(P_{\{B,B^\vartheta\}}\langle d(B)-r(B)+n\rangle\big)=P_{\{B,B^\vartheta\}}\langle d(B)-r(B)-n\rangle.$$
Now, the last assertion follows from Theorem \ref{cat-main}.
\end{proof}

\begin{thm}
    There is a unique involution of $\mathbb Z$-algebra $\bar{}:\mathbf{S}^{\mathfrak{c},\square}_{n,d}\to\mathbf{S}^{\mathfrak{c},\square}_{n,d}$ (called the bar involution)  
such that $\overline{v}=v^{-1}$ and $\overline{\{A\}}=\{A\}\ (\forall A\in\Theta_\square^\mathfrak{c})$.
\end{thm}
\begin{proof}
Since $\mathrm{ker}(\mathrm{ch})$ is stable under $\mathfrak{D}_\square^\mathfrak{c}$, by Lemma \ref{Verdier}, we have a well-defined involution of $\mathbb{Z}$-algebra given by
$$\overline{\mathfrak{D}}_\square:K^\oplus_0(\mathrm{Semis}^{\mathfrak{c},W}_\square)/\mathrm{ker}(\mathrm{ch})\to K^\oplus_0(\mathrm{Semis}^{\mathfrak{c},W}_\square)/\mathrm{ker}(\mathrm{ch}),\quad [\mathcal{K}]\mapsto [\mathfrak{D}_\square^\mathfrak{c}(\mathcal{K})].$$ Together with Theorem \ref{cat-main}, we have a $\mathbb{Z}$-algebra involution $\bar{}:\mathbf{S}^{\mathfrak{c},\square}_{n,d}\to\mathbf{S}^{\mathfrak{c},\square}_{n,d}$. Since $\mathrm{ch}\big([\underline{\overline{\mathbb{Q}}_\ell}_{\Delta \pmb{\sF}_\square^\mathfrak{c}}\langle \pm1\rangle]\big)=v^{\mp1}$, $\mathrm{ch}\big([\widetilde{\mathrm{IC}}_A^+\langle d(A)-r(A)\rangle]\big)=\{A\}$ and such elements generate the whole algebra $\mathbf{S}^{\mathfrak{c},\square}_{n,d}$ over $\mathbb{Z}$, the theorem follows from Proposition \ref{D_C}.\end{proof}

The bar involution can be extended to ${}_\mathbb{Q}\mathbf{S}^{\mathfrak{c},\square}_{n,d}$ naturally.

\begin{cor}
The elements $\mathbf{e}_i,\mathbf{f}_i$, $\mathbf{t}',\mathbf{t}$ are all bar invariant, while $\overline{\mathbf{d}_a^{\pm1}}=\mathbf{d}_a^{\mp1}$. 
\end{cor}
\begin{proof}
The generators $\mathbf{e}_i$ and $\mathbf{f}_i$ are bar invariant since they are sums of canonical basis elements. The element $\mathbf{t}'$ is bar invariant by our canonical construction in \eqref{def:tprime}, which further implies $\overline{\mathbf{t}}=\mathbf{t}$ since the difference between $\mathbf{t}'$ and $\mathbf{t}$ is bar invariant by \eqref{eq:difftt}. For $\mathbf{d}_a^{\pm1}=\sum_{\lambda}v^{\mp2\lambda_a}1_{\lambda}\in \mathbf{S}^{\mathfrak{c}, D}_{n,d}$ (resp. $\sum_{\lambda}v^{\mp\lambda_a}1_{\lambda}\in \mathbf{S}^{\mathfrak{c}, O}_{n,d}$), we have $\overline{\mathbf{d}_a^{\pm1}}=\sum_{\lambda}v^{\pm2\lambda_a}1_{\lambda}$ (resp.  $\sum_{\lambda}v^{\pm\lambda_a}1_{\lambda}$) since $1_\lambda=e_\lambda=[\lambda]=\{\lambda\}$ is bar invariant for each diagonal matrix $\lambda=\mathrm{diag}(\lambda_1,\ldots,\lambda_{n+1},\ldots,\lambda_1)$. That is, $\overline{\mathbf{d}_a^{\pm1}}=\mathbf{d}_a^{\mp1}$.
\end{proof}

\begin{rem}
Under the specializations in Theorem~\ref{thm:specialofS}, with
$(\mathbf u,\mathbf v)=(v^3,v^2),\allowbreak(v,v^2),\allowbreak(v^2,v)$, respectively,
our geometric bar involutions agree with the algebraic ones defined in
\cite[\S4.1]{LL21}, compatible with the bar involutions
of iquantum groups introduced in \cite[Lemmas~2.1 \& 4.1]{BWW18}
(cf.~\cite[Remark~7.3.3]{LL21}). 
\end{rem}

\section{$\mathrm{i}$Quantum groups and $\mathrm{i}$Schur algebras}
\label{sec:iquantumgroups}
In this section, we study the iquantum groups associated with $\mathbf{S}^{\mathfrak{c},\square}_{n,d}$. In particular, we provide a weak categorification for these iquantum groups, with explicit formulations for generators.

\subsection{iquantum groups}\label{subsec:iquantum-groups}
The quasi-split iquantum group $\mathbb{U}_n^\jmath$ of type $\mathrm{AIII}_{2n}$ (with two parameters $\mathbf{u}$ and $\mathbf{v}$) is the $\mathbb{Q}(\mathbf{u},\mathbf{v})$-algebra with generators
$$e_i, f_i \ (i=1,2,\ldots,n)\quad \text{and}\quad d_{a}^{\pm1} \ (a=1,2,\ldots,n+1)$$
subject to the following relations: for $1\leq i,j\leq n$ and $1\leq a,b\leq n+1$,
\begin{gather}
d_ad_a^{-1}=1=d_{a}^{-1}d_a, \quad d_ad_b=d_bd_a,\label{re:1}\\
d_{n+1}e_{i}d_{n+1}^{-1}=\mathbf{v}^{-2\delta_{n,i}}e_i, \quad d_{n+1}f_{i}d_{n+1}^{-1}=\mathbf{v}^{2\delta_{n,i}}f_i,\label{re:2}\\
d_ae_id_a^{-1}=\mathbf{v}^{\delta_{a,i}-\delta_{a,i+1}}e_i,\quad
d_af_id_a^{-1}=\mathbf{v}^{\delta_{a,i+1}-\delta_{a,i}}f_i \quad(a\neq n+1),\label{re:3}\\
e_if_j-f_je_i=\delta_{i,j}\frac{d_id_{i+1}^{-1}-d_i^{-1}d_{i+1}}{\mathbf{v}-\mathbf{v}^{-1}} \quad\big((i,j)\neq(n,n)\big), \label{re:4}\\
e_i^2e_j+e_je_i^2=\llbracket2\rrbracket_{\mathbf{v}}e_ie_je_i,\quad f_i^2f_j+f_jf_i^2=\llbracket2\rrbracket_{\mathbf{v}}f_if_jf_i \quad (|i-j|=1),\label{re:5}\\
e_ie_j=e_je_i,\quad f_if_j=f_jf_i \quad (|i-j|>1), \label{re:6}\\
e_n^2f_n+f_ne_n^2=\llbracket2\rrbracket_{\mathbf{v}}\big(e_nf_ne_n-e_n(d_nd_{n+1}^{-1}+d_n^{-1}d_{n+1})\big), \label{re:7}\\
f_n^2e_n+e_nf_n^2=\llbracket2\rrbracket_{\mathbf{v}}\big(f_ne_nf_n-(d_nd_{n+1}^{-1}+d_n^{-1}d_{n+1})f_n\big). \label{re:8}
\end{gather}

\begin{rem}
The iquantum group $\mathbb{U}_n^\jmath$ is isomorphic to the one introduced in \cite{BWW18, LL21}, where the last two relations are 
\begin{gather*}
e_n^2f_n+f_ne_n^2=\llbracket2\rrbracket_{\mathbf{v}}\big(e_nf_ne_n-e_n(\mathbf{u}\mathbf{v}d_nd_{n+1}^{-1}+\mathbf{u}^{-1}\mathbf{v}^{-1}d_n^{-1}d_{n+1})\big),\\
f_n^2e_n+e_nf_n^2=\llbracket2\rrbracket_{\mathbf{v}}\big(f_ne_nf_n-(\mathbf{u}\mathbf{v}d_nd_{n+1}^{-1}+\mathbf{u}^{-1}\mathbf{v}^{-1}d_n^{-1}d_{n+1})f_n\big). 
\end{gather*}
It suffices to replace their $d_{n+1}$ by our $\mathbf{u}\mathbf{v}d_{n+1}$. The advantage of our presentation is that the relations do not involve the parameter $\mathbf{u}$. 
\end{rem}

The quasi-split iquantum group $\mathbb{U}_n^\imath$ of type $\mathrm{AIII}_{2n-1}$ is the $\mathbb{Q}(\mathbf{u},\mathbf{v})$-algebra with generators 
$$e_i, \ f_i \ (i=1,2,\ldots,n-1), \quad t \quad \text{and}\quad d_{a}^{\pm1} \ (a=1,2,\ldots,n),$$ 
subject to the relations \eqref{re:1}, 
\eqref{re:3}-\eqref{re:6} and 
\begin{gather*}
td_a=d_at,\qquad te_i=e_it,\quad tf_i=f_it\quad(i\neq n-1),\\
 e^2_{n-1}t-\llbracket2\rrbracket_{\mathbf{v}}e_{n-1}te_{n-1}+te^2_{n-1}=0,\quad
    f^2_{n-1}t-\llbracket2\rrbracket_{\mathbf{v}}f_{n-1}tf_{n-1}+tf^2_{n-1}=0,\\
    t^2e_{n-1}-\llbracket2\rrbracket_{\mathbf{v}}te_{n-1}t+e_{n-1}t^2=e_{n-1},\quad
    t^2f_{n-1}-\llbracket2\rrbracket_{\mathbf{v}}tf_{n-1}t+f_{n-1}t^2=f_{n-1}.
\end{gather*}

\subsection{Stabilization construction}
 
Imitating the stabilization construction of quantum $\mathfrak{gl}_n$ due to Beilinson-Lusztig-MacPherson \cite{BLM90}, Lai and the first author \cite{LL21} proved that the family $\{\mathbb{S}_{n,d}^\mathfrak{c}~|~d\in\mathbb{N}\}$ of iSchur algebras admits a stabilization property, and hence constructed a stabilization algebra isomorphic to (the modified version of) the iquantum group $\mathbb{U}^\mathfrak{c}$ ($\mathfrak{c}=\imath,\jmath$). Thanks to Theorem~\ref{thm:specialofS}, the approach developed in \cite{LL21} is also valid for the iSchur algebras $\mathbf{S}_{n,d}^{\mathfrak{c},\square}$ ($\mathfrak{c}=\imath,\jmath$; $\square=D,O$). Namely, we can establish a stabilization property for $\{\mathbf{S}_{n,d}^{\mathfrak{c},\square}~|~d\in\mathbb{N}\}$, and then realize some specializations of (the modified version of) the iquantum group $\mathbb{U}^\mathfrak{c}_n$ as stabilization algebras. Here the specializations are as follows: $\mathbf{u}^2=\q^3, \mathbf{v}^2=\q^2$ if $\square=2d+1$; $\mathbf{u}^2=\q, \mathbf{v}^2=\q^2$ if $\square=2d$; and $\mathbf{u}^2=\q^2, \mathbf{v}^2=\q$ if $\square=O$. We still denote by $e_i, f_i, t, d_a^{\pm1}$ the generators of these specializations of $\mathbb{U}^\mathfrak{c}$.

The stabilization construction provides an explicit algebra homomorphism from iquantum groups to iSchur algebras as follows. Recall $\q=v^2$.
\begin{prop}\label{prop:surjqsch}
\begin{itemize}
\item[(1)]
There are surjective homomorphisms of $\mathbb{Q}(v)$-algebras 
$$\mathbb{U}^\jmath_n|_{\mathbf{u}^2=\q^3, \mathbf{v}^2=\q^2}\to {}_{\mathbb{Q}}\mathbf{S}^{\jmath,2d+1}_{n,d},\quad \mathbb{U}^\jmath_{n}|_{\mathbf{u}^2=\q, \mathbf{v}^2=\q^2}\to {}_{\mathbb{Q}}\mathbf{S}^{\jmath,2d}_{n,d}, \quad \mathbb{U}^\jmath_{n}|_{\mathbf{u}^2=\q^2, \mathbf{v}^2=\q}\to {}_{\mathbb{Q}}\mathbf{S}^{\jmath,O}_{n,d}$$
given by 
$$e_i \mapsto\mathbf{e}_i, \quad f_i\mapsto\mathbf{f}_i,\quad d_i \mapsto\mathbf{d}_i \quad (1\leq i\leq n), \quad d_{n+1}\mapsto \left\{\begin{array}{ll}v^{-3}\mathbf{d}_{n+1}, & \text{for } {}_{\mathbb{Q}}\mathbf{S}^{\jmath,D}_{n,d},\\
v^{-1}\mathbf{d}_{n+1}, & \text{for } {}_{\mathbb{Q}}\mathbf{S}^{\jmath,O}_{n,d}.
\end{array}\right.$$
\item[(2)]
There are surjective homomorphisms of $\mathbb{Q}(v)$-algebras 
$$\mathbb{U}^\imath_n|_{\mathbf{u}^2=\q^3, \mathbf{v}^2=\q^2}\to {}_{\mathbb{Q}}\mathbf{S}^{\imath,2d+1}_{n,d},\quad \mathbb{U}^\imath_{n}|_{\mathbf{u}^2=\q, \mathbf{v}^2=\q^2}\to {}_{\mathbb{Q}}\mathbf{S}^{\imath,2d}_{n,d}, \quad \mathbb{U}^\imath_{n}|_{\mathbf{u}^2=\q^2, \mathbf{v}^2=\q}\to {}_{\mathbb{Q}}\mathbf{S}^{\imath,O}_{n,d}$$ given by 
$$e_i \mapsto\mathbf{e}_i,\quad  f_i\mapsto\mathbf{f}_i,\quad t\mapsto \mathbf{t},\quad d_a \mapsto\mathbf{d}_a \quad(1\leq i\leq n-1, 1\leq a\leq n).$$
\end{itemize}
\end{prop}

\subsection{Categorification of iquantum groups}\label{subsec:iquantum-categorification}
For
$\square=D,SO$ and $\mathfrak{c}=\imath,\jmath$, put $\mathcal C=\mathrm{Semis}^{\mathfrak c,W}_\square$.
For $A\in\Theta_\square^{\mathfrak c}$, set
$$
\mathcal L_A=\widetilde{\mathrm{IC}}_A^+
  \langle d(A)-r(A)\rangle,$$
whose image under the characteristic homomorphism $\mathrm{ch}$ given in Theorem~\ref{cat-main} is $$\mathrm{ch}([\mathcal L_A])=\{A\}.$$
For diagonal $\lambda$, we have
$\mathcal L_\lambda=\widetilde{\mathrm{IC}}_\lambda^+$ and
$\mathrm{ch}([\mathcal L_\lambda])=1_\lambda$.

Let $\kappa_\square=2$ if $\square=D$, and $1$ if $\square=SO$.
Denote
\begin{align*}
\mathcal E_i&=
  \bigoplus_{\substack{B\in\Theta_\square^{\mathfrak c}\\
  B-E^\theta_{i+1,i}\ \mathrm{diagonal}}}\mathcal L_B,\quad
\mathcal F_i&=
  \bigoplus_{\substack{C\in\Theta_\square^{\mathfrak c}\\
  C-E^\theta_{i,i+1}\ \mathrm{diagonal}}}\mathcal L_C,\quad
\mathcal D_a^\pm&=
  \bigoplus_{\lambda\in\Theta_\square^{\mathfrak c,\mathrm{diag}}}
  \mathcal L_\lambda\langle\pm\kappa_\square\lambda_a\rangle,
\end{align*}
where the indices range over $1\le i\le n$, $1\le a\le n+1$ if $\mathfrak{c}=\jmath$; and $1\le i\le n-1$, $1\le a\le n$ if $\mathfrak{c}=\imath$. Since $\mathrm{ch}([L\langle m\rangle])
=v^{-m}\mathrm{ch}([L])$, the definitions of the generators give
\begin{equation}\label{eq:asscatefd}
\mathrm{ch}([\mathcal E_i])=\mathbf e_i,\quad
\mathrm{ch}([\mathcal F_i])=\mathbf f_i,\quad
\mathrm{ch}([\mathcal D_a^\pm])=\mathbf d_a^{\pm1}.
\end{equation}
For the generator $d_{n+1}^{\pm1}$, the rescaling in Proposition~\ref{prop:surjqsch} is
realized by $\mathcal D_{n+1}^\pm\langle\pm3\rangle$ in the case of $\square=D$, and by $\mathcal D_{n+1}^\pm\langle\pm1\rangle$ in the case of $\square=SO$.

For $\mathfrak c=\imath$, write
$A_\lambda=\lambda-E^\theta_{n,n}+E^\theta_{n,n+2}$.
When $\lambda_n\ge1$, the calculations in
\eqref{KL-2d}, \eqref{KL-2d+1} and \eqref{KL-O} give
$$
\mathcal T'_\lambda=
\begin{cases}
\mathcal L_{A_\lambda}\oplus\mathcal L_\lambda,
  &\square=SO,\ \lambda_n=1,\\
\mathcal L_{A_\lambda},&\text{otherwise},
\end{cases}
\qquad
\mathrm{ch}([\mathcal T'_\lambda])=\mathbf t'1_\lambda.
$$
In particular, the component with $\lambda_n=1$ in the orthogonal case
requires an additional diagonal IC object, as explained after
\eqref{KL-O}. If $\lambda_n=0$, we set
$$
\mathcal T'_\lambda=
\begin{cases}
0,&\square=2d,\\
\mathcal L_\lambda\langle1\rangle\oplus
\mathcal L_\lambda\langle-1\rangle,&\square=2d+1,SO.
\end{cases}
$$
Thus $\mathcal T'=\bigoplus_\lambda\mathcal T'_\lambda$ realizes
$\mathbf t'$. i.e.,
\begin{equation}\label{eq:asscatt}
\mathrm{ch}([\mathcal T'])=\mathbf t'.
\end{equation}
The algebraic generator $t$ maps to $\mathbf t$.
In the case of $\square=D$, it is recovered from $\mathbf t'$ by the
diagonal correction in \eqref{eq:difftt} after extending scalars
to $\mathbb Q(v)$.

By Theorem~\ref{cat-main}, convolution induces an isomorphism
$$
\mathbb Q(v)\otimes_{\mathcal A}
\bigl(K_0^\oplus(\mathcal C)/\ker(\mathrm{ch})\bigr)
\ \simeq\ {}_{\mathbb Q}\mathbf S_{n,d}^{\mathfrak c,\square}.
$$
Together with the surjections given in Proposition~\ref{prop:surjqsch},
The formulas in \eqref{eq:asscatefd} and \eqref{eq:asscatt} give a weak categorification of the specialized iquantum
groups through their finite-dimensional iSchur quotients.
Moreover, Proposition~\ref{D_C} shows that normalized Verdier
duality fixes $\mathcal E_i,\mathcal F_i,\mathcal T'$ up to
isomorphism and interchanges $\mathcal D_a^+$ and
$\mathcal D_a^-$, realizing the bar involution geometrically.

\appendix
\section{Convolution of mixed complexes on flag varieties}\label{app:mixed-convolution}

In this appendix, we study the convolution of mixed complexes on flag varieties, by which we give a proof of Lemma~\ref{prop:semis-convolution}. In fact, this procedure also justify the usual monoidal categorifications of quantum Schur algebras in the case of equal parameters by semisimple complexes, corresponding to the geometric
constructions in \cite{BLM90,BKLW18,FL15}.

\subsection{Convolution of mixed complexes}\label{app:convolution-models}
Let $\mathbf{G}$ be a split connected reductive group over $\mathbb F_q$,
and $\mathbf{P}, \mathbf{Q}_1, \mathbf{Q}_2$ its parabolic subgroups. 

For $K\in D^b_{\mathrm m,\mathbf{G}}(\mathbf{G}/\mathbf{Q}_1\times\mathbf{G}/\mathbf{P})$ and
$L\in D^b_{\mathrm m,\mathbf{G}}(\mathbf{G}/\mathbf{P}\times\mathbf{G}/\mathbf{Q}_2)$, define
\begin{equation}\label{eq:app-double-convolution}
K\star L=Rp_{13*}(p_{12}^*K\otimes^{\mathbb L}p_{23}^*L),
\end{equation}
where $p_{ij}$ are the obvious projections from $\mathbf{G}/\mathbf{Q}_1\times \mathbf{G}/\mathbf{P}\times\mathbf{G}/\mathbf{Q}_2$.
This is the unshifted convolution used in the main text (see \S\ref{subsec:convolution}).

Consider
\begin{equation}\label{eq:app-twisted-diagram}
\mathbf{G}/\mathbf{P}\times \mathbf{G}/\mathbf{Q}_2\xleftarrow{\ a\ }\mathbf{G}\times \mathbf{G}/\mathbf{Q}_2
\xrightarrow{\ b\ }\mathbf{G}\times^{\mathbf{P}}\mathbf{G}/\mathbf{Q}_2
\xrightarrow{\ m\ }\mathbf{G}/\mathbf{Q}_2,
\end{equation}
where
$$
a(h,z)=(h\mathbf{P},z),\quad b(h,z)=[h,z],\quad m([h,z])=hz.
$$
Here $\mathbf{Q}_1\times \mathbf{P}$ acts on $\mathbf{G}\times \mathbf{G}/\mathbf{Q}_2$ by
$(p,r)(h,z)=(phr^{-1},rz)$ and on $\mathbf{G}/\mathbf{P}\times\mathbf{G}/\mathbf{Q}_2$ by
$(p,r)(y,z)=(py,rz)$. The map $b$ is the quotient by $\mathbf{P}$,
and the remaining $\mathbf{Q}_1$-action on $\mathbf{G}\times^{\mathbf{P}} \mathbf{G}/\mathbf{Q}_2$ is the left multiplication on the first factor.
For $\mathcal F\in D^b_{\mathrm m,\mathbf{Q}_1}(\mathbf{G}/\mathbf{P})$ and
$\mathcal H\in D^b_{\mathrm m,\mathbf{P}}(\mathbf{G}/\mathbf{Q}_2)$, let
$\mathcal F\widetilde\boxtimes\mathcal H\in D^b_{\mathrm m,\mathbf{Q}_1}(\mathbf{G}\times^{\mathbf{P}} \mathbf{G}/\mathbf{Q}_2)$ be
the unshifted twisted external tensor product
(cf. \cite[\S7.2]{Ac21}), characterized by equivariant descent:
$$
b^*(\mathcal F\widetilde\boxtimes\mathcal H)
\simeq a^*(\mathcal F\boxtimes\mathcal H).
$$
Set
\begin{equation}\label{eq:app-single-convolution}
\mathcal F*_{\mathbf{P}}\mathcal H
=Rm_*(\mathcal F\widetilde\boxtimes\mathcal H).
\end{equation}
The $\mathbf{G}$-equivariant isomorphism $\mathbf{G}\times^{\mathbf{P}} \mathbf{G}/\mathbf{Q}_2\simeq\mathbf{G}/\mathbf{P}\times\mathbf{G}/\mathbf{Q}_2$, $[h,z]\mapsto(h\mathbf{P},hz)$,
identifies $m$ with the second projection. Hence $m$ is proper.

For parabolic subgroups $\mathbf{H}$ and $\mathbf{J}$,
let $i_{\mathbf{H},\mathbf{J}}:\mathbf{G}/\mathbf{J}\to \mathbf{G}/\mathbf{H}\times \mathbf{G}/\mathbf{J}:x\mapsto(\mathbf{H},x)$. We use the
unshifted form of the equivariant-descent equivalence in
\cite[Theorem~6.5.10]{Ac21}:
\begin{equation}\label{eq:app-raw-restriction}
\mathsf R_{\mathbf{H},\mathbf{J}}=i_{\mathbf{H},\mathbf{J}}^*\circ\operatorname{For}^{\mathbf{G}}_{\mathbf{H}}:
D^b_{\mathrm m,\mathbf{G}}(\mathbf{G}/\mathbf{H}\times \mathbf{G}/\mathbf{J})\to
D^b_{\mathrm m,\mathbf{H}}(\mathbf{G}/\mathbf{J}).
\end{equation}
\begin{prop}\label{prop:app-model-comparison}
 There is a natural isomorphism
\begin{equation}\label{eq:app-model-comparison}
\mathsf R_{\mathbf{Q}_1,\mathbf{Q}_2}(K\star L)
\simeq
\mathsf R_{\mathbf{Q}_1,\mathbf{P}}(K)*_{\mathbf{P}}\mathsf R_{\mathbf{P},\mathbf{Q}_2}(L), 
\end{equation}
which is compatible with the Frobenius structures.
\end{prop}

\begin{proof}
This is a mixed-sheaf analog of \cite[Exercise~7.7.1]{Ac21} for
partial flag varieties, with our unshifted conventions. The proof
follows the four steps outlined in
the hints to that exercise, using equivariant descent, proper base
change, and the projection formula. These constructions are defined
over $\mathbb F_q$. We omit the details.
\end{proof}

We use unshifted equivariant descent and convolution.
Thus \eqref{eq:app-model-comparison} involves no shift
or Tate twist.

\begin{prop}\label{prop:app-purity}
The equivalence $\mathsf R_{\mathbf{H},\mathbf{J}}$ preserves and reflects
purity of any given weight. Both convolutions $*_{\mathbf{P}}$ and $\star$ preserve purity of weight zero.
\end{prop}

\begin{proof}
Put $\mathbf{W}=\mathbf{G}/\mathbf{H}\times\mathbf{G}/\mathbf{J}$.
Let $\mu:\mathbf{G}\times \mathbf{W}\to \mathbf{W}$ be the diagonal action,
and let $\operatorname{pr}_2:\mathbf{G}\times \mathbf{W}\to\mathbf{W}$
be the projection.
Define
$$
\begin{aligned}
j:\mathbf{G}\times\mathbf{G}/\mathbf{J}
&\longrightarrow \mathbf{G}\times \mathbf{W},
& (g,z)&\longmapsto (g,(\mathbf{H},z)),\\
c:\mathbf{G}\times\mathbf{G}/\mathbf{J}
&\longrightarrow \mathbf{W},
& (g,z)&\longmapsto (g\mathbf{H},gz).
\end{aligned}
$$
Let $p:\mathbf{G}\times\mathbf{G}/\mathbf{J}\to\mathbf{G}/\mathbf{J}$
be the projection.
All these maps are defined over $\mathbb F_q$.
The following diagrams commute:
\[
\begin{tikzcd}[column sep=2em,row sep=2em]
\mathbf{G}\times\mathbf{G}/\mathbf{J}
  \arrow[r,"j"] \arrow[d,"c"']
  & \mathbf{G}\times \mathbf{W} \arrow[d,"\mu"]\\
\mathbf{W} \arrow[r,equal]
  & \mathbf{W}
\end{tikzcd}
\qquad
\begin{tikzcd}[column sep=2em,row sep=2em]
\mathbf{G}\times\mathbf{G}/\mathbf{J}
  \arrow[r,"j"] \arrow[d,"p"']
  & \mathbf{G}\times\mathbf{W} \arrow[d,"\operatorname{pr}_2"]\\
\mathbf{G}/\mathbf{J} \arrow[r,"i_{\mathbf{H},\mathbf{J}}"']
  & \mathbf{W} .
\end{tikzcd}
\]

For $K\in D^b_{\mathrm m,\mathbf{G}}(\mathbf W)$,
its $\mathbf{G}$-equivariant structure gives an isomorphism compatible with the Frobenius:
\[
\mu^*K\simeq\operatorname{pr}_2^*K.
\]
Pulling back along $j$ and using the diagrams gives
\[
c^*K
\simeq j^*\mu^*K
\simeq j^*\operatorname{pr}_2^*K
\simeq p^*i_{\mathbf{H},\mathbf{J}}^*K
\simeq p^*\mathsf R_{\mathbf{H},\mathbf{J}}(K).
\]
Both $c$ and $p$ are smooth and surjective.
Hence smooth pullback and descent show that $K$ is pure
of weight $w$ if and only if $\mathsf R_{\mathbf{H},\mathbf{J}}(K)$ is
pure of weight $w$.

Now let $\mathcal F$ and $\mathcal H$ be pure of
weight zero under the convolution $*_{\mathbf{P}}$.
Their external product is pure of weight zero.
The defining isomorphism
\[
b^*(\mathcal F\widetilde\boxtimes\mathcal H)
\simeq a^*(\mathcal F\boxtimes\mathcal H)
\]
and smooth pullback and descent imply that
$\mathcal F\widetilde\boxtimes\mathcal H$ is pure
of weight zero.
Since $m$ is proper,
\[
\mathcal F*_{\mathbf{P}}\mathcal H
=Rm_*(\mathcal F\widetilde\boxtimes\mathcal H)
\]
is also pure of weight zero.
Then the assertion for the convolution $\star$ follows from
Proposition~\ref{prop:app-model-comparison}
and the first assertion.
\end{proof}

\subsection{Proof of Lemma~\ref{prop:semis-convolution}}
\label{app:proof-semis-convolution}

For an $\mathbb{F}_q$ variety $\mathbf{X}$, we use the unshifted intersection complex $\IC_{\mathbf{X}}$, whose restriction
to the smooth dense stratum is $\Ql$ in degree zero, with its canonical Frobenius structure. The unshifted equivalences
$\mathsf R_{\mathbf{H},\mathbf{J}}$ identify the corresponding unshifted IC complexes
and intertwine the two convolutions by
Proposition~\ref{prop:app-model-comparison}. Thus it suffices to work for the convolution $*_\mathbf{P}$ instead of $\star$.

Choose standard representatives $\mathbf{P}, \mathbf{Q}_1,\mathbf{Q}_2$ containing a fixed Borel subgroup $\mathbf{B}$. Write $X_\mathbf{B}(w)=\overline{\mathbf{B}w\mathbf{B}/\mathbf{B}}$ for the Schubert variety corresponding to $w\in W$, and
$\pi_H:\mathbf{G}/\mathbf{B}\to \mathbf{G}/\mathbf{H}$ for a standard parabolic $\mathbf{H}$.

We use the well-known Tate decomposition for the product $*_\mathbf{B}$,
i.e., the finite-type case of \cite[Theorem~2.2.7]{dCHL18} with both
parabolic subgroups equal to $\mathbf{B}$:
\begin{equation}\label{eq:app-full-flag-tate}
\IC_{X_\mathbf{B}(u)}*_\mathbf{B}\IC_{X_\mathbf{B}(v)}
\simeq\bigoplus_{w,j}
\IC_{X_\mathbf{B}(w)}(-j)[-2j]^{\oplus m_{w,j}},
\qquad m_{w,j}\in\mathbb N.
\end{equation}
The isomorphism is over $\mathbb F_q$, and the IC complexes have their canonical
Frobenius structures.

First, $R\pi_{\mathbf{H}*}\IC_{X_\mathbf{B}(w)}$ also has an IC decomposition with
Tate multiplicity spaces. In fact, there is a Cartesian diagram
$$
\begin{tikzcd}[column sep=large,row sep=large]
\mathbf{G}/\mathbf{B}\times_{\mathbf{G}/\mathbf{H}}\mathbf{G}/\mathbf{B} \arrow[r,"r_1"] \arrow[d,"r_2"']
  &\mathbf{G}/\mathbf{B}\arrow[d,"\pi_\mathbf{H}"]\\
\mathbf{G}/\mathbf{B}\arrow[r,"\pi_\mathbf{H}"']&\mathbf{G}/\mathbf{H} .
\end{tikzcd}
$$
Identify its upper-left corner with $\mathbf{G}\times^\mathbf{B}(\mathbf{H}/\mathbf{B})$ by
$[g,h\mathbf{B}]\mapsto(g\mathbf{B},gh\mathbf{B})$. Since $\mathbf{H}/\mathbf{B}=X_\mathbf{B}(w_\mathbf{H})$ is smooth, where $w_\mathbf{H}$
is the longest element of the Weyl group $W_\mathbf{H}$ of $\mathbf{H}$, proper base change gives
\begin{equation}\label{eq:app-projection-push-pull}
\pi_\mathbf{H}^*R\pi_{\mathbf{H}*}K
\simeq Rr_{2*}r_1^*K
\simeq K*_\mathbf{B}\underline{\Ql}_{\mathbf{H}/\mathbf{B}}.
\end{equation}
For $K=\IC_{X_\mathbf{B}(w)}$, put $C=R\pi_{\mathbf{H}*}K$.
By the decomposition theorem over $\mathbb F_q$
\cite[Proposition~5.1.1]{dCHL18}, $B$-equivariance and
connectedness of orbit stabilizers, we have
$$
C\simeq\bigoplus_{D,n}\IC_D\otimes V_{D,n}[-n],
$$where $D$ ranges over the generalized Schubert varieties of $\mathbf{G}/\mathbf{H}$.
Choose $x\in D^\circ(\mathbb F_q)$ and a rational lift
$y\in\pi_\mathbf{H}^{-1}(x)(\mathbb F_q)$, where $D^\circ$ is the smooth locus of $D$.
Since $(\IC_D)_{x}=\Ql$, $V_{D,n}$ is a direct summand of
$$
\mathsf H^n(C)_{x}
\simeq\mathsf H^n(\pi_\mathbf{H}^*C)_{y}.
$$
By \eqref{eq:app-full-flag-tate},
\eqref{eq:app-projection-push-pull} and the stalk properties
of IC complexes of ordinary Schubert variety, this space vanishes for odd $n$
and is a direct sum of $\Ql(-n/2)$ for even $n$. Moreover, $R\pi_{\mathbf{H}*}$ preserves the weight.
Thus $C$ has a split Tate IC decomposition of the form
\eqref{eq:app-full-flag-tate}.

Now let $\mathbf{X}_1\subset \mathbf{G}/\mathbf{P}$ be a $\mathbf{Q}_1$-orbit closure and
$\mathbf{X}_2\subset \mathbf{G}/\mathbf{Q}_2$ a $P'$-orbit closure. Let $\widetilde{\mathbf{X}}_1\subset \mathbf{G}$
be the inverse image of $\mathbf{X}_1$. Denote $\boldsymbol{\mathcal{Z}} =\widetilde {\mathbf{X}}_1\times^{\mathbf{P}}\mathbf{X}_2$. The map
$$m:\boldsymbol{\mathcal{Z}}\to \mathbf{G}/\mathbf{Q}_2,\quad [g,z]\mapsto gz
$$
is projective, because $\boldsymbol{\mathcal Z}$ is a closed subvariety of
$\mathbf{G}/\mathbf{P}\times \mathbf{G}/\mathbf{Q}_2$ under $[g,z]\mapsto(g\mathbf{P},gz)$.
Compatibility of intersection complexes with external products and
smooth pullback gives
\begin{equation}\label{eq:app-twisted-ic}
\IC_{\mathbf{X}_1}\widetilde\boxtimes\IC_{\mathbf{X}_2}\simeq\IC_{\boldsymbol{\mathcal Z}}
\quad\text{and}\quad
\IC_{\mathbf{X}_1}*_{\mathbf{P}}\IC_{\mathbf{X}_2}\simeq Rm_*\IC_{\boldsymbol{\mathcal Z}};
\end{equation}
cf. \cite[Lemma~4.1.2 \& the proof of Theorem~2.2.1]{dCHL18}.

Choose $w_1,w_2$ such that
$\pi_{\mathbf{P}}^{-1}(\mathbf{X}_1)=\mathbf{X}_\mathbf{B}(w_1)$ and $\pi_{\mathbf{Q}_2}^{-1}(\mathbf{X}_2)=\mathbf{X}_\mathbf{B}(w_2)$.
They are the longest representatives of the corresponding double
cosets. Let $\widetilde{\mathbf{X}}_\mathbf{B}(w_1)\subset\mathbf{G}$ be the inverse image of $\mathbf{X}_\mathbf{B}(w_1)$. Denote
$\boldsymbol{\mathcal Y}=\widetilde{\mathbf{X}}_{\mathbf{B}}(w_1)\times^{\mathbf{B}}\mathbf{X}_\mathbf{B}(w_2)$, which satisfies the following commutative
diagram
\begin{equation}\label{eq:app-full-flag-reduction}
\begin{tikzcd}[column sep=large,row sep=large]
\boldsymbol{\mathcal Y}\arrow[r,"\alpha"]\arrow[d,"p_\mathbf{B}"']
  &\boldsymbol{\mathcal Z}\arrow[d,"m"]\\
\mathbf{G}/\mathbf{B}\arrow[r,"\pi_{\mathbf{Q}_2}"']&\mathbf{G}/\mathbf{Q}_2 ,
\end{tikzcd}
\qquad
\begin{gathered}
\alpha([g,h\mathbf{B}])=[g,h\mathbf{Q}_2],\\
p_\mathbf{B}([g,h\mathbf{B}])=gh\mathbf{B} .
\end{gathered}
\end{equation}
By the convolution formula \eqref{eq:app-full-flag-tate} and the preceding projection argument, respectively, we have
\begin{align*}
Rp_{\mathbf{B}*}\IC_{\boldsymbol{\mathcal Y}}
&\simeq\bigoplus_{w,j}\IC_{X_\mathbf{B}(w)}(-j)[-2j]^{\oplus m_{w,j}},\\
R\pi_{\mathbf{Q}_2*}\IC_{\mathbf{X}_\mathbf{B}(w)}
&\simeq\bigoplus_{D,a}\IC_D(-a)[-2a]^{\oplus b_{w,D,a}},
\end{align*}
where $D$ ranges over Schubert varieties in $\mathbf{G}/\mathbf{Q}_2$.
Composing these two decompositions gives a split Tate IC
decomposition of $R\pi_{\mathbf{Q}_2*}Rp_{\mathbf{B}*}\IC_{\boldsymbol{\boldsymbol{\mathcal Y}}}$.
Since $\alpha$ is proper and surjective, combining this with
\cite[Theorem~2.1.2]{dCHL18} applied to $\alpha$ gives a split Tate IC
decomposition of $M=Rm_*\IC_{\boldsymbol{\mathcal Z}}$.
Since all morphisms in \eqref{eq:app-full-flag-reduction} are
$\mathbf{Q}_1$-equivariant, all summands in the above decomposition are
$\mathbf{Q}_1$-equivariant.
Moreover, convolution preserves purity of weight zero by Proposition~\ref{prop:app-purity}. Thus, we arrive at Lemma~\ref{prop:semis-convolution} by applying Proposition~\ref{prop:app-model-comparison} over $\mathbb F_{q^2}$.

\begin{rem}
Similarly to the proof of Lemma~\ref{prop:semis-convolution}, we obtain the known (though possibly never written down) result that the geometric constructions in \cite{BLM90,BKLW18,FL15} admit the usual interpretation in terms of Lusztig-type convolution categories generated by split equivariant intersection cohomology complexes and their normalized shifts. 
With appropriate normalizations, their graded split Grothendieck groups recover the corresponding quantum Schur algebras and their canonical bases. As an application, a normalized Verdier duality similar to that in \S\ref{subsec:bar-involution} gives the bar involutions on these quantum Schur algebras. This shows that the bar involution defined in \cite{BKLW18} is indeed an $\mathbb Z$-algebra involution, whereas the statement therein only asserts that the bar is an anti-linear involution.
\end{rem}

\section{Negative coefficients and IC descents}\label{AppendixB}
This appendix shows, taking the unitary case as an example, that the positive descent objects do not form a monoidal subcategory. The negative descent objects and exchange objects occur naturally in the convolution of two positive ones. This phenomenon forces us to consider the category $\mathrm{Semis}^{\mathfrak{c},W}_\square$, which contains negative and exchange objects to categorify the iSchur algebras in this paper. Throughout, we work over $\overline{\mathbb F_q}$. All varieties are
equipped with $F$-structures, and all complexes are equipped with Weil
structures with respect to $F$.
\subsection{A small resolution in the unitary case}
\label{sec:fc-common}

In this section, we take $n=1$, $d=2$ and $D=2d=4$ in the unitary $\jmath$-case.
The partial flag varieties
\[
V=\overline{\mathbb F_q}^4,\qquad X=\Flag(1,3;V),\qquad Y=\Gr(2,V)
\]
have types $(1,2,1)$ and $(2,0,2)$, dimensions $5$ and $4$, and split
$\Fqq$-forms $\Flag(1,3;\Fqq^4)$ and $\Gr(2,\Fqq^4)$.
The Frobenius is given by $F(L,H)=(H^\perp, L^\perp)$ and
$F(U)=U^\perp$ for $(L,H)\in X$ and $U\in Y$. 

Put
\begin{equation}\label{eq:fc-T}
T=\{((L,H),(L',H'))\in X\times X \mid L\subset H',\ L'\subset H\},
\end{equation}
which is the closure of the orbit $\mathcal{O}_{A_T}$ with $A_T=\sm{0&1&0\\1&0&1\\0&1&0}$,
and let $\pi\colon\widetilde T\to T$ forget $U$, where
\begin{equation}\label{eq:fc-resolution}
\widetilde T=\{((L,H),(L',H'),U)\in X\times X\times Y \mid
 L+L'\subset U\subset H\cap H'\}.
\end{equation}

Denote $\widetilde{\mathrm{IC}}^{\pm}_T:=\widetilde{\mathrm{IC}}^\pm_{A_T}$ for short. We shall always use an orbit closure (instead of the associated index matrix) as a subscript to denote the corresponding IC complex.

\begin{prop}\label{prop:fc-small}
The map $\pi$ is a nontrivial $F$-compatible small resolution.
Moreover, $\dim T=\dim\widetilde T=8$, and
\begin{equation}\label{eq:fc-small-IC}
R\pi_*\const{\widetilde T}\simeq\tIC_{T}^+
\end{equation}
equips with the canonical descent structures.
\end{prop}

\begin{proof}
Over $U\in Y$, the projection $\widetilde T\to Y$ has fiber
$(\PP^1)^4$, given by $L,L'\subset U$ and $H,H'\supset U$. Hence
$\widetilde T$ is smooth and irreducible of dimension $8$.
The map $\pi$ is projective and surjective. If $L\ne L'$, then
$U=L+L'$; if $H\ne H'$, then $U=H\cap H'$. Thus $\pi$ is an isomorphism
off the diagonal $\Delta_X$, over which its fiber is
$\PP(H/L)\cong\PP^1$. As $\operatorname{codim}_T\Delta_X=3>2$, the map
$\pi$ is small. Since $\perp$ reverses inclusions and preserves
dimension two, the defining condition of $\widetilde T$ becomes
$H^\perp+(H')^\perp\subset U^\perp\subset L^\perp\cap(L')^\perp$.
Hence
\[
F((L,H),(L',H'),U)=((H^\perp,L^\perp),((H')^\perp,(L')^\perp),U^\perp)
\]
defines an $\Fq$-form of $\widetilde T$ with $\pi\circ F=F\circ\pi$.
Equation~\eqref{eq:fc-small-IC} follows from the IC extension property
; the isomorphism is the identity on
the dense orbit, so it respects the positive descent structures.
\end{proof}

We also use the smooth orbit closures
\[
R=\{((L,H),U)\in X\times Y \mid L\subset U\subset H\},
\qquad C=X\times Y,
\]
namely the closed orbit, of dimension $6$, and the closure of the open
orbit, of dimension $9$. Their indices are
$A_R=\sm{1&0&0\\1&0&1\\0&0&1}$ and $A_C=\sm{0&0&1\\1&0&1\\1&0&0}$.
The
normalized objects are
\begin{equation}\label{eq:fc-normalizations}
E_T^+=\tIC_T^+\sh3,\qquad
E_R^+=\tIC_R^+\sh1,\qquad
E_C^\pm=\tIC_C^\pm\sh4.
\end{equation}
Let $\varepsilon$ be the rank-one Weil module on which the Frobenius acts
by $-1$. It becomes trivial over $\Fqq$.

\subsection{A convolution containing negative IC descents}
\label{sec:fc-negative}

\begin{prop}\label{prop:fc-negative}
There is an isomorphism
\begin{equation}\label{eq:fc-negative-product}
E_T^+\diamond E_C^+\simeq
E_C^+\sh{\pm3}
\oplus\bigl((E_C^+)^{\oplus2}\oplus E_C^-\bigr)\sh{\pm1},
\end{equation}
where $K\sh{\pm a}$ means $K\sh a\oplus K\sh{-a}$. In particular,
$E_C^-\sh1$ and $E_C^-\sh{-1}$ occur.
\end{prop}

\begin{proof}
Since $C=X\times Y$, the convolution is computed on
$T\times_XC=T\times Y$ via
\begin{equation}\label{eq:fc-negative-diagram}
\begin{tikzcd}[column sep=large]
\widetilde T\times Y\arrow[r,"\pi\times\mathrm{id}_Y"]
 \arrow[rr,bend right=18,"g"']
&T\times Y\arrow[r,"m"]&X\times Y,
\end{tikzcd}
\end{equation}
where $m(x,x',U)=(x,U)$ and $g(x,x',W,U)=(x,U)$. As
$\pi\times\mathrm{id}_Y$ is small and $F$-compatible and $\tIC_C^+$ is
constant, \eqref{eq:fc-small-IC} and proper base change give
$\tIC_T^+\diamond\tIC_C^+\simeq Rg_*\const{\widetilde T\times Y}$.

We have $g=g'\times\mathrm{id}_Y$, where $g'\colon\widetilde T\to X$,
$(x,x',W)\mapsto x$, is $\mathrm{GL}(V)$-equivariant. Writing
$X=\mathrm{GL}(V)/P_x$ with $P_x$ the stabilizer of $x=(L,H)$, we get
$\widetilde T\simeq\mathrm{GL}(V)\times^{P_x}Z_x$, where
\begin{equation}\label{eq:fc-negative-fiber}
Z_x=\{(W,L',H')\mid L\subset W\subset H,\ L'\subset W\subset H'\}.
\end{equation}
Since $\mathrm{GL}(V)\to X$ is Zariski-locally trivial and $Z_x$ is
smooth and projective, $g$ is a Zariski-locally trivial smooth
projective fibration with fiber $Z_x$. As $P_x$ is a connected
parabolic subgroup, the local systems $R^ig_*\const{\widetilde T\times Y}$
are geometrically constant.

Fix $x\in X^F$. Via $W$, $Z_x$ is the fiber product of two
$\PP^1$-bundles over $\PP(H/L)\cong\PP^1$, parametrizing $L'\subset W$
and $H'\supset W$. Let $h$ be the pullback of the hyperplane class of
$\PP(H/L)$, and let $a,b$ be the pullbacks of the hyperplane classes
under $(W,L',H')\mapsto L'\in\PP(H)$ and
$(W,L',H')\mapsto H'/L\in\PP((V/L)^*)$. Frobenius acts on $Z_x$ by
$(W,L',H')\mapsto(W^\perp,(H')^\perp,(L')^\perp)$, so
$F(h)=qh$, $F(a)=qb$ and $F(b)=qa$. By the projective bundle formula,
odd cohomology vanishes, and the $\varepsilon$-parts are spanned by
$a-b$ and $h(a-b)$:
\begin{equation}\label{eq:fc-negative-Weil}
\begin{tabular}{|c|c|c|c|c|}
\hline
degree&$0$&$2$&$4$&$6$\\
\hline
basis of $H^r_{\et}(Z_x)$&$1$&$h,a,b$&$ha,hb,ab$&$hab$\\
$H^r_{\et}(Z_x)$&$\Ql$&$(\Ql^{\oplus2}\oplus\varepsilon)(-1)$
 &$(\Ql^{\oplus2}\oplus\varepsilon)(-2)$&$\Ql(-3)$\\
$\mathrm{tr}(F)$&$1$&$q$&$q^2$&$q^3$\\
\hline
\end{tabular}
\end{equation}
The convolution has only full-support IC summands, with these
multiplicity spaces. The input shifts add up to $\sh7$ and
$E_C^\pm=\tIC_C^\pm\sh4$, so degree $2j$ contributes
$E_C^\pm\sh{3-2j}$. The proof is complete.
\end{proof}

\subsection{A convolution containing an exchange descent}
\label{sec:fc-exchange}

In $C=X\times Y$, the subvarieties
\[
D_1=\{((L,H),U)\mid L\subset U\},\qquad
D_2=\{((L,H),U)\mid U\subset H\}
\]
are smooth $\PP^2$-bundles over $X$ of dimension $7$ with
$D_1\cap D_2=R$. The geometric orbits are $O=C\setminus(D_1\cup D_2)$,
$D_i^\circ=D_i\setminus R$ for $i=1,2$, and $R$; the indices of
$D_1^\circ$ and $D_2^\circ$ are
$B_1=\sm{1&0&0\\0&0&2\\1&0&0}$ and $B_2=\sm{0&0&1\\2&0&0\\0&0&1}$.
Since $L\subset U$ if and only if $U^\perp\subset L^\perp$, Frobenius
exchanges $D_1^\circ$ and $D_2^\circ$, and $B_1^\vartheta=B_2$.
Put $\omega=\{B_1,B_2\}$ and $\cP_\omega=P_\omega\sh2$.

\begin{prop}\label{prop:fc-exchange}
There is an isomorphism
\begin{equation}\label{eq:fc-exchange-product}
E_T^+\diamond E_R^+
 \simeq E_C^+\oplus\cP_\omega\oplus E_R^+\sh{\pm1}.
\end{equation}
Thus an exchange summand occurs.
\end{prop}

\begin{proof}
Let $Z=T\times_XR$, formed with the second flag of $T$ and the first
flag of $R$, and $\widetilde Z=\widetilde T\times_TZ$:
\begin{equation}\label{eq:fc-exchange-diagram}
\begin{tikzcd}[column sep=large,row sep=large]
&\widetilde Z\arrow[r]\arrow[d,"\rho"']\arrow[dl,"f"']
 &\widetilde T\arrow[d,"\pi"]\\
X\times Y&Z\arrow[l,"m"]\arrow[r]&T
\end{tikzcd}
\end{equation}
Here the square is Cartesian, $m(x,x',U)=(x,U)$ and $f=m\circ\rho$.
Since $R\to X$ is a $\PP^1$-bundle, so is $\widetilde Z\to\widetilde T$. Hence $\widetilde Z$ is smooth of dimension $9$, and $\rho$, a smooth
base change of $\pi$, is small. All maps are $F$-compatible, and proper
base change gives $\tIC_T^+\diamond\tIC_R^+\simeq Rf_*\const{\widetilde Z}$.

The fiber of $f$ over $((L,H),U)$ consists of the triples $(W,L',H')$
with
\begin{equation}\label{eq:fc-exchange-fiber}
L\subset W\subset H,\qquad L'\subset U\cap W,\qquad U+W\subset H'.
\end{equation}
Thus $U\cap W\ne0$. If $\dim(U\cap W)=1$, then $L'=U\cap W$ and
$H'=U+W$; if $W=U$, then $L'\in\PP(U)$ and $H'\in\PP((V/U)^*)$ are
arbitrary. On $O$, $J=U\cap H$ is a line different from $L$, and
$U\cap W\subset J$ forces $W=L+J$, $L'=J$ and $H'=U+L$. On $D_1^\circ$,
where $U\cap H=L$, and on $D_2^\circ$, where $U\subset H$, every
$W\in\PP(H/L)$ occurs, with $(L',H')=(L,U+W)$ and $(U\cap W,H)$, respectively. On $R$, the choices with $(L',H')=(L,H)$ form $\PP(H/L)\cong\PP^1$, those with $W=U$ form
$\PP(U)\times\PP((V/U)^*)\cong\PP^1\times\PP^1$, and the two meet at $(U,L,H)$. For a point fixed by $F^2$ in each stratum, this gives the following fibers and cohomology with the split $\Fqq$-Weil structure;
all other cohomology vanishes:
\begin{equation}\label{eq:fc-exchange-fiber-table}
\begin{tabular}{|c|c|c|c|c|}
\hline
stratum&fiber&$H^0_{\et}$&$H^2_{\et}$&$H^4_{\et}$\\
\hline
$O$&$\mathrm{pt}$&$\Ql$&$0$&$0$\\
$D_1^\circ,\ D_2^\circ$&$\PP^1$&$\Ql$&$\Ql(-1)$&$0$\\
$R$&$\PP^1\cup_{\mathrm{pt}}(\PP^1\times\PP^1)$
 &$\Ql$&$\Ql(-1)^{\oplus3}$&$\Ql(-2)$\\
\hline
\end{tabular}
\end{equation}
The four orbit closures are smooth and the geometric stabilizers are connected, so the equivariant local systems are geometrically
constant. By the split Tate convolution decomposition, the stalks determine the multiplicities successively:
\begin{equation}\label{eq:fc-exchange-split}
\tIC_T\star\tIC_R\simeq
\tIC_C\oplus(\tIC_{D_1}\oplus\tIC_{D_2})\sh{-2}
\oplus\tIC_R\sh{-2}\oplus\tIC_R\sh{-4}.
\end{equation}

Over an $\Fq$-rational point of $O$, the fiber is an $\Fq$-rational
point, so the full-support summand has positive descent. The two $D_i$-summands are exchanged by $F$ and form $P_\omega\sh{-2}$. At $x\in R^F$, Frobenius acts on the fiber by $(W,L',H')\mapsto(W^\perp,(H')^\perp,(L')^\perp)$. It preserves both components and their intersection point, and exchanges the two factors of $\PP^1\times\PP^1$. Mayer--Vietoris sequence gives the first row of the table below. The exchange summand has two constant stalks
interchanged by $F$, so it contributes $(\Ql\oplus\varepsilon)(-1)$ in
degree two; subtracting it and the full-support contribution gives the
last row:
\begin{equation}\label{eq:fc-exchange-descent-table}
\begin{tabular}{|c|c|c|c|}
\hline
degree&$0$&$2$&$4$\\
\hline
$H^\bullet_{\et}(f^{-1}(x))$&$\Ql$
 &$(\Ql^{\oplus2}\oplus\varepsilon)(-1)$&$\Ql(-2)$\\
full-support summand&$\Ql$&$0$&$0$\\
exchange summand&$0$&$(\Ql\oplus\varepsilon)(-1)$&$0$\\
summands supported on $R$&$0$&$\Ql(-1)$&$\Ql(-2)$\\
\hline
\end{tabular}
\end{equation}
We can read from \eqref{eq:fc-exchange-split} and \eqref{eq:fc-exchange-descent-table} that $$\tIC_T^+\diamond\tIC_R^+\simeq
\tIC_C^+\oplus P_\omega\sh{-2}
\oplus\tIC_R^+\sh{-2}\oplus\tIC_R^+\sh{-4}.$$ The input
shifts add up to $\sh4$, which turns the summands into $E_C^+$,
$P_\omega\sh2=\cP_\omega$ and $E_R^+\sh{\pm1}$. The proof is complete.
\end{proof}

\end{document}